\documentclass[11pt]{amsart}

\usepackage[utf8]{inputenc}
\usepackage[T1]{fontenc}
\usepackage[english]{babel}

\usepackage{amsaddr}

\usepackage{amsthm}
\usepackage{amsmath}
\usepackage{amssymb}
\usepackage{mathrsfs}
\usepackage{dsfont}

\usepackage{graphicx}

\usepackage{color}
\definecolor{marin}{rgb}   {0.,   0.1,   0.5} 
\definecolor{rouge}{rgb}   {0.8,   0.,   0.} 
\definecolor{sepia}{rgb}   {0.4,   0.25,   0.} 
\definecolor{mag}{rgb}   {0.3,   0,   0.3} 

\usepackage[colorlinks,citecolor=sepia,linkcolor=marin,urlcolor=mag ,bookmarksopen,bookmarksnumbered]{hyperref}

\newcommand{\ic}{\mathrm{i}}

\newcommand{\ul}{\underline{\boldsymbol{\ell}}}
\newcommand{\us}{\underline{\boldsymbol{\sigma}}}
\newcommand{\uq}{\underline{q}}

\newcommand{\bili}{\mathcal{M}}

\newcommand{\cste}{}

\newcommand{\bk}{\boldsymbol{k}}
\newcommand{\bm}{\boldsymbol{m}}
\newcommand{\bi}{\boldsymbol{i}}
\newcommand{\bj}{\boldsymbol{j}}
\newcommand{\bh}{\boldsymbol{h}}
\newcommand{\bl}{\boldsymbol{\ell}}
\newcommand{\bs}{\boldsymbol{\sigma}}
\newcommand{\bI}{\mathbb{I}}
\def\Z{{\mathbb Z}}

\def\R{{\mathbb R}}
\def\C{{\mathbb C}}
\newcommand{\N}{\mathbb N}
\newcommand{\bro}{\boldsymbol{\rho}}

\newtheorem{theorem}{Theorem}[section]
\newtheorem{corollary}[theorem]{Corollary}
\newtheorem{lemma}[theorem]{Lemma}
\newtheorem{proposition}[theorem]{Proposition}
\newtheorem{definition}[theorem]{Definition}
\newtheorem{remark}[theorem]{Remark}
\newtheorem{example}[theorem]{Example}

\newcommand\quotientthree[2]{{^{\displaystyle #1}}/{_{\displaystyle #2}}}

\begin{document}

\title[Full dimensional invariant tori for NLS]{
Lagrangian invariant tori for completely resonant nonlinear Schr\"odinger equations
}

\author{Joackim Bernier$^*$, Michela Procesi$^{\dagger}$}

\thanks{$^*$Nantes Universit\'e, CNRS, Laboratoire de Math\'ematiques Jean Leray, LMJL, F-44000 Nantes, France}
\thanks{$^\dagger$Dipartimento di Matematica e Fisica, Universit\`a Roma Tre, Roma, 00146, Italy}



\email{joackim.bernier@univ-nantes.fr}

\email{michela.procesi@uniroma3.it}

\begin{abstract}

We study small solutions of the parameterless nonlinear Schr\"odinger equations on the circle. In this regime, the nonlinearity is perturbative and the system is thus nearly integrable. We prove the persistence, in the presence of nonlinearity, of full-dimensional tori. We justify this denomination by proving that they are Lagrangian in the sense that, although they are not submanifolds (since they are compact), at every point their tangent cone is a Lagrangian subspace. This is the first result of this type for an equation without an infinity of external random parameters. 
We achieve this by establishing new fine decay and uniformity properties of the amplitude-to-frequency  map. This is the main technical novelty of this paper. It allows us to modulate all the frequencies and thus obtain non-resonance conditions without any assumption on the choice of the tangential sites.

\end{abstract}

\maketitle

\setcounter{tocdepth}{1} 
\tableofcontents

\section{Introduction}

We are interested in the dynamics of  solutions, generated by small and smooth initial data, to analytic nonlinear Schr\"odinger equations on $\mathbb{T} = \mathbb{R}/2\pi \mathbb{Z}$
\begin{equation}
\label{eq:NLS}
\tag{NLS}
\ic \partial_t u + \partial_x^2 u =  f(|u|^2)u, \quad x\in \mathbb{T}
\end{equation}
where $u(t,x)\in \mathbb{C}$, $f : \mathbb{R} \to \mathbb{R}$ is analytic in a neighborhood of the origin. Moreover, without loss of generality, we assume that $f(0) = 0$.  Because of its simplicity, this equation is an emblematic model for studying the stability of nearly integrable systems without external parameters \cite{KP96}.

\medskip

Indeed, since we consider small solutions, the nonlinearity is, at least for some time, perturbative. It is therefore natural to ask whether the solutions behave, for arbitrarily long times, like those of the linear Schr\"odinger equation
\begin{equation}
\label{eq:LS}
\tag{LS}
\ic \partial_t u + \partial_x^2 u = 0, \quad x\in \mathbb{T}.
\end{equation}
The solutions to \eqref{eq:LS} can be written as
$$
u(t,x) = \sum_{k\in \mathbb{Z}} e^{\ic (t k^2 + k x)}  u_k(0), \ x\in \mathbb{T}, \, t\in \mathbb{R},
$$
where the Fourier coefficients are defined by
\begin{equation}
\label{fourier}
\forall k\in \mathbb{Z}, \quad u_k := \int_{\mathbb{T}} u(x) e^{-\ic k x}\frac{\mathrm{d}x}{2\pi}.
\end{equation}

\medskip

This is a subtle problem as we are considering perturbations of a fully resonant linear oscillatory integrable equation (its frequencies are the integers $k^2$ with $k\in \mathbb{Z}$). In this paper, we always assume the existence of a cubic term in \eqref{eq:NLS}, that is
$$
f'(0) \neq 0.
$$
Thanks to an observation due to Kuksin--P{\"o}schel \cite{KP96}, this term is known to provide stability for solutions. It induces a slight modulation of the frequencies by the initial data which prevents most of the resonance phenomena from occurring (e.g. turbulence phenomena for $x\in \mathbb{T}^2$ \cite{CKSTT10}).  Indeed, it has been proven in \cite{Bou00,BFG20,BC24} the so called almost global stability of small, smooth and random initial data to \eqref{eq:NLS}. That is, given any $r$ chosen arbitrarily large, provided that $s$ is sufficiently large and $\varepsilon$ is small enough, most of the initial data of size $\varepsilon$ in $H^s$ (drawn with standard probability laws) generate solutions to \eqref{eq:NLS} which satisfy
\begin{equation}
\label{eq:almost_pres_actions}
\forall t\leq \varepsilon^{-r}, \quad \sum_{k\in \mathbb{Z}} \langle k \rangle^{2s} \big| |u_k(t)|^2 - |u_k(0)|^2 \big| \leq \| u(0) \|_{H^s}^{2+\nu}
\end{equation}
where $\nu\in (0,2]$ is a given number. This is the almost preservation, for very long times, of the actions of \eqref{eq:LS} by the flow of \eqref{eq:NLS}. These results have been extended in higher dimension in \cite{BC24}, for perturbations of the Korteweg–de Vries and Benjamin--Ono equations  in \cite{BG21}, for Kirchhoff equations in \cite{LX26} and for longer time scale in \cite{LX24} for solutions in Gevrey spaces. Let us also mention that, for algebraic reasons, in the completely integrable case for which $f''=0$, the actions of \eqref{eq:LS} are almost preserved for all times and all sufficiently small initial data \cite{GK14}.

\medskip

Actually, a simple and naive analogy with respect to what is known in Kolmogorov--Arnold--Moser (KAM) theory for finite dimensional Hamiltonian systems suggests that one might hope, even in the non integrable case, to remove this time restriction  to prove the almost global preservation of the actions for all times. Such a result has been recently proven in \cite{BCGP24}, for solutions in Gevrey space, when adding a random convolution potential to \eqref{eq:NLS}
\begin{equation}
\label{eq:NLS*}
\tag{NLS*}
\ic \partial_t u + \partial_x^2 u =  V \ast u + f(|u|^2)u, \quad x\in \mathbb{T} \quad \mathrm{where} \quad \forall k\in \mathbb{Z}, \ V_k \in  \mathbb{R} .
\end{equation}
 However, apart from this particular case, this problem is still fully open. Thus it is natural to wonder for which kind of initial data such a result holds.

\medskip

To gain further insight, it is more meaningful to reformulate the problem in terms of the existence of invariant tori. Indeed, solutions to \eqref{eq:LS} live on straight tori
$$
\mathbf{T}_{\xi} := \Big\{ \sum_{k\in \mathbb{Z}} u_k e^{\ic k x} \ | \ \forall k\in \mathbb{Z}, \ |u_k|^2 = \xi_k  \Big\}, \quad \xi \in \mathbb{R}^{\mathbb{Z}}_+
$$
and a natural question in KAM theory is to wonder which of these tori remain invariant, being slightly deformed, in presence of the nonlinear perturbation. Since we only consider tori which are slightly deformed, picking initial data on such invariant tori, we get solutions whose actions are almost preserved for all times.

\medskip

Thirty years ago Kuksin--P{\"o}schel \cite{KP96} proved that most of the finite dimensional tori survive. They are slight deformations of straight tori $\mathbf{T}_\xi$ in which the support of $\xi$ is finitely supported and its coefficients are sufficiently small with respect to the size of the support. Many other similar results concerning the existence of finite-dimensional invariant tori have been obtained for nearly integrable Hamiltonian PDEs (see e.g. \cite{Kuk87, Way90, CW93, Bou94, Pos96} for the earliest results for 1d PDEs, \cite{Bou98, EK10, BB13, PP15, Wan16, EGK16} for some results in higher dimension, \cite{BBHM18, BHM23, FG24,FMT25} for some results about quasi-linear equations and \cite{BKM18} for large amplitude solutions) and often at the cost of adding external parameters. 
Regarding the existence of infinite dimensional tori, the results are relatively few, we mention  \cite{Pos02,Bou05,GX13,BMP20,CY21,Con23,CGP23,BCGP24,WY26} which all deal with models with external parameters and convolution potentials (with the notable exception of \cite{Bou96,Pos02}, where the potential is multiplicative).
Recently, Bernier--Grébert--Robert \cite{BGR24} proved the existence of infinite dimensional invariant tori for \eqref{eq:NLS}. They are slight deformations of some tori $\mathbf{T}_\xi$ given by some fast decaying sequences $\xi$ with an infinite but (very) lacunary support. This was the only result of its kind for PDEs without external parameters (to the notable exception of completely integrable equations and incompressible Euler equations on $\mathbb{T}^d$, $d\in \{3\} \cup 2 \mathbb{Z}$, for which there exists explicit infinite dimensional invariant tori \cite{FM24}).

\medskip

In this paper, we  remove this lacunarity assumption. We prove the existence of invariant tori which are deformations of $\mathbf{T}_\xi$ for sequences $\xi$ with arbitrary support (which include sequences $\xi$ with full support and so of full dimensional tori). The significance of this result is, of course, highly dependent on the precise definition of 'deformation' (e.g. in $H^s$ topology all infinite dimensional tori are homeomorphic to $\mathbb{T}^\mathbb{N}$ equipped with the product topology). We introduce a natural notion of \emph{symplectic deformation} which in particular applies for the tori we construct. It turns out that the number of normal sites (i.e. the cardinal of the complementary of the support of $\xi$) is associated with a quantity which is invariant by this kind of transformation. This allows us to claim that
\begin{center}
\emph{there exist full dimensional invariant tori for \eqref{eq:NLS}.}
\end{center}
Actually, we prove that they are \emph{Lagrangian} in the sense that,  even though they are not differential submanifolds (because they are infinite dimensional and compact), at any point their tangent cone is a Lagrangian vector space.

\medskip

To the best of our knowledge, this is the first geometric definition of full dimensional tori in the context of Hamiltonian PDEs. While this terminology is not new — it has been used in many papers concerning existence of smooth solutions for equations with convolution potentials like \eqref{eq:NLS*} %\mic { I would move this before \cite{Pos02,Bou05,GX13,BMP20,CY21,Con23,CGP23,BCGP24,WY26}} (
%with the notable exception of \cite{Bou96},\cite{Pos02}, where the potential is multiplicative.
% \mic{also Poschel is multiplicative...even if the proof is very sketched} 
 — the fact that a full set of frequencies is considered (i.e., that non-lacunary tori are treated) was not especially more challenging\footnote{to the notable exception of the construction infinite dimensional tori of non smooth solutions for which only the existence of lacunary tori is known \cite{BMP23}.} and was not previously highlighted with a precise definition. %\mic{there are various discussions on this topic in \cite{BMP20}, however they are not really pertinent,so I will not cite them here.}

\medskip

This situation contrasts sharply with that of equations with no (or very few) parameters. The only way to obtain non resonance conditions (i.e. small divisor estimates) is to try to modulate the frequencies through the amplitude-frequency map \cite{Bou05}. In \cite{BB11} the authors used this ``twist'' property to construct a finite sequence of  tori of increasing dimension, in the spirit of \cite{Pos02}; they could not construct an infinite sequence of tori converging to an almost-periodic solution due to lack of sufficient information on the amplitude-frequency map.
	 This is a very delicate point which is the main technical advance of this paper. It was to partially overcome this issue that only lacunary tori were considered in \cite{BGR24}. It requires to identify subtle cancellation properties of the frequencies and their derivatives in small divisors. Moreover, it requires to prove dispersive smoothing effects at very low regularity in normal forms techniques in order to invert the amplitude-frequency map (up to some canceling terms). Our estimates are sharp in the sense that, as for the completely integrable case $f''=0$ \cite[Corollary 18.2 page 84]{Mol16} (but for completely different reasons), we manage to reduce the proof of the invertibility of the amplitude-frequency map to the boundedness of a kind of discrete Hilbert transform.

\medskip

Finally, we emphasize that, although our tori are full-dimensional, they are not maximal in the sense that there is no reason to expect them to satisfy Bourgain’s estimate \cite{Bou05}
$$
\forall v\in \mathcal{T},\forall k\in \mathbb{Z}, \quad 0< \sqrt{\xi_k}/2\leq  |v_k| \leq  2\sqrt{\xi_k}
$$
where $\mathcal{T}$ is an invariant torus which is a deformation of $\mathbf{T}_\xi$. 
Technically, it requires dealing with sequences decaying as little  as possible, in the sense that their decay has to match with the one given by the regularity of the phase space (which is limited by the algebra property). 
Constructing such invariant tori for \eqref{eq:NLS} would add serious issues related to the singularity of the divisors for parameterless equations \cite{Bou00,Bou05}.
At the moment, existence of such maximal tori have only be proven for equations with random convolution potentials \cite{Bou05,BMP20,CY21,Con23,CGP23,BCGP24}. % and their existence for nearly integrable PDEs without parameters remains an open very challenging problem.  \jo{Should write this now that there are AIs ? I am very tempted to cut this sentence after the reference. A human would understand that this is a good open problem but maybe not an AI...}\mic{Yes maybe cut...}
In the case of infinite dimensional lattice systems there are some results on persistence (and also on the measure) of invariant maximal tori without the use of external parameters, we mention for example \cite{CP95,CGP24,DFP26}, this is mainly due to the fact that one may impose conditions on the non-linearity which dramatically simplify the study of the amplitude-frequency modulation.

\subsection*{Notations}
{The Japanese bracket is defined by $\langle a \rangle := \sqrt{1+a^2}$ for $a\in \mathbb{R}$.  We shall use the notation $A\lesssim B$ to denote $A\le C B$ where $C$ is a positive constant
depending on  parameters fixed once for all. We will emphasize by writing $\lesssim_{\alpha}$ when the constant $C$ depends on some other parameter $\alpha$ and $A \sim_\alpha B$ if $A\lesssim_\alpha B \lesssim_\alpha A$.  We denote $\mathbb{R}^*_+$ the set of the positive integers. Given a property $P$, $\mathds{1}_P$ is equal to $1$ if $P$ is true and $0$ else. If $a,b\in \mathbb{R}$, we denote $a\wedge b := \min(a,b)$ and $a\vee b := \max(a,b)$. {Given a normed real vector space $E$, $\mathscr{L}(E)$ denotes the space of the continuous linear maps from $E$ to $E$.}

Given $S\subset \mathbb{Z}$ we always identify $\mathbb{C}^{\mathcal{S}}$ as the subspace of $\mathbb{C}^{\mathbb{Z}}$ whose all components indexed by $\mathcal{S}^c := \mathbb{Z} \setminus \mathcal{S}$ vanish.
}

\subsection*{Acknowledgments}   We warmly thank Benoît Grébert for valuable  discussions about preliminary versions of this work and in particular on the notion of action T\"oplitz Hamiltonians. J.B. thanks Baptiste Chantraine, St\'ephane Guillermou, and R\'emi Leclercq for many enriching and stimulating discussions on geometry, particularly around the notion of Lagrangian subsets.   M.P. thanks Luca Biasco, Livia Corsi and Roberto Feola for helpful discussions. J.B. was partially supported by the ANR projects ANR-22-CE40-0016 and ANR-26-CE40-5907-01.  M.P. was partially supported by Miur: PRIN 2022FPZEES  and GNAMPA-INdAM, ''Growth of Sobolev norms and energy cascades in Dispersive Hamiltonian PDEs'' Cup: E53C25002010001.  
 
\subsection*{Declaration of AI Use} We did not use any AI for the mathematical content or for brainstorming. However, we used one to assist us with translating some explanatory paragraphs in English. 

\section{Main results} \label{sec:main}
\subsection{More context}
To state our main result, we need a phase space. Since we only consider perturbations of tori $\mathbf{T}_{\xi}$ where $\xi$ is fast decaying, it turns out that the choice of the phase space does not play a crucial role for us. So for simplicity and to avoid introducing additional notation, we only consider solutions to \eqref{eq:NLS} in the phase space
$$
H^s:= \{ u \in L^2(\mathbb{T};\mathbb C)  \ | \ \|  u \|_{H^s}^2 := \sum_{k \in \mathbb{Z}} |u_k|^2 \langle k \rangle^{2s} <\infty\} \quad \mathrm{with} \quad s\geq 3.
$$
Thus, straight tori $\mathbf{T}_\xi$ are subsets of $H^s$ that are naturally parametrized by sequences $\xi$ in
$$
\ell^1_{2s}(\mathbb{Z};\mathbb{R}_+) =\Big\{ \xi \in \mathbb{R}_+^\mathbb{Z} \ | \ \sum_{k\in \mathbb{Z}} \langle k \rangle^{2s}\xi_k <\infty \Big\}\,,\quad \mathbb R_+:=\{x\in \mathbb R\,|\quad x\geq 0\}.
$$
Given $\xi \in \ell^1_{2s}(\mathbb{Z};\mathbb{R}_+) $, the support of $\xi$ (also called \emph{tangential sites} of $\mathbf{T}_\xi$) is naturally defined by
$$
\mathrm{supp} \, \xi := \{ k\in \mathbb{Z} \ | \ \xi_k \neq 0\}.
$$
As in \cite{BGR24}, we could easily consider more general phase spaces (for example spaces of analytic or Gevrey functions). This would prove that the almost periodic solutions are analytic in the space variable. Nevertheless, it would just be a direct extension of the result and would require less standard notations.

\medskip

We also need to recall the structure of the equation. Indeed, \eqref{eq:NLS} is an infinite dimensional Hamiltonian system. It can be written as
$$
\ic \partial_t u = \nabla H^{\eqref{eq:NLS}}(u)
$$
where $\nabla$ denotes the $L^2(\mathbb{T})$ gradient (see subsection \ref{sub:ham_form} for more details)
\begin{equation}
\label{eq:def_Ham_NLS}
H^{\eqref{eq:NLS}}(u) = \frac12 \int_{\mathbb{T}}  |\partial_x u|^2 + F(|u|^2) \frac{\mathrm{d}x}{2\pi},
\end{equation}
and $F$ is the primitive of $f$ vanishing at the origin. It is thus a Hamiltonian system on the weak symplectic space $(H^s,\varpi_0)$ (this notion is recalled in subsection \ref{sub:weak_symp}) where $\varpi_0=(\ic \,\cdot, \cdot)_{L^2}$ is the usual symplectic form on $L^2$ and $(\cdot,\cdot)_{L^2}$ denotes the usual real $L^2$ scalar product (recalled in subsection \ref{sub:ham_form} below). 

\medskip

As explained in the introduction, an important issue is to precisely define what full dimensional tori are. One could consider very strong definitions based, for example, on the results of \cite{Bou05,BMP20}, but they do not seem suitable for the tori we construct in this paper. We therefore propose a natural and fairly minimalist definition that generalizes the notion of Lagrangian torus. More precisely, we only consider invariant tori that are symplectic deformations of the straight tori $\mathbf{T}_\xi$.

\begin{definition}[Symplectic deformation] \label{def:symp_def} Given a weak symplectic space $(E,\varpi)$ and $X,Y\subset E$, a map $\tau : X \to Y$ is a symplectic deformation if there exists $L \in C^0(X ;\mathscr{L}(E))$, such that
\begin{itemize}
\item $\tau$ is a homeomorphism,
\item for all $x\in X$, $L(x)$ is symplectic invertible and we have
\begin{equation}
\label{eq:lim_sup}
\limsup_{\substack{x' \to x \\ X\ni  x'\neq x}}  \frac{\| \tau(x') - \tau(x) - L(x)(x'-x) \|_{E}}{\|x-x'\|_{E}} = 0.
\end{equation}
\end{itemize}
We say that the map $L$ is a derivative of $\tau$ and that $Y$ is a symplectic deformation of $X$.
\end{definition}
\begin{remark} Intuitively, the last property means that $\tau$ is as differentiable as possible in  $x$ and $L(x)$ is an extension of its derivative. We require $L$ to be a kind of Fr\'echet derivative but everything below would also work if we mimicked the notion of Gateaux derivative (although this would be more cumbersome). We do not currently have any potential applications in mind for the Gateaux derivative approach.
\end{remark}

This notion is studied and discussed in detail in Section \ref{sec:symp_def} below. It is particularly well-suited because it generates an equivalence relation on the subsets of $H^s$, and we can prove (Proposition \ref{prop:geom_key}) that the number of normal sites is invariant under such transformations. It is then natural to define full dimensional tori as follows.

\begin{definition}[Full dimensional tori] A subset $\mathcal{T} \subset H^s$ is a \emph{full dimensional torus} if it is a symplectic deformation of a straight torus $\mathbf{T}_\xi$, $\xi \in \ell^1_{2s}(\mathbb{Z};\mathbb{R}_+)$ , with no normal sites (i.e. satisfying $\mathrm{supp} \, \xi = \mathbb{Z}$).
\end{definition}

This definition is also geometrically natural because we show in Section \ref{sec:symp_def} that these tori are Lagrangian in the sense that at each point their tangent cone is a Lagrangian space.

\subsection{The result} Thanks to these notions, we can now state the main result of this paper. This is a result of existence of arbitrarily small invariant tori for \eqref{eq:NLS} based on arbitrary sets of tangential sites. We prove that that the smaller they are the closer they are to the straight tori (relatively to their size). {The proof is done in Section \ref{sec:proof}.}

%\mic{I think that some mention of the hypotheses on the equation should appear. we say them in the introduction but they are not grouped and maybe  it is better to make the result more sefl contained....} \jo{It is a matter of taste, I do not use to do it but I fully understand your point, so ok for me  :)}
\begin{theorem} \label{thm:main} 
{We consider \eqref{eq:NLS} with $f$ an analytic function satisfying $f(0)=0$ and $f'(0)\neq 0$.}
	For all  infinite set $\mathcal{S}_\infty \subset \mathbb{Z}$ and $\varepsilon \lesssim_{\mathcal{S}_\infty} 1$, there exists $\xi \in \ell^1_{2s}(\mathbb{Z};\mathbb{R}_+)$, a sympectic deformation $\Psi :  \mathbf{T}_\xi \to \mathcal{T} \subset H^s$ and frequencies $\omega \in \mathbb{R}^{\mathcal{S}_\infty}$ such that
\begin{itemize}
\item  $\mathcal{T}$ is an invariant set on which the flow of \eqref{eq:NLS} is conjugated to a rotation of frequencies $\omega$, i.e.
\begin{equation}
\label{eq:thm_inv_tor}
\forall v\in \mathbf{T}_\xi,\quad  \displaystyle t\mapsto \Psi ( \sum_{j\in \mathcal{S}_\infty} v_j e^{-\ic \omega_j t} e^{ijx} ) \ \mathrm{is \ a \ global\ solution\ of\ \eqref{eq:NLS}},
\end{equation} 
\item $\mathcal{S}_\infty$ is the set of the tangential sites of $\mathbf{T}_\xi$
$$
\mathrm{supp} \, \xi = \mathcal{S}_\infty 
$$
\item $\mathbf{T}_\xi$ is of size $\varepsilon$ and $\Psi$ is close to the identity
$$
\varepsilon^2/4 \leq \| \xi \|_{\ell^1_{2s}}\leq 2\varepsilon^2 \quad \mathrm{and} \quad \forall v\in \mathbf{T}_\xi, \quad \| \Psi(v) - v\|_{H^s} \leq \varepsilon^{5/4},
$$
\item the frequencies are close to those of \eqref{eq:LS}, 
$$
\forall j\in \mathcal{S}_\infty, \quad |\omega_j - j^2|\leq \varepsilon^{1/4}.
$$ 
\end{itemize}
\end{theorem}
Since this theorem includes the case $\mathcal{S}_\infty= \mathbb{Z}$, we deduce that
\begin{corollary} There exists full dimensional invariant tori for \eqref{eq:NLS}.

\end{corollary}

\subsection{Comments about the result}
We group here some comments regarding our result. For an informal presentation of the main strategic ideas see the subsection \ref{proofidee}.
\\
Let us start with a brief comparison with the results of \cite{BGR24}. The two main differences are that
\begin{itemize}
	\item In our setting the non-linearity $f$  is just analytic, while  in \cite{BGR24} $f$ must be entire (and with some quantitative control) so that one may perform Wick renormalization.
	\item  Our result  applies to {\em any choice of tangential sites} $\mathcal S_\infty$, while in \cite{BGR24} only lacunary sets were considered. This is the core of the paper and requires (among other things) a sharp control of the amplitude to frequency map. We refer the reader to subsections \ref{smalldiv}  and \ref{melnikov} for details.
\end{itemize}
To evidence the fact that the invariant tori  we construct are intrinsically different from the ones of \cite{BGR24}, we introduced the definition of symplectic deformation and of full dimensional  tori. In essence we need some way to characterize tori in terms of the ''number of normal modes''

\begin{itemize}
	\item The main point is that our result allows to construct, and distinguish,  invariant tori with arbitrary choices of $\mathcal S_\infty$, see subsection \ref{num_normal} for more detail.
	\item Theorem \ref{thm:main}  also holds if $\mathcal{S}_\infty$ is finite but, in this case, it is just a corollary of \cite{KP96}.
\end{itemize}

 To make the result more readable, we formulated our main Theorem \ref{thm:main}  mainly in terms of the invariant tori. Several direct consequences follow by rephrasing in our setting the corresponding proofs of \cite{BGR24}:
\begin{itemize}
	\item  We could prove that arbitrarily close to most of the finite-dimensional invariant tori constructed in \cite{KP96}, there exist full-dimensional invariant tori. %This is a fairly direct byproduct of the proof which is done in \cite{BGR24} for infinite dimensional tori. 
	\end{itemize}
%	We did not state it here to avoid overloading the statement.
	Our construction actually produces something stronger: we show the existence of a set of higher-dimensional (but finite) tori of asymptotically full measure, accumulating on most of the tori constructed in \cite{KP96}. We then show that the same holds for the latter, and so on. We thus construct families of tori of arbitrarily large dimension, accumulating on one another (a finite number of steps in this direction was done in \cite{BB11}, for example). It is, moreover, by taking the limit in this procedure that we obtain infinite-dimensional tori (see Section \ref{sec:disc} for a more detailed presentation of the proof).
	 \begin{itemize}
	 \item 
	Proceeding just as in \cite{BGR24} one sees that, up to multiplication by a scalar, there exists at most one non trivial vanishing rational linear combination of the frequencies. It follows that, as in \cite{BGR24}, solutions to \eqref{eq:NLS} on these tori are almost periodic functions which are not quasi periodic. Moreover their trajectories are dense and ergodic in sub-invariant tori of co-dimension\footnote{in some weak but natural sense} at most $1$. 
\end{itemize}
In fact, adding a term $m u$ to \eqref{eq:NLS}, we could prove that for almost all $m \in \mathbb{R}$, the frequencies are rationally independent. 
\begin{itemize}
	\item The invariant tori of Theorem \ref{thm:main} are obtained as limit of quadratically stable finite dimensional invariant tori. Thus they certainly satisfy a form of stability property. We did not go further in this interesting direction which would  deserve a proper study. We mention also \cite{CLSY18} for a  stability result for maximal tori for the NLS with external parameters.
	\item In our construction the invariant tori are parameterized by sequences $\xi\in \ell^1_{2s}(\Z,\mathbb R_+)$ which turn out to be super-exponentially decreasing. We remark that this does not imply that the solutions in \eqref{eq:thm_inv_tor} have space Fourier coefficients which decay super-exponentially; in fact in our construction we can only say that solutions are in $H^s$.  Since the choice of the $H^s$ setting is purely for reasons of convenience, we could prove that our solutions are in fact analytic (but similarly to the quasi-periodic case, we do not expect higher regularity). 
	\end{itemize}

\section{Formalism and discussions about the proof}
\label{sec:disc}
In this section we give a brief presentation of the formalism and of the new ideas in the paper. 
\subsection{Lebesgue spaces and Hamiltonian formalism}

\label{sub:ham_form} 

As usual, discrete Lebesgue spaces are defined, for $1\leq p<\infty$, by
 $$
 \ell^{p} := \{ u\in \mathbb{C}^{\mathbb{Z}} \ | \ \|u\|_{\ell^p}^p := \sum_{k\in \mathbb{Z}}  |u_k|^p <\infty \}.
 $$
 We extend these definitions for sequences indexed by subsets of $\mathbb{Z}$. Recalling that we identify functions with their sequence of Fourier coefficients, the Fourier--Plancherel isometry writes
$$
\forall u \in \ell^2, \quad \|u\|_{\ell^2}^2 = \int_{\mathbb{T}} |u(x)|^2 \frac{\mathrm{d}x}{2\pi} =: \| u\|_{L^2}^2.
$$ 
We equip $\ell^2$ of the standard real scalar product $(u,v)_{\ell^2}$  defined by
$$
\forall u,v\in \ell^2(\mathbb{T}), \quad  (u,v)_{\ell^2} :=  \Re \sum_{k\in \mathbb{Z}} u_k \overline{v_k} = \Re \int_{\mathbb{T}} u(x) \overline{v(x)} \frac{\mathrm{d}x}{2\pi}.
$$
We define the Lebesgue norms by
$$
\forall q\geq 1, \quad \| u\|_{L^{2q}}^{2q} := \int_{\mathbb{T}} |u|^{2q} \frac{\mathrm{d}x}{2\pi}
$$
Given an open subset $\mathcal{U} \subset H^s$, a smooth function $P :\mathcal{U} \to \mathbb{R}$ and $u\in \mathcal{U}$, its gradient $\nabla P(u)$ is the unique element of $h^{-s}$  satisfying
$$
\forall v\in H^s, \ (\nabla P(u),v)_{\ell^2} = \mathrm{d}P(u)(v).
$$
It can be easily checked that
\begin{equation}
\label{eq:formula_grad}
\forall k \in \mathbb{Z}, \quad (\nabla P(u))_k = 2\partial_{\overline{u_k}} P(u):= (\partial_{\Re u_k} +\ic \partial_{\Im u_k} )P(u).
\end{equation}
We always equip $H^s$ with the usual symplectic form $(\ic\cdot,\cdot)_{\ell^2}$.  Therefore \eqref{eq:NLS} is a Hamiltonian system as it writes
\begin{equation}
\label{eq:sympl_struct}
\ic\partial_t u = \nabla H^{\eqref{eq:NLS}}(u), \quad \mathrm{with} \quad H^{\eqref{eq:NLS}}(u) := \frac12 \int_{\mathbb{T}} |\partial_x u(x)|^2 + F(|u(x)|^2) \frac{\mathrm{d}x}{2\pi} 
\end{equation}
where $F$ is the primitive of $f$ vanishing at the origin.
Finally, we recall that a $C^1$ map $\tau : \mathcal{U} \to H^s$ is \emph{symplectic} if
$$
\forall u \in \mathcal{U}, \forall v,w \in H^s, \quad (\ic v,w)_{\ell^2} = (i\mathrm{d}\tau(u)(v),\mathrm{d}\tau(u)(w))_{\ell^2}.
$$
Moreover, if  $H,K :  \mathcal{U}  \to \mathbb{R}$ are two smooth functions such that $\nabla H$ (or $\nabla K$) is  $ H^s$ valued then the \emph{Poisson bracket} of $H$ and $K$ is defined by
$$
\{  H,K\}(u):= (\ic \nabla H(u),\nabla K(u))_{\ell^2}.
$$
Note that, as usual, it also satisfies
\begin{equation}
\label{eq:Poisson_formula}
\{  H,K\}
  =2\ic \sum_{k\in \mathbb{Z}} \partial_{\overline{u_k}}H \partial_{u_k} K - \partial_{u_k}H \partial_{\overline{u_k}} K.
\end{equation}

\subsection{Global constants}
Now, we fix some objects for the rest of the paper 
\begin{itemize}
\item a strictly increasing sequence of finite subsets $(\mathcal{S}_p)_{p\geq 0}$ of $\mathbb{Z}$,
$$
\emptyset =: \mathcal{S}_0 \varsubsetneq \mathcal{S}_1 \varsubsetneq \mathcal{S}_2 \varsubsetneq \cdots \varsubsetneq \mathcal{S}_\infty  \subset \mathbb{Z}, \quad \mathrm{where} \quad \mathcal{S}_\infty = \bigcup_{p\geq 1} \mathcal{S}_p.
$$
\item a sufficiently fast decaying sequence of radii $(r_p)_{p \geq 0} \in (\mathbb{R}_+^*)^{\mathbb{N}}$. 
\item  a minimal radius of analyticity that we arbitrarily fix as
$$
\eta_{\min} := 10.
$$
\item sets in which we will draw our internal parameters
$$
\mathcal{O}_p^{\mathrm{full}} := \prod_{1\leq p' \leq p }  \big(\frac{r_{p'}^{2(1-\delta)}}2,  r_{p'}^{2(1-\delta)}\big)^{\mathcal{S}_{p'} \setminus \mathcal{S}_{p'-1}} \subset \mathbb{R}^{\mathcal{S}_p}, \quad \mathrm{where} \quad \delta = \frac1{74}.
$$ 
The choice of $\delta$ is rather arbitrary it just has to be small enough. {As explained in the notation subsection, we identify $\mathcal{O}_p^{\mathrm{full}}$ with the subset of $\mathbb{R}^\mathbb{Z}$ of sequences vanishing on $\mathcal{S}_p^c$ and whose restriction to $\mathcal{S}_p$ belongs to $\mathcal{O}_p^{\mathrm{full}}$.} %I checked that we have at least the constraint \delta^{-1} > 29
\item in the proof, to avoid irrelevant constant, we fix arbitrarily 
$$
f'(0) = -1
$$
 but its precise value and its sign plays no role (we only use that it is different from $0$).
\end{itemize}
In the paper, we will careful track the dependencies with respect to $p$ and $r$ but not with respect to $\mathcal{S}$. We will impose a finite number of constraints on the coefficients $r_p$ (that $r_1$ is small enough and that $r_{p+1}$ is sufficiently small with respect to $r_p$ for all $p\geq 1$). We did not try to optimize such constraints, which depend on the sequence $\mathcal S_p$. However, by construction, the sequence will be at least  superexponentially decaying (due to a condition of the type $r_{p+1} \le r_p^\alpha$, with $\alpha > 1$). Moreover  one can easily verify that $r_{p+1}\to 0$ as the cardinality of $\mathcal S_{p+1}\setminus \mathcal S_p$ increases to infinity.

\subsection{Discussions about the proof}\label{proofidee}  In this subsection we present some of the most important aspects of the proof.

\subsubsection{General strategy} 

The general strategy of the proof relies on the so called P\"oschel's approach (in reference to \cite{Pos02}). It consists in proving the existence of infinite dimensional invariant tori as limit of sequences of finite dimensional ones whose dimension goes to $+\infty$. The sequences are constructed inductively. Each torus is obtained as a higher dimensional perturbation of the previous one using Kolmogorov--Arnold--Moser (KAM) theory.
The strength of this method is its flexibility because the finite dimensional tori can be constructed arbitrarily close from each other. So it provides infinity of smallness parameters\footnote{here the radii $(r_p)_p$ mentioned above} which are essentially at the end the amplitudes of the infinite dimensional torus. 
Unfortunately,  this flexibility is also its main limitation because if --as is always the case in the literature-- the amplitudes decay too fast the torus cannot be maximal and so typical. This approach was followed in \cite{Pos02,GX13,WY26} for PDEs with external parameters and in \cite{BGR24} in the  much more delicate parameterless case.

More concretely, in our context it consists in proving the existence (for some exponents $\alpha_1,\alpha_2,\alpha_3>0$ which are explicit in the paper), by induction, of Borel sets $\mathcal{O}_p \subset \mathcal{O}_{p-1} \times  \big(\frac{r_{p}^{2(1-\delta)}}2,  r_{p}^{2(1-\delta)}\big)^{\mathcal{S}_{p} \setminus \mathcal{S}_{p-1}}  \subset \mathcal{O}_p^{\mathrm{full}}$  of relative measure larger than  $ 1- r_p^{\alpha_1}$,
and for all $\xi \in \mathcal{O}_p $ 
\begin{itemize}
\item of a symplectic change of variable $\tau_{\xi}^{(p)}$, close to the identity on a $r_p^{\alpha_2}$ neighborhood of $\mathbf{T}_\xi$,
\item of frequencies $W_\xi^{(p)} \in \ell^\infty(\mathbb{Z};\mathbb{R})$ of size $\| W_\xi^{(p)} \|_{\ell^\infty} \leq r_p^{\alpha_3}$ 
\end{itemize}
such that, setting $\xi^{(q)} = (\xi_{j})_{j\in \mathcal{S}_q}$,
$$
\forall v\in \mathbf{T}_\xi, \ t\mapsto \tau_{\xi^{(1)}}^{(1)} \circ \cdots \circ \tau_{\xi^{(p)}}^{(p)} \left(  \sum_{k\in \mathbb{Z}} e^{\ic \omega_k(\xi) t } v_k e^{-\ic k x} \right) \quad \mathrm{solves \ \eqref{eq:NLS} }
$$
where
$$
\omega_k^{(p)}(\xi) := k^2 + ( W_{\xi^{(1)}}^{(1)} +\cdots +W_{\xi^{(p)}}^{(p)})_k.
$$
Finally, we pass to the limit by proving that $\tau_{\xi^{(1)}}^{(1)}, \cdots,\tau_{\xi^{(p)}}^{(p)}$  are sufficiently close to the identity so that $\tau_{\xi^{(1)}}^{(1)} \circ \cdots \circ \tau_{\xi^{(p)}}^{(p)}$ converges (this is done in Section \ref{sec:conv}).

\subsubsection{Small divisors and twist condition} \label{smalldiv}

The main difficulty when working only with internal parameters is to obtain small divisor estimates because the frequencies have to be modulated through the amplitude-frequency map. In particular, this requires a very fine understanding of this map.

More precisely, in Pöschel's approach, to pass from step $p-1$ to $p$ using KAM, we have to solve cohomological equations involving small divisors of the form (among others)
$$
\bk \cdot \omega^{(p)}(\xi) := \sum_{j\in \mathcal{S}_p} \bk_j \omega_{j}^{(p)}(\xi) \quad \mathrm{where} \quad \bk \in \mathbb{Z}^{\mathcal{S}_p} \setminus \{0\}.
$$
We have to prove a quantitative lower bound on them for most of the values of $\xi$. The simplest and most natural way to do this is to prove a lower bound on the derivative of the amplitude-frequency map $\omega^{(p)}$, that is
\begin{equation}
\label{eq:trouve_e}
\forall \bk \in \mathbb{Z}^{\mathcal{S}_p} \setminus \{0\}, \exists \mathbf{e} \in \mathbb{R}^{\mathcal{S}_p} \setminus \{0\}, \quad \bk \cdot \mathrm{d} \omega^{(p)}(\xi)(\mathbf{e})  = \sum_{i,j\in \mathcal{S}_p} \bk_i \mathbf{e}_j \partial_{\xi_j} \omega^{(p)}_i (\xi) \gtrsim \|\mathbf{e}\|_{\ell^2}
\end{equation}
uniformly in $\xi$. In practice, $\xi$ lives on a Borel set and $\omega^{(p)}$ is only Lipschitz, so we have to give a weak sense to these derivatives, but this is standard in KAM theory.

The naive estimates we can prove on these derivatives (in the spirit of \cite{KP96}) are of the form\footnote{In fact, it is $\partial_{\xi_j} \omega^{(p)}_i (\xi) + 2$ that takes this form. We omit this term $2$ because it somewhat obscures the discussion, even though it poses no difficulty. It can, for example, be removed by an obvious gauge transformation.}
$$
\partial_{\xi_j} \omega^{(p)}_i  = \mathds{1}_{i=j} + M_{i,j}^{(p)} \quad \mathrm{with} \quad |M_{i,j}^{(p)}| \leq C_p r_1^{\alpha_4},
$$
for an explicit exponent $\alpha_4 > 0$ and a constant $C_p \geq 1$ depending only on $p$. This implies that, if $r_1$ is small enough, the coefficients of the matrix of $\mathrm{d} \omega^{(p)}(\xi)$ are close to those of the identity matrix. The identity matrix comes from the cubic term of the nonlinearity. However, to obtain \eqref{eq:trouve_e}, we need more: we would like $\mathrm{d} \omega^{(p)}(\xi)$ to be close to the identity in a standard matrix norm (i.e., typically associated with a $\ell^q$ space). In other words, we need the matrix $M^{(p)}$ to be small in a standard matrix norm. But the size of the matrix is $\# \mathcal{S}_p \times \# \mathcal{S}_p$, so to ensure that $M$ is small in a matrix norm, we have to assume that $r_1$ is smaller than a constant that goes to $0$ as $\# \mathcal{S}_p$ goes to $+\infty$. As a consequence, these estimates are sufficient to obtain finite-dimensional tori but not infinite-dimensional ones. To go further, decay or summability properties of the coefficients of $\mathrm{d}\omega^{(p)}$ have to be understood.

A significant progress was made in \cite{BGR24}. By leveraging dispersive smoothing effects, the authors managed to show a decomposition of the form
\begin{equation}
\label{eq:vieille_dec}
\partial_{\xi_j} \omega^{(p)}_i  = D_i \mathds{1}_{i=j} + M_{i,j}^{(p)} + A_j^{(p)} \quad \mathrm{with} \quad |D_i-1| + |A_j| \leq r_1^{\alpha_4}  \quad \mathrm{and} \quad   |M_{i,j}| \leq \Gamma_{i,j} r_1^{\alpha_4}
\end{equation}
$$
\mathrm{with} \quad \Gamma_{i,j} := \left( 1 \wedge \frac{\langle i \rangle^{4}}{\langle j\rangle} \wedge \frac{\langle j \rangle^{4}}{ \langle i\rangle} \right)  \vee \frac{1}{ \langle i\rangle}.
$$
The $A_j^{(p)}$ terms do not decay but they depend only on $j$. The point is that, due to the gauge invariance of \eqref{eq:NLS} (which is equivalent to the preservation of the mass $\|u\|_{L^2}^2$), it actually suffices to prove \eqref{eq:trouve_e} for coefficients $\bk \in \mathbb{Z}^{\mathcal{S}_p} \setminus \{0\}$ satisfying $\sum_{j\in \mathcal{S}_p} \bk_j = 0$. Thus, the contribution of the $A^{(p)}$ term vanishes in \eqref{eq:trouve_e}, and so it can be ignored. It suffices to prove that $M^{(p)}$ is small in a well-chosen matrix norm. {In \cite{BGR24} the authors prove a bound $\| M^{(p)}\|_{\ell^1 \to \ell^1} \lesssim r_1^{\alpha_4}$ (uniformly in $p$) by exploiting  the decay of the coefficients $\Gamma_{i,j}$ and requiring  sparsity condition on $\mathcal S_\infty$, mainly used to  control the  $\langle i \rangle^{4}, \langle j \rangle^{4}$ factors. Note that typically the distance between two sites should grow like $\exp \circ \exp$. Consequently  this estimate was sufficient to prove the existence of infinite-dimensional tori but not full-dimensional ones.} 
%\\
%	Thanks to the decay of the coefficients $\Gamma_{i,j}$, it is proven in \cite{BGR24} that $\| M^{(p)}\|_{\ell^1 \to \ell^1} \lesssim r_1^{\alpha_4}$ (uniformly in $p$). However, this decay being rather weak (mainly due to the $\langle i \rangle^{4}, \langle j \rangle^{4}$ factors), they needed a sparsity assumption to obtain this uniform estimate (typically, that the distance between two sites grows like $\exp \circ \exp$). In any case, this estimate was sufficient to prove the existence of infinite-dimensional tori but not full-dimensional ones.

In this paper, we obtain a much finer decomposition. We succeed in proving (in Section \ref{sec:res_coho}) a decomposition of the form (up to simpler terms)
\begin{equation}
\label{eq:jolie_dec}
\begin{split}
\partial_{\xi_j} \omega^{(p)}_i  = {D_i} \mathds{1}_{i=j} + M_{i,j}^{(p)} + A_j^{(p)} + B_i^{(p)} \quad \mathrm{with}  \\
{|D_i-1| +}  |A_j^{(p)}| + |B_i^{(p)}| \leq r_1^{\alpha_4}  \quad \mathrm{and} \quad |M_{i,j}^{(p)}| \leq \frac{r_1^{\alpha_4}}{1 + |i| + |j|}.
\end{split}
\end{equation}
As previously, we can ignore the contributions of $A^{(p)}, B^{(p)}$ in \eqref{eq:trouve_e} by imposing $\sum_{j\in \mathcal{S}_p} \bk_j = 0$ and $\mathbf{e} = \bk$  (which is a very natural choice). Then, the smallness of the matrix $M^{(p)}$ in the $\| \cdot \|_{\ell^2 \to \ell^2}$ norm simply follows from the following estimate (see Lemma \ref{lem:Hilb}):
$$
\forall u, v \in \ell^2, \quad \Big|  \sum_{i,j\in \mathbb{Z}}  \frac{u_i v_j}{1 + |i| + |j|} \Big| \lesssim \| u\|_{\ell^2} \| v\|_{\ell^2}.
$$
This is the key that allows us to prove the existence of full-dimensional tori. Of course, obtaining a decomposition of the form \eqref{eq:jolie_dec} is non trivial and is one of the main technical novelty of this paper.

% \mic{In our estimate with the $\ell^2$ norms there are some factors depending on $p$ right? (to have a completely uniform bound one would need $\ell_1$, should we put $\lesssim_p$ in (9)?} \jo{I do not  really see what you mean. For me bounds are uniform in $p$ here. In particular the implicit constant we prove for (9) is $1/2$. Nevertheless, I agree with you that it would not be a problem if it depended on $p$.} \mic{sorry, what I meant to say is that in the measure estimates if you know the $\ell^2$-bound of (9) then you will get some p dependence, which you would not get if you were using $\ell^1$ operator norms, but of course this is not a problem...the thing that I had in mind is that maybe we should explain the difference between having an $\ell^1$ operator bound as in \cite{BGR24} or an $\ell^2$ one. Maybe it is better to say this inside the paper}
%\jo{Yes, I think it is a remark we could do when do the measure estimates. I do not think it is crucial because in any case the constant is harmless in any case, it is $\sqrt{\# \mathcal{S}_p}$, so it is bounded by $\|\bk\|_{\ell^1}^{1/2}$. A reference is [Projection bodies, Bourgain, Lindenstrauss,  pp 250–270  \href{https://link.springer.com/chapter/10.1007/BFb0081746}{https://link.springer.com/chapter/10.1007/BFb0081746}]. It is surprisingly good ! And very helpful to get maximal tori...   } \mic{very interesting}

\subsubsection{Smoothing effects and Melnikov's conditions}\label{melnikov} One important difference between the decomposition \eqref{eq:vieille_dec} of \cite{BGR24} and the new decomposition \eqref{eq:jolie_dec} is the absence of the large factors $\langle i \rangle^{4}$ and $\langle j \rangle^{4}$. Let us first explain where they came from. The point in \eqref{eq:vieille_dec} was to obtain the small factors $\langle i \rangle^{-1}, \langle j \rangle^{-1}$. They come from the smoothing properties of \eqref{eq:NLS}, i.e., that up to a gauge transform, the flow of \eqref{eq:NLS} is a $1$-smoothing perturbation of the identity (see, e.g., \cite{BGR24}). To prove this smoothing property, one typically uses the fact that when one averages the nonlinearity by the flow of the linear part of \eqref{eq:NLS} (or when solving a cohomological equation), we gain one derivative. Technically, this materializes through large divisor estimates of the form
\begin{equation}
\label{eq:gd_old}
\langle \bs_1 \bl_1^2 + \cdots +\bs_q \bl_q^2\rangle \gtrsim |\bl_1| \langle \bl_3 \rangle^{-2} q^{-1}
\end{equation}
where $q \geq 2$, $\bl \in \mathbb{Z}^{q}$, $\bs \in \{-1,1\}^q$ satisfy $|\bl_1| \geq \cdots \geq |\bl_q|$, 
\begin{equation}
\label{eq:assump_gd}
\bs_1 \bl_1 + \cdots +\bs_q \bl_q = 0 \quad \mathrm{and} \quad (\bl_1,\bs_1) \neq (\bl_2,-\bs_2)
\end{equation}
and $\bl_j:=0$ if $q<j$. The one-smoothing property is due to the factor $|\bl_1|$ in \eqref{eq:gd_old}, whereas the factor $\langle \bl_3 \rangle^{-2}$ forces us to work in higher regularity spaces. In particular, it is the cause of the $\langle i \rangle^{4}$ and $\langle j \rangle^{4}$ factors in \eqref{eq:jolie_dec}. The assumption $(\bl_1,\bs_1) \neq (\bl_2,-\bs_2)$ explains why the smoothing effect holds only up to a gauge transform.

This discussion might surprise some readers familiar with KAM theory. Usually, there is no need to consider Melnikov conditions of order higher than $2$. In other words, in Pöschel's approach, when considering an equation with external parameters as in \cite{Pos02,GX13,WY26}, one can assume to pass from step $p-1$ to step $p$ that
$\bl$ contains at most two sites not in $\mathcal S_p^c$, so that
$$
|\bl_3| \leq \max_{j \in \mathcal{S}_{p}} |j| = \mathcal{O}(r_{p}^{0^+}).
$$
In other words, the estimate \eqref{eq:gd_new} allows us to gain one derivative at the cost of only an arbitrarily small power of the homogeneity parameter {$r_p$}. This loss can be absorbed without issue, which means that somehow the $\langle \bl_3 \rangle^{-2}$ factor could be ignored. However, the situation is slightly different here. The parameters must be extracted from the initial conditions through the cubic terms of the nonlinearity. As a result, for fairly natural homogeneity reasons, we also need to consider the third Melnikov condition, and so we only have
\begin{equation}
\label{eq:what_we_only_have}
|\bl_4| \leq \max_{j \in \mathcal{S}_{p}} |j| = \mathcal{O}(r_{p}^{0^+}).
\end{equation}

To remedy this point, we prove a stronger estimate than \eqref{eq:gd_old}. More precisely, we prove (Lemma \ref{lem:gain_one_der}) that, under the same assumptions \eqref{eq:assump_gd}, by convexity arguments, we have 
\begin{equation}
\label{eq:gd_new}
\langle \bs_1 \bl_1^2 + \cdots +\bs_q \bl_q^2\rangle \gtrsim |\bl_1| \langle \bl_4 \rangle^{-2} q^{-1}
\end{equation}
This estimate is much stronger for us because of \eqref{eq:what_we_only_have}.
In other words, the estimate \eqref{eq:gd_new} allows us to gain one derivative at the cost of only an arbitrarily small power of the homogeneity parameter. This loss can be absorbed without issue, which explains why in \eqref{eq:jolie_dec} we have the gain $(1 + |i| + |j|)^{-1}$ without the large numerator as in \eqref{eq:vieille_dec}.

Finally, we point out that the estimate \eqref{eq:gd_new} also essentially reduces the proof of the third Melnikov condition to that of the $0$-Melnikov condition. Indeed, it allows us to derive, from a small divisor estimate depending on the three largest indices, a uniform estimate in these indices. In \cite{BGR24}, this difficulty was overcome by using the polynomial decay of the frequencies (modulo integers and a constant). However, this second approach broke the natural homogeneity of the small divisor estimates and led to highly singular estimates. In other words, instead of having small divisors of size $r_{p}^{2(1-\delta)}$ (corresponding to the size of $(\xi_j)_{j\in \mathcal{S}_{p}}$), which is the natural lower bound we can expect by homogeneity arguments, they could only be controlled by $r_p^{\beta}$ with $\beta \gg 2$ ($\beta=68$ in \cite{BGR24} for the third Melnikov condition). This singularity created numerous difficulties that were overcome by imposing the assumption $\# \mathcal{S}_{p} \setminus \mathcal{S}_{p-1} = 1$ and performing many Birkhoff normal form steps at each stage to reduce the size of the perturbation in order to implement a KAM scheme. The estimate \eqref{eq:gd_old} thus allows us to avoid having to do the same, making our proof significantly simpler.

\subsubsection{Action-T\"oplitz Hamiltonians} The estimates on the twist matrices (such as \eqref{eq:vieille_dec} or \eqref{eq:jolie_dec}) are naturally obtained as consequences of estimates on the structure of the Hamiltonians involved at each step of the construction (notably $H^{\eqref{eq:NLS}} \circ \tau_{\xi^{(1)}}^{(1)} \circ \cdots \circ \tau_{\xi^{(p)}}^{(p)}$). The latter are obtained by introducing classes of Hamiltonians that must
\begin{itemize}
    \item allow us to take advantage of the regularizing effects given by \eqref{eq:gd_new},
    \item be stable under the various transformations we need to apply to them,
    \item include $H^{\eqref{eq:NLS}}$.
\end{itemize}
One of the main challenges is to ensure that these classes allow us to analyze what happens for terms whose largest index is associated with an action. These are the terms for which we have $(\bl_1,\bs_1) = (\bl_2,-\bs_2)$, and thus for which we do not have the gain given by \eqref{eq:gd_new}. In particular, we want to show that such terms either have been created by the changes of variables and are associated with coefficients decaying like $\langle \ell_1 \rangle^{-1}$, or that they do not depend on $\ell_1$. The latter generate the $A^{(p)}, B^{(p)}$ terms in the twist matrix \eqref{eq:jolie_dec}. One can also easily verify the existence of such terms in $H^{\eqref{eq:NLS}}$.

This problem also appeared in \cite{BGR24} (where the decay was $\langle \ell_1 \rangle^{-1}\langle \ell_3 \rangle^{2}\langle \ell_4 \rangle^{2}$ instead of $\langle \ell_1 \rangle^{-1}$). It was resolved by noting that all terms in $H^{\eqref{eq:NLS}}$ containing an action (also called pairing) could be factorized and thus removed from the Hamiltonian\footnote{with the notable exception of the overpaired terms for which $\bl_1 = \bl_2 = \bl_3$, but these are simple and special enough to be treated separately.} provided that $\| u \|_{L^2}^2$ is considered as a new variable of the Hamiltonian (this is the Wick resummation). Since $\| u \|_{L^2}^2$ is a constant of motion, it can essentially be ignored, and thus one can consider that there are no terms in $H^{\eqref{eq:NLS}}$ whose largest index is associated with an action. It then suffices to verify that all new terms of this form created by the changes of variables indeed have a decay in $\langle \ell_1 \rangle^{-1}$. This approach, which is very algebraic, has the advantage of being relatively simple but requires more regularity on the nonlinearity $f$ (which must in particular be an entire function not growing too fast) due to the constants appearing in the Wick resummation.

We could have followed the same strategy in this paper, but we preferred a more pedestrian approach to remove this restriction on $f$ and only require that the latter is analytic. More precisely, we introduce classes of Hamiltonians that directly have the desired properties, and we show that they are stable under our transformations. This is relatively simple but tedious due to all the possible cases that need to be considered. Most of these verifications are carried out in the appendix.

%\mic{maybe a brief explanation of the idea of Action Toplitz?} \jo{Yes, why not... if you have an idea how to present it without being too long.}

 Let us give a heuristic presentation. As is well known frequency modulation in KAM or normal form schemes comes form  action terms, and  in our scheme this amounts to controlling all the terms in the Hamiltonian where  the  highest two    modes form an action.
Let us first consider the case of just one action term. Assume for simplicity that $H$ only contains terms of such type, and only depends on the normal variables so that is it is  a linear combination of monomials of the type 
	\[ v_{\bl_q}^{\bs_q}\dots v_{\bl_3}^{\bs_3}|v_{\bl_1}|^2  \,,\quad (\bl_3,\bs_3) \neq (\bl_4,-\bs_4)
	\] 
	then we group H into two summands. In the first term (the 1-action-T\"oplitz part) the coefficients of the Hamiltonian do not depend on $\ell_1$  while in the second part (1-action smoothing) the coefficients decay as $|\ell_1|^{-1}$.
	\\
In the case of  at least two action terms we are considering combinations of monomials of the type $v_{\bl_q}^{\bs_q}\dots v_{\bl_5}^{\bs_5} |v_{\bl_3}|^2 |v_{\bl_1}|^2$. 
The  terms where $\ell_4=\ell_1$ are easily controlled because they only produce a small diagonal correction,  
%\jo{for me it is not the reason.\mic{you are right it was bullshit} They are the over-paired terms. We control them because they generate, in the twist matrix, diagonal terms of size $r_1^{\alpha_4}$. Actually I forgot to include them in the "discussion about the proof" subsection. Now I included them. }
 so let us assume that $\ell_4=\ell_3<\ell_2=\ell_1$. 
Also in this case we group $H$ into two summands. In the first term (the 2-action-T\"oplitz part) the coefficients of the Hamiltonian do not depend on $\ell_1,\dots,\ell_4$  while the second part  either  the coefficients decay as $|\ell_1|^{-1}$ (1-action smoothing), or they do not depend on $\ell_1$ and  decay as $|\ell_3|^{-1}$ (2-action smoothing). With this structure in mind, we essentially recover the decomposition \eqref{eq:jolie_dec} of the twist matrix : the $D_i$ factors come from the  $|v_{\bl_1}|^4$ terms, $M_{i,j}$ from the 1-action smoothing terms and $A_i,B_j$ from the terms which do not depend on $\bl_1$.
	The idea of deducing asymptotics of the eigenvalues from symmetry properties of the Hamiltonian appears in many contexts, for instance using the T\"oplitz-Lipschitz property (like say in \cite{EK10}) or paralinearizations  like \cite{BD17}.  These definitions, however, only control the ``para-quadratic'' part of an Hamiltonian and this is not enough to obtain  a modulation like \eqref{eq:jolie_dec}.

\section{Symplectic deformations} \label{sec:symp_def} Now we present the new notion of \emph{symplectic deformation} define above in Definition \ref{def:symp_def}. It is a natural way to classify invariant tori. In particular, it provides robust definitions of the intuitive concepts of \emph{full dimensional tori} and \emph{number of normal sites} (see Proposition \ref{prop:geom_key} below). The classification is certainly finer but we did not try to go further in this direction.

In this paper, we prove the existence of invariant tori for \eqref{eq:NLS} of the form $\mathcal{T}=\tau(\mathbf{T}_\xi)$ where $\tau :\mathbf{T}_\xi \to  \mathcal{T}$ is a homeomorphism. However, the fact $\tau$ that is a homeomorphism provides very little information on $\mathcal{T}$: whenever $\# \mathrm{supp} \, \xi = +\infty$, $\mathbf{T}_\xi$ is homeomorphic to $\mathbb{T}^{\mathbb{N}}$. To proceed further, we need to consider properties analogous to those of $C^1$ maps. The ideal situation would be to have a deformation $\tau$ admitting a symplectic extension to a neighborhood of $\mathbf{T}_\xi$ (and even better if this neighborhood is also a neighborhood of the origin as in \cite{Bou05,BMP20}). 
%\jo{A remark for a future possible discussion: one of my friend doing symplectic topology (Rémi Leclerc) told me it is very nice because it allows to give a robust definition to the notion of amplitudes. In other words, we could prove that the $(\xi_j)_{j\in \mathcal{S}_\infty}$ are invariant up to permutation by this kind of transformation. Did you know ? Is it obvious for you ? I have the intuition from what he told me but I do not have the full proof clear in my mind.}  \mic{I am not sure i understand..discuss?} \jo{Yes, with pleasure :) I do not think it is difficult but it is something I would like to fully understand at some point.} 
However, the possible existence of such an extension is not known in many cases and seems, in any case, out of reach without external parameters. The minimal property we would like to impose on $\tau$ would be to be derivable and that, in any point, its derivative admits a symplectic invertible extension (to discuss the properties of tangent space). Unfortunately, we face another problem:  $\mathbf{T}_\xi$ is compact and so it is not a submanifold of $H^s$ (unless $\# \mathrm{supp} \, \xi < \infty$). To the best of our knowledge, the notions of tangent space and derivative are not standard in this context. Thus, in this subsection, we naturally extend these notions to our context.  For possible future applications, we present these concepts in an abstract framework. {Note that even if we decided to work in an $\ell^\infty$ based topology our tori would in any case be compact (because the actions decay superexponentially) so the problem would persist.}

\subsection{Weak symplectic spaces}
\label{sub:weak_symp}
To avoid ambiguities, first we recall basic definitions and results about weak symplectic spaces.
\begin{definition}[Weak symplectic space] A (real) weak symplectic space is a couple $(E,\varpi)$ where $E$ is a real Banach space and $\varpi$ is bilinear continuous alternate form which is non degenerate:
$$
\{ x\in E \ | \ \forall y\in E, \ \varpi(x,y)=0 \}=\{0\}.
$$
\end{definition}

\begin{example} { The realization of a complex Hilbert space $(\mathbb H, \langle\cdot,\cdot \rangle)$ has a natural (strong) symplectic structure, with scalar product 
	$(\cdot,\cdot)_{L^2}:=\Re \langle \cdot,\cdot \rangle $ and symplectic form $\varpi_0(\cdot,\cdot):=\Im\langle \cdot,\cdot\rangle$. We shall mostly consider $\mathbb H= L^2(\mathbb T;\mathbb C)$, so that
$(H^s, \varpi_0)$ is a weak symplectic space. }
\end{example}

%\begin{definition}[$\mathscr{L}(E)$] Given a normed space $E$, $\mathscr{L}(E)$ denotes the space of the continuous linear maps from $E$ to $E$.
%\end{definition}

\begin{definition}[Symplectic orthogonal]  Given a weak symplectic space $(E,\varpi)$ and a subset $X\subset E$, we set
$$
X^{\perp_{\mathrm{symp}}}  := \{ y\in E \ | \ \forall x\in X, \quad \varpi(x,y) = 0 \}.
$$
\end{definition}

\begin{definition}[Symplectic map] Given a weak symplectic space $(E,\varpi)$ and an open set $\Omega \subset E$, a $C^1$ map $\Psi:\Omega \to E$ is symplectic if
$$
\forall x\in \Omega,\forall y,z \in E, \quad \varpi(\mathrm{d}\Psi(x)(y),\mathrm{d}\Psi(x)(z)) = \varpi(y,z).
$$ 
\end{definition}

From the definition, it is straightforward, to deduce the following standard lemma.
\begin{lemma} \label{lem:symp_perp} Given a weak symplectic space $(E,\varpi)$, a vector subspace $V\subset E$ and an invertible linear symplectic map $L : E \to E$, we have 
$$
(L V)^{\perp_{\mathrm{symp}}} = L (V^{\perp_{\mathrm{symp}}} ).
$$
\end{lemma}

\begin{definition}[isotropic and Lagrangian vector spaces] Given a weak symplectic space $(E,\varpi)$, a vector subspace $X \subset E$ is  \emph{isotropic} if $V\subset V^{\perp_{\mathrm{symp}}}  $ and \emph{Lagrangian} if $V= V^{\perp_{\mathrm{symp}}}  $.
\end{definition}

\subsection{First properties of symplectic deformation} We recall that the notion of symplectic deformation is defined in Definition \ref{def:symp_def} above. { First, we state that the inverse of a symplectic deformation is a symplectic deformation.}

\begin{lemma} \label{lem:inverse_symp_def} Given a weak symplectic space $(E,\varpi)$ and $X,Y\subset E$, a symplectic deformation $\tau : X \to Y$ and a derivative  $L:X\to \mathscr{L}(E)$ of $\tau$, we have that $\tau^{-1}$ is a symplectic deformation and that $(L\circ \tau^{-1})^{-1}$ is a derivative of $\tau^{-1}$.
\end{lemma}
\begin{proof}
It suffices to mimic the standard proof for $C^1$ functions on open sets.%, there is no particular difficulty, so we omit the proof.
\end{proof}
{Then, we deduce that this notion allows to classify subsets.}
\begin{proposition} Given a weak symplectic space $(E,\varpi)$, being a symplectic deformation is a equivalence relation on the subsets of $E$.
\end{proposition}
\begin{proof} One only needs to check that the composition of two symplectic deformation is a symplectic deformation. Again, it suffices to mimic the proof for $C^1$ functions on open sets. %There is no extra difficulty, so we omit the proof.
\end{proof}

\subsection{Tangent cones} Unfortunately infinite dimensional tori are not submanifolds in $H^s$ (because they are compact). Thus their tangent space is not a priori well defined. So instead, we work with tangent cones (which turn out to be vector spaces).

\begin{definition}[Tangent cone]  Given a real normed vector space $E$, $X\subset E$ and $x\in X$. Its  differential tangent cone is defined by
$$
\mathrm{T}_x X = \big\{  \gamma'(0) \ | \ \gamma \in C^1([0,1]; X) \big\}
$$
%\emph{Bouligand tangent cone} (or \emph{contingent cone}) of $X$ in $x$ is defined by
%$$
%\mathrm{T}_x X := \bigcap_{\varepsilon>0} \overline{(\varepsilon^{-1},+\infty) (X-x)}
%$$ 
\end{definition}
\begin{remark} We could actually work with the Bouligand tangent cone and the Clarke tangent cone (which are more standard in geometry). However, it turns out that, in our applications, all these cones coincide.
\end{remark}

\begin{lemma} \label{lem:geo_expli} Let $\xi \in \ell^1_{2s}(\mathbb{Z};\mathbb{R}_+)$ and $u\in \mathbf{T}_\xi$. Then, in $(H^s, \varpi_0)$ we have
$$
\mathrm{T}_u \mathbf{T}_\xi =  \{ v \in H^s \ | \ \forall k\in \mathbb{Z}, \ v_k \in \ic u_k \mathbb{R} \}.
$$
\end{lemma}
\begin{proof}
\emph{ $\bullet$ Step $1$: "$\subset$".} Let $\gamma \in  C^1([0,1]; \mathbf{T}_\xi)$. For all $k\in \mathbb{Z}$, we have that for all $t\in [0,1]$, $|\gamma_k(t)| = \sqrt{\xi}_k$. Thus there exists $\theta_k \in C^1([0,1]; \mathbb{R})$ such that $\gamma_k = e^{\ic \theta_k}$. Thus, by continuity of $w\mapsto w_k$ in $H^s$, we have
$$
(\gamma'(0))_k = (\gamma_k)'(0) = \ic \theta_k'(0) e^{\ic \theta_k(0)} = \ic \theta_k'(0) u_k(0) \in \ic u_k \mathbb{R}.
$$ 
\emph{ $\bullet$ Step $2$: "$\supset$".} Let $v\in H^s$ such that  for all $k\in \mathbb{Z}$ we have  $v_k \in \ic u_k \mathbb{R}$. Let $\gamma : \mathbb{R} \to H^s$ be defined by 
$$
\forall t\in \mathbb{R},\forall k\in \mathbb{Z}, \  \gamma_k(t) = u_k e^{\ic t c_k} \quad \mathrm{where} \quad \ic c_k = \left\{ \begin{array}{lll} v_k u_k^{-1}& \mathrm{if} \ u_k \neq 0 \\
0 &\mathrm{else}.\end{array} \right.
$$
Note that by assumption $c\in \mathbb{R}^{\mathbb{Z}}$. Then we define $d: \mathbb{R} \to H^s$ by
$$
\forall t\in \mathbb{R},\forall k\in \mathbb{Z}, \  d_k(t) = v_k e^{\ic t c_k} .
$$
The continuity of $d$ is a consequence of the dominated convergence theorem. 
So we just have to prove that $\gamma' = d$. We note that for all $N>0$
$$
\gamma_{\leq N} = \sum_{|k|\leq N} \gamma_k e^{\ic k \cdot} \in C^1(\mathbb{R};H^s) \quad \mathrm{and} \quad  (\gamma_{\leq N})' = \sum_{|k|\leq N} d_k e^{\ic k \cdot} 
$$
converges uniformly to $\gamma$ and $d$ respectively as $N$ goes to $+\infty$. Thus, to conclude, it suffices to pass to the limit, in the formula
$$
\gamma_{\leq N}(t) = \gamma_{\leq N}(0) + \int_0^t  (\gamma_{\leq N}(\mathfrak{t}))' \mathrm{d}\mathfrak{t}.
$$
\end{proof}
\begin{corollary} \label{cor:co-iso} We have that 
$$
 (\mathrm{T}_u \mathbf{T}_\xi )^{\perp_{\mathrm{symp}}} = {\{ w\in H^s \ | \ \forall k\in \mathrm{supp}(\xi), \quad w_k \in \ic u_k \mathbb{R} \}.}
$$
In particular,  $\mathrm{T}_u \mathbf{T}_\xi $ is \emph{isotropic} and if $\mathrm{supp}\, \xi = \mathbb{Z}$ then it is \emph{Lagrangian}.
\end{corollary}
\begin{proof}
	$w\in (\mathrm{T}_u \mathbf{T}_\xi )^{\perp_{\mathrm{symp}}}$ may be rephrased as $  w_k  \overline{ u_k}\in \ic \R$ for all $k\in \mathrm{supp}(\xi)$. Thus  (multiplying both sides by $u_k\neq 0$) $w\in (\mathrm{T}_u \mathbf{T}_\xi )^{\perp_{\mathrm{symp}}}$ if and only if $w_k \in \ic \R u_k$ for all
	$k\in \mathrm{supp}(\xi)$. 
\end{proof}
Finally, we note that as expected, tangent cones are transported by derivatives of symplectic deformation.
\begin{lemma} \label{lem:deform_cone} Given a weak symplectic space $E$, $X,Y\subset E$, a symplectic deformation $\tau : X\to Y$ and a derivative $L : X \to \mathscr{L}(E)$ of $\tau$, we have that for all $x\in X$,
$$
\mathrm{T}_{\tau(x)} Y = L(x) \mathrm{T}_x X.
$$
\end{lemma}
\begin{proof} This is a direct corollary of Lemma \ref{lem:inverse_symp_def}.

\end{proof}

\subsection{Number of normal sites}\label{num_normal}

Finally, we state the main result of this subsection. 
\begin{proposition}
\label{prop:geom_key}  Let $\xi \in \ell^1_{2s}(\mathbb{Z};\mathbb{R}_+)$ and $\mathcal{T}\subset H^s$ be a symplectic deformation of $\mathbf{T}_\xi$. Then, for all $u\in \mathcal{T}$, the tangent cone $\mathrm{T}_u \mathcal{T}$ is a isotropic real vector space and we have
\begin{equation}
\label{eq:geom}
\# (\mathbb{Z} \setminus \mathrm{supp} \, \xi) = \frac12 \mathrm{dim}_{\mathbb{R}} \quotientthree{(\mathrm{T}_u \mathcal{T})^{\perp_{\mathrm{symp}}}\! \! }{\mathrm{T}_u \mathcal{T}}.
\end{equation}
In particular, if $\mathrm{supp} \, \xi = \mathbb{Z}$ then $\mathrm{T}_u \mathcal{T}$ is a Lagrangian real vector space.
\end{proposition}
Since the right-hand side of \eqref{eq:geom} depends only on $\mathcal{T}$ (and not on $\xi$ and the symplectic deformation), this shows that it makes sense to speak about the \emph{number of normal sites} of $\mathcal{T}$. The second property justifies calling the torus \emph{full dimensional} or \emph{Lagrangian} when its number of normal sites vanishes.

\begin{proof}[Proof of Proposition \ref{prop:geom_key}] Let $\tau : \mathbf{T}_\xi \to  \mathcal{T}$ be a symplectic deformation of derivative $L$. Let $u\in \mathcal{T}$ and set  $v= \tau^{-1}(u)$. By Lemma \ref{lem:deform_cone}, we have that
$$
\mathrm{T}_u \mathcal{T} = L(v) \mathrm{T}_v \mathbf{T}_\xi.
$$
Then, since $L(v)$ is symplectic and invertible, we have by Lemma \ref{lem:symp_perp} that
$$
(\mathrm{T}_u \mathcal{T})^{\perp_{\mathrm{symp}}} = L(v)  (\mathrm{T}_v \mathbf{T}_\xi)^{\perp_\mathrm{symp}}.
$$
Thus, we get by Corollary \ref{cor:co-iso} that $\mathrm{T}_u \mathcal{T}$ is isotropic and, since $L(v)$ is invertible, that
$$
\mathrm{dim}_{\mathbb{R}} \quotientthree{(\mathrm{T}_u \mathcal{T})^{\perp_{\mathrm{symp}}}\! \! }{\mathrm{T}_u \mathcal{T}} = \mathrm{dim}_{\mathbb{R}} L(v) \quotientthree{(\mathrm{T}_v \mathbf{T}_\xi)^{\perp_{\mathrm{symp}}}\! \! }{\mathrm{T}_v  \mathbf{T}_\xi} =  \mathrm{dim}_{\mathbb{R}} \quotientthree{(\mathrm{T}_v \mathbf{T}_\xi)^{\perp_{\mathrm{symp}}}\! \! }{\mathrm{T}_v  \mathbf{T}_\xi}.
$$
Finally, it suffices to note by Lemma \ref{lem:geo_expli}  that
\begin{center}
$\quotientthree{(\mathrm{T}_v \mathbf{T}_\xi)^{\perp_{\mathrm{symp}}}\! \! }{\mathrm{T}_v  \mathbf{T}_\xi}$ is isomorphic to $\{u\in H^s \ | \ \forall k\in \mathrm{supp} \, \xi, \, u_k=0 \}$.
\end{center}

\end{proof}

\section{Analytic Hamiltonians in partial action-angle variables}

\subsection{A first standard class}

In this paper, the Hamiltonians will always be given by formal series written in partial action angle variables whose coefficients satisfy some natural symmetry conditions. We start by introducing a first basic class of formal series to impose these symmetry conditions. Then, we introduce some basic projections. Finally, we prove vector fields estimates and flow estimates for these Hamiltonians. 

\subsubsection{Definition of the class.}
 We start by introducing a first basic space of Hamiltonians. Later, we will introduce some refined spaces of Hamiltonians but they will always be subspaces of this one. In the definition just below, the variable $y$  should be thought as an action variable whereas the variable $z$ will correspond to  exponentials of angles (see Definition \ref{def:z,y} below for a precise definition).
\begin{definition}[Formal series without parameters $ \mathscr{F}_{\eta,r,p}^{\cste}$] \label{def:class_F} Let $p\geq 0$,  $r>0$ and $\eta \geq 0$. We denote by $\mathscr{F}_{\eta,r,p}^{\cste}$, the set of the formal series $P$, of the form
$$
P(y,z,v) :=  \sum_{ \boldsymbol{k} \in \mathbb{Z}^{\mathcal{S}_p}  }\sum_{ \boldsymbol{m} \in \mathbb{N}^{\mathcal{S}_p}  } \sum_{q\in \mathbb{N} }    \sum_{ \boldsymbol{\ell} \in (\mathcal{S}_p^c)^{q} }  \sum_{ \boldsymbol{\sigma} \in  \{-1,1\}^{q}  } P_{\boldsymbol{k},\boldsymbol{m}}^{\boldsymbol{\ell},\boldsymbol{\sigma}}  \,   \underbrace{ \prod_{i=1}^q v_{\boldsymbol{\ell}_i}^{\boldsymbol{\sigma}_i} }_{=:v_{\boldsymbol{\ell}}^{\boldsymbol{\sigma}} } \, \underbrace{ \prod_{j\in \mathcal{S}} y^{\boldsymbol{m}_j}_j z^{\boldsymbol{k}_j}_j  }_{=: y^{\boldsymbol{m}} \, z^{ \boldsymbol{k} }} \, 
$$
with the usual convention $ v_{\ell}^{-1} :=  \overline{v_{\ell}}$ and whose coefficients $ P_{\boldsymbol{k},\boldsymbol{m}}^{\boldsymbol{\ell},\boldsymbol{\sigma}} \in \mathbb{C}$ satisfy the relations
\begin{itemize}
\item reality : $P_{-\boldsymbol{k},\boldsymbol{m}}^{\boldsymbol{\ell},-\boldsymbol{\sigma}} = \overline{ P_{\boldsymbol{k},\boldsymbol{m}}^{\boldsymbol{\ell},\boldsymbol{\sigma}} }$,
\item symmetry : $P_{\boldsymbol{k},\boldsymbol{m}}^{\varphi \boldsymbol{\ell},\varphi \boldsymbol{\sigma}} = P_{\boldsymbol{k},\boldsymbol{m}}^{\boldsymbol{\ell},\boldsymbol{\sigma}} $ for all $\varphi \in \mathfrak{S}_q$,
\item preservation of the mass : if $P_{\boldsymbol{k},\boldsymbol{m}}^{\boldsymbol{\ell},\boldsymbol{\sigma}} \neq 0$ then
$
\displaystyle \sum_{j\in \mathcal{S}_p} \boldsymbol{k}_j + \sum_{1\leq i \leq q} \boldsymbol{\sigma}_i  =0,
$
\item zero momentum :  if $P_{\boldsymbol{k},\boldsymbol{m}}^{\boldsymbol{\ell},\boldsymbol{\sigma}} \neq 0$ then
$
\displaystyle \sum_{j\in \mathcal{S}_p} j\boldsymbol{k}_j + \sum_{1\leq i \leq q} \boldsymbol{\sigma}_i \boldsymbol{\ell}_i=0,
$
\end{itemize}
and the bound
\begin{equation}
\label{eq:def_norm_F}
\| P \|_{\mathscr{F}_{\eta,r,p}^{\cste}}:= \! \! \! \! \! \!   \!\!\!\!\sum_{\substack{({\boldsymbol{k}},{\boldsymbol{m}})\in \mathbb{Z}^{\mathcal{S}_p} \times \mathbb{N}^{\mathcal{S}_p}\\ q\in \mathbb{N}}} \!\!\! \! \! \! \! \!  \max(1,q ,\|\bk\|_{\ell^1} , \|\bm\|_{\ell^{\infty}})  r^{q+2\|\bm\|_{\ell^1}} e^{\eta(\|\bk\|_{\ell^1} + q+2\|\bm\|_{\ell^1})} \!\! \!\!\! \sup_{\substack{  {\boldsymbol{\ell}}\in (\mathcal S_p^c)^{q}\\  \boldsymbol{\sigma}\in\{-1,1\}^{q}}}  \! \! \! \! \! \! |P^{\boldsymbol{\ell},\boldsymbol{\sigma}}_{{\boldsymbol{k}},{\boldsymbol{m}}}|<\infty\,.
\end{equation}
\end{definition}
We are going to prove that these formal series define functions on the following annuli.
\begin{definition}[Annuli] Given $\xi \in (\mathbb{R}_+)^{\mathbb{Z}}$ and $\rho>0$, we set
$$
\mathcal{A}_\xi(\rho) = \Big\{ u\in H^s \ | \ \sum_{k \in \mathbb{Z}}  \big| |u_k|^2 - \xi_k \big| \langle k \rangle^{2s} \, \leq \rho^2 \Big\}.
$$
\end{definition}
Then, we define the associated action-angle variables.
\begin{definition}[Functions $y,\theta,z$] \label{def:z,y} Given $u \in L^2(\mathbb{T})$, $\xi \in \ell^1(\mathbb{Z};\mathbb{R})$, $k\in \mathbb{Z}$ we define
\begin{itemize}
\item the re-centered action $y_k(\xi;u) = |u_k|^2 - \xi_k$,
\item the angle $\theta_k(u) \in \mathbb{T}$ as an angle satisfying $|u_k| e^{i \theta_k(u)}= u_k$,
\item $z_k(u) \in \mathbb{C}$, is defined by $z_k(u) =e^{i\theta_k}$.
\end{itemize}
\end{definition}
 \begin{remark}
 	Of course if $u_k=0$ the choice of $\theta_k(u)$ is arbitrary. Otherwise, if we fix $\xi\in \ell^1(\mathbb Z;\mathbb R_+)$ with  support in the finite set $\mathcal S_p$, then  for $\rho< \min_{j\in \mathcal S_p} \sqrt{\xi_j}$,  the map\footnote{We are denoting by $ B_{r}(\mathbb R^\mathcal S_p)$ the ball centered at zero in $\mathbb R^{\mathcal S_p}$ endowed with the $\ell^1$ topology. }
 	$$\Psi_\xi:\mathcal A_\xi(\rho)\to B_{\rho^2}(\mathbb R^\mathcal S_p)\times \mathbb T^{\mathcal S_p}\times B_\rho(h^s(\mathcal S^c_p,\mathbb C))\,:\quad u  \mapsto \big((y_k(\xi;u), \theta_k(\xi;u))_{k\in \mathcal S},(u_j)_{j\in\mathcal S^c} \big)$$ is analytic with inverse $\mathcal A_\xi(\rho)\to B_{\rho^2}(\mathbb R^\mathcal S_p)\times \mathbb T^{\mathcal S_p}\times B_\rho(h^s(\mathcal S^c_p,\mathbb C))\to \mathcal A_\xi(2\mathtt j_{p}^{s}\rho)$, where $\mathtt j_{p}:=\max(\langle j \rangle )_{j\in \mathcal S_p}$. 
 	 \end{remark}

Now, we prove that these formal series define functions on these annuli.  

\begin{lemma}[Formal series define functions] \label{lem:ouf_cest_des_fonctions} Let $p\geq 0$, $\xi \in \mathcal{O}_p^{\mathrm{full}} $, $r \leq 2 r_p$, $\eta \geq \eta_{\min}$, $P\in  \mathscr{F}_{\eta,r,p}^{\cste}$. Then 
\begin{equation}
\label{identi}
P[\xi] := u\mapsto P(y(\xi;u),z(u),u )=:P(\xi;u)
\end{equation}
 defines a smooth function from $\mathcal{A}_\xi(3r) $ to $\mathbb{R}$ and enjoys the estimate
$$
\sup_{u\in \mathcal{A}_\xi(3r)}|P(\xi;u)| \lesssim_p \| P\|_{\mathscr{F}_{\eta,r,p}^{\cste}}\,,\qquad \sup_{u\in \mathcal{A}_\xi(3r)} \| \nabla P(\xi;u) \|_{H^s} \lesssim_p  \frac{r_0}{r^2 e^{2\eta}}\| P\|_{\mathscr{F}_{\eta,r,p}^{\cste}}
$$
\end{lemma}
\begin{proof}[Proof of Lemma \ref{lem:ouf_cest_des_fonctions}]
%	\mic{Maybe we can also remove the convolution, what I have in mind is this:
%		We have a basis on $H^s$, which I denote $\mathbf{e}_j$. Then I have the usual convolution
%		\[
%		\mathbf{e}_{i}\ast \mathbf{e}_j= \mathbf{e}_{i+j}
%		\]
%So, it seems to me that
%\[
%\sum_{ \substack{\boldsymbol{\ell},\bs \in (\mathcal{S}_p^c)^{q} \times  \{-1,1\}^{q} \\ \sum \bs_j \bl_j= k } } \prod_{i=1}^q v_{\boldsymbol{\ell}_i}^{\boldsymbol{\sigma}_i} =(\underbrace{\sum_{j,\sigma} v^\sigma_j \mathbf{e}_{\sigma j}  \ast \dots \ast  \sum_{j,\sigma} v^\sigma_j \mathbf{e}_{\sigma j} }_{q \, \mbox{\tiny times}})_k
%\]
%or if you prefer the notation $ v+\bar v= \sum_{j,\sigma} v^\sigma_j \mathbf{e}_{\sigma j} $ and 
%\[
%\sum_{ \substack{\boldsymbol{\ell},\bs \in (\mathcal{S}_p^c)^{q} \times  \{-1,1\}^{q} \\ \sum \bs_j \bl_j= k } } \prod_{i=1}^q v_{\boldsymbol{\ell}_i}^{\boldsymbol{\sigma}_i} =(\underbrace{(v+\bar v) \ast \dots \ast  (v+\bar v)}_{q \, \mbox{\tiny times}})_k \sim \int_{\mathbb T} (v(x)+\bar v(x))^q e^{-i  k x}
%\]
%	} 
%	\jo{Yes, it is true. Just a remark : when you denote $\overline{v}$ you are considering $v$ as a function but when you denote $\ast$ you are considering $v$ as a sequence... (but it is not a serious and we can let as it is.}
	%	 $v= \sum_j v_j \mathbb{e}_j$  and $w= \sum_j w_j \mathbb{e}_j$ 
	 For compactness we denote $D_\rho:= B_{\rho^2}(\mathbb R^{\mathcal S_p})\times \mathbb T^{{\mathcal S_p}}\times B_\rho(h^s(\mathcal S_p^c,\mathbb C))$ . 
	 We also denote $v= (v_j)_{j\in \mathbb Z}$, $\bar v= (\overline{v_{-j}})_{j\in \mathbb Z}$ so that
	 \[
	 \sum_{ \substack{(\boldsymbol{\ell},\bs) \in (\mathcal{S}_p^c)^{q} \times  \{-1,1\}^{q} \\ \sum \bs_j \bl_j= k } } \prod_{i=1}^q v_{\boldsymbol{\ell}_i}^{\boldsymbol{\sigma}_i} =(\underbrace{(v+\bar v) \ast \dots \ast  (v+\bar v)}_{q \, \mbox{\tiny times}})_k\,,
	 \]
	 where $u\ast w$ is the standard convolution.
	 \\
%	 For $v,w\in h^s(\mathcal S_p^c,\mathbb C)\subset H^s$, we set  for $k\in \mathbb Z$
%	 $$
%	 (v\star w)_k := \sum_{\substack{\sigma_1,\sigma_2=\pm , j_1,j_2\in {\mathcal S_p}^c\\ \sigma_1j_1+\sigma_2 j_2+k=0}}v_{j_1}^{\sigma_1} w_{j_2}^{\sigma_2}\,,%\quad \mbox{so that}\quad \|v\star w\|_{H^s} \le  \mathtt{C_{alg}}\|v\|_{H^s}\|w\|_{H^s}\,,
%	 $$
%	 so that, by the algebra property,  we have that $v\,\star\, w\in H^s$. 
%%	By construction one has 
%%	\[
%%	\sup_{u\in \mathcal{A}_\xi(3r)}|P(\xi;u)|  \le \sup_{D_{3r }}|{ P}| %\le  %\, ,\,  \sup_{D_{\eta, r e^{\eta}}}|\underline X_{P}|_E   \le \sup_{D_{\eta, r e^{\eta}}}|\underline X_{\hat P}|_E   
%%	\]
Setting
\[
\mathcal P_{\boldsymbol{k},\boldsymbol{m},q}:= \sup_{\substack{  {\boldsymbol{\ell}}\in (\mathcal S_p^c)^{q}\\  \boldsymbol{\sigma}\in\{-1,1\}^{q}}}  \! \! \! \! \! \! |P^{\boldsymbol{\ell},\boldsymbol{\sigma}}_{{\boldsymbol{k}},{\boldsymbol{m}}}|
\]
and noting that  
\begin{equation}
	\label{convq}
|\sum_{ \substack{\boldsymbol{\ell},\bs \in (\mathcal{S}_p^c)^{q} \times  \{-1,1\}^{q} \\ \sum \bs_j \bl_j= k } } \,    \prod_{i=1}^q v_{\boldsymbol{\ell}_i}^{\boldsymbol{\sigma}_i} | 	= |(\underbrace{(v+\bar v) \ast \dots \ast  (v+\bar v)}_{q \, \mbox{\tiny times}})_k |\lesssim  2^q\|v\|_{\ell^1}^q
%\|\underbrace{(v+\bar v) \ast \dots \ast  (v+\bar v)}_{q \, \mbox{\tiny times}} \|_{\ell^1} 
\lesssim  4^q  \|v\|_{H^s}^q,
\end{equation}
%	$
%	|(\underbrace{v\star \dots\star v}_{q \, \mbox{\tiny times}})_{\sum j k_j} |	\lesssim \|(\underbrace{v\star \dots\star v}_{q \, \mbox{\tiny times}})\|_{\ell^1} \lesssim(4 \|v\|_{\ell^1})^q\lesssim  2^q  \|v\|_{H^s}^q,
%	$ 
%	

	one obtains (provided that $\eta \ge \eta_{\min}$ is  sufficiently large)
\begin{equation}
		\label{stimaP}
	\begin{aligned}
		\sup_{u\in \mathcal{A}_\xi(3r)}|P(\xi;u)|  &\le \sup_{D_{3r }}|{ P}| \le  \sum_{\substack{({\boldsymbol{k}},{\boldsymbol{m}})\in \mathbb{Z}^{\mathcal{S}_p} \times \mathbb{N}^{\mathcal{S}_p}\\ q\in \mathbb{N}}} \mathcal P_{\boldsymbol{k},\boldsymbol{m},q} 	\sup_{D_{3r }} |y^{\boldsymbol{m}} z^{\boldsymbol{k}}
		|\sum_{ \substack{\boldsymbol{\ell},\bs \in (\mathcal{S}_p^c)^{q} \times  \{-1,1\}^{q} \\ \sum \bs_j \bl_j= -\sum  j \bk_j } } \,    \prod_{i=1}^q v_{\boldsymbol{\ell}_i}^{\boldsymbol{\sigma}_i}
		% (\underbrace{(v+\bar v) \ast \dots \ast  (v+\bar v)}_{q \, \mbox{\tiny times}})_{-\sum j k_j} 
		| \\ &\le  \sum_{\substack{({\boldsymbol{k}},{\boldsymbol{m}})\in \mathbb{Z}^{\mathcal{S}_p} \times \mathbb{N}^{\mathcal{S}_p}\\ q\in \mathbb{N}}}\mathcal P_{\boldsymbol{k},\boldsymbol{m},q}  (12 r)^{2\|\bm\|_{\ell^1}+q }  \lesssim  \| P\|_{\mathscr{F}_{\eta,r,p}^{\cste}}.
	\end{aligned}
\end{equation}
		Regarding the vector field, we note  that
		\begin{equation}
		\label{derivo}
	\partial_{\overline {u_\ell}} P = 
		\partial_{\overline{v_\ell}} P   \quad \mathrm{if} \ \ell\in \mathcal S_p^c \quad  \mathrm{and} \quad
		\partial_{\overline{u_\ell}} P  = -\frac{1}{2 \overline{ u_\ell}} (z_\ell \partial_{z_\ell}  P ) +  u_\ell  ( \partial_{y_\ell}  P) \quad \mathrm{else}.
	\end{equation}
We remark that 
	\[
	\|(\partial_{\overline{v_\ell}}\partial P )_{\ell\in \mathcal S_p^c}\|_{H^s} \le \sum_{\substack{({\boldsymbol{k}},{\boldsymbol{m}})\in \mathbb{Z}^{\mathcal{S}_p} \times \mathbb{N}^{\mathcal{S}_p}\\ q\in \mathbb{N}}}q  \mathcal P_{\bk,\bm,q} |y|^{\bm} |z|^{\bk}\big\|	\sum_{ \substack{\boldsymbol{\ell},\bs \in (\mathcal{S}_p^c)^{q} \times  \{-1,1\}^{q} \\ \sum \bs_j \bl_j= -\sum  j \bk_j +\ell } } \,    \prod_{i=1}^q v_{\boldsymbol{\ell}_i}^{\boldsymbol{\sigma}_i}\big\|_{H^s}\,.
	\]
Now, for any $\alpha>0$, one has
	\[
\begin{aligned}\label{jens}
		\| \sum_{ \substack{\boldsymbol{\ell},\bs \in (\mathcal{S}_p^c)^{q} \times  \{-1,1\}^{q} \\ \sum \bs_j \bl_j= -\sum  j \bk_j +\ell } } \,    \prod_{i=1}^q v_{\boldsymbol{\ell}_i}^{\boldsymbol{\sigma}_i} \|_{H^s} &\le \sup_{\ell}\frac{\langle\sum j \bk_j - \ell \rangle^s }{\langle\ell\rangle^s} \| \sum_{ \substack{\boldsymbol{\ell},\bs \in (\mathcal{S}_p^c)^{q} \times  \{-1,1\}^{q} \\ \sum \bs_j \bl_j= \ell } } \,    \prod_{i=1}^q v_{\boldsymbol{\ell}_i}^{\boldsymbol{\sigma}_i}\|_{H^s}\\ & \lesssim_{\alpha,p,s} 2^{q-1}e^{\alpha|\bk|_{\ell^1}} (q-1)^s (2 \|v\|_{H^s} )^{q-1}
\end{aligned}
	\]
	since 
	\[
	\sup_{\boldsymbol{k}} \sup_{\ell}\frac{\langle\sum j \bk_j + \ell \rangle^s }{\langle\ell\rangle^s} e^{-\alpha |\boldsymbol{k}|_{\ell^1}}\le \Big(\sup_{\boldsymbol{k}} \sup_{\ell}\frac{\mathtt j_{p}  |\boldsymbol{k}|_{\ell^1}+ \langle \ell \rangle }{\langle\ell\rangle} e^{-\frac{\alpha}{s} |\boldsymbol{k}|_{\ell^1}}\Big)^s \lesssim_{\alpha,p,s} 1\,,
	\]
	and (applying Jensen and Young inequalities)
	\[
\| \sum_{ \substack{\boldsymbol{\ell},\bs \in (\mathcal{S}_p^c)^{q} \times  \{-1,1\}^{q} \\ \sum \bs_j \bl_j= \ell } } \,    \prod_{i=1}^q v_{\boldsymbol{\ell}_i}^{\boldsymbol{\sigma}_i} \|_{H^s} \lesssim 2^{(q-1)}(q-1)^{s-1}  \|v\|_{H^s}^{q-1}\,.
	\]
	Hence
	\[
	\begin{aligned}
		\sup_{D_{3 r }}&\| (\partial_{\overline{v_\ell}}\partial P )_{\ell\in \mathcal S_p^c} \|_{H^s}   \le \sum_{\substack{({\boldsymbol{k}},{\boldsymbol{m}})\in \mathbb{Z}^{\mathcal{S}_p} \times \mathbb{N}^{\mathcal{S}_p}\\ q\in \mathbb{N}}} q  \mathcal P_{\bk,\bm,q}(3r)^{2\|\bm\|_{\ell^1}} \sup_{D_{3 r }} \|  \sum_{ \substack{\boldsymbol{\ell},\bs \in (\mathcal{S}_p^c)^{q} \times  \{-1,1\}^{q} \\ \sum \bs_j \bl_j= -\sum  j \bk_j +\ell } } \,    \prod_{i=1}^q v_{\boldsymbol{\ell}_i}^{\boldsymbol{\sigma}_i}  \|_{H^s}\\  
		&\lesssim_{\alpha,p,s} \sum_{\substack{({\boldsymbol{k}},{\boldsymbol{m}})\in \mathbb{Z}^{\mathcal{S}_p} \times \mathbb{N}^{\mathcal{S}_p}\\ q\in \mathbb{N}}} q  \mathcal P_{\bk,\bm,q}(3r)^{2\|\bm\|_{\ell^1}}  (6 r)^{q-1} (q-1)^s e^{\alpha \|\bk\|_{\ell^1}}\\ & \lesssim_{p} \sup_{q'} (q')^s e^{-\eta_{\min} \frac{q'}2} \sum_{\substack{({\boldsymbol{k}},{\boldsymbol{m}})\in \mathbb{Z}^{\mathcal{S}_p} \times \mathbb{N}^{\mathcal{S}_p}\\ q\in \mathbb{N}}} q  \mathcal P_{\bk,\bm,q}r^{2\|\bm\|_{\ell^1}+q-1} e^{\eta_{\min} (\|\bk\|_{\ell^1}+2\|\bm\|_{\ell^1}+ q-1)}\\ &  \lesssim_{p,s}  \frac{1}{r }\| P\|_{\mathscr{F}_{\eta,r,p}^{\cste}}
	\end{aligned}
	\]
	provided 	$\alpha=\eta_{\min}$ is sufficiently large. Regarding the other components we  reason as in \eqref{stimaP}, to show that for all $\ell\in \mathcal S_p$
		\begin{equation}
		\label{provo}
		\begin{aligned}
			\sup_{u\in \mathcal{A}_\xi(3r)}|\partial_{\theta_\ell} P(\xi;u) +r^2  \partial_{y_\ell} P(\xi;u)|   &\le   \sum_{\substack{({\boldsymbol{k}},{\boldsymbol{m}})\in \mathbb{Z}^{\mathcal{S}_p} \times \mathbb{N}^{\mathcal{S}_p}\\ q\in \mathbb{N}}}  (\|\bk\|_{\ell^1}+\|\bm\|_{\ell^1})\mathcal P_{\boldsymbol{k},\boldsymbol{m},q}  (12 r)^{2\|\boldsymbol{m}\|_{\ell^1} +q } \\ & \lesssim_{p,s}  \| P\|_{\mathscr{F}_{\eta,r,p}^{\cste}}\,.
		\end{aligned}
	\end{equation}
	Since by construction $|u_\ell| \ge \sqrt{\xi_\ell}/2$ for all $\ell\in \mathcal S_p$, by \eqref{derivo}
	\begin{equation*}
	\begin{aligned}
		\sup_{u\in \mathcal{A}_\xi(3r)}|  \partial_{\overline{u_\ell}} P(\xi;u)|  &\le  \sum_{\substack{({\boldsymbol{k}},{\boldsymbol{m}})\in \mathbb{Z}^{\mathcal{S}_p} \times \mathbb{N}^{\mathcal{S}_p}\\ q\in \mathbb{N}}}   \mathcal P_{\boldsymbol{k},\boldsymbol{m},q}	\sup_{D_{3r }}  (\frac{|\bk_\ell|}{\sqrt{\xi_\ell}} + \frac{\sqrt{\xi_\ell} \bm_\ell }{y_\ell})  |y^{\boldsymbol{m}} z^{\boldsymbol{k}} \sum_{ \substack{\boldsymbol{\ell},\bs \in (\mathcal{S}_p^c)^{q} \times  \{-1,1\}^{q} \\ \sum \bs_j \bl_j= -\sum  j \bk_j } } \,    \prod_{i=1}^q v_{\boldsymbol{\ell}_i}^{\boldsymbol{\sigma}_i} | \\ &\le r_0 \sum_{\substack{({\boldsymbol{k}},{\boldsymbol{m}})\in \mathbb{Z}^{\mathcal{S}_p} \times \mathbb{N}^{\mathcal{S}_p}\\ q\in \mathbb{N}}}  (\|\bk\|_{\ell^1}+\|\bm\|_{\ell^1})\mathcal P_{\boldsymbol{k},\boldsymbol{m},q}  (12 r)^{2\|\bm\|_{\ell^1} +q -2}  \lesssim  \frac{r_0}{r^2 e^{2\eta}} \| P\|_{\mathscr{F}_{\eta,r,p}^{\cste}}
	\end{aligned}
	\end{equation*}
	and the result follows.

\end{proof}

\begin{remark} {We could actually prove that these Hamiltonians are smooth on $\ell^1$ and that their vector fields map $\ell^1$ in $\ell^1$. We chose to work in $H^s$ because it is more standard and to ensure that the Hamiltonian of \eqref{eq:NLS} is locally bounded (which is not essential for us but it is convenient).}
\end{remark}

\begin{remark} Thanks to the symmetry condition on the coefficients of the formal Hamiltonian, for all $\xi \in \mathcal{O}_p^{\mathrm{full}}$, the map $P \mapsto P[\xi] $ is injective. In other words, it makes sense to identity the function with its expansion.
\end{remark}

As a corollary of the proof of Lemma \ref{lem:ouf_cest_des_fonctions}, we have the following Poisson bracket estimate which allows to control the variation of the actions.
\begin{corollary} \label{cor:var_actions} Let $p\geq 0$, $\xi \in \mathcal{O}_p^{\mathrm{full}} $, $r \leq 2 r_p$, $\eta \geq \eta_{\min}$, $P\in  \mathscr{F}_{\eta,r,p}^{\cste}$. Then 
$$
\sup_{u\in \mathcal{A}_\xi(3r)} \sum_{j\in \mathbb{Z}} \langle j \rangle^{2s} |\{ I_j, P[\xi] \}(u)| \lesssim_p   \| P\|_{\mathscr{F}_{\eta,r,p}^{\cste}} 
$$
where for all $u\in H^s$ and all $j\in \mathbb{Z}$, $I_j(u) := |u_j|^2$.
% \jo{I had in mind a loss of $r^{-1}$ in the estimate but doing the proof, it seems to me that the estimate is better than I thought... (but uselessly better)}
\end{corollary}
\begin{proof} 
	 On the one hand, given $u\in \mathcal{A}_\xi(3r)$, by Cauchy--Schwarz, we have
\begin{equation*}
\begin{split}
\sum_{j\in \mathcal{S}_p^c}  \langle j \rangle^{2s} | \{ I_j, P[\xi] \}(u) | &\lesssim \sum_{j\in \mathcal{S}_p^c}  \langle j \rangle^{2s} | u_j \partial_{\overline{u_j}} P(\xi;u) | \lesssim \sqrt{\sum_{j\in \mathcal{S}_p^c } \langle j \rangle^{2s}  |u_j|^2 } \| \nabla_{\mathcal{S}_p^c}  P(\xi;u)\|_{H^s}\lesssim_p \| P\|_{\mathscr{F}_{\eta,r,p}}.
\end{split}
\end{equation*}
On the other hand, given $j\in \mathcal{S}_p$, $I_j$ commutes with $y_j$. Moreover, we have $\{ I_j, e^{\ic \theta_j} \} = 4 \ic e^{\ic \theta_j} $.
Thus,  using \eqref{provo}%estimating $\partial_{\theta_j} P[\xi]$ as $P[\xi]$, we have
\begin{equation*}
\begin{split}
\sum_{j\in \mathcal{S}_p}  \langle j \rangle^{2s} | \{ I_j, P[\xi] \}(u) |  \lesssim_p \sup_{j\in \mathcal{S}_p} |\partial_{\theta_j} P(\xi;u)|  \lesssim_p \| P\|_{\mathscr{F}_{\eta,r,p}}.
\end{split}
\end{equation*}
\end{proof}

\subsubsection{Projections} \label{sub:projection} Now, we define some basic projections we will use in all the paper. %\mic{I noticed now that you use $\|\cdot\|_{\ell^1}$ instead of $|\cdot|_{\ell^1}$...I shall have to correct my part..} \jo{yes, sorry}
We remark that, by construction, all the projections defined by restrictions of the coefficients are continuous with respect to $\|\cdot\|_{\mathscr{F}_{\eta,r,p}^{\cste}}$.
\begin{definition}[Degree] Let $p\geq 0$, $\eta \geq 0$ and $r>0$. A Hamiltonian $P\in  \mathscr{F}_{\eta,r,p}^{\cste}$ is homogeneous of degree $d\geq 0$ if, for all $\boldsymbol{k} \in \mathbb{Z}^{\mathcal{S}_p},\boldsymbol{m} \in \mathbb{N}^{\mathcal{S}_p},q\in \mathbb{N}, \boldsymbol{\ell} \in (\mathcal{S}_p^c)^q ,\boldsymbol{\sigma} \in  \{-1,1\}^q $ either $P_{\boldsymbol{k},\boldsymbol{m}}^{\boldsymbol{\ell},\boldsymbol{\sigma}} =0$ or $2n + 2\| \boldsymbol{m} \|_{\ell^1} +q = d$. We denote by $\Pi_{\mathrm{deg}=d},\Pi_{\mathrm{deg}\leq d}, \Pi_{\mathrm{deg}\geq d}$ the associated projections defined by restriction of the coefficients, i.e.
$$
(\Pi_{\mathrm{deg}=d} P)_{n,\boldsymbol{k},\boldsymbol{m}}^{\boldsymbol{\ell},\boldsymbol{\sigma}} := \mathds{1}_{2n + 2\| \boldsymbol{m} \|_{\ell^1} +q = d} P_{n,\boldsymbol{k},\boldsymbol{m}}^{\boldsymbol{\ell},\boldsymbol{\sigma}}, \quad \Pi_{\mathrm{deg}\leq d} := \sum_{0\leq d'\leq d }\Pi_{\mathrm{deg}=  d'}, \quad \! \! \Pi_{\mathrm{deg}\geq d} = \mathrm{Id} - \Pi_{\mathrm{deg}\leq d-1} .
$$
\end{definition}

\begin{definition}[$y$-homogeneity] Let $p\geq 0$, $\eta \geq 0$ and $r>0$. A Hamiltonian $P\in  \mathscr{F}_{\eta,r,p}^{\cste}$  is $y$-homogeneous of degree $d\geq 0$ if, for all $\boldsymbol{k} \in \mathbb{Z}^{\mathcal{S}_p},\boldsymbol{m} \in \mathbb{N}^{\mathcal{S}_p},q\in \mathbb{N}, \boldsymbol{\ell} \in (\mathcal{S}_p^c)^q ,\boldsymbol{\sigma} \in  \{-1,1\}^q $ either $P_{\boldsymbol{k},\boldsymbol{m}}^{\boldsymbol{\ell},\boldsymbol{\sigma}} =0$ or $ \| \boldsymbol{m} \|_{\ell^1} = d$. We denote by $\Pi_{\| \boldsymbol{m} \|_{\ell^1} =d},\Pi_{\| \boldsymbol{m} \|_{\ell^1} \geq d},\Pi_{\| \boldsymbol{m} \|_{\ell^1} \leq d}$ the associated projections defined by restriction of the coefficients.
\end{definition}

\begin{definition}[Integrable Hamiltonians] Let $p\geq 0$, $\eta \geq 0$ and $r>0$. A Hamiltonian $P\in  \mathscr{F}_{\eta,r,p}^{\cste}$ is integrable if for all $\boldsymbol{k} \in \mathbb{Z}^{\mathcal{S}_p},\boldsymbol{m} \in \mathbb{N}^{\mathcal{S}_p},q\in \mathbb{N}, \boldsymbol{\ell} \in (\mathcal{S}_p^c)^q ,\boldsymbol{\sigma} \in  \{-1,1\}^q $ either $P_{\boldsymbol{k},\boldsymbol{m}}^{\boldsymbol{\ell},\boldsymbol{\sigma}} =0$ or $\boldsymbol{k}  = 0$ and there exists a permutation $\varphi$ such that $\varphi \boldsymbol{\ell} = \boldsymbol{\ell}$ and $\varphi \boldsymbol{\sigma} = - \boldsymbol{\sigma}$. We denote by $\Pi_{\mathrm{int}}$ the associated projection defined by restriction of the coefficients.
\end{definition}

\subsubsection{Vector field estimate and flow estimate}

The following lemma provides an estimate on the derivatives of the vector fields. Note that a standard Cauchy estimate would give a homogeneity factor of the type $r^{-2(h+1)}$ due to derivatives in the variables $y$. To get better bounds we restrict to Hamiltonians at most linear in $y$.
\begin{lemma}[Vector field estimates] \label{lem:vector_fields} Let $p\geq 0$, $r\leq 2r_p$,  $\eta \geq \eta_{\min}$ and $\chi    \in \Pi_{\| \boldsymbol{m} \|_{\ell^1} \leq 1} \mathscr{F}_{\eta,r,p}^{\cste}$ be a Hamiltonian of $y-$homogeneity smaller or equal to $1$. Then, for all $\xi \in \mathcal{O}_p^{\mathrm{full}} $, $\nabla \chi[\xi] : \mathcal{A}_\xi(3r) \to H^s$ is $C^\infty$ and satisfies, for all $ h\geq 0$, $u\in   \mathcal{A}_\xi(3r)$, $v^{(1)},\cdots,v^{(h)}\in H^s$, 
$$
 \|\mathrm{d}^h \nabla\chi(\xi;u)(v^{(1)},\cdots,v^{(h)})\|_{H^s} \lesssim_h \| \chi \|_{\mathscr{F}_{\eta,r,p}^{\cste}} r^{-2-h} \prod_{1\leq i\leq h} \|v^{(i)}\|_{H^s}.
 $$
\end{lemma} 
\begin{proof} 
We shall take advantage of the fact that $H$ is real analytic and a polynomial in $y$ in order to make an appropriate holomorphic extension and prove the bounds by Cauchy estimates. We recall  \cite[Theorem 2 of Appendix  A]{PT87}, which states that for any holomorphic function {$f: \mathcal A \to F$} where $E,F$ are Banach spaces  (and ${\mathcal A \subset E}$ is an open set) one has\footnote{The key point is that, as a consequence of Cauchy's integral formula,  the $h$'th differential can be computed, as continuous $n$--linear operator as:
		\begin{equation}\label{ildifn}
			d^hf(z_0)[v^{(1)},\dots,v^{(h)}] =  \frac 1{(2\pi \ic )^h} \!\!\!\!\!\!\!\underset{|\zeta_1|,\dots,|\zeta_n|=\varrho}{\iint} \!\!\!\!\!\!\!\frac{f(z_0+ \zeta_1 v^{(1)}+\dots \zeta_h v^{(h)})}{ \zeta_1^2\dots \zeta_h^2} \mathrm{d} \zeta_1 \dots \mathrm{d}\zeta_h \quad \text{for $\varrho$ small enough. }
		\end{equation}
		}
	\begin{equation}\label{prce}
		\begin{aligned}
	| d^hf(z_0)|_F^{\rm op}:=	\sup_{|v^{(1)}|_E,\dots,|v^{(h)}|_E\le 1}|\mathrm{d}^hf(z_0)[v^{(1)},\dots,v^{(h)}]|_F \le h^h\varrho^{-h}\sup_{z\in B_\varrho(z_0)}|f(z)|_F\,,%\\
		%	\sup_{|v|_E\le 1}|d^nf(z_0)[v,\dots,v]|_F  &\le n! \varrho^{-n}\sup_{z\in \mathcal B_\varrho(z_0)}|f(z)|_F\,,
		\end{aligned}
	\end{equation} 
	for all $0<\varrho<\varrho_{\max}$, where $\varrho_{\max}$ is the distance between $z_0$ and the boundary of $\mathcal A$.

%	\medskip
%	
%	%\mic{qui mi sembra sia l'errore di PoTru che dice $\varrho < \min (R(z_0)/\|v_i\|)$ a me sembra che la cosa corretta sia $\varrho< R/(\sum_i \|v^{(i)}\|)$.}  
%	
%	Denoting by $\delta$ the distance between $z_0$ and the boundary of $\mathcal A$  one has the Cauchy estimates	for its operator norm:\begin{equation}\label{prce}
%		\begin{aligned}
%			|d^nf(z_0)|^{\rm op}:= \!\!\!\!\sup_{|v^{(1)}|_E,\dots,|v^{(n)}|_E\le 1}|d^nf(z_0)[v^{(1)},\dots,v^{(n)}]|_F &\le n^n \varrho^{-n}\sup_{z\in B_\varrho(z_0)}|f(z)|_F\,,\\
%			\sup_{|v|_E\le 1}|d^nf(z_0)[v,\dots,v]|_F  &\le n! \varrho^{-n}\sup_{z\in \mathcal B_\varrho(z_0)}|f(z)|_F\,,
%		\end{aligned}
%	\end{equation} 
%	for all $0<\varrho<\delta$.
%	So that, for all $z_0\in \mathcal A$, one has
%	\begin{equation}
%		\label{tota0}
%		f(z)= \sum_{n=0}^\infty \frac{1}{n!} d^n f(z_0)[(z-z_0,\dots, z-z_0]
%	\end{equation}
%	where the series on the right hand side is totally convergent in any ball $B_\varrho(z_0)$ with $\varrho<\delta$.	
	
	\medskip
	By definition $\chi(\xi;u)$ on $\mathcal A_\xi(3r)$ is  a real analytic function defined by the totally convergent series 
	\begin{equation}
		\label{est}
		\chi(\xi;u)= \hspace{-0.5cm} \sum_{\substack{({\boldsymbol{k}},{\boldsymbol{m}})\in \mathbb{Z}^{\mathcal{S}_p} \times \mathbb{N}^{\mathcal{S}_p}\\ q\in \mathbb{N}}}   \sum_{ \boldsymbol{\ell} \in (\mathcal{S}_p^c)^{q} }  \sum_{ \boldsymbol{\sigma} \in  \{-1,1\}^{q}  } \chi _{\boldsymbol{k},\boldsymbol{m}}^{\boldsymbol{\ell},\boldsymbol{\sigma}}  \,   { \prod_{i=1}^q u_{\boldsymbol{\ell}_i}^{\boldsymbol{\sigma}_i} } \,  \prod_{j\in \mathcal{S}_p} ( u_j^{+1} u_j^{-1} -\xi_j )^{\boldsymbol{m}_j} \left(\frac{u_j^{+1}}{\sqrt{u_j^{+1} u_j^{-1}}}\right)^{\boldsymbol{k}_j}\,.
	\end{equation}
	where $u^{+1}_j = u_j$ and $  u^{-1}_j =\overline{ u_j}$. 
	%In fact, as we have seen in the proof of Lemma \ref{lem:ouf_cest_des_fonctions}, $\nabla \chi$ is a (real) analytic function  $\mathcal{A}_\xi(3r) \to H^s$. 
For $\boldsymbol{\rho} \in (\mathbb{R}_+^*)^{\mathcal{S}_p}$  and  for $0<r\le 2 r_p$,  we  define 
	$$\mathcal{A}^\C_\xi(\boldsymbol{\rho}, r) = \Big\{ u^{\pm 1} \in H^s \  | \sup_{j\in \mathcal{S}_p}\big| u_j^{+ 1} u_j^{-1} - \xi_j \big| \vee  \big| |u_j^{\pm 1}|^2 - \xi_j \big| \, \leq \bro_j^2 \,,\quad %  \big| u_j^{+ 1} - \overline {u_j^{-1}}  \big| \, \leq \bro_j/4\,,\quad 
	\sum_{j \in \mathcal S^c_p}  \big| u_j^{\pm 1}   \big|^2 \langle j \rangle^{2s} \, \leq  r ^2\Big\}.
	$$
	Let us now assume that $\chi$ is {$y$-homogeneous of degree $d_y=0,1$.}
	We claim that $\chi(\xi;u)$ can be extended to $\mathcal{A}^\C_\xi(\boldsymbol{\rho}, 4r)$ with $\boldsymbol{\rho}:=(\frac{\sqrt{\xi_j}}{2})_{j\in \mathcal S_p}$.   In each summand of \eqref{est} the only term which might not be holomorphic is  $ \sqrt{u_j^{+1} u_j^{-1}}$,  however  in  $\mathcal{A}^\C_\xi(\boldsymbol{\rho}, 4r)$ one has Re$(u_j^{+1} u_j^{-1})>\xi_j/2>0$ so the polydromy of the square root is avoided. By  reasoning as in \eqref{stimaP},  one verifies that  \eqref{est} is totally convergent and satisfies the bounds
	 \[
	\begin{aligned}
		&\sum_{\substack{({\boldsymbol{k}},{\boldsymbol{m}})\in \mathbb{Z}^{\mathcal{S}_p} \times \mathbb{N}^{\mathcal{S}_p}\\ q\in \mathbb{N}}}   \sum_{ \boldsymbol{\ell} \in (\mathcal{S}_p^c)^{q} }  \sum_{ \boldsymbol{\sigma} \in  \{-1,1\}^{q}  } |\chi _{\boldsymbol{k},\boldsymbol{m}}^{\boldsymbol{\ell},\boldsymbol{\sigma}} |\sup_{\mathcal{A}^\C_\xi(\boldsymbol{\rho}, 4r)}| \,   { \prod_{i=1}^q u_{\boldsymbol{\ell}_i}^{\boldsymbol{\sigma}_i} } \,  \prod_{j\in \mathcal{S}} ( u_j^{+1} u_j^{-1} -\xi_j )^{\boldsymbol{m}_j} \left(\frac{u_j^{+1}}{\sqrt{u_j^{+1} u_j^{-1}}}\right)^{\boldsymbol{k}_j}| \\ &
	\le \sum_{\substack{({\boldsymbol{k}},{\boldsymbol{m}})\in \mathbb{Z}^{\mathcal{S}_p} \times \mathbb{N}^{\mathcal{S}_p}\\ q\in \mathbb{N}}}  \sup_{ \substack{\boldsymbol{\ell} \in (\mathcal{S}_p^c)^{q}\\  \boldsymbol{\sigma} \in  \{-1,1\}^{q}  }} |\chi _{\boldsymbol{k},\boldsymbol{m}}^{\boldsymbol{\ell},\boldsymbol{\sigma}}| \,  ( 16 r)^q  \, \boldsymbol{\rho}^{2 \boldsymbol{m}}  4^{|\boldsymbol{k}|_{\ell^1}} \lesssim \Big(\frac{1}{r}\Big)^{2 d_y}  \| \chi \|_{\mathscr{F}_{\eta,r,p}^{\cste}}\,. 
	\end{aligned}
	\]
	for $\eta_{\min}$ large enough.
	\\
	We conclude that $\chi(\xi;u)$ is holomorphic $\mathcal{A}^\C_\xi(\boldsymbol{\rho},4r)\to \mathbb C$. Moreover, proceeding word by word as in the proof of Lemma \ref{lem:ouf_cest_des_fonctions}, we have that
	$\nabla_{\mathcal S^c_p} \chi(\xi;u):= 2(\partial_{ u^{-1}_j}\chi (\xi;u))_{j \in \mathcal S_p^c}$ is holomorphic $\mathcal{A}^\C_\xi(\boldsymbol{\rho},4r)\to h^s({\mathcal S^c},\C)$ and satisfies the bounds
	\[
	\sup_{u^{\pm 1}\in  \mathcal{A}^\C_\xi(\boldsymbol{\rho},4r)}\| \nabla_{\mathcal S^c} \chi(\xi;u)\|_{H^s} \lesssim \frac{1}{r^{2 d_y+1}} \| \chi \|_{\mathscr{F}_{\eta,r,p}^{\cste}}\,. 
	\]
	Similarly, given $j \in \mathcal{S}_p$, by using Cauchy estimates we have 
	\[
	\sup_{u^{\pm 1}\in  \mathcal{A}^\C_\xi(\frac{2}{3}\boldsymbol{\rho},4r)}| \partial_{u_j^-} \chi(\xi;u) | \lesssim \frac{1}{\min_k \bro_k} \sup_{u^{\pm 1}\in  \mathcal{A}^\C_\xi(\boldsymbol{\rho},4r)}|  \chi (\xi;u)| \lesssim  \frac{1}{r^{2 d_y+1}} \| \chi \|_{\mathscr{F}_{\eta,r,p}^{\cste}}\, .
	\]
	
	If $d_y=0$ the desired result follows directly by the Cauchy estimates
	\[
	\sup_{u \in  \mathcal{A}_\xi(3r)}\| \mathtt d^h \nabla \chi(\xi;u)\|^{\rm op}_s \lesssim_h \frac{1}{\varrho^h }  \sup_{u^{\pm 1}\in  \mathcal{A}^\C_\xi(\frac23 \boldsymbol{\rho},4r)}\| \nabla  \chi(\xi;u)\|_{H^s}  \lesssim_{h,p} \frac{1}{r^{h+1}} \| \chi \|_{\mathscr{F}_{\eta,r,p}^{\cste}}
	\]
	 since the distance $\varrho$ between  $\mathcal{A}_\xi(3r)$  and the border of $\mathcal{A}^\C_\xi(\frac23 \boldsymbol{\rho},4r) $ is at least of order $r< \min_{k\in \mathcal S_p} \bro_k$.

 If $d_y=1$, it is convenient to rewrite
 \[
 \chi= \sum_{j\in \mathcal S_p} ( u_j^{+1} u_j^{-1} -\xi_j ) \chi_j\,,\quad \chi_j:= \sum_{ \boldsymbol{k} \in \mathbb{Z}^{\mathcal{S}_p}  } \sum_{q\in \mathbb{N} }    \sum_{ \boldsymbol{\ell} \in (\mathcal{S}_p^c)^{q} }  \sum_{ \boldsymbol{\sigma} \in  \{-1,1\}^{q}  } \chi _{\boldsymbol{k},\boldsymbol{e}_j}^{\boldsymbol{\ell},\boldsymbol{\sigma}}  \,   { \prod_{i=1}^q v_{\boldsymbol{\ell}_i}^{\boldsymbol{\sigma}_i} } \,  \prod_{j\in \mathcal{S}} \left(\frac{u_j^{+1}}{\sqrt{u_j^{+1} u_j^{-1}}}\right)^{\boldsymbol{k}_j}\,.
 \]
 Where, reasoning as in the case $d_y=0$, we have that  each $\chi_j, \nabla \chi_j$ is holomorphic in $ \mathcal{A}^\C_\xi(\frac23 \boldsymbol{\rho},4r)$ with the bounds
 \[
  \begin{aligned}
  	&\sup_{u^{\pm 1}\in  \mathcal{A}^\C_\xi(\boldsymbol{\rho},4r)}|  \chi_j | \lesssim_p \frac{1}{r^{2}} \| \chi \|_{\mathscr{F}_{\eta,r,p}^{\cste}}\,,  \quad \sup_{u^{\pm 1 }\in  \mathcal{A}^\C_\xi(\boldsymbol{\rho},4r)}\| \nabla \chi_j \|_{H^s} \lesssim_p \frac{1}{r} \| \chi_j\|_{\eta,r,p} \lesssim_p \frac{1}{r^{3}} \| \chi \|_{\mathscr{F}_{\eta,r,p}^{\cste}} \\
  	&\sup_{u\in  \mathcal{A}_\xi(3r)}|\mathrm{d}^k\chi_j|^{\rm op} \lesssim_p \frac{1}{r^{2+k}} \| \chi \|_{\mathscr{F}_{\eta,r,p}^{\cste}}\,,  \quad
 \sup_{u\in  \mathcal{A}_\xi(3r)}\| \mathrm{d}^{k}\nabla \chi_j (u)\|^{\rm op}_s \lesssim_{k,p} \frac{1}{r^{3+k}} \| \chi \|_{\mathscr{F}_{\eta,r,p}^{\cste}}
  \end{aligned}
 \]
  By construction 
  $$
  \begin{aligned}
  	&\mathrm{d}^{h}\nabla_{\mathcal S^c} \chi=\sum_{j\in \mathcal S_p}\sum_{k=0}^{\min(2,h)} \binom{h}{k} \mathrm{d}^{k} ( u_j^{+1} u_j^{-1} -\xi_j ) \mathrm{d}^{h-k}  \nabla_{\mathcal S^c} \chi_j \,,\\ &
  \mathrm{d}^{h}\partial_{u_j^{-1}} \chi=\sum_{j\in \mathcal S_p}\sum_{k=0}^{\min(2,h)}  \binom{h}{k}  \mathrm{d}^{k} ( u_j^{+1} u_j^{-1} -\xi_j ) \mathrm{d}^{h-k}  \partial_{u_j^{-1}} \chi_j   +  \sum_{k=0}^{1}  \binom{h}{k}  \mathrm d^k( u_j^{+1} )\mathrm{d}^{h-k}  \chi_j  
  \end{aligned}$$   
  Since these are  finite sums we may consider each term separately.  We have,  for all $u\in  \mathcal{A}_\xi(3r)$,
 \[
\| \mathrm{d}^{k} ( u_j^{+1} u_j^{-1} -\xi_j ) \mathrm{d}^{h-k}  \nabla_{\mathcal S^c} \chi_j\|_{H^s}^{\rm op} \lesssim_{p,h} (r^{2} \mathds{1}_{k=0}+  \mathds{1}_{k\geq 1}) \frac{1}{r^{3+h-k}}\| \chi \|_{\mathscr{F}_{\eta,r,p}^{\cste}} \le  \frac{1}{r^{2+h}}\| \chi \|_{\mathscr{F}_{\eta,r,p}^{\cste}}
 \]
 and
 \[
 |\mathrm d^k( u_j^{+1} )\mathrm{d}^{h-k}  \chi_j |^{\rm op}  \lesssim_{p,k}  \frac{1}{r^{2+h}}\| \chi \|_{\mathscr{F}_{\eta,r,p}^{\cste}}\,.
 \]
\end{proof}

By Cauchy--Lipchitz, Lemma \ref{lem:vector_fields} ensures the existence and uniqueness of maximal solutions to $\ic \partial_t u = \nabla \chi(\xi;u)$.  In the following lemma, we emphasize some of the basic properties of its flow. As usual in normal form paper, these flows will be the building block for all the changes of variables we are going to construct.

\begin{lemma}[Hamiltonian flows] \label{lem:flows} Let $p\geq 0$, $r\leq 2r_p$,  $\eta \geq \eta_{\min}$ and $\chi    \in \Pi_{\| \boldsymbol{m} \|_{\ell^1} \leq 1}\mathscr{F}_{\eta,r,p}^{\cste}$ be a Hamiltonian of $y-$homogeneity smaller or equal to $1$.
If 
$$
\| \chi \|_{\mathscr{F}_{\eta,r,p}^{\cste}} \leq r^{2+\delta}
$$
then, for all $ \xi \in \mathcal{O}_p^{\mathrm{full}}$, there exist a smooth map 
$$
\Phi_{\chi[\xi]} : \left\{ \begin{array}{lll} [-1,1]\times  \mathcal{A}_\xi(\frac{5}2r)& \to & \mathcal{A}_\xi(3r) \\
(t,u) & \mapsto &  \Phi_{\chi[\xi]}^t(u) \end{array}\right.
$$
 satisfying the following properties.
\begin{itemize}
\item It is the Hamiltonian flow of $\chi[\xi]$ : for all $t\in [-1,1]$ and $u\in \mathcal{A}_\xi(\frac{5}2 r)$, we have 
$$
\ic \partial_t \Phi_{\chi[\xi]}^t(u) = \nabla \chi(\xi;\Phi_{\chi[\xi]}^t(u) ), \quad \mathrm{and} \quad \Phi_{\chi[\xi]}^0(u) = u.
$$
\item It satisfies the flow property : for all $t,\tau \in [-1,1]$ and $u\in \mathcal{A}_\xi(\frac{5}2 r)$, provided that $|t+\tau|\leq 1$ and $\Phi_{\chi[\xi]}^t(u)\in \mathcal{A}_\xi(\frac{5}2 r)$, we have
$$
\Phi_{\chi[\xi]}^{\tau}(\Phi_{\chi[\xi]}^t(u)) = \Phi_{\chi[\xi]}^{t+\tau}(u).
$$
\item For all  $t\in [-1,1]$, for all $\rho < \frac{5}2 r$,   
$$
\Phi_{\chi[\xi]}^t (\mathcal{A}_\xi(\rho)) \subset \mathcal{A}_\xi(\sqrt{\rho^2+C_p\| \chi \|_{\mathscr{F}_{\eta,r,p}^{\cste}} } ),
$$
where $C_p>0$ is a constant depending only on $p$.
\item For all $t\in [-1,1]$, $\Phi_{\chi[\xi]}^t$ is symplectic.
\item It is close to the identity: for all $ j\geq 0$, $t\in [-1,1]$, $u\in   \mathcal{A}_\xi(\frac{5}2 r)$, $v^{(1)},\cdots,v^{(j)}\in H^s$,
$$
 \|\mathrm{d}^j \big(\Phi_{\chi[\xi]}^t - \mathrm{Id}  \big)(u)  (v^{(1)},\cdots,v^{(j)}) \|_{H^s} \lesssim_j \| \chi \|_{\mathscr{F}_{\eta,r,p}^{\cste}} r^{-2-j} \prod_{1\leq i\leq j} \|v^{(i)}\|_{H^s}.
 $$
\end{itemize}
\end{lemma}
\begin{proof}  The three first items are standard consequences of the Cauchy--Lipchitz theorem combined with the estimates of Lemma \ref{lem:vector_fields} and Corollary \ref{cor:var_actions}. The fourth item is a basic consequence of the fact that $\Phi_{\chi[\xi]}$ is a Hamiltonian flow. Finally, for the last item, the case $j=0$ can be deduced from Lemma \ref{lem:ouf_cest_des_fonctions} by a standard bootstrap argument while the case $j\geq 1$ can be proven by induction using the Gr\"onwall estimate (and estimates of Lemma \ref{lem:vector_fields}). For more details about the proof of this kind of properties, we refer for example to \cite{BGR24}.
\end{proof}

\subsection{Action-T\"oplitz Hamiltonians and decay properties} Now, we aim at introducing two refined subspaces of $\mathscr{F}_{\eta,r,p}^{\cste}$. The point is that we want to control a rather subtle decay property of the coefficients of the Hamiltonian. This decay property will be crucial to modulate the frequencies in order to get non resonance conditions (see Section \ref{sec:res_coho}).

\medskip

To define this class of Hamiltonians, first we have to introduce the polar-lexicographic order, the associated decreasing rearrangement and the concatenation operator for indices. 
\begin{definition}[Polar-lexicographic order $\prec$] We denote by $\prec$ the total order on $\mathbb{Z}$ such that 
$$
0 \prec -1\prec 1\prec -2\prec 2\prec\cdots.
$$
\end{definition}
\begin{definition}[Decreasing rearrangement $\boldsymbol{\ell}^*$] Given $q\geq 0$ and $\boldsymbol{\ell} \in \mathbb{Z}^d$, we denote by $\boldsymbol{\ell}^* \in \mathbb{Z}^d$  the decreasing rearrangement of $\boldsymbol{\ell}$ (for the polar-lexicographic order): $\boldsymbol{\ell}^*$ is the permutation of $\boldsymbol{\ell}$ satisfying $\boldsymbol{\ell}_q^*\preccurlyeq \cdots \preccurlyeq \boldsymbol{\ell}_1^*$. If $q'>q$, we set by convention $\boldsymbol{\ell}_{q'}^*:=0$.
\end{definition}

\begin{definition}[Concatenation operator $(\cdot \, ; \, \cdot )$]\label{def:concat}
	 Given $q\geq 0$, $i\in \mathbb{Z}$, $\varsigma \in \{-1,1\}$, $\boldsymbol{\ell}\in \mathbb{Z}^q$, $\boldsymbol{\sigma}\in \{-1,1\}^q$, we set
$$
(\boldsymbol{\ell},\boldsymbol{\sigma};i,\varsigma):= ((\boldsymbol{\ell},i),(\boldsymbol{\sigma},\varsigma))
$$
and, to shorten the notation,
$$
(\boldsymbol{\ell},\boldsymbol{\sigma};i,\pm 1) := (\boldsymbol{\ell},\boldsymbol{\sigma};i,1;i,-1).
$$
In many cases, we will omit the parentheses to keep the notation uncluttered.
\end{definition}
Now we define the sets of indices associated with monomials whose $2n$ highest modes are distinct actions.
 \begin{definition}[Sets of indices with high modes pairing $ \mathbf{C}_{q,n} $] \label{def:highpair} For all $q\geq 0$ and $n\leq q$, we set
$$
 \mathbf{C}_{q,n} :=  \mathfrak{S}_q \{ (\boldsymbol{\ell},\boldsymbol{\sigma}; \boldsymbol{i}_n,\pm 1; \cdots ;\boldsymbol{i}_1,\pm 1) \ | \ \boldsymbol{\ell}_{1}^* \prec \boldsymbol{i}_n \prec \cdots \prec \boldsymbol{i}_1 \} \subset \mathbb{Z}^q \times \{-1,1\}^q.
 $$
\end{definition}

\medskip

		In order to define our class of Hamiltonians, we start by introducing two spaces of Hamiltonians characterized by some algebraic properties. First we define the $n$-action T\"oplitz Hamiltonians as being the Hamiltonians whose non zero coefficients correspond to monomials whose highest modes are given by $n$ unpaid actions and do not depend on the indices of these $n$ actions.
\begin{definition}[$n$-action T\"oplitz Hamiltonians $\mathscr{T}_{n,\eta,r,p}^{\cste} $] Given $\eta \geq 0,r>0,p\geq 0,n\geq 0$, we denote by $ \mathscr{T}_{n,\eta,r,p}^{\cste} $  the space of the Hamiltonian $P\in \mathscr{F}_{\eta,r,\mathcal{S}}^{\cste}$ which are \emph{ $n$-action T\"oplitz} in the sense that,  for all $\boldsymbol{k} \in \mathbb{Z}^{\mathcal{S}_p},\boldsymbol{m} \in \mathbb{N}^{\mathcal{S}_p},q\in \mathbb{N}, \boldsymbol{\ell} \in (\mathcal{S}_p^c)^q ,\boldsymbol{\sigma} \in  \{-1,1\}^q$, its coefficients satisfy
$$
P^{\boldsymbol{\ell},\boldsymbol{\sigma}}_{\boldsymbol{k},\boldsymbol{m}} = 0 \quad \mathrm{if} \quad (\boldsymbol{\ell},\boldsymbol{\sigma}) \notin  \mathbf{C}_{q,n} 
$$
and
$$
P_{{\boldsymbol{k}},{\boldsymbol{m}}}^{{\boldsymbol{\ell}},{\boldsymbol{\sigma}};\boldsymbol{i}_n,\pm 1; \cdots ; \boldsymbol{i}_1,\pm 1} = P_{{\boldsymbol{k}},{\boldsymbol{m}}}^{{\boldsymbol{\ell}},{\boldsymbol{\sigma}};\boldsymbol{j}_n,\pm 1; \cdots ; \boldsymbol{j}_1,\pm 1}
	$$
{whenever  $\boldsymbol{i},\boldsymbol{j} \in (\mathcal{S}_p^c)^n$ are such that both $({\boldsymbol{\ell}},{\boldsymbol{\sigma}};\boldsymbol{i}_n,\pm 1; \cdots ; \boldsymbol{i}_1,\pm 1)$ and $({\boldsymbol{\ell}},{\boldsymbol{\sigma}};\boldsymbol{i}_n,\pm 1; \cdots ; \boldsymbol{i}_1,\pm 1)$ belong to $ \mathbf{C}_{q,n}$.}
%	$\boldsymbol{\ell}_h \prec \boldsymbol{j}_n \prec \cdots \prec \boldsymbol{j}_1$ and $\boldsymbol{\ell}_h \prec \boldsymbol{j}_n \prec \cdots \prec \boldsymbol{j}_1$, for all $1\leq h\leq q$. 
\end{definition}

Now, we define the action smoothing Hamiltonians of depth $n$. They are Hamiltonians whose coefficients associated with monomials $v_{\boldsymbol{\ell}}^{\boldsymbol{\sigma}}$ with  $ (\boldsymbol{\ell},\boldsymbol{\sigma}) \in  \mathbf{C}_{q,n} $ satisfying a decay property.
\begin{definition}[Action smoothing Hamiltonians of depth $n$, $\mathscr{S}_{n,\eta,r,p}^{\cste} $] Given $\eta \geq 0,r>0,p\geq 0,n\geq 0$, we denote by $ \mathscr{S}_{n,\eta,r,p}^{\cste} $  the space of the Hamiltonian $P\in \mathscr{F}_{\eta,r,\mathcal{S}}^{\cste}$ which are \emph{action smoothing of depth $n$} in the sense that,  for all $\boldsymbol{k} \in \mathbb{Z}^{\mathcal{S}_p},\boldsymbol{m} \in \mathbb{N}^{\mathcal{S}_p},q\in \mathbb{N}, \boldsymbol{\ell} \in (\mathcal{S}_p^c)^q ,\boldsymbol{\sigma} \in  \{-1,1\}^q$, the following limits exist and satisfy
$$
\lim_{|\boldsymbol{i}_n|\to \infty} \cdots \lim_{|\boldsymbol{i}_1|\to \infty} P^{\boldsymbol{\ell},\boldsymbol{\sigma};\boldsymbol{i}_n,\pm 1; \dots ; \boldsymbol{i}_1,\pm 1 }_{\boldsymbol{k},\boldsymbol{m}} = 0.
$$
\end{definition}

\begin{remark} \label{rem:direct_sum} Thes spaces defined above are complementary, in the sense that
$$
\mathscr{S}_{n,\eta,r,p}^{\cste}   \cap \mathscr{T}_{n,\eta,r,p}^{\cste}  = \{0\}.
$$
In particular, it will be useful to note that:
\begin{center}
the spaces $ \mathscr{S}_{1,\eta,r,p}^{\cste} ,  \ \mathscr{T}_{1,\eta,r,p}^{\cste} \cap \mathscr{S}_{2,\eta,r,p}^{\cste}$ and $\mathscr{T}_{2,\eta,r,p}^{\cste}$ form a direct sum.
\end{center}
\end{remark}

Now, we are in position to define the spaces of Hamiltonians we will use in most of the paper. They are Hamiltonians admitting decomposition in a sum of action T\"oplitz Hamiltonians satisfying some quantified action smoothing properties. Concrete examples  of such decomposition, naturally arising in \eqref{eq:NLS},  will be given in next subsection (see Lemma \ref{lem:explicit_dec} below). The space $\mathscr{C}_{\eta,r,p}^{\cste}$ will correspond to Hamiltonian obtained from the resolution of a cohomological equation. They satisfy a stronger decay property.

\begin{definition}[Hamiltonians $\mathscr{H}_{\eta,r,p}^{\cste}$ and $\mathscr{C}_{\eta,r,p}^{\cste}$] \label{def:class_take3} Let $p\geq 0$, $r>0$ and $\eta \ge 0$. We denote by $ \mathscr{H}_{\eta,r,p}^{\cste}$, the set of the formal series $P \in \mathscr{F}_{\eta,r,p}^{\cste}$ of the form 
		\begin{equation}
			\label{eq:cond2}
			P = P_0+P_1+P_2, \quad \mathrm{with} \ P_0\in \mathscr{S}_{1,\eta,r,p}^{\cste},\, P_1\in  \mathscr{T}_{1,\eta,r,p}^{\cste} \cap \mathscr{S}_{2,\eta,r,p}^{\cste} , \, P_2\in \mathscr{T}_{2,\eta,r,p}^{\cste}
		\end{equation}
		satisfying the bound
		\begin{equation}\label{eq:normcst2}
		\| P \|_{\mathscr{H}_{\eta,r,p}^{\cste}} :=	\sum_{\substack{(\boldsymbol{k},\boldsymbol{m})\in \mathbb{Z}^{\mathcal{S}_p} \times \mathbb{N}^{\mathcal{S}_p}\\ q\in \mathbb{N}}} \!\!\max(1,q ,\|\bk\|_{\ell^1} , \|\bm\|_{\ell^{\infty}})|P|^{\Theta,q}_{\boldsymbol{k},\boldsymbol{m}}   r^{q+2\|\bm\|_{\ell^1}} e^{\eta(\|\bk\|_{\ell^1} +q+2\|\bm\|_{\ell^1})}<\infty\,,\
		\end{equation}
		where 
		\begin{equation}\label{eq:Theta2}
			|P|^{\Theta,q}_{\boldsymbol{k},\boldsymbol{m}} := \hspace{-0.7cm} \sup_{\substack{ n=0,1,2 \\ {\boldsymbol{\ell}}\in (\mathcal S_p^c)^{q-2n}\\ \boldsymbol{\ell}_1^* \prec{\boldsymbol{i}}_n\prec\dots\prec{\boldsymbol{i}}_1\\ \boldsymbol{\sigma}\in\{-1,1\}^{q-2n}}}  \! \! \! \! \! \!  |P^{\boldsymbol{\ell},\boldsymbol{\sigma};{\boldsymbol{i}}_n,\pm 1;\dots;\boldsymbol{i}_1,\pm 1}_{n,{\boldsymbol{k}},{\boldsymbol{m}}}| \Theta_{n,{\boldsymbol{\ell}},\boldsymbol{\sigma}}^{-1}\,, \quad 
			\Theta_{n,\boldsymbol{\ell},\boldsymbol{\sigma}}^{-1}  :=  1 + \mathds{1}_{n\ne 2} \mathds{1}_{(\boldsymbol{\ell},\boldsymbol{\sigma}) \in {\mathbf{C}_{q-2n,1}}}  \| \boldsymbol{\ell} \|_{\ell^\infty}.
		\end{equation}
%		\mic{	\begin{equation}\label{eq:Theta2bis}
%			|P|^{\Theta,q}_{\boldsymbol{k},\boldsymbol{m}} := \hspace{-1cm} \sup_{\substack{ n=0,1,2 \\ {\boldsymbol{\ell}}\in (\mathcal S_p^c)^{q-2n}\\ \boldsymbol{\ell}_{1} \preccurlyeq \cdots \preccurlyeq \boldsymbol{\ell}_{q-2n} \prec{\boldsymbol{i}}_n\prec\dots\prec{\boldsymbol{i}}_1\\ \boldsymbol{\sigma}\in\{-1,1\}^{q-2n}}}  \! \! \! \! \! \!  |P^{\boldsymbol{\ell},\boldsymbol{\sigma};{\boldsymbol{i}}_n,\pm 1;\dots;\boldsymbol{i}_1,\pm 1}_{n,{\boldsymbol{k}},{\boldsymbol{m}}}| \Theta_{n,{\boldsymbol{\ell}},\boldsymbol{\sigma}}^{-1}\,, \quad 
%			\Theta_{n,\boldsymbol{\ell},\boldsymbol{\sigma}}^{-1}  :=  1 + \mathds{1}_{n\ne 2} \mathds{1}_{(\boldsymbol{\ell},\boldsymbol{\sigma}) \in \mathbf{C}_{q-2n,1} }  \| \boldsymbol{\ell} \|_{\ell^\infty}.
%		\end{equation}}
		We denote by $\mathscr{C}_{\eta,r,p}^{\cste}$ the subspace of the formal series $P \in \mathscr{H}^{\cste}_{\eta,r,p}$ satisfying  $P_2=0$ and the extra bound
		\begin{equation}\label{eq:norm_C}
		\| P \|_{\mathscr{C}_{\eta,r,p}^{\cste}} :=	\sum_{\substack{(\boldsymbol{k},\boldsymbol{m})\in \mathbb{Z}^{\mathcal{S}_p} \times \mathbb{N}^{\mathcal{S}_p}\\ q\in \mathbb{N}}} \!\!\max(1,q ,\|\bk\|_{\ell^1} , \|\bm\|_{\ell^{\infty}})|P|^{\Upsilon,q}_{\boldsymbol{k},\boldsymbol{m}}   r^{q+2\|\bm\|_{\ell^1}} e^{\eta(\|\bk\|_{\ell^1} +q+2\|\bm\|_{\ell^1})}<\infty\,,\
		\end{equation}
		where  
		\begin{equation}\label{eq:Ups}
			|P|^{\Upsilon,q}_{\boldsymbol{k},\boldsymbol{m}} := \! \! \! \! \! \! \sup_{\substack{ n=0,1 \\ {\boldsymbol{\ell}}\in (\mathcal S_p^c)^{q-2n}\\  \boldsymbol{\ell}_1^* \prec{\boldsymbol{i}}_n\prec\dots\prec{\boldsymbol{i}}_1\\ \boldsymbol{\sigma}\in\{-1,1\}^{q-2n}}}  \! \! \! \! \! \! |P^{\boldsymbol{\ell},\boldsymbol{\sigma};{\boldsymbol{i}}_n,\pm 1;\dots;\boldsymbol{i}_1,\pm 1}_{n,{\boldsymbol{k}},{\boldsymbol{m}}}| \Upsilon_{n,{\boldsymbol{\ell}},\boldsymbol{\sigma}}^{-1}\,, \quad 
			\Upsilon_{n,\boldsymbol{\ell},\boldsymbol{\sigma}}^{-1}  :=  \langle \| \boldsymbol{\ell} \|_{\ell^\infty} \rangle.
		\end{equation}
	\end{definition}

	\begin{remark}
	This definition makes sense because the decomposition \eqref{eq:cond2} is unique (see Remark \ref{rem:direct_sum}).
\end{remark}
Now, we introduce a notation for the projection $P \mapsto P_n$.
\begin{definition}[Projection $\Pi_{\mathrm{Top}=n}$] Given $n\in \{0,1,2\}$ and 
$P=P_0+P_1+P_2 \in   \mathscr{S}_{1,\eta,r,p}^{\cste} \oplus (\mathscr{T}_{1,\eta,r,p}^{\cste} \cap \mathscr{S}_{2,\eta,r,p}^{\cste})\oplus\mathscr{T}_{2,\eta,r,p}^{\cste}$, we set $\Pi_{\mathrm{Top}=n} P := P_n$.
\end{definition}
\begin{remark} These projections are clearly contractions in $ \mathscr{H}_{\eta,r,p}^{\cste}$ and $ \mathscr{C}_{\eta,r,p}^{\cste}$.
\end{remark}

We conclude this subsection with the observation that all the projections defined in subsection \ref{sub:projection} are contractions in $ \mathscr{H}_{\eta,r,p}^{\cste}$ (and $ \mathscr{C}_{\eta,r,p}^{\cste}$ also but we will not need it). It is based on the following lemma whose proof follows directly from the definitions.
\begin{lemma} \label{stc}
	Let $p\geq 0$, $r>0$, $\eta > 0$, $n\geq 1$, $d\geq 0$. The spaces $\mathscr{T}_{n,\eta,r,p}^{\cste}$ and $\mathscr{S}_{n,\eta,r,p}^{\cste}$ are stable for the projections $\Pi_{\mathrm{deg}=d} $, $\Pi_{\| \boldsymbol{m} \|_{\ell^1}=d} $ and $\Pi_{\mathrm{int}} $.
\end{lemma}
\begin{corollary}  All the projection defined in subsection \ref{sub:projection} are contractions in $ \mathscr{H}_{\eta,r,p}^{\cste}$ and $ \mathscr{C}_{\eta,r,p}^{\cste}$.
\end{corollary}

\subsection{Examples : application to \eqref{eq:NLS}} In this subsection we prove that the nonlinearity of \eqref{eq:NLS} belongs to this class with $p=0$. This is somehow the initialization of our construction.
\begin{remark}
	When $p=0$, we have by definition $\mathcal S_0=\emptyset$ (and $\# \mathcal{O}_p^{\mathrm{full}} = 1$). So given $r>0$, $\eta \geq \eta_{\min}$, $  \mathscr{F}_{\eta,r,0}\equiv \mathscr{F}_{0,re^{\eta},0}^{\cste}$ for any $\eta,r$, same for $  \mathscr{F}_{\eta,r,0}^{\cste}$. 
As a consequence, for $P\in  \mathscr{H}_{\eta,r,0}$, we do not write the parameter $\xi$ when referring to the function associated with $P$.
\end{remark}
We start by proving that the ``building blocks'' of the \eqref{eq:NLS} Hamiltonian $H^{\eqref{eq:NLS}}$ belong to $\mathscr{H}_{0,r,0}^{\cste}$, with an explicit decomposition of the form \eqref{eq:cond2}.
\begin{lemma} \label{lem:explicit_dec} Let $\mathfrak{q}\geq 2$. For all $u\in H^s$  we have 
	\begin{equation}
		\label{eq:obvious}
		N_{\mathfrak q}(u):=	\binom{2\mathfrak{q}}{ \mathfrak{q}} \| u \|_{L^{2\mathfrak{q}}}^{2\mathfrak{q}} := \binom{2\mathfrak{q}}{ \mathfrak{q}}\int_{\mathbb{T}} |u(x)|^{2\mathfrak{q}} \frac{\mathrm{d}x}{2\pi}
= P_0(u) + P_1(u) +P_2(u)
	\end{equation}
	where $P_0\in \mathscr{S}_{1,0,1,0}^{\cste},\, P_1\in  \mathscr{T}_{1,0,1,0}^{\cste} \cap \mathscr{S}_{2,0,1,0}^{\cste} , \, P_2\in \mathscr{T}_{2,0,1,0}^{\cste}$ are explicitly given by
	$$
	P_0(u) = \hspace{-0.7cm}\sum_{\substack{  (\boldsymbol{\ell} , \boldsymbol{\sigma} ) \notin \mathbf{C}_{2\mathfrak{q},1}   \\   (\boldsymbol{\ell} , \boldsymbol{\sigma} ) \in \mathcal{M}_{2\mathfrak{q}} }} \hspace{-0.3cm}  u_{\boldsymbol{\ell}}^{\boldsymbol{\sigma}}, \quad
	P_1(u) = \hspace{-0.7cm}\sum_{\substack{  (\boldsymbol{\ell} , \boldsymbol{\sigma} ) \in \mathbf{C}_{2\mathfrak{q},1} \setminus \mathbf{C}_{2\mathfrak{q},2}    \\ (\boldsymbol{\ell} , \boldsymbol{\sigma} ) \in \mathcal{M}_{2\mathfrak{q}}  }  } \hspace{-0.3cm}  u_{\boldsymbol{\ell}}^{\boldsymbol{\sigma}}, \quad \mathrm{and} \quad
	P_2(u) = \hspace{-0.7cm}\sum_{\substack{  (\boldsymbol{\ell} , \boldsymbol{\sigma} ) \in  \mathbf{C}_{2\mathfrak{q},2}   \\ (\boldsymbol{\ell} , \boldsymbol{\sigma} ) \in \mathcal{M}_{2\mathfrak{q}}  }  } \hspace{-0.3cm}  u_{\boldsymbol{\ell}}^{\boldsymbol{\sigma}}.
	$$
	and $\mathcal{M}_{2\mathfrak{q}} := \{ (\boldsymbol{\ell},\boldsymbol{\sigma}) \in \mathbb{Z}^q \times \{-1,1\}^q \ | \  \boldsymbol{\sigma}_1 \boldsymbol{\ell}_1 + \cdots + \boldsymbol{\sigma}_{2\mathfrak{q}} \boldsymbol{\ell}_{2\mathfrak{q}} = \boldsymbol{\sigma}_1 + \cdots + \boldsymbol{\sigma}_{2\mathfrak{q}}  = 0 \}$.
	Consequently, for all $r>0$, we have
	\[
	\| N_{\mathfrak{q}} \|_{\mathscr{H}_{0,r,0}^{\cste}}= \| P_0 \|_{\mathscr{H}_{0,r,0}^{\cste}}=\| P_1 \|_{\mathscr{H}_{0,r,0}^{\cste}}=\| P_2 \|_{\mathscr{H}_{0,r,0}^{\cste}} = 2\mathfrak{q}r^{2 \mathfrak{q}}.
	\]
\end{lemma}
\begin{proof} To obtain  \eqref{eq:obvious}, we just decompose the integral in Fourier series
	\[
	N_{\mathfrak{q}} = \binom{2\mathfrak{q}}{ \mathfrak{q}} \| u \|_{L^{2\mathfrak{q}}}^{2\mathfrak{q}} = \sum_{\substack{ \boldsymbol{\sigma}_1 \boldsymbol{\ell}_1 + \cdots + \boldsymbol{\sigma}_{2\mathfrak{q}} \boldsymbol{\ell}_{2\mathfrak q} =0  \\ \boldsymbol{\sigma}_1 + \cdots + \boldsymbol{\sigma}_{2\mathfrak q} =0  }  } u_{\boldsymbol{\ell}_1}^{\boldsymbol{\sigma}_1} \cdots u_{\boldsymbol{\ell}_{2\mathfrak q}}^{\boldsymbol{\sigma}_{2\mathfrak q}}\equiv \sum_{\substack{    (\boldsymbol{\ell} , \boldsymbol{\sigma} ) \in \mathcal{M}_{2\mathfrak{q}} }} \hspace{-0.3cm}  u_{\boldsymbol{\ell}}^{\boldsymbol{\sigma}}\,.
	\]
	By construction $N_{\mathfrak{q}} \in \mathscr{F}_{0,r,0}$ for all $r>0$ and $\| N_{\mathfrak{q}} \|_{\mathscr{F}_{0,r,0}^{cste}}=  2{\mathfrak{q}} $.	What we have to check is that $P_0\in \mathscr{S}_{1,0,1,0}^{\cste},\, P_1\in  \mathscr{T}_{1,0,1,0}^{\cste} \cap \mathscr{S}_{2,0,1,0}^{\cste}$ and $P_2\in \mathscr{T}_{2,0,1,0}^{\cste}$.  We focus only on $ P_1$. The proof for $P_0$ and $P_2$ are similar. By a slight abuse of notation, we denote by $\emptyset$ the unique element of $\mathbb{Z}^{\emptyset}$ and of $\mathbb{N}^{\emptyset}$.
	
	\medskip
	
	First, we focus on the action-T\"oplitz property. Given $\boldsymbol{\ell} \in \mathbb{Z}^{2\mathfrak{q}-2} ,\boldsymbol{\sigma} \in  \{-1,1\}^{2\mathfrak{q}-2} $ and $i\in \mathbb{Z}$ such that $\boldsymbol{\ell}_h \prec i$ for all $1\leq h\leq 2\mathfrak{q}-2$, we have by definition
	\begin{equation}
		\label{eq:pas_bete}
		P^{\boldsymbol{\ell},\boldsymbol{\sigma};i,\pm 1}_{1,\emptyset,\emptyset} = \mathds{1}_{(\boldsymbol{\ell} , \boldsymbol{\sigma} ) \notin \mathbf{C}_{2\mathfrak{q}-2,1} } \mathds{1}_{(\boldsymbol{\ell} , \boldsymbol{\sigma} ) \in \mathcal{M}_{2\mathfrak{q}-2}}
	\end{equation}
	because $(\boldsymbol{\ell},\boldsymbol{\sigma};i,\pm 1)\in \mathcal{M}_{2\mathfrak{q}}$ if and only if $(\boldsymbol{\ell} , \boldsymbol{\sigma} ) \in \mathcal{M}_{2\mathfrak{q}-2}$. Since $P^{\boldsymbol{\ell},\boldsymbol{\sigma};i,\pm 1}_{1,\emptyset,\emptyset} $ does not depend on $i$, we have that $ P_1\in  \mathscr{T}_{1,0,1,0}^{\cste}$.
	
	\medskip
	
	Now, we focus on the action smoothing property.  Given $\boldsymbol{\ell} \in \mathbb{Z}^{2\mathfrak{q}-4} ,\boldsymbol{\sigma} \in  \{-1,1\}^{2\mathfrak{q}-4} $ and $\boldsymbol{i}_2\in \mathbb{Z}$, thanks to the identity  \eqref{eq:pas_bete}, we have that
	\begin{equation}\label{gino}
		\lim_{|\boldsymbol{i}_1|\to \infty } P^{\boldsymbol{\ell},\boldsymbol{\sigma};\boldsymbol{i}_2,\pm 1;\boldsymbol{i}_1,\pm 1}_{1,\emptyset,\emptyset} = \mathds{1}_{(\boldsymbol{\ell} , \boldsymbol{\sigma} ; \boldsymbol{i}_2,\pm 1) \notin \mathbf{C}_{2\mathfrak{q}-2,1} } \mathds{1}_{(\boldsymbol{\ell} , \boldsymbol{\sigma} ;\boldsymbol{i}_2,\pm 1) \in \mathcal{M}_{2\mathfrak{q}-2}}.
	\end{equation}
	{since $(\boldsymbol{\ell} , \boldsymbol{\sigma} ; \boldsymbol{i}_2,\pm 1) \in \mathbf{C}_{2\mathfrak{q}-2,1}  $ whenever  $\boldsymbol{\ell}_1^* \prec \boldsymbol{i}_2$, we deduce that the  limit as $|\boldsymbol{i}_2|$ goes to $\infty$  of \eqref{gino} exists and vanishes.}
	%Thus, it suffices to note that this quantity vanishes as soon as $\boldsymbol{\ell}_h \prec \boldsymbol{i}_2$ for all $1\leq h\leq 2\mathfrak{q}-4$, to deduce that it limit as $|\boldsymbol{i}_2|$ goes to $\infty$ exists and vanishes. 
	This means that $ P_1\in  \mathscr{S}_{2,0,1,0}^{\cste}$. The bound on the norm follows directly from the definition.
\end{proof}
Once we have studied the ``building blocks'' $N_\mathfrak{q}$, we pass to studying totally convergent series
we have the following 
\begin{lemma}\label{nemitski}  Given an analytic function $G: (-\rho_0,\rho_0)\to \mathbb R$, with $G(0)=0$ 
we set\footnote{this is well defined as a formal power series for all $u\in H^s$ and is totally convergent for $\|u\|_{H^s} \lesssim \rho_0$.} %\mic{I am getting confused with the estimates, if we were in $\ell^1$ the same radius...otherwise we have to take a smaller one?} \jo{Yes there is a factor in $H^s$. But it seems to me we do not really care here, we just have to estimate the coefficient. This factor is taken into account in Lemma \ref{lem:ouf_cest_des_fonctions}}\mic{yes, but then maybe we have to say that this is an identity at the level of formal power series?} \jo{I dot really see your point}\mic{ Otherwise i just put $r$ sufficiently small} \jo{Yes, this is essentially what I proposed in Proposition \ref{prop:preparation_NLS}. My idea was to avoid introducing the analyticity radius of $F$ as a notation (outside of a proof).}
	$$
	P(u):= \int_{\mathbb T}  G(|u(x)|^2) \frac{\mathrm{d}x}{2\pi} = \sum_{k=1}^\infty  \frac{G^{(k)}(0)}{k!} \|u\|_{L^{2k}}^{2k}
	$$  Then for all  $r < \rho_0$
	one has $P\in {\mathscr{H}_{0,r,0}^{\cste}}$
	with the bound 
	\[
	\|P\|_{{\mathscr{H}_{0,r,0}^{\cste}}} \lesssim  \sum_{k=1}^\infty \frac{G^{(k)}(0)}{k!}  r^{2k} \,. %\quad F_k:= \frac{1}{k!} F^{(k)}(0)\,.
	\]
\end{lemma}

\begin{proof}
	By construction
	$P(u):=  \sum_{k=1}^\infty \binom{2k}{k}^{-1} \frac{G^{(k)}(0)}{k!} N_{k}(u)$, so the bound follows by the triangular inequality.
\end{proof}

We conclude with  two Hamiltonian which will play a central role in all the paper.
\begin{definition}[$Z_2$ and $Z_4$] \label{def:Z_24} We set
	$$
	Z_2(u) := \frac12 \sum_{j\in \mathbb{Z}} j^2 |u_j|^2 \quad \mathrm{and} \quad Z_4(u) := \frac14 \sum_{j\in \mathbb{Z}} |u_j|^4 -  \frac12 \| u\|_{L^2}^4.
	$$
\end{definition}

Then, we prove that $Z_4$ is the integrable part of the quartic part of the Hamiltonian of \eqref{eq:NLS} (recall that by convention we set $f'(0) = -1$).
\begin{lemma} \label{lem:part_int_P4} We have that $Z_4,\| \cdot \|_{L^4}^4 \in  \mathscr{H}_{0,1,0}^{\cste}$ and 
	\begin{equation}
		\label{eq:KP96}
		- \frac14 \Pi_{\mathrm{int}} \| \cdot \|_{L^4}^4 = Z_4.
	\end{equation}
\end{lemma}
\begin{proof} 
	%	First, decomposing the integral in Fourier series, we note that 
	%	$$
	%	\int_{\mathbb{T}} \, |u(x)|^4 \frac{\mathrm{d}x}{2\pi}  =  \sum_{  \boldsymbol{\ell}_1 + \boldsymbol{\ell}_2 - \boldsymbol{\ell}_3 -   \boldsymbol{\ell}_4 =0   } u_{\boldsymbol{\ell}_1}  u_{\boldsymbol{\ell}_2} \overline{u_{\boldsymbol{\ell}_3}} \overline{u_{\boldsymbol{\ell}_4}}= \binom{4}{2}^{-1}  \! \! \! \! \sum_{\substack{ \boldsymbol{\sigma}_1 \boldsymbol{\ell}_1 + \cdots + \boldsymbol{\sigma}_4 \boldsymbol{\ell}_4 =0  \\ \boldsymbol{\sigma}_1 + \cdots + \boldsymbol{\sigma}_4 =0  }  } u_{\boldsymbol{\ell}_1}^{\boldsymbol{\sigma}_1} \cdots u_{\boldsymbol{\ell}_4}^{\boldsymbol{\sigma}_4}.
	%	$$
	%	Thus, we have that $Z_4,\| \cdot \|_{L^4}^4 \in  \mathscr{F}_{0,1,0}^{\cste}$. Then by definition of $\Pi_{\mathrm{int}}$, we have
	%	$$
	%	- \frac14 \big( \Pi_{\mathrm{int}} \| \cdot \|_{L^4}^4\big) (u) = - \frac14 \sum_{ \substack{ (\boldsymbol{\ell}_1,\boldsymbol{\ell}_2)  = (\boldsymbol{\ell}_3,\boldsymbol{\ell}_4) \\ \mathrm{or} \ (\boldsymbol{\ell}_1,\boldsymbol{\ell}_2)  = (\boldsymbol{\ell}_4,\boldsymbol{\ell}_3)  } }  |u_{\boldsymbol{\ell}_1} |^2 |u_{\boldsymbol{\ell}_2} |^2.
	%	$$
	Applying the Poincaré's formula,  we get the identity \eqref{eq:KP96}. Then the bounds follow from  Lemma \ref{stc}
\end{proof}

We conclude this subsection by the following proposition in which we prove that the nonlinear part of \eqref{eq:NLS} belongs to the $\mathscr{H}$ class. We recall that $f'(0)$ does not appear in the formula below because, to avoid notations, we have set $f'(0)=-1$. 
\begin{proposition} \label{prop:preparation_NLS}  There exists $\rho_0>0$, there exists $R^{(0)}\in \Pi_{\mathrm{deg}\geq 4} \mathscr{H}_{4 \eta_{\min},\rho_0,0}^{\cste}$, containing only terms of degree larger than or equal to $4$ and no quartic integrable term (i.e. $\Pi_{\mathrm{deg} =4} \Pi_{\mathrm{int}} R^{(0)} =0 $), such that for all $u\in B_{H^s}(3\rho_0)$,
	$$
	H^{\eqref{eq:NLS}}(u)  =Z_2(u) + Z_4(u) + R^{(0)}(u).
	$$
	where $H^{\eqref{eq:NLS}}$ denotes the Hamiltonian of \eqref{eq:NLS} defined by \eqref{eq:def_Ham_NLS}.
\end{proposition}

\begin{proof}
By definition, passing to  the Fourier representation \eqref{fourier} 
\[
H^{\eqref{eq:NLS}}(u) = \frac12 \int_{\mathbb{T}} |\partial_x u(x)|^2 + F(|u(x)|^2) \frac{\mathrm{d}x}{2\pi}  =   \frac12 \sum_{j\in \mathbb{Z}} j^2 |u_j|^2  +\frac{1}{2}  \sum_{k=2}^\infty \frac{F^{(k)}(0)}{k!}\|u\|_{L^{2k}}^{2k}.
\] 
Now, recalling that $F^{(2)}(0)=  f'(0) =-1 $ we obtain
\[
H^{\eqref{eq:NLS}}(u) = Z_2 - \frac14 \|u\|^4_{L^4} + \frac{1}{2}  \sum_{k=3}^\infty \frac{F^{(k)}(0)}{k!}\|u\|_{L^{2k}}^{2k}\,,
\]
and the result follows, by Lemma \ref{nemitski},  setting
\[
R^{(0)}(u)=  - \frac14( \mathds{1}- \Pi_{\rm{int}}) \|u\|^4_{L^4} + \frac{1}{2}  \sum_{k=3}^\infty \frac{F^{(k)}(0)}{k!}\|u\|_{L^{2k}}^{2k}\,.
\]
\end{proof}

\subsection{Hamiltonians depending on parameters}

As usual in KAM theory, the coefficients of the Hamiltonians will depend in a Lipschitz way on the parameter $\xi$ which will leave in Borel subsets of $\mathcal{O}_p^{\mathrm{full}}$. So, now, we extend the previous definition to include this dependence. 
\begin{definition}[Hamiltonians with parameters $\mathscr{H}_{\eta,r,p,\mathcal{O}},\mathscr{C}_{\eta,r,p,\mathcal{O}},\mathscr{F}_{\eta,r,p,\mathcal{O}}$] \label{def:lip} Let $p\geq 0$, $r\leq 2r_p$, $\mathcal{O} \subset \mathcal{O}_p^{\mathrm{full}}$,  $\eta \geq 0$. We denote by
$$
 \mathscr{H}_{\eta,r,p,\mathcal{O}} =: \mathrm{Lip}(\mathcal{O}; \mathscr{H}_{\eta,r,p}^{\cste} )
 $$
 the space of the Lipschitz functions from $\mathcal{O}$ into the space of the formal series $\mathscr{H}_{\eta,r,p}^{\cste} $. Naturally we equip $ \mathscr{H}_{\eta,r,p,\mathcal{O}} $ with the norms
 $$
 \| P\|_{ \mathscr{H}_{\eta,r,p,\mathcal{O}}^{\mathrm{sup}}} = \sup_{\xi \in \mathcal{O}} \| P(\xi) \|_{ \mathscr{H}_{\eta,r,p}^{\cste} },
\quad 
 \| P\|_{ \mathscr{H}_{\eta,r,p,\mathcal{O}}^{\mathrm{lip}}} = \sup_{\substack{\xi,\zeta \in \mathcal{O}\\ \xi \neq \zeta}} \frac{\| P(\xi) - P(\zeta)  \|_{\mathscr{H}_{\eta,r,p}^{\cste}}}{\| \xi -\zeta \|_{\ell^2}},
 $$
 and
 $$
\| P\|_{ \mathscr{H}_{\eta,r,p,\mathcal{O}}^{\mathrm{tot}}} := \| P\|_{ \mathscr{H}_{\eta,r,p,\mathcal{O}}^{\mathrm{sup}}}  + r^2 \| P\|_{ \mathscr{H}_{\eta,r,p,\mathcal{O}}^{\mathrm{lip}}}. 
$$
The subspace $ \mathscr{C}_{\eta,r,p,\mathcal{O}}$ and $ \mathscr{F}_{\eta,r,p,\mathcal{O}}$ as well as the associated norms are defined in the same way.
\end{definition}

\begin{remark} We naturally extend the notations we introduced in the parameter-less case. More precisely,
given $\xi \in \mathcal{O}$, thanks to Lemma \ref{lem:ouf_cest_des_fonctions}, to shorten the notation, we set
$$
P(\xi;u) := P(\xi)(\xi;u).
$$
To avoid possible confusion, we denote by $P[\xi]:= P(\xi;\cdot)$ the smooth function from $ \mathcal{A}_\xi(3r)$ into $\mathbb{R}$ defined just above.
\end{remark}

\subsection{Translations} 

From step $p$ to step $p+1$, some normal sites become tangential sites. In particular,  we have to do a partial action-angle change of variables $v_j\rightsquigarrow \sqrt{\xi_j +y_j}e^{\ic \theta_j} $ for $\mathcal{S}_{p+1}\setminus \mathcal{S}_{p}$, and show that the resulting formal power series is in $\mathscr{H}_{\eta'.r',p+1}$ for appropriate $\eta',r'$.
This is what we do in the following proposition whose proof is postponed to Appendix \ref{sub:trans}.

\begin{proposition} \label{prop:open} Let $p\geq 0$, $\eta \in [\eta_{\min},9 \eta_{\min}]$, $\mathcal{O} \subset \mathcal{O}_p^{\mathrm{full}}$ be a Borel set and $P\in \mathscr{H}_{\eta,\rho,p,\mathcal{O}}^{\mathrm{tot}}$ where 
$$
2r_p \geq \rho \gtrsim_{p} r_{p+1}^{1-\delta} \gtrsim_{p} r.
$$
 There exists a Hamiltonian $R\in \mathscr{H}_{\eta,r,p+1,\mathcal{O}^\sharp}$, where
$$
\mathcal{O}^\sharp := \mathcal{O}\times  \big(\frac{r_{p+1}^{2(1-\delta)}}2,  r_{p+1}^{2(1-\delta)}\big)^{\mathcal{S}_{p+1}\setminus \mathcal{S}_{p}},
$$ 
such that
\begin{equation}
\label{eq:bound_open}
\| R\|_{\mathscr{H}_{\eta,r,p+1,\mathcal{O}^\sharp}^{\mathrm{tot}}} \lesssim_p \|P\|_{\mathscr{H}_{\eta,\rho,p,\mathcal{O}}^{\mathrm{tot}}}
\end{equation}
and for all $\xi=:(\xi^{\flat},\xi^{\mathrm{new}})\in \mathcal{O}^\sharp$ and all $u\in \mathcal{A}_{\xi}(3r_{p+1})$, 
\begin{equation}
\label{identita}
R(\xi;u) = P(\xi^\flat;u).
\end{equation}
Moreover, we have  
$$
\Pi_{\mathrm{deg}\leq d-1} R  =0 \quad \mathrm{where} \quad d= \min_{ P_{\boldsymbol{k},\boldsymbol{m}}^{\boldsymbol{\ell},\boldsymbol{\sigma}} \neq 0 } 2 \| \boldsymbol{m} \|_{\ell^1} + \# \{ j \ | \ \boldsymbol{\ell}_j \in \mathcal{S}_{p+1}^c\}
$$
and if $\Pi_{\mathrm{int}} P=0$ then $\Pi_{\mathrm{int}} R = 0$.
\end{proposition}
Note that, due to the translation,  the coefficients of $R$ have a  dependence on $\xi^{\rm new}$, however when one substitutes  $y_j = |u_j|^2 -\xi_j^{\rm new}$ for $j\in \mathcal{S}_{p+1}\setminus \mathcal{S}_{p}$ such artificial  dependence on $\xi^{\rm new}$ cancels out.

\subsection{Poisson bracket stability and composition by Hamiltonian flows} In this subsection, we prove that our classes of Hamiltonians are stable by Poisson bracket and  composition by Hamiltonian flow.

\begin{proposition}[Poisson bracket estimates]\label{prop:poisson_bracket_estimate} Let $p\geq 0$, $\eta \geq \eta_{\min}$, $r\leq 2r_p$, $\mathcal{O} \subset \mathcal{O}_p^{\mathrm{full}}$.
		For any  $\alpha,r>0$,  and for any  $P\in \mathscr{H}_{\eta+\alpha,r,p,\mathcal O}^{\mathrm{tot}}$ and  $\chi\in \mathscr{C}_{\eta+\alpha,r,p,\mathcal O}^{\mathrm{tot}}$, there exists a unique Hamiltonian $\{ P,Q\} \in \mathscr{H}_{\eta,r,p,\mathcal O}^{\mathrm{tot}}$ satisfying the bound
	\begin{equation}\label{estim:poisson3hom} 
		\begin{aligned}
			&	\| \{P,\chi\}\|_{\mathscr{H}^{\mathrm{tot}}_{\eta,r,p,\mathcal O}}  \lesssim_{p,\eta}\ \frac{1}{r^2 \alpha}\| P\|_{\mathscr{H}^{\mathrm{tot}}_{\eta+\alpha,r,p,\mathcal O}}\| \chi\|_{\mathscr{C}^{\mathrm{tot}}_{\eta+\alpha,r,p,\mathcal O}}
		\end{aligned}
	\end{equation}
	and such that for all $\xi \in \mathcal{O}$ and all $u\in \mathcal{A}_\xi(3r)$, $\{P,\chi\}(\xi;u) = \{P[\xi],\chi[\xi]\}(u).$ 
	Moreover, if $d_1,d_2 \in \mathbb{N}$ are such that $\Pi_{\mathrm{deg} \geq d_1} P = P$ and $\Pi_{\mathrm{deg} \geq d_2} \chi = \chi$ then 
\begin{equation}
\label{eq:deg_prop_poisson}
\Pi_{\mathrm{deg} \geq d_1+d_2-2} \{P,\chi\} = K.
\end{equation}
\end{proposition}
	\begin{proof}
		It is due to Lemma \ref{poisson_base} and the fact that setting  $\Delta_{\xi,\zeta}Q:= (Q(\xi)-Q(\zeta))/\|\xi-\zeta\|_{\ell^2}$, for any $Q\in \mathscr{H}_{\eta,r,p,\mathcal O}$, one has 
		\[
		\Delta_{\xi,\zeta}\{P,\chi\} = \{\Delta_{\xi,\zeta}P, \chi(\xi) \}+ \{P(\zeta), \Delta_{\xi,\zeta} \chi\}\,.
		\]
		The degree property is standard and is a direct corollary of the proof of Lemma \ref{lem:poisson_F} below.
	\end{proof}
	
	\begin{corollary} \label{cor:ad_k} In the setting of Proposition \ref{prop:poisson_bracket_estimate}, we have that 
	\[
	\forall k\geq 0, \quad \|\mathrm{ad}_\chi^{k} P\|_{\mathscr{H}^{\mathrm{tot}}_{\eta,r,p,\mathcal O}}\le  \left(\frac{C \| \chi\|_{\mathscr{C}^{\mathrm{tot}}_{\eta+\alpha,r,p,\mathcal O}}k}{r^2 \alpha}\right)^k \|P\|_{\mathscr{H}^{\mathrm{tot}}_{\eta+\alpha,r,p,\mathcal O}}\,,
	\]
	where $C=C(p,\eta)$ is the constant in Proposition \ref{prop:poisson_bracket_estimate}.
	\end{corollary}
	\begin{proof} Let us drop all the indexes except $\eta$ in the norm (since they are fixed). We prove  inductively that for any $0\le n\le k$ 
	\[
	\|\mathrm{ad}_\chi^{n} P\|_{\eta+\alpha \frac{k-n}{k}}\le  \left(\frac{C \|\chi\|_{\eta+\alpha} k}{r^2 \alpha}\right)^n \|P\|_{\eta+\alpha}\,,
	\]
	The bound holds true trivially  if $n=0$, then assume that it is correct up to $n<k$ so
	\[
	\begin{aligned}
		\|\mathrm{ad}_\chi^{n+1} P\|_{\eta+\alpha \frac{k-(n+1)}{k}} &= \|(\{\chi,\mathrm{ad}_\chi^{n }P\}\|_{\eta+\alpha \frac{k-n}{k} - \frac{\alpha}{k}} 
		\\ 
		&\le  C \frac{k}{r^2\alpha} \|\chi\|_{\eta+\alpha \frac{k-n}{k}} \|({\rm ad}_\chi )^{n} P\|_{\eta+\alpha \frac{k-n}{k}}
		\\	
		& \le 	C \frac{k}{r^2\alpha} \|\chi\|_{\eta+\alpha} C^n  \left(\frac{k}{r^2 \alpha}\right)^n \|\chi\|_{\eta+\alpha}^{n} \|P\|_{\eta+\alpha} \le \left(\frac{C \|\chi\|_{\eta+\alpha} k}{r^2 \alpha}\right)^{n+1} \|P\|_{\eta+\alpha}\
	\end{aligned}
	\]
by the inductive hypothesis (recall that the norms are increasing in $\eta$).
	\end{proof}

\begin{proposition}[Composition with Hamiltonian flows]\label{prop:comp_ham_flow}  Let $p\geq 0$, $\eta > \eta'\geq \eta-1 \geq  \eta_{\min}$, $r\leq 2r_p$, $\mathcal{O} \subset \mathcal{O}_p^{\mathrm{full}}$,  $P \in \mathscr{C}_{\eta,r,p,\mathcal{O}}$ and $\chi    \in \Pi_{\| \boldsymbol{m} \|_{\ell^1} \leq 1}\mathscr{C}_{\eta,r,p,\mathcal{O}}$ be a Hamiltonian of $y$-homogeneity smaller or equal to $1$ such that 
$$
\| \chi \|_{\mathscr{C}_{\eta,r,p,\mathcal{O}}^{\mathrm{tot}}} \leq r^{2+\delta} (\eta-\eta')
$$
then there exists a smooth map  $K \in C^\infty([0,1];\mathscr{H}_{\eta',r,p,\mathcal{O}})$  such that for all $\xi \in \mathcal{O}$, all $u\in \mathcal{A}_\xi(2r)$ and all $t\in [0,1]$ %I put a continuity result to ensure the integrability with respect to t
\begin{equation}
\label{eq:K_given_by_flot}
K(t)(\xi;u) = P(\xi;\Phi_{\chi[\xi]}^{t}(u))
\end{equation}
where $\Phi_{\chi[\xi]}^{t}$ denotes the Hamiltonian flow of $\chi$ defined by Lemma \ref{lem:flows}, and we have the bound
$$
\| K(t) \|_{\mathscr{H}_{\eta',r,p,\mathcal{O}}^{\mathrm{tot}}} \leq 2 \| P \|_{\mathscr{H}_{\eta,r,p,\mathcal{O}}^{\mathrm{tot}}}. 
$$
Moreover, $\Pi_{\mathrm{deg} \geq d_1} P = P$ and $\Pi_{\mathrm{deg} \geq d_2} \chi = \chi$ with $d_1\geq 0$ and $d_2 \geq 2$ then 
$
\Pi_{\mathrm{deg} \geq d_1} K = K.
$
\end{proposition}
\begin{proof} Let us drop all the indexes except $\eta$ in the norm (since they are fixed). As a consequence of the smallness assumption on $\chi$ and of Corollary \ref{cor:ad_k}, the entire series $t \mapsto \sum_{k} t^k \mathrm{ad}_{\chi}^{k} P /k!$ has a radius of converge larger than $1$ in $\mathscr{H}_{\eta',r,p,\mathcal{O}}^{\mathrm{tot}}$. It follows that, for $|t|<1$
$$
K(t) := \sum_{k\geq 0} \frac{(-t)^k}{k!} \mathrm{ad}_{\chi}^{k} P 
$$
is well defined and defines a smooth map from $[-1,1]$ to $\mathscr{H}_{\eta',r,p,\mathcal{O}}^{\mathrm{tot}}$ satisfying
\begin{equation}
\label{eq:formal_transport}
\partial_t K = -\{ \chi,K\} \quad \mathrm{and} \quad K(0) =P.
\end{equation}
The degree property is a direct consequence of the degree property of the Poisson bracket (see Proposition \ref{prop:poisson_bracket_estimate}) and the uniform bound on $K$ is a direct consequence of the estimate of Corollary \ref{cor:ad_k}.

It only remains to prove that $K$ is given by the flow of $\chi$ (i.e. \eqref{eq:K_given_by_flot}). To prove it, is suffices to notice that \eqref{eq:formal_transport} is a transport equation. Indeed, \eqref{eq:formal_transport} implies that for all $\xi \in \mathcal{O}$, we have
$$
\partial_t K[\xi]= -\{ \chi[\xi],K[\xi]\} \quad \mathrm{on} \quad \mathcal{A}_\xi(3r).
$$
Thanks to the smallness assumption on $\chi$, Lemma \ref{lem:flows} ensures that $K(t)[\xi]\circ \Phi_{\chi[\xi]}^{-t}$ is well defined and smooth on  $\mathcal{A}_\xi(5r/2)$ for $t\in [-1,1]$. Moreover, the chain rule implies that for all $u\in \mathcal{A}_\xi(5r/2)$,
\begin{equation*}
\begin{split}
\partial_t \Big(K(t)(\xi; \Phi_{\chi[\xi]}^{-t}(u))\Big) &= -\{ \chi[\xi],K[\xi]\}(\Phi_{\chi[\xi]}^{-t}(u)) + \mathrm{d}K(t)(\xi; \Phi_{\chi[\xi]}^{-t}(u))( \partial_t \Phi_{\chi[\xi]}^{-t}(u)) \\
&= \{ \chi[\xi],K[\xi]\}(\Phi_{\chi[\xi]}^{-t}(u)) + \mathrm{d}K(t)(\xi; \Phi_{\chi[\xi]}^{-t}(u))( \ic \nabla \chi[\xi](\Phi_{\chi[\xi]}^{-t}(u)))=0
\end{split}
\end{equation*}
by definition of the Poisson bracket. It follows that
$$
K(t)(\xi; \Phi_{\chi[\xi]}^{-t}(u)) = K(0)(\xi;u) = P(\xi;u).
$$ 
Finally assuming that $\mathcal{A}_\xi(2r)$ (to ensure that everything is well defined) and applying this last relation for $u \rightsquigarrow \Phi_{\chi[\xi]}^{t}(u)$, we get  \eqref{eq:K_given_by_flot} by the group property of the flow.
\end{proof}

\section{Resolution of cohomological equations} \label{sec:res_coho}

In this section, we solve our cohomological equations (see \eqref{eq:coho} below). It requires the derivation of small divisor estimates which is the most critical part of the paper (see Lemma \ref{lem:sd_estimates} below).
{In order to state the main result of this section, we first need the following notation.}
\begin{definition}[Largest fourth normal mode $\mu_P$]  Let $p\geq 0$, $\eta \geq 0$, $r>0$, $\mathcal{O} \subset \mathcal{O}_p^{\mathrm{full}}$ and $P\in  \mathscr{H}_{\eta,r,p,\mathcal{O}}$. The largest fourth normal mode of $P$, denoted $\mu_P$, 
is defined by
$$
\mu_P := \sup_{\xi \in \mathcal{O}}\sup_{P_{\boldsymbol{k},\boldsymbol{m}}^{\boldsymbol{\ell},\boldsymbol{\sigma}}(\xi) \neq 0} \boldsymbol{\ell}_4^*.
$$
\end{definition}

In the following proposition, which is the main result of this section, we solve the cohomological equations. The rest of this section is devoted to its proof. 
\begin{proposition}[Resolution of cohomological equations] \label{prop:res_coho} Let $p\geq 0$, $4 \eta_{\mathrm{min}} \geq \eta > \eta''> \eta' \geq 2\eta_{\min}$, $r\leq r_p$, $\gamma \in (0,1)$ and $\mathcal{O} \subset \mathcal{O}_p^{\mathrm{full}}$ be a Borel set. 
Given
\begin{itemize}
\item a homogeneous  Hamiltonian $P\in  \Pi_{\mathrm{deg} = d}  \mathscr{H}_{\eta,r,p,\mathcal{O}}$ of degree $d\leq 4$,   with a finite largest fourth normal mode $\mu_P<\infty$ and containing no integrable terms, i.e.$
\Pi_{\mathrm{int}} P = 0,$
\item a family of integrable Hamiltonians $Z^{(p')} \in  \Pi_{\mathrm{deg}\leq 4}  \Pi_{\mathrm{int}} \mathscr{H}_{\eta,r_{p'},p',\mathcal{O}_{p'}},$  for $1\leq p'\le p$ of degree smaller than or equal to $4$, defined for parameters in a Borel set $\mathcal{O}_{p'} \subset \mathcal{O}_{p'}^{\mathrm{full}}$ satisfying $\mathcal{O} \subset \mathcal{O}_{p'} \times \mathbb{R}^{\mathcal{S}_{p} \setminus \mathcal{S}_{p'}} $ and satisfying the bound 
$$
\| Z^{(p')} \|_{ \mathscr{H}_{\eta,r_{p'},p',\mathcal{O}_{p'}}^{\mathrm{tot}}} \leq r_{p'}^{4+\frac13}
$$
\end{itemize}
there exists a Borel set $\mathcal{O}^\sharp \subset \mathcal{O}$ and two Hamiltonians $\chi \in \mathscr{C}_{\eta'',r,p,\mathcal{O}^\sharp}$ and $R \in   \mathscr{H}_{\eta',r,p,\mathcal{O}^\sharp}$ such that 
\begin{itemize}
\item $R$ is homogeneous of degree $d+2$ and does not contain any integrable term, i.e.
$$
R = \Pi_{\mathrm{deg} = d+2} R \quad \mathrm{and} \quad \Pi_{\mathrm{int}} R =0,
$$ 
\item for all $\xi \in \mathcal{O}^\sharp$, $u\in \mathcal{A}_\xi(3r)$, we have
\begin{equation}
\label{eq:coho}
\{ Z_2 + Z_4 + \sum_{1\leq p'\leq p}Z^{(p')}[\xi^{(p')}] , \chi[\xi] \}(u) = (P+R)(\xi;u) 
\end{equation}
where $\xi^{(p')} := (\xi_j)_{j\in \mathcal{S}_{p'}}$ {and we recall that the notations $Z^{(p')}[\xi^{(p')}],\chi[\xi]$ are  defined in Lemma \ref{lem:ouf_cest_des_fonctions}},
\item $ \mathcal{O}^\sharp$ is relatively large in  $\mathcal{O}$, 
$$
\mathrm{meas}(\mathcal{O} \setminus \mathcal{O}^\sharp) \leq \gamma \, \mathrm{meas} (\mathcal{O}_p^{\mathrm{full}}).
$$
\item $\chi$ and $R$ satisfy the bounds
\begin{equation}
\label{eq:est_ChiR}
\begin{split}
\| \chi \|_{\mathscr{C}_{\eta'',r,p,\mathcal{O}^\sharp}^{\mathrm{tot}}} &\lesssim \gamma^{-2}  \langle \mu_P \rangle^{21} (\eta - \eta')^{- 2 \# \mathcal{S}_p-25} r_p^{-2+7\delta/4} \| P \|_{\mathscr{H}_{\eta,r,p,\mathcal{O}}^{\mathrm{tot}}}, \\
\| R \|_{\mathscr{H}_{\eta',r,p,\mathcal{O}^\sharp}^{\mathrm{tot}}} &\lesssim r^{2-\delta} (\eta'' - \eta')^{- 9 }  \| \chi \|_{\mathscr{C}_{\eta'',r,p,\mathcal{O}^\sharp}^{\mathrm{tot}}}  .
\end{split}
\end{equation}
%\mic{tQuestions: 
%Do you really need $\gamma^{-2}$? This comes from Lipschitz estimates right? but maybe one can compensate with the small factor in the definition of the Tot norm? } \jo{First, I would say that it is harmless. In the definition of the tot norm we compensate the critical part which is $r_p^{2}$ (we could not pay $r_p^2$). Then, I kink we cannot really fully compensate the $\gamma$ in the tot norm because it is going to $0$ in the KAM iteration of the step $p$ and we still need Lipschitz estimates after this step. Maybe there is a more elegant way to proceed but it not clear to me.} \mic{I am not sure I understand this sentence, you dont need to send $\gamma$ to zero in the KAM iteration I have in mind (I am pretty sure I can go from $\gamma$ to $c\gamma$ with $c<1$ by eventually tweaking the smallness conditions), in any case as you say here it is harmless, maybe we could discuss some other time} \jo{Yes, it looks like a discussion we had when I visited you in Rome.}
\item the supports of $\chi$ and $P$ coincide, for all $n \in \{0,1,2\}$, $\xi \in \mathcal{O}^\sharp$, $\boldsymbol{k} \in \mathbb{Z}^{\mathcal{S}_p},\boldsymbol{m} \in \mathbb{N}^{\mathcal{S}_p},q\in \mathbb{N}, \boldsymbol{\ell} \in (\mathcal{S}_p^c)^q ,\boldsymbol{\sigma} \in  \{-1,1\}^q $,
$$
 \Pi_{\mathrm{Top}=n} P_{\boldsymbol{k},\boldsymbol{m}}^{\boldsymbol{\ell},\boldsymbol{\sigma}}(\xi) \neq 0 \iff   \Pi_{\mathrm{Top}=n}\chi_{\boldsymbol{k},\boldsymbol{m}}^{\boldsymbol{\ell},\boldsymbol{\sigma}}(\xi) \neq 0.
$$
\end{itemize}
\end{proposition}

\begin{remark}
The cohomological equation \eqref{eq:coho} is non standard. To be standard, there should not be a higher degree remainder term like $R$ and the integral part should be homogeneous of degree $2$. Actually the proof of this proposition is based on the resolution of such a standard equation. The point here is that, to get the appropriate quadratic integrable Hamiltonian, we have to translate the Hamiltonians $Z^{(p')}$ (as in \eqref{eq:dec_quart}). This requires many notations which will not be useful in the rest of the paper, so we did not want to introduce them now. That is why we chose this more compact formulation. Moreover, in any case, we  also have to estimate somewhere in paper $R$ which is nothing but  the Poisson bracket between the quartic part of the integrable Hamiltonians (after translation) and $P$.
\end{remark}

\subsection{Preparatory lemmas} In Proposition \ref{prop:res_coho}, we gain one derivative by solving the cohomological equation ($P$ lives in a $\mathscr{H}$ space while $\chi$ lives in a $\mathscr{C}$ space). The price to pay is the factor $\mu_P^{21}$. This gain will mainly be a corollary of the following technical lemma. 
\begin{lemma}[Gain of one derivative] \label{lem:gain_one_der} Let $\boldsymbol{\ell} \in \mathbb{Z}^3$, $\boldsymbol{\sigma}\in\{-1,1\}^3$  such that 
$$
|\boldsymbol{\ell}_1| \geq |\boldsymbol{\ell}_2| \geq |\boldsymbol{\ell}_3|, \quad (\boldsymbol{\ell}_1,\boldsymbol{\sigma}_1)\neq (\boldsymbol{\ell}_2,-\boldsymbol{\sigma}_2)  \quad \mathrm{and} \quad |\boldsymbol{\sigma}_1\boldsymbol{\ell}_1+\boldsymbol{\sigma}_2\boldsymbol{\ell}_2+\boldsymbol{\sigma}_3\boldsymbol{\ell}_3|^2\leq \frac14 |\boldsymbol{\ell}_1|
$$
then
$$
\Big| \boldsymbol{\sigma}_1\boldsymbol{\ell}_1^2+\boldsymbol{\sigma}_2\boldsymbol{\ell}_2^2+\boldsymbol{\sigma}_3\boldsymbol{\ell}_3^2\Big| \ge \frac14  |\boldsymbol{\ell}_1| .
$$
\end{lemma}
\begin{proof}
Without loss of generality, we assume that $|\boldsymbol{\ell}_1|\geq1$.
If $\boldsymbol{\sigma}_1=\boldsymbol{\sigma}_2$ then 
$$
\Big| \boldsymbol{\sigma}_1\boldsymbol{\ell}_1^2+\boldsymbol{\sigma}_2\boldsymbol{\ell}_2^2+\boldsymbol{\sigma}_3\boldsymbol{\ell}_3^2\Big|\geq  |\boldsymbol{\ell}_1|^2 \geq |\boldsymbol{\ell}_1|.
$$
Now we assume $\boldsymbol{\sigma}_1=-\boldsymbol{\sigma}_2$ and we note that our assumption $(\boldsymbol{\ell}_1,\boldsymbol{\sigma}_1)\neq (\boldsymbol{\ell}_2,-\boldsymbol{\sigma}_2)$ implies $\boldsymbol{\ell}_1\neq\boldsymbol{\ell}_2$. Furthermore the other assumption, $|\boldsymbol{\sigma}_1\boldsymbol{\ell}_1+\boldsymbol{\sigma}_2\boldsymbol{\ell}_2+\boldsymbol{\sigma}_3\boldsymbol{\ell}_3|^2\leq \frac12 |\boldsymbol{\ell}_1|$, implies  that $\boldsymbol{\ell}_1$ and $\boldsymbol{\ell}_2$ have the same sign (if not $|\boldsymbol{\sigma}_1\boldsymbol{\ell}_1+\boldsymbol{\sigma}_2\boldsymbol{\ell}_2+\boldsymbol{\sigma}_3\boldsymbol{\ell}_3|\geq |\boldsymbol{\ell}_1|$).

\medskip

Then we note that 
$$
 \Big| \boldsymbol{\sigma}_1\boldsymbol{\ell}_1^2+\boldsymbol{\sigma}_2\boldsymbol{\ell}_2^2+\boldsymbol{\sigma}_3\boldsymbol{\ell}_3^2\Big|\geq \boldsymbol{\ell}_1^2-\boldsymbol{\ell}_2^2-\boldsymbol{\ell}_3^2=2\boldsymbol{\ell}_1(\boldsymbol{\ell}_1-\boldsymbol{\ell}_2)- (\boldsymbol{\ell}_1-\boldsymbol{\ell}_2)^2-\boldsymbol{\ell}_3^2.
$$
On the other hand we know that (with $\delta =-\boldsymbol{\sigma}_1\boldsymbol{\sigma}_3$)
\begin{equation}\label{eq:ok}(\boldsymbol{\ell}_1-\boldsymbol{\ell}_2-\delta\boldsymbol{\ell}_3)^2\leq \frac14 |\boldsymbol{\ell}_1|\end{equation}
and thus
$$
-\boldsymbol{\ell}_3^2\geq (\boldsymbol{\ell}_1-\boldsymbol{\ell}_2)^2-2\delta \boldsymbol{\ell}_3(\boldsymbol{\ell}_1-\boldsymbol{\ell}_2)-\frac14 |\boldsymbol{\ell}_1|.
$$
Therefore we get
\begin{align*}
 \Big| \boldsymbol{\sigma}_1\boldsymbol{\ell}_1^2+\boldsymbol{\sigma}_2\boldsymbol{\ell}_2^2+\boldsymbol{\sigma}_3\boldsymbol{\ell}_3^2\Big|&\geq 2(\boldsymbol{\ell}_1-\boldsymbol{\ell}_2)(\boldsymbol{\ell}_1-\delta\boldsymbol{\ell}_3)-\frac14 |\boldsymbol{\ell}_1|= 2|\boldsymbol{\ell}_1-\boldsymbol{\ell}_2||\boldsymbol{\ell}_1-\delta\boldsymbol{\ell}_3|-\frac14 |\boldsymbol{\ell}_1|.
\end{align*}
Now if $|\boldsymbol{\ell}_1-\delta\boldsymbol{\ell}_3|\geq \frac14|\boldsymbol{\ell}_1|$ we get since $ |\boldsymbol{\ell}_1-\boldsymbol{\ell}_2|\geq1$
$$ \Big| \boldsymbol{\sigma}_1\boldsymbol{\ell}_1^2+\boldsymbol{\sigma}_2\boldsymbol{\ell}_2^2+\boldsymbol{\sigma}_3\boldsymbol{\ell}_3^2\Big|\geq \frac12|\boldsymbol{\ell}_1| -\frac14 |\boldsymbol{\ell}_1|\geq \frac14 |\boldsymbol{\ell}_1|.
$$
If on the contrary $|\boldsymbol{\ell}_1-\delta\boldsymbol{\ell}_3|< \frac14|\boldsymbol{\ell}_1|$, then  $\delta\boldsymbol{\ell}_3$ must have the same sign than $\boldsymbol{\ell}_1$ and $|\boldsymbol{\ell}_2|\geq|\boldsymbol{\ell}_3|>\frac34 |\boldsymbol{\ell}_1|$.  Therefore
$$|\boldsymbol{\ell}_1-\boldsymbol{\ell}_2-\delta\boldsymbol{\ell}_3|> \frac12 |\boldsymbol{\ell}_1|$$
which is incompatible with \eqref{eq:ok}. 
\end{proof}

Our measure estimates will mainly rely on the following basic lemma.
\begin{lemma} \label{lem:measure_basic} Let $d\geq 1$, $\mathcal{C} \subset \mathbb{R}^d$ be a compact convex subset of $\mathbb{R}^d$,  $\mathcal{O}\subset \mathcal{C} $ be a Borel subset, $L$ be a vector subspace of dimension $1$ and $f:\mathcal{O} \to \mathbb{R}$ be a measurable function such that
$$
\forall \xi,\zeta \in \mathcal{O}, \quad \xi-\zeta \in L \implies  |f(\xi) - f(\zeta)|\geq \| \xi-\zeta \|_{\ell^2},
$$
then for any projector $\mathfrak{p} : \mathbb{R}^d \to \mathbb{R}^d$ of kernel $\mathrm{Ker} \, \mathfrak{p} = L$ and all $\gamma>0$, we have 
$$
\mathrm{meas}\{ \xi \in \mathcal{O} \ | \ |f(\xi) |<\gamma  \} \leq 2 \gamma \, \mathrm{meas} \, \mathfrak{p}( \mathcal{C}) .
$$
\end{lemma}
\begin{proof} It suffices to disintegrate the Lebesgue measure following the decomposition $\mathbb{R}^d = \mathrm{Ker} \,\mathfrak{p} \oplus \mathrm{Im} \,\mathfrak{p} $.
\end{proof}
\begin{corollary} \label{cor:give_measure_est} In particular if $ \mathcal{C} = (0,\boldsymbol{a}_1)\times \cdots \times (0,\boldsymbol{a}_d)$ where  $(\boldsymbol{a}_j)_{1\leq j\leq d} \in \mathbb{R}_+^*$ be a non-increasing sequence and $L = \mathbb{R} \mathrm{w} $ where $\mathrm{w} \in \mathbb{R}^d \setminus \{0\}$, we have
$$
\mathrm{meas}\{ \xi \in \mathcal{O} \ | \ |f(\xi) |<\gamma  \} \lesssim_d \gamma \prod_{j\neq k} \boldsymbol{a}_j \quad \mathrm{where} \quad k= \mathrm{max} \{  j \ |  \mathrm{w}_j \neq 0 \}. 
$$
\end{corollary}
\begin{proof}
We set $A = \mathrm{diag}(\boldsymbol{a}_1,\cdots,\boldsymbol{a}_d)$ the diagonal matrix with coefficients $\boldsymbol{a}_1,\cdots,\boldsymbol{a}_d$ and we define
$
g: A^{-1} \mathcal{O} \to \mathbb{R}
$
by the relation 
$$
g(x) = \boldsymbol{a}_k^{-1} f(Ax) \quad \mathrm{for \ all} \ x\in A^{-1} \mathcal{O}.
$$
Thus, $g$ satisfies the assumptions of Lemma \ref{lem:measure_basic} (the direction is $L = \mathbb{R} A^{-1} \mathrm{w}$ and the convex $\mathcal{C} = [0,1]^d$). As a consequence, we have
$$
\forall \gamma>0, \quad \mathrm{meas}\{ x \in A^{-1} \mathcal{O} \ | \ |  \boldsymbol{a}_k^{-1} f(Ax)  |<\gamma  \} \leq 2 \gamma \max_{\mathfrak{p} \in \mathcal{P}^{\mathrm{orth}}_{\mathrm{rg} = d-1}}  \mathrm{meas} \, \mathfrak{p}( [0,1]^d) \lesssim_d \gamma
$$
where $\mathcal{P}^{\mathrm{orth}}_{\mathrm{rg} = d-1}$ denotes the set of the orthogonal projectors of rank $d-1$ (note that it is compact). Finally, to get the result, it suffices to apply the changes of variables $\xi = Ax$ and $\gamma \leftarrow \boldsymbol{a}_k^{-1} \gamma$. 
\end{proof}

A crucial ingredient to prove the twist condition is the boundedness of the following operator which is a kind of discrete Hilbert transform.
\begin{lemma} \label{lem:Hilb} For all $u,v\in L^2$, we have
$$
\Big|  \sum_{i,j\in \mathbb{Z}}  \frac{u_i v_j}{ 1+|i|+|j| } \Big| \lesssim \| u\|_{L^2} \| v\|_{L^2}. 
$$
\end{lemma}
\begin{proof}
If suffices to notice that
$$
x \mapsto \sum_{j\neq 0} \frac{e^{\ic jx}}{j} \in L^\infty(\mathbb{T})
$$
to decompose the sum in several sub-regions depending on the signs of $i$ and $j$, to apply the Plancherel identity and to conclude by applying the Cauchy--Schwarz inequality.
\end{proof}

{\begin{lemma}\label{lem:mu_P_finite} Let $p\geq 0$, $\eta \geq 0$, $r>0$, $\mathcal{O} \subset \mathcal{O}_p^{\mathrm{full}}$ and $P\in  \mathscr{H}_{\eta,r,p,\mathcal{O}}$ be such that $\mu_P<\infty$. Then, setting $P_n = \Pi_{\mathrm{Top}=n}P$ for $0\leq n\leq 2$, we have 
$$
P_2 =0 \quad \mathrm{and} \quad \max(\mu_{P_0},\mu_{P_1}) \leq \mu_P.
$$ 
\end{lemma}
\begin{proof}
First, since $\mu_P<\infty$, we notice that for $\boldsymbol{k} \in \mathbb{Z}^{\mathcal{S}_p},\boldsymbol{m} \in \mathbb{N}^{\mathcal{S}_p},q\in \mathbb{N}, \boldsymbol{\ell} \in (\mathcal{S}_p^c)^q ,\boldsymbol{\sigma} \in  \{-1,1\}^q$ such that $\bl_1^*\prec \bi_2 \prec \bi_1$ we have
$$
P_{2,\bk,\bm}^{\bl,\bs;\bi_1,\pm1;\bi_2,\pm 1} =\lim_{|\bj_2| \to \infty} \lim_{|\bj_1| \to \infty} P_{\bk,\bm}^{\bl,\bs;\bj_1,\pm1;\bj_2,\pm 1} =  0.
$$
Since $P_2$ is $2$-action T\"oplitz, it follows that $P_2=0$.

Then, assume that $\bl_2^* \succ \mu_P$. In that case, we have
$$
P_{1,\bk,\bm}^{\bl,\bs;\bi_2,\pm 1} =\lim_{|\bj_1| \to \infty} P_{\bk,\bm}^{\bl,\bs;\bj_1,\pm 1} =  0
$$
because if $\bj_1$ is large enough, $(\bl,\bj_1,\bj_1)$ contains at least $4$ indices larger than $\mu_P$. Since $P_1$ is $1$-action T\"oplitz, it follows that $\mu_{P_1}\leq \mu_P$. Thus, since $P=P_0+P_1+P_2$, it follows  directly that we
also have $\mu_{P_0}\leq \mu_P$.

\end{proof}}

\subsection{Frequencies and their properties} In this subsection, we are going to translate the integrable Hamiltonians appearing in Proposition \ref{prop:res_coho} according to step $p$. It will generate frequencies of which we will carefully analyse the properties. We start with the simplest (but leading one) $Z_4$.
%\mic{be;ow it seems to me that the only dependence on $p$ is that you are deciding that $\xi\in \mathcal O_p$ and same for u, it is not clear to me why we put this dependence, I mean this is just the expansion at $\xi$, then if xi and u are in the right domanis you have the estimate $r_p^{4-\delta}$,} \jo{If I understand well, you do not like the fact that I wrote a subscript $p$ in $Z_4^{(\mathrm{lead},p)}$, isn't it ? I agree that a this point is is rather heavy but I think it convenient for later use to point out the fact that it means that it needs parameters $\xi$ in $\mathcal{O}_p^{\mathrm{full}}$. This notation should be compared with $Z_4^{(p',p)}$ which appear later. I mean that $Z_4^{(p',p)}$ and $Z_4^{(\mathrm{lead},p)}$ are objects of the same nature and satisfying the same kind of estimates. That is why I wanted similar notations.}

\begin{lemma} \label{lem:trans_Z4} Let $p\geq 1$. For all $\xi \in \mathcal{O}_p^{\mathrm{full}}$ and all $u\in \mathcal{A}_\xi(3r_p)$, we have 
\begin{equation}
\label{eq:dec_Z4}
Z_4 (u) = Z_4(\sqrt{\xi}) +  \frac12  \sum_{j\in \mathbb{Z}} \omega_j^{\mathrm{lead}}(\xi) (|u_j|^2 - \xi_j)  + Z_4^{(\mathrm{lead},p)}   (\xi;u)
\end{equation}
where for all $j\in \mathbb{Z}$, $(\sqrt{\xi})_j := \sqrt{\xi_j}$,
$$
\omega_j^{\mathrm{lead}}(\xi) :=  \xi_j - 2\|\xi\|_{\ell^1}
$$
and $Z_4^{(\mathrm{lead},p)}  \in \Pi_{\mathrm{deg}= 4}  \Pi_{\mathrm{int}} \mathscr{H}_{4 \eta_{\min},r_{p},p}^{\cste}$  is a homogeneous integrable Hamiltonian of degree $4$ satisfying the bound
$$
\| Z_4^{(\mathrm{lead},p)}  \|_{ \mathscr{H}_{4 \eta_{\min},r_{p},p}^{\cste}} \lesssim r_p^{4-\delta}.
$$
%\mic{should we do the $\mathscr{H}^{\rm tot}$ norm?} \jo{I would say no because it does not depend on the parameter $\xi$. The heavy notation with the subscript $\mathrm{cste}$ had the advantage to highlight this (non crucial) point.}
\end{lemma}
\begin{proof}  Recalling the Definition \ref{def:Z_24} of $Z_4$, we note that the decomposition \eqref{eq:dec_Z4} holds with 
\begin{equation*}
\begin{split}
 Z_4^{(\mathrm{lead},p)}   (\xi;u) &= \frac14 \sum_{j\in \mathbb{Z}} (|u_j|^2 - \xi_j)^2 - \frac12 \sum_{i,j\in \mathbb{Z}} (|u_j|^2 - \xi_j)(|u_i|^2 - \xi_i) \\
 &= -\frac14 \sum_{j\in \mathbb{Z}} (|u_j|^2 - \xi_j)^2 -  \sum_{i\prec j} (|u_j|^2 - \xi_j)(|u_i|^2 - \xi_i) 
\end{split}
\end{equation*}
Thus the only point is to prove that $Z_4^{(\mathrm{lead},p)}  \in \Pi_{\mathrm{deg}= 4}  \Pi_{\mathrm{int}} \mathscr{H}_{4 \eta_{\min},r_{p},p}^{\cste}$ and the associated bound.

\medskip
%\mic{would you mind if I write them directly in the power series notation? I think it makes it clearer that they do in fact depend on $p$ and not on $\xi$, feel free to change it back} \jo{it is fine for me}
We notice the following decomposition
$$
 Z_4^{(\mathrm{lead},p)} =  Z_{4,0}^{(\mathrm{lead},p)}+ Z_{4,1}^{(\mathrm{lead},p)} + Z_{4,2}^{(\mathrm{lead},p)}
$$
where $Z_{4,0}^{(\mathrm{lead},p)}  \in \mathscr{S}_{1,0,1,p}^{\cste},\, Z_{4,1}^{(\mathrm{lead},p)} \in  \mathscr{T}_{1,0,1,p}^{\cste} \cap \mathscr{S}_{2,0,1,0}^{\cste}$ and $Z_{4,2}^{(\mathrm{lead},p)}\in \mathscr{T}_{2,0,1,p}^{\cste}$
are explicitly given by
\[
\begin{aligned}
	&Z_{4,0}^{(\mathrm{lead},p)}= -\frac14( \sum_{j\in   \mathcal{S}_p} y_j^2 + \sum_{j\in   \mathcal{S}^c_p} |v_j|^4 ) -  \sum_{\substack{i\prec j \\  i,j\in \mathcal{S}_p}} y_i y_j  - \!\! \! \! \sum_{\substack{i\prec j \\
		 i\in\mathcal{S}^c_p, \ j\in \mathcal{S}_p}} y_j |v_i|^2.
\\
&Z_{4,1}^{(\mathrm{lead},p)}  =  - \! \! \!  \! \! \!  \sum_{\substack{i\prec j \\ i\in \mathcal{S}_p, \ j\notin \mathcal{S}_p}} \! \! \!  y_i |v_j|^2\quad \mathrm{and} \quad Z_{4,2}^{(\mathrm{lead},p)} =  -  \sum_{\substack{i\prec j \\ i,j\notin \mathcal{S}_p}} |v_i|^2|v_j|^2
\end{aligned}
\]
%$$
%Z_{4,1}^{(\mathrm{lead},p)} (\xi;u) =  - \! \! \!  \! \! \!  \sum_{\substack{i\prec j \\ i\in \mathcal{S}_p, \ j\notin \mathcal{S}_p}} \! \! \!  (|u_j|^2 - \xi_j)(|u_i|^2 - \xi_i) \quad \mathrm{and} \quad Z_{4,2}^{(\mathrm{lead},p)}(\xi;u) =  -  \sum_{\substack{i\prec j \\ i,j\notin \mathcal{S}_p}} (|u_j|^2 - \xi_j)(|u_i|^2 - \xi_i)
%$$
%and
%$$
%Z_{4,0}^{(\mathrm{lead},p)}(\xi;u)  =  -\frac14 \sum_{j\in \mathbb{Z}} (|u_j|^2 - \xi_j)^2   -  \sum_{\substack{i\prec j \\  i,j\in \mathcal{S}_p}} (|u_j|^2 - \xi_j)(|u_i|^2 - \xi_i)  - \!\! \! \! \sum_{\substack{i\prec j \\ i\notin \mathcal{S}_p, \ j\in \mathcal{S}_p}} (|u_j|^2 - \xi_j)(|u_i|^2 - \xi_i).
%$$
The two last terms clearly satisfy 
$$
\| Z_{4,1}^{(\mathrm{lead},p)}  \|_{ \mathscr{H}_{0,1,p}^{\cste}} + \| Z_{4,2}^{(\mathrm{lead},p)}  \|_{ \mathscr{H}_{0,1,p}^{\cste}} \lesssim 1.
$$
whereas due to its last term $Z_{4,0}^{(\mathrm{lead},p)}$ satisfies
$$
\| Z_{4,1}^{(\mathrm{lead},p)}  \|_{ \mathscr{H}_{0,1,p}^{\cste}} \lesssim 1+\max_{j \in \mathcal{S}_p} |j|.
$$
It follows that $Z_4^{(\mathrm{lead},p)}  \in \Pi_{\mathrm{deg}= 4}  \Pi_{\mathrm{int}} \mathscr{H}_{0,1,p}^{\cste}$ with the same bound as $Z_{4,1}^{(\mathrm{lead},p)} $.
Finally we deduce that $Z_4^{(\mathrm{lead},p)}  \in \Pi_{\mathrm{deg}= 4}  \Pi_{\mathrm{int}} \mathscr{H}_{4 \eta_{\min},r_{p},p}^{\cste}$ satisfies the expected estimate by a scaling argument (and a smallness condition on $r_p$ to control $\max_{j \in \mathcal{S}_p} |j|$ by $r_p^{-\delta}$ ).

\end{proof}

The following lemma is crucial in this paper. It is the key to solve the cohomological equation \eqref{eq:coho} and thus to prove Proposition \ref{prop:res_coho}. It  is essentially the need to obtain a lemma such as the present one, which motivates the constructions we have introduced so far. However its statement is technical. So let us first explain its meaning and its motivations. 

\medskip

 We are going to consider an integrable Hamiltonian $Z^{(p')}$ as in Proposition \ref{prop:res_coho}. To be able to solve the cohomological equation \eqref{eq:coho}, first, we will have to translate $Z^{(p')}$ according to the step $p$. Thanks to this translation we will get, in \eqref{eq:Z_2p}, a quadratic Hamiltonian (at step $p$) associated with frequencies $\omega_{j}^{(p',p)}$ depending on $\xi$. The resolution of the cohomological equation \eqref{eq:coho}, is going to make appear small divisor estimates involving all these frequencies $\omega_{j}^{(p',p)}$. The part of these small divisors associated with $\omega_{j}^{(p',p)}$ will be naturally of the form
 \begin{equation}
 \label{eq:small_div_explications}
 \sum_{j\in \mathcal{S}_p} \boldsymbol{k}_j \omega_{j}^{(p',p)} + \sum_{1\leq i\leq q} \boldsymbol{\sigma}_i \omega_{\boldsymbol{\ell}_i}^{(p',p)}
 \end{equation}
 where $q\leq 4$, $\boldsymbol{k} \in \mathbb{Z}^{\mathcal{S}_p} \setminus \{0\}$, $\boldsymbol{\ell} \in (\mathcal{S}_p^c)^q$, $\boldsymbol{\sigma} \in \{-1,1\}^q$ satisfy 
 $$
 \sum_{j\in \mathcal{S}_p} \boldsymbol{k}_j + \sum_{1\leq i \leq q} \boldsymbol{\sigma}_i =0
 $$ 
 due to the conservation of the mass\footnote{Note that we could also use the momentum conservation to add another constraint on $(\bl,\bk,\bs)$ but since we do not need it here, we do not.}. %\mic{momentum also?} \jo{we do not need it, as I said we could include the case where quintic terms are inhomogeneous.} \mic{I am very undecided, it is true that we do not use it, but it appears in our definitions, so sometimes seems strange to me not to mention it...but of course maybe this helps to clarify that we could work with grater generality} \jo{Would the following sentence sufficient for you ? (we put it in footnote ?) }The fact that we only consider $q\leq 4$ is due to the fact that we will not consider Hamiltonian of degree larger than $4$ (which is not a real restriction in this section, we just need a fixed number larger than or equal to $4$). As is classical in KAM theory, to obtain the small divisor estimates, we show that the variations of the $p'$ part \eqref{eq:small_div_explications} of the small divisor (with respect to $\xi$) are sufficiently small in a well-chosen direction. This is precisely the purpose of the estimates \eqref{eq:le_truc_le_plus_impotant_du_papier}, which are the central point of this lemma. This will allow to prove that the whole variation of the small divisor is of order $1$ in a well chosen direction and thus that the small divisors are not small for most of the values of $\xi$.

\begin{lemma} \label{lem:the_key}Let $p'\geq 1$, $\eta \geq \eta_{\min}$, $Z^{(p')} \in  \Pi_{\mathrm{deg}\leq 4}  \Pi_{\mathrm{int}} \mathscr{H}_{\eta,r_{p'},p',\mathcal{O}_{p'}},$ be an integrable Hamiltonian of degree smaller than or equal to $4$, defined for parameters in a Borel set $\mathcal{O}_{p'} \subset \mathcal{O}_{p'}^{\mathrm{full}}$ and set 
$$
\lambda =r_{p'}^{-4}\| Z^{(p')} \|_{ \mathscr{H}_{\eta,r_{p'},p',\mathcal{O}_{p'}}^{\mathrm{tot}}}.
$$
Then, for all $p\geq  p'$, there exists $3$ homogeneous Hamiltonians, $Z_{d}^{(p',p)} \in  \Pi_{\mathrm{deg}=d}  \Pi_{\mathrm{int}} \mathscr{H}_{\eta,r_{p},p,\mathcal{O}_{p',p}}$, $d\in \{0,2,4\}$ with
$$
\mathcal{O}_{p',p} := \mathcal{O}_{p'} \times \prod_{p'+1\leq j\leq p}   \big(\frac{r_{j}^{2(1-\delta)}}2,  r_{j}^{2(1-\delta)}\big)^{\mathcal{S}_{j} \setminus \mathcal{S}_{j-1}} \subset \mathcal{O}_{p}^{\mathrm{full}}
$$
such that for all $\xi = (\xi',\xi^\sharp) \in  \mathcal{O}_{p',p}$ and all $u\in \mathcal{A}_\xi(3 r_p)$,
\begin{equation}
\label{eq:dec_quart}
Z^{(p')}(\xi';u) = (Z_{0}^{(p',p)}+Z_{2}^{(p',p)}+Z_{4}^{(p',p)})(\xi;u)
\end{equation}
where $Z_{4}^{(p',p)}$ satisfies the estimate
\begin{equation}
\label{eq:bound_on_Z4p}
\| Z_{4}^{(p',p)} \|_{\mathscr{H}_{\eta,r_{p},p,\mathcal{O}_{p',p}}^{\mathrm{tot}}} \leq \lambda \, r_p^{4-\delta}
\end{equation}
and $Z_{2}^{(p',p)}$ admits the following decomposition for all 
\begin{equation}
\label{eq:Z_2p}
Z_{2}^{(p',p)}(\xi;u) =\frac12 \sum_{j\in \mathbb{Z}} \omega_{j}^{(p',p)}(\xi) (|u_j|^2 - \xi_j)
\end{equation}
with frequencies $\omega_{j}^{(p',p)} : \mathcal{O}_{p',p} \to \mathbb{R}$, $j\in \mathbb{Z}$, which satisfy 
\begin{itemize}
\item for all $j\in \mathbb{Z}$, $\omega_{j}^{(p',p)}$ is uniformly bounded by  $\lambda$,%it will be sufficient to converge by compactness. Note that we could prove that it is a actually a Cauchy sequence.
\item for all $j\in \mathbb{Z}$, $\xi,\zeta \in \mathcal{O}_{p',p}$, we have
\begin{equation}
\label{eq:bad_lip}
| \omega_{j}^{(p',p)}(\xi) - \omega_{j}^{(p',p)}(\zeta) | \leq \lambda  r_p^{-\delta}r_{p'}^{-\delta}  \| \xi - \zeta \|_{\ell^2},
\end{equation}
\item for all $\xi,\zeta \in \mathcal{O}_{p',p}$ and all $\boldsymbol{k} \in \mathbb{Z}^{\mathcal{S}_p} \setminus \{0\}$ satisfying $\displaystyle |\sum_{j\in \mathcal{S}_p} \boldsymbol{k}_j| \leq 4$, we have
\begin{equation}
\label{eq:le_truc_le_plus_impotant_du_papier}
\xi-\zeta \in \mathbb{R}  \boldsymbol{k}  \implies \left\{  \begin{array}{lll} \displaystyle | \sum_{j\in \mathcal{S}_p} \boldsymbol{k}_j (\omega_{j}^{(p',p)}(\xi) - \omega_{j}^{(p',p)}(\zeta))| \leq r_{p'}^{-\delta} \lambda \| \boldsymbol{k} \|_{\ell^2} \| \xi-\zeta \|_{\ell^2} \\
\forall j\notin \mathcal{S}_p, \ |\omega_{j}^{(p',p)}(\xi) - \omega_{j}^{(p',p)}(\zeta)| \leq  r_{p'}^{-\delta} \lambda \| \xi-\zeta \|_{\ell^2}  \end{array} \right. .
\end{equation}
\end{itemize}
\end{lemma}
\begin{proof} 
%	  \mic{I suffer this part a bit...not the substance but the notations, I probably would have tried to write everything in the partial action angles, is should be harmless and it seems to me it could be clearer...I will try and just comment the previous version} \jo{It is fine for me, we can keep your version :)}
	To lighten the notations, we set 
{\begin{equation}\label{ypsilon}
	Y_j^{(p')} = \begin{cases}
		y_j \quad \mbox{if}\; j\in \mathcal S_{p'} \\
		|v_j|^2  \quad \mbox{if}\; j\in \mathcal S_{p'}^c
	\end{cases} 
	\end{equation}
		Since $Z^{(p')}$ is integrable, it contains only terms of degree  $0,2$ and $4$. More precisely, by definition, it admits a decomposition of the form
		\begin{equation*}
			\begin{split}
				Z^{(p')}=& c_1^{(p')}(\xi')+ c_2^{(p')}(\xi') \sum_{i\in \mathcal{S}_{p'}^c}  Y_i^{(p')} + \sum_{j \in \mathbb{Z}} a_j^{(p')}(\xi') Y_j^{(p')}  +  c_3^{(p')} (\xi') \sum_{\substack{ j  \prec i \\ i,j\in \mathcal{S}_{p'}^c }}   Y_j^{(p')} Y_i^{(p')}  \\
				& +  \sum_{i,j \in \mathbb{Z}} b_{i,j}^{(p')}(\xi')  Y_i^{(p')}  Y_j^{(p')} +  \sum_{\substack{j \prec i \\ i\in \mathcal{S}_{p'}^c}} e_j^{(p')}(\xi')  Y_i^{(p')} Y_j^{(p')} 
			\end{split}
		\end{equation*}
		where 
		$r_{p'}^{-4}c_1^{(p')},r_{p'}^{-2} c_2^{(p')},c_3^{(p')},r_{p'}^{-2} \langle j \rangle a_j^{(p')}, (1+\mathds{1}_{i\neq j}(|i|+|j|)) b_{i,j}^{(p')}, \langle j \rangle  e_j^{(p')} : \mathcal{O}_{p'} \to \mathbb{R} $ are real valued $(r_{p'}^{-2-\delta/2} \lambda)$-Lipschitz functions uniformly bounded by $r_{p'}^{-\delta/2} \lambda$ (and $b_{i,j} = b_{j,i}$). Indeed, if we come back to the definition, we directly have that $b_{i,j}^{(p')}$ is bounded by $\lambda (1+\mathds{1}_{i\neq j} (\mathds{1}_{i\in \mathcal{S}_{p'}^c}|i| \vee \mathds{1}_{j\in \mathcal{S}_{p'}^c}|j|))^{-1}$. So, provided that $r_{p'}$ is small enough, we deduce that it is bounded by $\lambda r_{p'}^{-\delta/2} (1+\mathds{1}_{i\neq j}(|i|+|j|))^{-1}$. All the other estimates are simpler and only rely on the smallness of $r_{p'}$.

		\medskip
		
		Then, we perform the translation which, due to the integrability of $Z^{(p')}$, amounts to $|v_j|^2 \rightsquigarrow \xi_j + y_j$ for all $j\in \mathcal S_p\setminus\mathcal S_{p'}$. }
%$$
%\forall p\geq p', \quad Y_j^{(p)}(\xi;u) =  |u_j|^2 - \xi_j.
%$$ 
%Since $Z^{(p')}$ is integrable, it contains only terms of degree  $0,2$ and $4$. More precisely, by definition, it admits a decomposition of the form
%\mic{here it would be $Y^{(p')}_j(\xi',u)$ right? }
%\begin{equation*}
%\begin{split}
%Z^{(p')}(\xi';u) =& c_1^{(p')}+ c_2^{(p')} \sum_{i\in \mathcal{S}_{p'}^c}  Y_i^{(p')} + \sum_{j \in \mathbb{Z}} a_j^{(p')} Y_j^{(p')}  +  c_3^{(p')}  \sum_{\substack{ j  \prec i \\ i,j\in \mathcal{S}_{p'}^c }}   Y_j^{(p')} Y_i^{(p')}  \\
%& +  \sum_{i,j \in \mathbb{Z}} b_{i,j}^{(p')}  Y_i^{(p')}  Y_j^{(p')} +  \sum_{\substack{j \prec i \\ i\in \mathcal{S}_{p'}^c}} e_j^{(p')}  Y_i^{(p')} Y_j^{(p')} 
%\end{split}
%\end{equation*}
%where 
%$r_{p'}^{-4}c_1^{(p')},r_{p'}^{-2} c_2^{(p')},c_3^{(p')},r_{p'}^{-2} \langle j \rangle a_j^{(p')}, (1+\mathds{1}_{i\neq j}(|i|+|j|)) b_{i,j}^{(p')}, \langle j \rangle  e_j^{(p')} : \mathcal{O}_{p'} \to \mathbb{R} $ are real valued $(r_{p'}^{-2-\delta/2} \lambda)$-Lipschitz functions uniformly bounded by $r_{p'}^{-\delta/2} \lambda$ (and $b_{i,j} = b_{j,i}$). Indeed, if we come back to the definition, we directly have that $b_{i,j}^{(p')}$ is bounded by $\lambda (1+\mathds{1}_{i\neq j} (\mathds{1}_{i\in \mathcal{S}_{p'}^c}|i| \vee \mathds{1}_{j\in \mathcal{S}_{p'}^c}|j|))^{-1}$. So, provided that $r_{p'}$ is small enough, we deduce that it is bounded by $\lambda r_{p'}^{-\delta/2} (1+\mathds{1}_{i\neq j}(|i|+|j|))^{-1}$. All the other estimates are simpler and only rely on the smallness of $r_{p'}$.
%
%
%\medskip
%
%Then, we perform the translation. 
We get the decompositions \eqref{eq:dec_quart} and \eqref{eq:Z_2p} where for all $\xi = (\xi',\xi^\sharp) \in  \mathcal{O}_{p',p}$ and all $j\in \mathbb{Z}$
\begin{equation*}
\begin{split}
\omega_{j}^{(p',p)}(\xi) =& \mathds{1}_{j\in \mathcal{S}_{p'}^c} c_2^{(p')}(\xi') + a_j^{(p')}(\xi') + \mathds{1}_{j\in \mathcal{S}_{p'}^c} c_3^{(p')}(\xi') \sum_{\substack{ i  \neq j \\ i\in \mathcal{S}_{p} \setminus \mathcal{S}_{p'} }} \xi_i \\ &+ 2 \sum_{i \in \mathcal{S}_{p} \setminus \mathcal{S}_{p'}  } b_{i,j}^{(p')}(\xi') \xi_i + e_j^{(p')}(\xi') \sum_{\substack{j \prec i \\ i\in \mathcal{S}_{p} \setminus   \mathcal{S}_{p'}}}  \xi_i +  \mathds{1}_{j\in \mathcal{S}_{p'}^c}   \sum_{\substack{i \prec j \\  i\in \mathcal{S}_{p} \setminus   \mathcal{S}_{p'}}} e_i^{(p')}(\xi')  \xi_i. 
\end{split} 
\end{equation*}
The proof of the bound \eqref{eq:bound_on_Z4p} on $Z_{4}^{(p',p)}$ is a direct corollary of Proposition \ref{prop:open}  and its proof so we omit it. We focus on the properties of the frequencies.
First, using the bounds we have on the coefficients, we directly get the uniform bound on the frequencies: 
\begin{equation*}
\begin{split}
(\lambda r_{p'}^{-\delta/2})^{-1}|\omega_{j}^{(p',p)}(\xi)| \leq& r_{p'}^{2} +  r_{p'}^{2} \langle j \rangle ^{-1} + 4\sum_{p'<p''\leq p} r_{p''}^{2(1-\delta)} \# \mathcal{S}_{p''}\leq r_{p'}^{\delta/2}.
\end{split}
\end{equation*}
Finally, we focus on Lipschitz estimates \eqref{eq:bad_lip} and \eqref{eq:le_truc_le_plus_impotant_du_papier}.  We consider $\zeta = (\zeta',\zeta^\sharp) \in \mathcal{O}_{p',p}$ and, for all $j\in \mathbb{Z}$, we will always use the decomposition 
\begin{equation}
\label{eq:en_carton}
|\omega_{j}^{(p',p)}(\xi) - \omega_{j}^{(p',p)}(\zeta) | \leq |\omega_{j}^{(p',p)}(\xi',\xi^\sharp) -  \omega_{j}^{(p',p)}(\xi', \zeta^\sharp)| +  | \omega_{j}^{(p',p)}(\xi',\zeta^\sharp) -\omega_{j}^{(p',p)}(\zeta',\zeta^\sharp)  |. 
\end{equation}

\noindent \underline{\emph{Step $1$ : Rough Lipschitz estimate \eqref{eq:bad_lip}.}} Let $j\in \mathbb{Z}$.
On the one hand, we have 
\begin{equation*}
\begin{split}
(\lambda r_{p'}^{-\delta/2})^{-1} |\omega_{j}^{(p',p)}(\xi',\xi^\sharp) -  \omega_{j}^{(p',p)}(\xi', \zeta^\sharp)|  &\lesssim \sum_{\substack{ i\in \mathcal{S}_{p} \setminus \mathcal{S}_{p'} }} |\xi_i-\zeta_i| \leq 5 \| \xi - \zeta\|_{\ell^2} \sqrt{\# \mathcal{S}_p} \le r_p^{-\delta/2} \| \xi - \zeta\|_{\ell^2} 
\end{split}
\end{equation*}
On the other hand, we have
\begin{equation}
\label{eq:pas_si_mal_en_fait}
\begin{split}
 ( \lambda r_{p'}^{-\delta/2})^{-1} | \omega_{j}^{(p',p)}(\xi',\zeta^\sharp) - \omega_{j}^{(p',p)}(\zeta',\zeta^\sharp)  |&\leq \| \xi'-\zeta'\|_{\ell^2}  (2+ r_{p'}^{-2}  4\sum_{p'<p''\leq p} r_{p''}^{2(1-\delta)} \# \mathcal{S}_{p''} ) \\& \leq \frac12 r_{p'}^{-\delta/2} \|\xi - \zeta\|_{\ell^2} .
\end{split}
\end{equation}
Thus we get the rough Lipschitz estimate \eqref{eq:bad_lip} by \eqref{eq:en_carton}.

\medskip

It only remains to prove the non trivial Lipschitz estimate \eqref{eq:le_truc_le_plus_impotant_du_papier}. So, from now, we consider $\boldsymbol{k} \in \mathbb{Z}^{\mathcal{S}_p} \setminus \{0\}$ and we assume that there exists $\alpha \in \mathbb{R}$ such that 
\begin{equation}
\label{eq:super_important}
\xi- \zeta = \alpha \boldsymbol{k} \quad \mathrm{and} \quad \big| \sum_{i\in \mathcal{S}_p}\boldsymbol{k}_i \big|\leq 4.
\end{equation}

\noindent \underline{\emph{Step $2$ : Lipschitz estimate \eqref{eq:le_truc_le_plus_impotant_du_papier} for normal sites.}} Let $j\in \mathcal{S}_p^c$.
First, using \eqref{eq:super_important} and Cauchy-Schwarz, there exists a universal constant $C>0$ such that we have
\begin{equation*}
\begin{split}
( \lambda r_{p'}^{-\delta/2})^{-1} &|\omega_{j}^{(p',p)}(\xi',\xi^\sharp) -  \omega_{j}^{(p',p)}(\xi', \zeta^\sharp)| \\
 &\leq  \Big|  \sum_{\substack{ i  \neq j \\ i\in \mathcal{S}_{p} \setminus \mathcal{S}_{p'} }} \alpha \boldsymbol{k}_i \Big| + 2 \sum_{i \in \mathcal{S}_{p} \setminus \mathcal{S}_{p'}  } \frac{ |\xi_i - \zeta_i| }{ 1+ \mathds{1}_{i\neq j}(|i|+|j|) } + \langle j \rangle^{-1} \big| \sum_{\substack{j \prec i \\ i\in \mathcal{S}_{p} \setminus   \mathcal{S}_{p'}}}  \alpha \boldsymbol{k}_i \big| + \sum_{\substack{i \prec j \\  i\in \mathcal{S}_{p} \setminus   \mathcal{S}_{p'}}}  \frac{ |\xi_i - \zeta_i|}{\langle i \rangle}\\
 &\leq |\alpha| \Big(  4+ \sum_{ i\in  \mathcal{S}_{p'} } |  \boldsymbol{k}_i| \Big)+C \|\xi - \zeta\|_{\ell^2} + |\alpha| \langle j \rangle^{-1}  \Big(  4+ \sum_{ i \prec j } |  \boldsymbol{k}_i| + \sum_{ i\in  \mathcal{S}_{p'} }  |  \boldsymbol{k}_i|  \big) \\
 &\leq 5 |\alpha| \| \boldsymbol{k} \|_{\ell^2} (\# \mathcal{S}_{p'})^{1/2} + C \|\xi - \zeta\|_{\ell^2} +   |\alpha|  \| \boldsymbol{k} \|_{\ell^2} (4 + \frac{\sqrt{2|j|+1}}{\langle j \rangle} + (\# \mathcal{S}_{p'})^{1/2}) \\
 &\leq \frac12 r_{p'}^{-\delta/2} \|\xi - \zeta\|_{\ell^2} .
\end{split}
\end{equation*}
The estimate on $\omega_{j}^{(p',p)}(\xi',\zeta^\sharp) - \omega_{j}^{(p',p)}(\zeta',\zeta^\sharp) $ is directly given by \eqref{eq:pas_si_mal_en_fait}.

\noindent \underline{\emph{Step $3$ : Lipschitz estimate \eqref{eq:le_truc_le_plus_impotant_du_papier} for tangential sites.}} We use again the decomposition \eqref{eq:en_carton}. 
%On the one hand, u
Using the bounds we have on the coefficients, we have
\begin{equation*}
\begin{split}
( \lambda r_{p'}^{-\delta/2})^{-1} &|\sum_{j\in \mathcal{S}_p} \boldsymbol{k}_j (\omega_{j}^{(p',p)}(\xi',\xi^\sharp) -  \omega_{j}^{(p',p)}(\xi', \zeta^\sharp))| \\
 \leq & \Big| \sum_{j\in \mathcal{S}_{p'}^c} \boldsymbol{k}_j \sum_{\substack{ i  \neq j \\ i\in \mathcal{S}_{p} \setminus \mathcal{S}_{p'} }} ( \xi_i - \zeta_i) \Big| +2 \sum_{j\in \mathcal{S}_p} \sum_{i \in \mathcal{S}_{p} \setminus \mathcal{S}_{p'}  }     \frac{|\boldsymbol{k}_j||\xi_i - \zeta_i|}{1+\mathds{1}_{i\neq j} (|i|+|j|)} \\
 &+  \sum_{j\in \mathcal{S}_p}  \frac{|\boldsymbol{k}_j|}{\langle j \rangle} \Big|  \sum_{\substack{j \prec i \\ i\in \mathcal{S}_{p} \setminus   \mathcal{S}_{p'}}}  (\xi_i - \zeta_i)  \Big| \\
 & +  \Big| \sum_{j\in  \mathcal{S}_{p'}^c}    \boldsymbol{k}_j   \sum_{\substack{i \prec j \\  i\in \mathcal{S}_{p} \setminus   \mathcal{S}_{p'}}} ( \lambda r_{p'}^{-\delta/2})^{-1} e_i^{(p')}(\xi')  (\xi_i -\zeta_i)  \Big|  =: I+II+III+IV.
\end{split}
\end{equation*}
First, we use Lemma \ref{lem:Hilb}, to directly get that
$$
II \lesssim \| \boldsymbol{k} \|_{\ell^2} \| \xi-\zeta \|_{\ell^2}.
$$
In order to control the other terms, we have to exploit some cancellation using \eqref{eq:super_important}. We have
\begin{equation*}
\begin{split}
I &=|\alpha| \Big| \sum_{j\in \mathcal{S}_{p'}^c} \boldsymbol{k}_j \sum_{\substack{ i  \neq j \\ i\in \mathcal{S}_{p} \setminus \mathcal{S}_{p'} }}\boldsymbol{k}_i \Big| \leq |\alpha | \| \boldsymbol{k} \|_{\ell^2}^2 + |\alpha| \Big| \sum_{j\in \mathcal{S}_{p'}^c} \boldsymbol{k}_j \sum_{ i\in \mathcal{S}_{p} \setminus \mathcal{S}_{p'} }\boldsymbol{k}_i \Big| \\
&\leq |\alpha | \| \boldsymbol{k} \|_{\ell^2}^2 + \alpha( 4 +  \sum_{j\in \mathcal{S}_{p'}} |\boldsymbol{k}_j| )^2 \leq r_{p'}^{-\delta/3} |\alpha | \| \boldsymbol{k} \|_{\ell^2}^2  = r_{p'}^{-\delta/3} \| \boldsymbol{k} \|_{\ell^2}  \| \xi-\zeta \|_{\ell^2}
\end{split}
\end{equation*}
Similarly, still by Lemma \ref{lem:Hilb} and Cauchy--Schwarz, we have
\begin{equation*}
\begin{split}
III &\leq |\alpha|  \sum_{j\in \mathcal{S}_p}  \frac{|\boldsymbol{k}_j|}{\langle j \rangle} \Big( 4 + \sum_{i \in \mathcal{S}_{p'} }  | \boldsymbol{k}_i|   + \sum_{i \preccurlyeq j }|\boldsymbol{k}_i| \Big) \\
&\leq r_p^{-\delta/3} |\alpha| \| \boldsymbol{k} \|_{\ell^2}  \sum_{j\in \mathcal{S}_p} \frac{|\boldsymbol{k}_j|}{\langle j \rangle}    + 3 |\alpha|  \sum_{j\in \mathcal{S}_p} \sum_{i \preccurlyeq j } \frac{|\boldsymbol{k}_j|}{1+|i|+|j|}  |\boldsymbol{k}_i|  \lesssim r_{p'}^{-\delta/3} \| \boldsymbol{k} \|_{\ell^2}  \| \xi-\zeta \|_{\ell^2}.
\end{split}
\end{equation*}
Finally, permuting the sum and using the bound we have on $e_i^{(p')}$, we similarly get
\begin{equation*}
\begin{split}
IV \leq \sum_{i\in  \mathcal{S}_{p'}^c} \frac{|\xi_i-\zeta_i|}{\langle i \rangle } \big| \sum_{\substack{i \prec j \\  j\in \mathcal{S}_{p'}^c}} \boldsymbol{k}_j\big| 
\leq  \sum_{i\in  \mathcal{S}_{p'}^c} \frac{|\xi_i-\zeta_i|}{\langle i \rangle }   \Big( 4 + \sum_{j \in \mathcal{S}_{p'} }  | \boldsymbol{k}_i|   + \sum_{j \preccurlyeq i }|\boldsymbol{k}_i| \Big) \leq r_{p'}^{-\delta/3} \| \boldsymbol{k} \|_{\ell^2}  \| \xi-\zeta \|_{\ell^2}.
\end{split}
\end{equation*}
On the other hand, using the bounds we have on the coefficients, we have
\begin{equation*}
\begin{split}
V:=&( \lambda  r_{p'}^{-\delta/2})^{-1} \| \xi' -\zeta'\|_{\ell^2}^{-1}|\sum_{j\in \mathcal{S}_p} \boldsymbol{k}_j (\omega_{j}^{(p',p)}(\xi', \zeta^\sharp)-  \omega_{j}^{(p',p)}(\zeta', \zeta^\sharp))| \\
 \leq & |\sum_{j\in \mathcal{S}_{p'}^c} \boldsymbol{k}_j| +  \sum_{j\in \mathcal{S}_{p}} \frac{|\boldsymbol{k}_j|}{\langle  j \rangle}  +  |\sum_{j\in \mathcal{S}_{p'}^c} \boldsymbol{k}_j| r_{p'}^{-2} \sum_{i\in \mathcal{S}_{p'}} \zeta_i + r_{p'}^{-2}\sum_{j\in \mathcal{S}_{p'}^c} |\boldsymbol{k}_j|\zeta_j +  2 \sum_{\substack{i \in \mathcal{S}_{p'}^c \\ j \in \mathcal{S}_p}  } \frac{ |\boldsymbol{k}_j| r_{p'}^{-2}\zeta_i}{1+\mathds{1}_{i\neq j} (|i|+ |j|)} \\
 &+ \sum_{j\in \mathcal{S}_p} \frac{|\boldsymbol{k}_j|}{\langle j \rangle} r_{p'}^{-2}\sum_{\substack{j \prec i \\ i\in \mathcal{S}_{p} \setminus   \mathcal{S}_{p'}}}  \zeta_i + \sum_{i\in \mathcal{S}_{p'}^c}  \frac{\zeta_i r_{p'}^{-2}}{\langle i \rangle}   \big|\sum_{\substack{j\in \mathcal{S}_{p'}^c \\ j \succ i}} \boldsymbol{k}_j |  .
\end{split}
 \end{equation*}
To conclude the proof, we just have to prove that 
$$
V \leq r_{p'}^{-\delta/3} \| \boldsymbol{k} \|_{\ell^2}.
$$
We use that $\zeta_j \lesssim r_{p''}^{2-2\delta}$ if $j\in \mathcal{S}_{p''}$, Lemma \ref{lem:Hilb}, and, as previously, that $\boldsymbol{k}$ enjoy good cancellations properties \eqref{eq:super_important}.  
Then, the bound on $V$ is a quite direct corollary of what have been done just before: it does not require any new arguments and does not present any new difficulties. So we omit its proof.

\end{proof}

\subsection{Extraction of the parameters : non resonance conditions} In the following lemma, we apply most of the previous results to construct the Borel set $\mathcal{O}^\sharp$ of Proposition \ref{prop:res_coho}.

\begin{lemma}[Small divisor estimates] \label{lem:sd_estimates} Let $p\geq 0$, $\gamma \in (0,1), \mathcal{O}\subset \mathcal{O}_p^{\mathrm{full}}$ be a Borel set. Given a family of integrable Hamiltonians $Z^{(p')} \in  \Pi_{\mathrm{deg}\leq 4}  \Pi_{\mathrm{int}} \mathscr{H}_{\eta,r_{p'},p',\mathcal{O}_{p'}},$ as in Proposition \ref{prop:res_coho}, there exists a Borel set $\mathcal{O}^\sharp \subset \mathcal{O}$ satisfying the measure estimate
$$
\mathrm{meas}(\mathcal{O} \setminus \mathcal{O}^\sharp) \leq \gamma \, \mathrm{meas} (\mathcal{O}_p^{\mathrm{full}})
$$
and such that for all indices $\boldsymbol{k}\in \mathbb{Z}^{\mathcal{S}_p}$, $0\leq q\leq 4$, $\boldsymbol{\ell} \in (\mathcal{S}_p^c)^q$, $\boldsymbol{\sigma}\in \{-1,1\}^q$ associated with non resonant monomials with no actions and commuting with the mass and the momentum, that is satisfying
\begin{equation}
\label{eq:cond_indice_0}
\boldsymbol{k} \neq 0 \quad \mathrm{or} \quad (\boldsymbol{k} = 0 \quad \mathrm{and} \quad \boldsymbol{\sigma}_1 \boldsymbol{\ell}_1^2 + \cdots + \boldsymbol{\sigma}_q \boldsymbol{\ell}_q^2\neq 0),
\end{equation}
\begin{equation}
\label{eq:cond_indice_1}
\forall i,j\leq q, \ \boldsymbol{\ell}_i = \boldsymbol{\ell}_j \implies \boldsymbol{\sigma}_i=\boldsymbol{\sigma}_j,
\end{equation}
\begin{equation}
\label{eq:cond_indice_2}
\sum_{j\in \mathcal{S}_p} \boldsymbol{k}_j + \boldsymbol{\sigma}_1+\cdots+ \boldsymbol{\sigma}_q  = \sum_{j\in \mathcal{S}_p} j \boldsymbol{k}_j + \boldsymbol{\sigma}_1\boldsymbol{\ell}_1+\cdots+ \boldsymbol{\sigma}_q \boldsymbol{\ell}_q =0,
\end{equation}
we have 
\begin{equation}
\label{eq:the_small_divisor_estimate}
\forall \xi \in \mathcal{O}^\sharp, \ |\Omega_{\boldsymbol{k},\boldsymbol{\ell},\boldsymbol{\sigma} }(\xi)|\gtrsim \gamma  \langle \bl_4^*\rangle ^{-10} \langle \| \boldsymbol{\ell} \|_{\ell^\infty} \rangle   r_p^{2-\frac{9\delta}5} \| \boldsymbol{k}\|_{\ell^1}^{-\# \mathcal{S}_p-12}
\end{equation}
where  the small divisors $\Omega_{\boldsymbol{k},\boldsymbol{\ell},\boldsymbol{\sigma} }$ are defined  by
$$
\Omega_{\boldsymbol{k},\boldsymbol{\ell},\boldsymbol{\sigma} } := \sum_{j\in \mathcal{S}_p} \boldsymbol{k}_j \omega^{\mathrm{full}}_j + \sum_{1\leq i \leq q} \boldsymbol{\sigma}_i \omega^{\mathrm{full}}_{\boldsymbol{\ell}_i} 
$$
with
$$
\forall j\in \mathbb{Z}, \quad \omega^{\mathrm{full}}_j :=  j^2 + \omega^{\mathrm{lead}}_j + \sum_{1\leq p'\leq p}  \omega^{(p',p)}_j 
$$
and the frequencies $\omega^{\mathrm{lead}}$ and $\omega^{(p',p)}$ are given by Lemma \ref{lem:trans_Z4} and Lemma \ref{lem:the_key}.
\end{lemma}
\begin{proof}
We set
\[
\begin{aligned}
	\mathcal{O}^\sharp &:=  \big\{\xi \in \mathcal{O} \ | \ | \Omega_{\boldsymbol{k},\boldsymbol{\ell},\boldsymbol{\sigma} }(\xi) | \geq  c\, \gamma\, r_p^{2-\frac{37\delta}{20}} \langle \| \boldsymbol{\ell}\|_{\ell^\infty}\rangle^{-5} \langle \| \boldsymbol{k}\|_{\ell^1}\rangle^{-\# \mathcal{S}_p-1} \\ 
	&\hskip4.5cm \forall \quad  \boldsymbol{k},\boldsymbol{\ell},\boldsymbol{\sigma}  s.t. \ \eqref{eq:cond_indice_0},\eqref{eq:cond_indice_1}, \eqref{eq:cond_indice_2} \ \mbox {hold} \big\} \\
	& \ = \mathcal{O} \setminus \bigcup_{  \substack{ \boldsymbol{k},\boldsymbol{\ell},\boldsymbol{\sigma} \\ s.t. \ \eqref{eq:cond_indice_0},\eqref{eq:cond_indice_1}, \eqref{eq:cond_indice_2}  } }  E_{\boldsymbol{k}}^{\boldsymbol{\ell},\boldsymbol{\sigma}}(  c\, \gamma\, r_p^{2-\frac{37\delta}{20}} \langle \| \boldsymbol{\ell}\|_{\ell^\infty}\rangle^{-5} \langle \| \boldsymbol{k}\|_{\ell^1}\rangle^{-\# \mathcal{S}_p-1} ) 
\end{aligned}
\]
where $c>0$ is a universal normalization constant that will be fixed later and
$$
\forall \epsilon>0, \quad E_{\boldsymbol{k}}^{\boldsymbol{\ell},\boldsymbol{\sigma}}(\epsilon) := \big\{\xi \in \mathcal{O} \ | \ | \Omega_{\boldsymbol{k},\boldsymbol{\ell},\boldsymbol{\sigma} }(\xi) | \leq  \epsilon \big\}.
$$
We divide the proof in two steps. In the first one we prove the measure estimate and in the second one we prove the improved small divisor estimate \eqref{eq:the_small_divisor_estimate}.

\medskip

\noindent \emph{$\bullet$ Step $1$ : measure estimate.} First, we fix some indices $\boldsymbol{k},\boldsymbol{\ell},\boldsymbol{\sigma}$ as in the statement of the lemma.
Using the bound on $\omega^{(p',p)}$ given by Lemma \ref{lem:the_key}, since $q\leq 4$, we have  
\begin{equation}
\label{eq:usefull_sd}
| \Omega_{\boldsymbol{k},\boldsymbol{\ell},\boldsymbol{\sigma} } |  \geq |\boldsymbol{\sigma}_1 \boldsymbol{\ell}_1^2 + \cdots + \boldsymbol{\sigma}_q \boldsymbol{\ell}_q^2|- 2\max_{j\in \mathcal{S}_p} \langle j \rangle^2 \| \boldsymbol{k}\|_{\ell^1} -\frac12.
\end{equation}
Thus, thanks to the assumption \eqref{eq:cond_indice_0}, if $\boldsymbol{k}=0$ and $\epsilon<1/2$,
 we have  $\mathrm{meas} \,  E_{\boldsymbol{k}}^{\boldsymbol{\ell},\boldsymbol{\sigma}}(\epsilon) = 0$. 

\medskip

As a consequence, now we focus on the case  $\boldsymbol{k}\neq 0$. We notice that, by conservation of the mass, the contribution of the $\|\xi\|_{\ell^1}$ from $\omega^{\mathrm{lead}}(\xi)$ vanishes in $\Omega_{\boldsymbol{k},\boldsymbol{\ell},\boldsymbol{\sigma} }$  and so that we have that for all $\xi,\zeta \in \mathcal{O}$
 \begin{equation*}
 \begin{split}
| \Omega_{\boldsymbol{k},\boldsymbol{\ell},\boldsymbol{\sigma} }(\xi) - \Omega_{\boldsymbol{k},\boldsymbol{\ell},\boldsymbol{\sigma} }(\zeta)   |  \geq&  \Big| \sum_{j\in \mathcal{S}_p} \boldsymbol{k}_j (\xi_j-\zeta_j) \Big| - \sum_{1\leq p' \leq p}  \Big| \sum_{j\in \mathcal{S}_p} \boldsymbol{k}_j (\omega^{(p',p)}_j(\xi)-\omega^{(p',p)}_j(\zeta)) \Big| \\
&-  q \sum_{1\leq p' \leq p} \sup_{j\notin \mathcal{S}_p}|\omega_{j}^{(p',p)}(\xi) - \omega_{j}^{(p',p)}(\zeta)| 
\end{split}
\end{equation*}
To go further, we assume that moreover $\xi -\zeta \in \mathbb{R}^* \boldsymbol{k}$.
Noticing that since $q\leq 4$, we have that $\displaystyle |\sum_{j\in \mathcal{S}_p} \boldsymbol{k}_j| \leq 4$ by conservation of the mass. So applying  the Lipschitz estimates of Lemma \ref{lem:the_key}, we get that
 \begin{equation*}
 \begin{split}
\frac{| \Omega_{\boldsymbol{k},\boldsymbol{\ell},\boldsymbol{\sigma} }(\xi) - \Omega_{\boldsymbol{k},\boldsymbol{\ell},\boldsymbol{\sigma} }(\zeta)   |}{\| \xi- \zeta\|_{\ell^2}} \geq&  \| \boldsymbol{k} \|_{\ell^2}  - \sum_{1\leq p' \leq p} r_{p'}^{1/4} \| \boldsymbol{k} \|_{\ell^2} - 4 \sum_{1\leq p' \leq p} r_{p'}^{1/4} \geq \frac12 \| \boldsymbol{k} \|_{\ell^2} \geq \frac12.
\end{split}
\end{equation*}
Thus, applying Corollary \ref{cor:give_measure_est}, we get, for all $\epsilon>0$, the measure estimate
$$
 \mathrm{meas} \,  E_{\boldsymbol{k}}^{\boldsymbol{\ell},\boldsymbol{\sigma}}(\epsilon) \lesssim_{\# \mathcal{S}_p} \, \epsilon \, r_p^{-2(1-\delta)} \mathrm{meas} (\mathcal{O}_p^{\mathrm{full}}) \leq \epsilon r_p^{-2+\frac{19}{10}\delta } \mathrm{meas} (\mathcal{O}_p^{\mathrm{full}}).
$$
As a consequence, by definition of $\mathcal{O}^\sharp $, the measure estimate $\mathrm{meas}(\mathcal{O} \setminus \mathcal{O}^\sharp) \leq \gamma \, \mathrm{meas} (\mathcal{O}_p^{\mathrm{full}})$ follows by a direct summation argument for a well chosen normalization constant $c\in (0,1)$. 
%\mic{the power $-5$ is enough since you fix one of the $\bl_i$ through momentu conservation and another from (56), right?} \jo{No it is much more basic. We have $q\leq 4$ so in the worst case $\bl \in \mathbb{Z}^4$. If we want a sum indexed by $\bl \in \mathbb{Z}^4$ to be convergent it suffices to put a weight $\| \bl\|_{\ell^\infty}^{-5}$.}
%\mic{so 3 would be fir (not that i want to change, just to be sure I understand)} \jo{fir = fine ? If so, I do not understand what you mean. I am just saying that if I want 
%$$
%\sum_{\bl \in \mathbb{Z}^4\setminus \{0\}} \frac1{ \|\bl\|_{\ell^\infty}^{\alpha}}<\infty
%$$
%I need $\alpha>4$ and so $\alpha = 5$ is a canonical choice.}
%\mic{Yes but you are not summing over all $\bl\in \mathbb Z^4\setminus\{0\}$ but only on those which satisfy
%	\[
%	 |\boldsymbol{\sigma}_1 \boldsymbol{\ell}_1^2 + \cdots + \boldsymbol{\sigma}_q \boldsymbol{\ell}_q^2| \le 4 \max_{j\in \mathcal{S}_p} \langle j \rangle^2 \| \boldsymbol{k}\|_{\ell^1}\,,\qquad  \boldsymbol{\sigma}_1 \boldsymbol{\ell}_1 + \cdots + \boldsymbol{\sigma}_q \boldsymbol{\ell}_q=-  \sum_{j\in \mathcal{S}_p} j \bk_j
%	\] The momentum conservation fixes one of the $\bl_i$ (in term of the others and of $\bk$ ) while the first condition gives $\sim\|\bk\|$ possible choices. So at most 2 are free and I could take -3 as an exponent...I am not saying that I want to use this, I was just checking }
%	\jo{Now I understand what you mean and I fully agree with you :) The bound is not sharp at all.}

\medskip

\noindent \emph{$\bullet$ Step $2$ : Improved small divisor estimate.} Let $\boldsymbol{k},\boldsymbol{\ell},\boldsymbol{\sigma}$ as in the statement of the lemma and $\xi \in \mathcal{O}^\sharp$. By definition, we have that
\begin{equation}
\label{eq:tempora}
| \Omega_{\boldsymbol{k},\boldsymbol{\ell},\boldsymbol{\sigma} }(\xi) | \gtrsim  \gamma\, r_p^{2-\frac{37\delta}{20}} \langle \| \boldsymbol{\ell}\|_{\ell^\infty}\rangle^{-5}  \| \boldsymbol{k}\|_{\ell^1}^{-\# \mathcal{S}_p-1}.
\end{equation}
We aim at improving this estimate to get \eqref{eq:the_small_divisor_estimate}. Without loss of generality, we assume that $|\boldsymbol{\ell}_1| \geq \cdots \geq |\boldsymbol{\ell}_q| $.  We have two cases to consider.

\begin{itemize} \item Either we  have
\begin{equation*}
\frac18 |\boldsymbol{\ell}_1|  \leq 2\langle \boldsymbol{\ell}_4^* \rangle^2 + 2\max_{j\in \mathcal{S}_p} \langle j \rangle^2 \| \boldsymbol{k} \|_{\ell^1}^2
\end{equation*}
an in that case it is clear that \eqref{eq:tempora} implies \eqref{eq:the_small_divisor_estimate}.
\item Or we have
\begin{equation}
\label{eq:temporara}
\frac18 |\boldsymbol{\ell}_1| \geq 2\langle \boldsymbol{\ell}_4^*\rangle^2 + 2\max_{j\in \mathcal{S}_p} \langle j \rangle^2 \| \boldsymbol{k} \|_{\ell^1}^2.
\end{equation}
Using the zero momentum condition and applying Lemma \ref{lem:gain_one_der}, we have that 
$$
| \boldsymbol{\sigma}_1\boldsymbol{\ell}_1^2+\boldsymbol{\sigma}_2\boldsymbol{\ell}_2^2+\boldsymbol{\sigma}_3\boldsymbol{\ell}_3^2| \geq \frac14  |\boldsymbol{\ell}_1| 
$$
 and so, by \eqref{eq:usefull_sd}, we have that
$$
| \Omega_{\boldsymbol{k},\boldsymbol{\ell},\boldsymbol{\sigma} } |  \geq \frac14  |\boldsymbol{\ell}_1|- \langle \boldsymbol{\ell}_4^*\rangle^2- 2\max_{j\in \mathcal{S}_p} \langle j \rangle^2 \| \boldsymbol{k}\|_{\ell^1} -\frac12 \mathop{\geq}^{\eqref{eq:temporara}} \frac18  |\boldsymbol{\ell}_1|.
$$
As a consequence, we also have  \eqref{eq:the_small_divisor_estimate}.
\end{itemize}
\end{proof}

\subsection{Proof of Proposition \ref{prop:res_coho}} The Borel set $\mathcal{O}^\sharp$ is given by Lemma \ref{lem:sd_estimates}. We use the notations of Lemmas \ref{lem:trans_Z4}, \ref{lem:the_key}, \ref{lem:sd_estimates} to translate (at step $p$) the Hamiltonians $Z_4$ and $Z^{(p)}$ and to define the small divisors $\Omega_{\boldsymbol{k},\boldsymbol{\ell},\boldsymbol{\sigma} }$. In order to deal with the assumption \eqref{eq:cond_indice_1} of Lemma \ref{lem:sd_estimates} (about the absence of actions in $v_{\boldsymbol{\ell}}^{\boldsymbol{\sigma}}$), we introduce a notion of reduced indices in order, somehow, to remove the actions from $v_{\boldsymbol{\ell}}^{\boldsymbol{\sigma}}$.

\medskip

More precisely, for all $q\geq 0$, $\boldsymbol{\sigma}\in \{-1,1\}^q$ and $\boldsymbol{\ell}\in (\mathcal{S}_p^c)^q$, the reduced indices $\underline{q}\leq q, \underline{\boldsymbol{\sigma}}\in \{-1,1\}^{\underline{q}}, \underline{\boldsymbol{\ell}}\in (\mathcal{S}_p^c)^{\underline{q}}$ are the only indices satisfying
$$
\underline{\boldsymbol{\ell}}_q \preccurlyeq \cdots \preccurlyeq \underline{\boldsymbol{\ell}}_{1} , \quad \forall i,j\leq \underline{q}, \  \underline{\boldsymbol{\ell}}_i =  \underline{\boldsymbol{\ell}}_j \implies  \underline{\boldsymbol{\sigma}}_i= \underline{\boldsymbol{\sigma}}_j,
$$
$$
\forall j\in \mathcal{S}_p^c, \quad \# \{ i \ | \ \underline{\boldsymbol{\ell}}_i =j \} = \sum_{\boldsymbol{\ell}_i = j} \boldsymbol{\sigma}_i =  \sum_{\underline{\boldsymbol{\ell}}_i = j} \underline{\boldsymbol{\sigma}}_i.
$$
Informally, $v_{\underline{\boldsymbol{\ell}}}^{\underline{\boldsymbol{\sigma}}}$ is nothing but what remains when actions have been factorized from $v_{\boldsymbol{\ell}}^{\boldsymbol{\sigma}}$.

\medskip

This notion, will be useful in this proof because it ensures the assumption \eqref{eq:cond_indice_1} and satisfies
\begin{equation}
\label{eq:reduce_preserve_ev}
\Omega_{\boldsymbol{k},\boldsymbol{\ell},\boldsymbol{\sigma} } = \Omega_{\boldsymbol{k},\underline{\boldsymbol{\ell}},\underline{\boldsymbol{\sigma}}}, \quad \sum_{1\leq i \leq q} \boldsymbol{\sigma}_i \boldsymbol{\ell}_i= \sum_{1\leq i \leq \uq} \us_i \ul_i, \quad \sum_{1\leq i \leq q} \boldsymbol{\sigma}_i= \sum_{1\leq i \leq \uq} \us_i.
\end{equation}
Finally, for $0\leq q\leq 4$, we set
$$
\mathbf{E}_q = \{ (\boldsymbol{k},\boldsymbol{\ell},\boldsymbol{\sigma})\in \mathbb{Z}^{\mathcal{S}_p}\times (\mathcal{S}_p^c)^q \times \{-1,1\}^q \ | \  \boldsymbol{k},\ul,\us \quad \mathrm{satisfy} \quad   \  \eqref{eq:cond_indice_0}, \eqref{eq:cond_indice_1}, \eqref{eq:cond_indice_2} \}
$$
and we recall that by Lemma \ref{lem:sd_estimates} that if $\xi \in \mathcal{O}^\sharp$ then
\begin{equation}
\label{eq:sd_reduced} 
 \forall (\boldsymbol{k},\boldsymbol{\ell},\boldsymbol{\sigma})\in \mathbf{E}_q, \quad |\Omega_{\boldsymbol{k},\boldsymbol{\ell},\boldsymbol{\sigma} }(\xi)|= | \Omega_{\boldsymbol{k},\ul,\us }(\xi)| \gtrsim \gamma  \langle \boldsymbol{\ell}_4^*\rangle ^{-10} \langle \| \ul\|_{\ell^\infty} \rangle   r_p^{2-\frac{9\delta}5} \| \boldsymbol{k}\|_{\ell^1}^{-\# \mathcal{S}_p-12}=: M_{\boldsymbol{k},\ul}.
\end{equation}

\medskip

Now, with these notations at hand, we divide the proof in $4$ step.

\subsubsection*{Step $1$: Construction of $\chi$ and $R$} Given $\xi \in \mathcal{O}^\sharp$, we define the coefficients of a formal Hamiltonians $\chi(\xi)$ by 
\begin{center}
$\chi_{\boldsymbol{k},\boldsymbol{m}}^{\boldsymbol{\ell},\boldsymbol{\sigma}}(\xi) = 0$  \quad if  \quad $  (\boldsymbol{k},\boldsymbol{\ell},\boldsymbol{\sigma}) \not\in  \mathbf{E}_q$ \quad and  \quad $
\displaystyle \chi_{\boldsymbol{k},\boldsymbol{m}}^{\boldsymbol{\ell},\boldsymbol{\sigma}}(\xi)  :=\frac{ P_{\boldsymbol{k},\boldsymbol{m}}^{\boldsymbol{\ell},\boldsymbol{\sigma}}(\xi) }{\ic \, \Omega_{\boldsymbol{k},\boldsymbol{\ell},\boldsymbol{\sigma} }(\xi)}
$  \quad else,
\end{center} 
for $\boldsymbol{k} \in \mathbb{Z}^{\mathcal{S}_p},\boldsymbol{m} \in \mathbb{N}^{\mathcal{S}_p},q\in \mathbb{N}, \boldsymbol{\ell} \in (\mathcal{S}_p^c)^q ,\boldsymbol{\sigma} \in  \{-1,1\}^q $. Note that due, to \eqref{eq:sd_reduced},  this definition makes sense (i.e. $\Omega_{\boldsymbol{k},\boldsymbol{\ell},\boldsymbol{\sigma} }(\xi)$ does not vanish). The other steps of the proof will be dedicated to proving that $\chi$ belongs to $\mathscr{C}_{\eta'',r,p,\mathcal{O}^\sharp}$ and satisfies the bound \eqref{eq:est_ChiR}. Assuming this properties for the moment, thanks to Proposition \ref{prop:poisson_bracket_estimate}, it makes sense to define  $R \in \mathscr{H}_{\eta',r,p,\mathcal{O}^\sharp}$ by
$$
R = \big\{ Z_4^{(\mathrm{lead},p)} + \sum_{1 \leq p' \leq p} Z_4^{(p',p)}, \chi \big\}.
$$
The properties of $R$ are then direct consequences of Proposition \ref{prop:poisson_bracket_estimate} on the stability of our Hamiltonian classes under Poisson brackets. Moreover, a direct calculation (standard in KAM theory) shows that, by construction of $\chi$ and $R$, there exists $Q\in \mathscr{H}_{\eta',r,p,\mathcal{O}^\sharp}$ such that
for all $\xi \in \mathcal{O}^\sharp$, $u\in \mathcal{A}_\xi(3r)$, we have
$$
\{ Z_2 + Z_4 + \sum_{1\leq p'\leq p}Z^{(p')}[\xi^{(p')}] , \chi[\xi] \}(u) = (P+R+Q)(\xi;u) 
$$
where
$$
Q_{\boldsymbol{k},\boldsymbol{m}}^{\boldsymbol{\ell},\boldsymbol{\sigma}}(\xi) = \mathds{1}_{  (\boldsymbol{k},\boldsymbol{\ell},\boldsymbol{\sigma}) \not\in \mathbf{E}_q} P_{\boldsymbol{k},\boldsymbol{m}}^{\boldsymbol{\ell},\boldsymbol{\sigma}}(\xi).
$$
Thus to prove that $\chi,R$ solve the cohomological equation \eqref{eq:coho}, we have to prove that $Q=0$, that is
\begin{equation}
\label{eq:pas_grand_chose2}
P_{\boldsymbol{k},\boldsymbol{m}}^{\boldsymbol{\ell},\boldsymbol{\sigma}}(\xi) \neq  0 \implies  (\boldsymbol{k},\boldsymbol{\ell},\boldsymbol{\sigma})\in \mathbf{E}_q.
\end{equation}
So let $\xi \in \mathcal{O}^\sharp$ and $\boldsymbol{k},\boldsymbol{m}, q ,\boldsymbol{\ell},\boldsymbol{\sigma}$ be such that $P_{\boldsymbol{k},\boldsymbol{m}}^{\boldsymbol{\ell},\boldsymbol{\sigma}}(\xi) \neq  0$.
Since $P$ commutes with the mass and the momentum, we have by \eqref{eq:reduce_preserve_ev} that $\boldsymbol{k},\ul,\us$ satisfy \eqref{eq:cond_indice_2}. Moreover, since $P$ is homogeneous of degree $d\leq 4$, we have that $\uq \leq 4$. The assumption \eqref{eq:cond_indice_1} is ensured by definition of the reduced indices. So it only remains to prove the assumption \eqref{eq:cond_indice_0}. Indeed, assuming by contradiction that $\boldsymbol{k}=0$  and $\us_1 \ul_1^2+\cdots + \us_{\uq} \ul_{\uq}^2=0$ then we would have (by preservation of the mass and of the momentum)
$$
\boldsymbol{\sigma}_1 +\cdots + \boldsymbol{\sigma}_q= \boldsymbol{\sigma}_1 \boldsymbol{\ell}_1+\cdots + \boldsymbol{\sigma}_q \boldsymbol{\ell}_q = \boldsymbol{\sigma}_1 \boldsymbol{\ell}_1^2+\cdots + \boldsymbol{\sigma}_q \boldsymbol{\ell}_q^2 = 0 
$$
which is known to imply (see \cite{KP96}) that $u_{\boldsymbol{\ell}}^{\boldsymbol{\sigma}}$ is integrable, i.e. that there exists $\varphi \in \mathfrak{S}_4$ such that $\varphi(\boldsymbol{\ell},\boldsymbol{\sigma}) =(\boldsymbol{\ell},-\boldsymbol{\sigma})  $. This is impossible because $P$ does not contains integrable terms (i.e. $\Pi_{\mathrm{int}} P = 0$). Thus, $\boldsymbol{k},\ul,\us$ satisfy all the assumptions of   Lemma \ref{lem:sd_estimates} and we have  $(\boldsymbol{k},\boldsymbol{\ell},\boldsymbol{\sigma})\in \mathbf{E}_q$.

\medskip

Actually, for later use we notice a slight refinement of \eqref{eq:pas_grand_chose2}. We denote by $P=P_0+P_1+P_2$ where $P_n := \Pi_{\mathrm{Top}=n} P$ for $n\in\{0,1,2\}$. Then, we have that for all $\xi \in \mathcal{O}^\sharp$
\begin{equation}
\label{eq:sd_true_reduced} 
\exists n,\quad P_{n,\boldsymbol{k},\boldsymbol{m}}^{\boldsymbol{\ell},\boldsymbol{\sigma}}(\xi) \neq  0 \implies  (\boldsymbol{k},\boldsymbol{\ell},\boldsymbol{\sigma})\in \mathbf{E}_q.
\end{equation}
To get it, it suffices to notice that, by construction, $P_0,\cdots,P_2$ also satisfy the assumptions we used on $P$ to get \eqref{eq:pas_grand_chose2} when $P_{\boldsymbol{k},\boldsymbol{m}}^{\boldsymbol{\ell},\boldsymbol{\sigma}}(\xi) \neq  0$ (because $\Pi_{\mathrm{int}}$ and $\Pi_{\mathrm{Top}=n}$ commute for $n=0,1,2$).

\medskip

Now, it remains to prove that $\chi$ belongs to $\mathscr{C}_{\eta'',r,p,\mathcal{O}^\sharp}$ and satisfies the bound \eqref{eq:est_ChiR}.

\subsubsection*{Step $2$: Decomposition of $\chi$.} Let $\xi \in \mathcal{O}^\sharp$. For $0\leq n \leq 2$, we define the formal Hamiltonians $\chi_n(\xi)$ by
\begin{center}
$\chi_{n,\boldsymbol{k},\boldsymbol{m}}^{\boldsymbol{\ell},\boldsymbol{\sigma}}(\xi) = 0$  \quad if  \quad $(\boldsymbol{k},\boldsymbol{\ell},\boldsymbol{\sigma})\not\in \mathbf{E}_q$ \quad and  \quad $
\displaystyle \chi_{n,\boldsymbol{k},\boldsymbol{m}}^{\boldsymbol{\ell},\boldsymbol{\sigma}}(\xi)  :=\frac{ P_{n,\boldsymbol{k},\boldsymbol{m}}^{\boldsymbol{\ell},\boldsymbol{\sigma}}(\xi) }{\ic \, \Omega_{\boldsymbol{k},\boldsymbol{\ell},\boldsymbol{\sigma} }(\xi)}$ \quad else.
\end{center}
The fact that $\chi_n$ is well and defined and that $\chi(\xi) = \chi_0(\xi)+\chi_1(\xi)+\chi_2(\xi)$ is due to \eqref{eq:sd_true_reduced}. Moreover, since $P_n(\xi) \in \mathscr{F}_{\eta,r,p}^{\cste}$, the bound \eqref{eq:sd_reduced} also implies that $\chi_n (\xi) \in \mathscr{F}_{\eta'',r,p}^{\cste}$ (because, by Lemma \ref{lem:mu_P_finite}, we have $\mu_{P_n}\leq \mu_P<\infty$ for $0\leq n\leq 2$).
% \mic{i dont understand (62) is not a bound, you want (60)?} \jo{yes, corrected but is this bound $|\boldsymbol{\ell}_4^*| \leq \mu_P<\infty$ clear for you ? No that I re-read it, it seems to me it is too fast. I mean, now, it seems to me that we want to use that if there exits $n$ such that $P_{n,\boldsymbol{k},\boldsymbol{m}}^{\boldsymbol{\ell},\boldsymbol{\sigma}}(\xi)\neq 0$ then $|\boldsymbol{\ell}_4^*| \leq \mu_P$. It is true but it would require more explanations, no ? Or does I miss something and there is something obvious ? I can write a small paragraph to prove it if you think it is useful. I tried many alternative definitions for $\chi_n$ and with some of them it was really obvious, it could explain why I wrote so few explanations...} \mic{this was the spirit of the remark below, I was wondering if one should not put a remark saying that if $P\in \mathscr{H}_{\eta,r,p}$ has $\nu_P<\infty$ then the same holds for $P_0,P_1$ and  moreover $P_2=0$} \jo{yes, I added Lemma \ref{lem:mu_P_finite} to state and prove it. }
 \\
We have to prove that this decomposition of $\chi(\xi)$ is the one associated with the $\mathscr{H}$ norm. This is a direct consequence of the fact that $P=P_0+P_1+P_2$ is the decomposition of $P$ associated with the definition of the $\mathscr{H}$ norm and that for all $\boldsymbol{k} \in \mathbb{Z}^{\mathcal{S}_p},0\leq q\leq 2, i\in \mathcal{S}_p^c, \boldsymbol{\ell} \in (\mathcal{S}_p^c)^q ,\boldsymbol{\sigma} \in  \{-1,1\}^q $,
$$
(\boldsymbol{k},\boldsymbol{\ell},\boldsymbol{\sigma})\in \mathbf{E}_q \iff (\boldsymbol{k},(\boldsymbol{\ell},\boldsymbol{\sigma};i,\pm 1))\in \mathbf{E}_{q+2}.
$$
Now that we have this decomposition of $\chi$, we have to prove the bound \eqref{eq:est_ChiR} on $\chi$.

\subsubsection*{Step $3$: bound on $\chi$.}   First, we notice that, since by assumption $\mu_P <\infty$, we have by Lemma \ref{lem:mu_P_finite} that $P_2=0$ and so we have $\chi_2=0$. Thus, it suffices to prove bounds on $\chi_0$ and $\chi_1$. All the estimates will rely on the following ones on the coefficients of $\chi$. So let $\xi,\zeta \in \mathcal{O}^\sharp$, $0\leq n \leq 1$, $\boldsymbol{k} \in \mathbb{Z}^{\mathcal{S}_p},\boldsymbol{m} \in \mathbb{N}^{\mathcal{S}_p},q\in \mathbb{N}, \boldsymbol{\ell} \in (\mathcal{S}_p^c)^q ,\boldsymbol{\sigma} \in  \{-1,1\}^q $ be such that $(\boldsymbol{k},\boldsymbol{\ell},\boldsymbol{\sigma})\in \mathbf{E}_q$. By  \eqref{eq:sd_reduced}, we have
$$
|\chi_{n,\boldsymbol{k},\boldsymbol{m}}^{\boldsymbol{\ell},\boldsymbol{\sigma}}(\xi) | \leq  \frac{|P_{n,\boldsymbol{k},\boldsymbol{m}}^{\boldsymbol{\ell},\boldsymbol{\sigma}}(\xi)  |}{M_{\boldsymbol{k},\ul}}
$$
and
	\begin{equation*}
	\begin{split}
|\chi_{n,\boldsymbol{k},\boldsymbol{m}}^{\boldsymbol{\ell},\boldsymbol{\sigma}}(\xi)  - \chi_{n,\boldsymbol{k},\boldsymbol{m}}^{\boldsymbol{\ell},\boldsymbol{\sigma}}(\zeta) | &= |\frac{ P_{n,\boldsymbol{k},\boldsymbol{m}}^{\boldsymbol{\ell},\boldsymbol{\sigma}}(\xi) }{\ic \, \Omega_{\boldsymbol{k},\boldsymbol{\ell},\boldsymbol{\sigma} }(\xi)}  - \frac{ P_{n,\boldsymbol{k},\boldsymbol{m}}^{\boldsymbol{\ell},\boldsymbol{\sigma}}(\zeta) }{\ic \, \Omega_{\boldsymbol{k},\boldsymbol{\ell},\boldsymbol{\sigma} }(\zeta)} | \\
&\leq \frac{|P_{n,\boldsymbol{k},\boldsymbol{m}}^{\boldsymbol{\ell},\boldsymbol{\sigma}}(\xi) -P_{n,\boldsymbol{k},\boldsymbol{m}}^{\boldsymbol{\ell},\boldsymbol{\sigma}}(\zeta) |}{M_{\boldsymbol{k},\ul}} + 
\frac{ |P_{n,\boldsymbol{k},\boldsymbol{m}}^{\boldsymbol{\ell},\boldsymbol{\sigma}}(\xi)  |  }{M_{\boldsymbol{k},\ul}^2}  |\Omega_{\boldsymbol{k},\boldsymbol{\ell},\boldsymbol{\sigma} }(\xi) - \Omega_{\boldsymbol{k},\boldsymbol{\ell},\boldsymbol{\sigma} }(\zeta)|\\
&\leq \frac{|P_{n,\boldsymbol{k},\boldsymbol{m}}^{\boldsymbol{\ell},\boldsymbol{\sigma}}(\xi) -P_{n,\boldsymbol{k},\boldsymbol{m}}^{\boldsymbol{\ell},\boldsymbol{\sigma}}(\zeta) |}{M_{\boldsymbol{k},\ul}} + 
\frac{ |P_{n,\boldsymbol{k},\boldsymbol{m}}^{\boldsymbol{\ell},\boldsymbol{\sigma}}(\xi)  |  }{M_{\boldsymbol{k},\ul}^2} r_p^{-3\delta/2} \langle \| \boldsymbol{k} \|_{\ell^1} \rangle \|\xi -\zeta \|_{\ell^2}
\end{split}
		\end{equation*}
		where the bound on $\Omega_{\boldsymbol{k},\boldsymbol{\ell},\boldsymbol{\sigma} }(\xi) - \Omega_{\boldsymbol{k},\boldsymbol{\ell},\boldsymbol{\sigma} }(\zeta)$ follows directly of the Lipschitz estimate \eqref{eq:bad_lip} on the frequencies $\omega^{(p',p)}$ given by Lemma \ref{lem:the_key}.

\medskip

		By definition of the $\mathscr{H}$ and $\mathscr{C}$ norms (Definition \ref{def:class_take3}), it follows that
			\begin{equation*}
	\begin{split}
		 \| \chi\|_{\mathscr{C}_{\eta'',r,p,\mathcal{O}^\sharp}^{\mathrm{tot}}} &\lesssim  \| P\|_{\mathscr{H}_{\eta'',r,p,\mathcal{O}^\sharp}^{\mathrm{tot}}}  \sup_{\substack{ 0\leq n\leq 1, \ 0\leq q\leq 4-2n \ \\ (\boldsymbol{k},\boldsymbol{\ell},\boldsymbol{\sigma})\in \mathbf{E}_{q-2n}  \\  \boldsymbol{\ell}_1^* \prec{\boldsymbol{i}}_n\prec\dots\prec{\boldsymbol{i}}_1\\ \boldsymbol{\ell}_4^* \preccurlyeq \mu_P}}  \frac{\Upsilon_{n,\boldsymbol{\ell},\boldsymbol{\sigma}}^{-1}  e^{-(\eta-\eta'')\| \boldsymbol{k}\|_{\ell^1}}}{\Theta_{n,\boldsymbol{\ell},\boldsymbol{\sigma}}^{-1} M_{\boldsymbol{k},\ul}}   \Big( 1 + \frac{r^2 r_p^{-\delta/9} \langle \| \boldsymbol{k} \|_{\ell^1} \rangle }{ M_{\boldsymbol{k},\ul}} \Big). 
		\end{split}
		\end{equation*}
		Thus, using the small divisor estimate \eqref{eq:sd_reduced}, it follows that
		\begin{equation*}
	\begin{split}
		 \| \chi\|_{\mathscr{C}_{\eta'',r,p,\mathcal{O}^\sharp}^{\mathrm{tot}}} 	&\lesssim \| P\|_{\mathscr{H}_{\eta'',r,p,\mathcal{O}^\sharp}^{\mathrm{tot}}}  \sup_{\substack{ 0\leq n\leq 1, \ 0\leq q\leq 4-2n \ \\ \boldsymbol{k}\in \mathbb{Z}^{\mathcal{S}_p}, \ \boldsymbol{\sigma}\in \{-1,1\}^{q-2n}\\  \boldsymbol{\ell}_1^* \prec{\boldsymbol{i}}_n\prec\dots\prec{\boldsymbol{i}}_1\\ \boldsymbol{\ell}_4^*\preccurlyeq \mu_P}} 
\Big( \frac{\langle \| \boldsymbol{k} \|_{\ell^1} \rangle^{2\# \mathcal{S}_p+25} e^{-(\eta-\eta'')\| \boldsymbol{k}\|_{\ell^1}}}{\gamma^2 r_p^{2-\frac{9\delta}5}}  \Big) \frac{\Upsilon_{n,\boldsymbol{\ell},\boldsymbol{\sigma}}^{-1} \langle \boldsymbol{\ell}_4^*\rangle ^{20}  }{\Theta_{n,\boldsymbol{\ell},\boldsymbol{\sigma}}^{-1}  \langle \| \ul \|_{\ell^\infty}\rangle}	 \\
&\lesssim \| P\|_{\mathscr{H}_{\eta'',r,p,\mathcal{O}^\sharp}^{\mathrm{tot}}} \! \! \! \frac{  \langle \mu_{P}\rangle ^{21}  (\eta-\eta'')^{2\# \mathcal{S}_p+25} }{\gamma^2 r_p^{2-\frac{7\delta}4}} \hspace{-0.7cm}\sup_{\substack{ 0\leq q\leq 4  \\ \boldsymbol{\ell} \in \mathbb{Z}^q, \  \boldsymbol{\sigma}\in \{-1,1\}^{q}}} 
  \underbrace{\frac{ \langle \| \boldsymbol{\ell}\|_{\ell^\infty} \rangle  }{(1+\mathds{1}_{(\boldsymbol{\ell},\boldsymbol{\sigma})\in \mathbf{C}_{q,1}} \| \boldsymbol{\ell}\|_{\ell^\infty})  \langle \| \ul \|_{\ell^\infty}\rangle \langle \boldsymbol{\ell}_4^*\rangle  }}_{=: A_{\boldsymbol{\ell},\boldsymbol{\sigma}}}.
		\end{split}
		\end{equation*}
To conclude this proof, it only remains to bound $A_{\boldsymbol{\ell},\boldsymbol{\sigma}}$ by $1$. So let $\sigma \in \{-1,1\}^q, \boldsymbol{\ell} \in (\mathcal{S}_p^c)^q$ with $0\leq q \leq 4$. The point is to note that at least one of the following cases holds.
\begin{itemize}
\item Case $1$: $ \| \ul \|_{\ell^\infty} = \| \boldsymbol{\ell}\|_{\ell^\infty}$. In that case it is clear that $ A_{\boldsymbol{\ell},\boldsymbol{\sigma}} \leq 1$.
\item Case $2$:  $(\boldsymbol{\ell},\boldsymbol{\sigma})  \in \mathbf{C}_{q,1}$.  
\item Case $3$: Up to permutation, $(\boldsymbol{\ell},\boldsymbol{\sigma}) = (i,\pm 1;i,\pm 1)$ for some $i\in \mathcal{S}_p^c$.
\end{itemize}
In the two first cases, it is clear that $ A_{\boldsymbol{\ell},\boldsymbol{\sigma}} \leq 1$.
 In the last case, we have $|\boldsymbol{\ell}_4^*| = |i| = \| \boldsymbol{\ell}\|_{\ell^\infty}$ and so, we also have that $ A_{\boldsymbol{\ell},\boldsymbol{\sigma}} \leq 1$.
		 
\section{The loop}

In this section, we pass from step $p-1$ to step $p$. From now, to simplify notations, we fix arbitrarily, once for all, a sequence of analyticity radii:
\begin{equation}
\label{eq:def_etap}
\eta_p = 2 \eta_{\min} + 2^{-p}.
\end{equation}
%\mic{I am afraid i am getting lost, should we not need also to fix $\gamma_p$ and assume some smallness on the $r_p$?} \jo{It is somehow what we do. Of course to prove the KAM theorem (that we do not prove) we need to take another sequence of $\gamma$ going to $0$ but they will easily be compensated by the smallness of the Hamiltonians we have to remove. The smallness assumption on $r_p$ is always, by convention, implicit. We only have to check that we have a finite number of constraints of the form $r_p \lesssim_{r_{p-1},p}  1$. It is something we could recall if you want. } \mic{yes I will recall it in a couple of places, I think it helps} \jo{Perfect!}
The following proposition is the main result of this section. 
\begin{proposition} \label{prop:the_loop} Let $p\geq 1$ and $\mathcal{O}_{p-1} \subset \mathcal{O}_{p-1}^{\mathrm{full}}$ be a Borel set. We set, for all $\xi \in \mathcal{O}_{p-1}$ and all $u\in \mathcal{A}_\xi(3r_{p-1})$  
$$
H(\xi;u) := Z_2(u)+Z_4(u) + \sum_{1\leq p'\leq p-1} Z^{(p')}(\xi^{(p')};u) +  R^{(p-1)}(\xi;u)
$$
 where $\xi^{(p')} := (\xi_j)_{j\in \mathcal{S}_{p'}}$,
\begin{itemize}
\item $Z_2,Z_4$ are given by Definition \ref{def:Z_24} and, as in Proposition \ref{prop:res_coho}, for all $p'\in \{1,\cdots,p-1\}$, $Z^{(p')} \in  \Pi_{\mathrm{deg}\leq 4}  \Pi_{\mathrm{int}} \mathscr{H}_{\eta_{p'},r_{p'},p',\mathcal{O}_{p'}}$ is an integrable Hamiltonian  of degree smaller than or equal to $4$, defined for parameters in a Borel set $\mathcal{O}_{p'} \subset \mathcal{O}_{p'}^{\mathrm{full}}$ satisfying $\mathcal{O}_{p-1}  \subset \mathcal{O}_{p'} \times \mathbb{R}^{\mathcal{S}_{p-1} \setminus \mathcal{S}_{p'}} $ and satisfying the bound
$$
\| Z^{(p')} \|_{ \mathscr{H}_{\eta_{p'},r_{p'},p',\mathcal{O}_{p'}}^{\mathrm{tot}}} \leq r_{p'}^{4+1/3},
$$
\item  $R^{(p-1)} \in \Pi_{\mathrm{deg} \geq 4} \mathscr{H}_{\eta_{p-1},r_{p-1},p-1,\mathcal{O}_{p-1}}$ is a Hamiltonian, containing only terms of degree larger than or equal to $4$ and no quartic integrable term (i.e. $\Pi_{\mathrm{deg} =4} \Pi_{\mathrm{int}} R^{(p-1)} =0 $), satisfying the bound
$$
\| R^{(p-1)} \|_{ \mathscr{H}_{\eta_{p-1},r_{p-1},p-1,\mathcal{O}_{p-1}}^{\mathrm{tot}}}  \leq r_{p-1}^{4-2\delta},
$$
\end{itemize}
Then, there exist
\begin{itemize}
\item a Borel set $\mathcal{O}_{p} \subset  \mathcal{O}_{p-1} \times \big(\frac{r_{p}^{2(1-\delta)}}2,  r_{p}^{2(1-\delta)}\big)^{\mathcal{S}_{p} \setminus \mathcal{S}_{p-1}}  $ satisfying
\begin{equation}
\label{eq:measure_est_loop}
\mathrm{meas}\Big( \mathcal{O}_{p-1} \times \big(\frac{r_{p}^{2(1-\delta)}}2,  r_{p}^{2(1-\delta)}\big)^{\mathcal{S}_{p} \setminus \mathcal{S}_{p-1}} \setminus \mathcal{O}_{p} \Big) \leq   r_p^{\delta/9} \mathrm{meas}(\mathcal{O}_{p}^{\mathrm{full}} ).
\end{equation}
\item an integrable Hamiltonian $Z^{(p)} \in  \Pi_{\mathrm{deg}\leq 4}  \Pi_{\mathrm{int}} \mathscr{H}_{\eta_p,r_{p},p,\mathcal{O}_p}$ of degree smaller than or equal to $4$ satisfying the bound
$$
\| Z^{(p)} \|_{ \mathscr{H}_{\eta_p,r_{p},p,\mathcal{O}_p}^{\mathrm{tot}}} \leq r_{p}^{4+1/3},
$$
\item a Hamiltonian $R^{(p)} \in \Pi_{\mathrm{deg} \geq 4} \mathscr{H}_{\eta_{p},r_{p},p,\mathcal{O}_{p}}$  containing only terms of degree larger than or equal to $4$ and no quartic integrable term (i.e. $\Pi_{\mathrm{deg} =4} \Pi_{\mathrm{int}} R^{(p)} =0 $), satisfying the bound 
$$
\| R^{(p)} \|_{ \mathscr{H}_{\eta_{p},r_{p},p,\mathcal{O}_{p}}^{\mathrm{tot}}}  \leq r_p^{4-2\delta},
$$
\end{itemize}
such that, for all $\xi \in \mathcal{O}_{p}$ and all $u\in \mathcal{A}_\xi(2 r_p)$  
$$
H(\xi;\tau_\xi^{(p)}(u)) =  Z_2(u)+Z_4(u) + \sum_{1\leq p'\leq p} Z^{(p')}(\xi^{(p')};u) +  R^{(p)}(\xi;u)
$$
where $\tau_\xi^{(p)} \in C^2( \mathcal{A}_\xi(r_p); \mathcal{A}_\xi(r_{p-1}/2) )$ is a symplectic map which is close to the identity: for all $ j\in \{0,1,2\}$, $v^{(1)},\cdots,v^{(j)}\in H^s$,
\begin{equation}
\label{eq:tau_close_id}
 \|\mathrm{d}^j \big(\tau_\xi^{(p)} - \mathrm{Id}  \big)(u)  (v^{(1)},\cdots,v^{(j)}) \|_{H^s} \leq r_p^{2-j-\frac13} \prod_{1\leq i\leq j} \|v^{(i)}\|_{H^s}.
 \end{equation}
\end{proposition}
The rest of this section is devoted to the proof of Proposition \ref{prop:the_loop}. Multilinear estimates and small divisor estimates have been proven in previous sections. In this proof it mainly remains to use them and to pay attention to homogeneity arguments. The proof involves extra analyticity radii $\eta_p^{(i)}$ defined by
\begin{equation}
\label{eq:def_etapa}
\eta_p^{(\alpha)} :=\big( 1 -\frac{\alpha}4 \big)\eta_{p-1} + \frac{\alpha}4 \eta_p, \quad \mathrm{for} \quad \alpha \in [0,4] \quad \mathrm{and} \quad p\geq 1.
\end{equation}

\subsection{Step $1$: A first decomposition and a zoom}  We decompose the remainder term $R^{(p-1)}$ as
$$
R^{(p-1)} = B + G+L
$$ 
where $L := \Pi_{\mathrm{deg}\geq 6} R^{(p-1)} $ contains all the terms of degree larger than or equal to $6$ of $R^{(p-1)} $ and $B$ contains all the terms of degree smaller than or equal to $5$ which are going to generate terms to be removed after the translation, i.e. for all $\boldsymbol{k} \in \mathbb{Z}^{\mathcal{S}_{p-1}},\boldsymbol{m} \in \mathbb{N}^{\mathcal{S}_{p-1}},q\in \mathbb{N}, \boldsymbol{\ell} \in (\mathcal{S}_{p-1}^c)^q ,\boldsymbol{\sigma} \in  \{-1,1\}^q $ we have
\begin{equation}
\label{eq:def_B}
B_{\boldsymbol{k} ,\boldsymbol{m} }^{\boldsymbol{\ell} ,\boldsymbol{\sigma} }  =  (\Pi_{\mathrm{deg} \leq 5} R^{(p-1)})_{\boldsymbol{k} ,\boldsymbol{m} }^{\boldsymbol{\ell} ,\boldsymbol{\sigma} }  \quad \mathrm{if} \quad \# \{j \ | \ \boldsymbol{\ell}_j \in \mathcal{S}_{p}\}\geq q+2\|\boldsymbol{m} \|_{\ell^1}-3
\end{equation}
 and $B_{\boldsymbol{k} ,\boldsymbol{m} }^{\boldsymbol{\ell} ,\boldsymbol{\sigma} } = 0$ else. Applying Lemma \ref{lem:non_trivial_proj} of the appendix, we deduce that $B$ belongs to $\mathscr{H}_{\eta_{p-1},r_{p-1},p-1,\mathcal{O}_{p-1}}$ and satisfies 
 $$
 \| B\|_{ \mathscr{H}_{\eta_{p-1},r_{p-1},p-1,\mathcal{O}_{p-1}}^{\mathrm{tot}}} \lesssim_p \| R^{(p-1)}\|_{ \mathscr{H}_{\eta_{p-1},r_{p-1},p-1,\mathcal{O}_{p-1}}^{\mathrm{tot}}}.
 $$
 Thus, since $\| R^{(p-1)} \|_{ \mathscr{H}_{\eta_{p-1},r_{p-1},p-1,\mathcal{O}_{p-1}}^{\mathrm{tot}}}  \leq r_{p-1}^{4-2\delta}$, these Hamiltonians satisfy the estimate 
 $$
 \| B\|_{ \mathscr{H}_{\eta_{p-1},r_{p-1},p-1,\mathcal{O}_{p-1}}^{\mathrm{tot}}}\vee  \| G \|_{ \mathscr{H}_{\eta_{p-1},r_{p-1},p-1,\mathcal{O}_{p-1}}^{\mathrm{tot}}} \vee  \| L \|_{ \mathscr{H}_{\eta_{p-1},r_{p-1},p-1,\mathcal{O}_{p-1}}^{\mathrm{tot}}}  \leq r_{p-1}^{4-3\delta}.
 $$
 Then we perform a zoom: we decrease the radius of the annulus. More precisely, we set 
 $$
 \rho := r_p^{1-8 \delta /7}. 
 $$
and, using that $\Pi_{\mathrm{deg} \geq 4} B = B$, $\Pi_{\mathrm{deg} \geq 4} G = G$ and $\Pi_{\mathrm{deg} \geq 6} L = L$, we get by homogeneity that
 \begin{equation}
 \label{eq:bound_B}
 \| B\|_{ \mathscr{H}_{\eta_{p-1},\rho,p-1,\mathcal{O}_{p-1}}^{\mathrm{tot}}}\vee  \| G \|_{ \mathscr{H}_{\eta_{p-1},\rho,p-1,\mathcal{O}_{p-1}}^{\mathrm{tot}}}  \leq (\rho/r_{p-1})^4  r_{p-1}^{4-3\delta}  \leq \rho^4 r_p^{- \delta/7 }
 \end{equation}
 and
 \begin{equation}
  \label{eq:bound_L}
 \| L \|_{ \mathscr{H}_{\eta_{p-1},\rho,p-1,\mathcal{O}_{p-1}}^{\mathrm{tot}}} \leq  \rho^6 r_{p-1}^{- 2 - 3\delta }\leq \rho^6 r_p^{- \delta/7 }\,,
 \end{equation}
provided we take $r_p$ sufficiently small with respect to $r_{p-1}$.
 
 At the next step,  we aim at removing $B$ with one step of Birkhoff normal form. So, we are going to apply Proposition \ref{prop:res_coho} to solve a cohomological equation (with $P=-B$). In particular, we have to control its largest fourth largest index $\mu_P$. Indeed, by definition, $B$ only contains terms of degree $5$ which are at least of degree $2$ with respect to $(u_j)_{j\in S_p}$ and terms of degree $4$ which are at least of degree $1$ with respect to $(u_j)_{j\in S_p}$. So its terms are at most of degree $3$ with respect $(u_j)_{j\in S_p^c}$. It follows that 
$$
\mu_B\leq \max_{j\in \mathcal{S}_p} |j| \leq r_p^{-\delta/20}.
$$

\subsection{Step $2$: One step of Birkhoff normal form to remove $B$} To lighten the notations, we set
$$
Z_{\leq p-1}[\xi] := Z_2 + Z_4 + \sum_{1\leq p'\leq p-1}Z^{(p')}[\xi^{(p')}].
$$
so that $H[\xi]= Z_{\leq p-1}[\xi] + R^{(p-1)}[\xi]$. Note that, in this definition, only $R^{(p-1)}$ corresponds to  a formal power series  via the identification \eqref{identi}, $H$ and $Z_{\le p-1}$ are just functions of $u$ parametrized by $\xi$. 
Now, we apply Proposition \ref{prop:res_coho} with $\gamma =r_p^{\delta/8}$ and we use the bound we proved on $\mu_B$. We get a Borel set $\mathcal{O}^{(1)} \subset  \mathcal{O}_{p-1}$, satisfying the measure estimate 
\begin{equation}
\label{eq:meas1}
\mathrm{meas}(\mathcal{O}_{p-1} \setminus \mathcal{O}^{(1)} ) \leq r_p^{\delta/8} \, \mathrm{meas} (\mathcal{O}_{p-1}^{\mathrm{full}}),
\end{equation}
and two Hamiltonians $\chi \in \mathscr{C}_{\eta_{p}^{(1)},\rho,p-1,\mathcal{O}^{(1)} }$ and $A \in   \Pi_{\mathrm{deg \geq 6}}\mathscr{H}_{\eta_{p}^{(1)},\rho,p-1,\mathcal{O}^{(1)} }$ solving the cohomological equation
\begin{equation}
\label{eq:coho1}
\forall \xi \in \mathcal{O}^{(1)} , \forall  u\in \mathcal{A}_\xi(3\rho), \quad \{ Z_{\leq p-1}[\xi] , \chi[\xi] \}(u) = (A-B)(\xi;u) 
\end{equation}
and satisfying the bounds
\begin{equation} \label{eq:bound_Achi}
\begin{split}
\| \chi \|_{\mathscr{C}_{\eta_{p}^{(1/2)},\rho,p-1,\mathcal{O}^{(1)} }^{\mathrm{tot}}} \! \! \! \! &\lesssim (r_p^{\delta/8})^{-2}  \langle \mu_B \rangle^{21}  (\eta_{p-1} - \eta_p^{(1/2)})^{- 2 \# \mathcal{S}_{p-1}-25} r_{p-1}^{-2+7\delta/4} \| B\|_{ \mathscr{H}_{\eta_{p-1},\rho,p-1,\mathcal{O}_{p-1}}^{\mathrm{tot}}} \! \! \! \! \! \! \leq \rho^{4}r_p^{-\frac{\delta}2} \\
\| A \|_{\mathscr{H}_{\eta_{p}^{(1)},\rho,p-1,\mathcal{O}^{(1)} }^{\mathrm{tot}} } \! \! \! \! &\lesssim \rho^{2-\delta} (\eta_{p}^{(1)} - \eta_p^{(1/2)})^{- 9} \| \chi \|_{\mathscr{C}_{\eta_{p}^{(1/2)},\rho,p-1,\mathcal{O}^{(1)} }^{\mathrm{tot}}}  \! \! \leq \rho^{6}r_p^{-\frac{3\delta}2}.\\
\end{split}
\end{equation}
Moreover, since the support of $\chi$ and $B$ coincide, we have 
$$
\Pi_{\|\boldsymbol{m} \| \geq 2} \chi = \Pi_{\|\boldsymbol{m} \| \geq 2} B =0.
$$
The fact that the $y-$homogeneity of $B$ is smaller than or equal to $1$ is due to the definition of \eqref{eq:def_B} and that $B$ only contains terms of degree $4$ and $5$.

\medskip

Thus, since $\rho^{4}r_p^{-\frac{\delta}2} \ll \rho^{2+\delta}$, we can apply Lemma \ref{lem:flows} to define the Hamiltonian flow $\Phi_\chi$. Moreover, using that $\chi$ solves the cohomological equation \eqref{eq:coho1}, we get that for all $\xi \in \mathcal{O}^{(1)} $ and all $u\in \mathcal{A}_\xi(2\rho)$
\begin{equation}
\label{eq:def_H1}
H(\xi; \Phi_{\chi[\xi]}^1(u)) =Z_{\leq p-1}(\xi;u) + G(\xi;u) + L^\sharp(\xi;u) =: H^{(1)}(\xi;u)
\end{equation}
where
\begin{equation*}
\begin{split}
L^\sharp(\xi;u) =& Z_{\leq p-1}[\xi]\circ \Phi_{\chi[\xi]}^1(u) - Z_{\leq p-1}(\xi;u) - \{Z_{\leq p-1}[\xi],\chi[\xi]\}(u) \\
&+ (B+G)(\xi;\Phi_{\chi[\xi]}^1(u)) - (B+G)(\xi;u) +  L(\xi;\Phi_{\chi[\xi]}^1(u)) +  A(\xi;u)  \\
=& \int_{0}^1 \! \! \Big\{ (1-t)\{Z_{\leq p-1}[\xi],\chi[\xi]\}  + (B+G)[\xi] , \chi[\xi]\Big\}\circ \Phi_{\chi[\xi]}^t(u) \, \mathrm{d}t +  L(\xi;\Phi_{\chi[\xi]}^1(u)) +   A(\xi;u)\\
=&  \int_{0}^1 \Big\{  (tB + (1-t)A +G)[\xi] , \chi[\xi]\Big\}\circ \Phi_{\chi[\xi]}^t(u) \, \mathrm{d}t +  L(\xi;\Phi_{\chi[\xi]}^1(u)) +   A(\xi;u).
\end{split}
\end{equation*}
As a consequence of Propositions \ref{prop:comp_ham_flow} and \ref{prop:poisson_bracket_estimate}, it follows that $L^\sharp\in \mathscr{H}_{\eta_{p}^{(2)},\rho,p-1,\mathcal{O}^{(1)} }$ and, thanks to the bounds \eqref{eq:bound_B},\eqref{eq:bound_L} and \eqref{eq:bound_Achi} on $A,B,G,\chi,L$, it satisfies the bound
\begin{equation}
\label{eq:bound_Ldiese}
\| L^\sharp \|_{\mathscr{H}_{\eta_{p}^{(2)},\rho,p-1,\mathcal{O}^{(1)} }^{\mathrm{tot}}} \leq \rho^6 r_p^{- 2\delta }.
\end{equation}

\subsection{Step $3$: A translation and another zoom} Now, we perform the translation: we apply Proposition \ref{prop:open} with 
$$
r = \varrho := r_p^{1-6\delta/7}.
$$
Setting 
\begin{equation}
\label{eq:meas2}
\mathcal{O}^{(2)} := \mathcal{O}^{(1)} \times \big(\frac{r_{p}^{2(1-\delta)}}2,  r_{p}^{2(1-\delta)}\big)^{\mathcal{S}_{p}} \subset \mathcal{O}_p^{\mathrm{full}},
\end{equation}
we get two Hamiltonians $G^{(2)},L^{(2)} \in \mathscr{H}_{\eta_{p}^{(2)},\varrho,p,\mathcal{O}^{(2)} }$, obtained as the traslation of $G$ and $L^\sharp$ respectively, such that for all $\xi  \in \mathcal{O}^{(2)}$ and all $u\in \mathcal{A}_\xi(3\varrho)$ 
%\mic{I dont really like that $\xi^{(p-1)}$ appearing only on the left hand side} \jo{yes but it is rather natural since we do a translation...}\mic{you are using formula (34) right? What bothers me is the apparent wierdness in writing an equality in which the left and the right hand side have two different sets of variables. In general I undestand better if I avoid the $(\xi,u)$ operation as much as possible (for example in (44) I did not change the notation to the formal power series only because one would need to define the traslation from $p'$ to $p$ which seems quite heavy) Maybe I can try to put some reminders or a remark someplace (afer the traslation proposition?) to highlight this} \jo{I am not completely sure to see what you mean. Somehow you do not like equalities between functions and you would prefer equalities between formal power series ?} \jo{I agree that one unusual aspect here is that we have a lot of equalities between function which are given by formal series of different variable (because given by different values of $p'$).} \mic{what bothers me is that in the line below $H^1$ and $Z_{\le p-1}$ are not power series while the last two are so,  I would like to add something whci makes it clearer}
$$
H^{(1)}(\xi^{(p-1)};u) = Z_{\leq p-1}(\xi;u) + G^{(2)}(\xi;u) + L^{(2)}(\xi;u)
$$
and 
$$
\| G^{(2)}\|_{\mathscr{H}_{\eta_{p}^{(2)},\varrho,p,\mathcal{O}^{(2)} }^{\mathrm{tot}}} \lesssim_p \| G \|_{ \mathscr{H}_{\eta_{p-1},\rho,p-1,\mathcal{O}_{p-1}}^{\mathrm{tot}}} \mathop{\leq}^{\eqref{eq:bound_B}}  \rho^4 r_p^{- \delta/6 },
$$
$$
\| L^{(2)}\|_{\mathscr{H}_{\eta_{p}^{(2)},\varrho,p,\mathcal{O}^{(2)} }^{\mathrm{tot}}} \lesssim_p \| L^\sharp \|_{\mathscr{H}_{\eta_{p}^{(2)},\rho,p-1,\mathcal{O}^{(1)} }^{\mathrm{tot}}} \mathop{\leq}^{\eqref{eq:bound_Ldiese}} \rho^6 r_p^{- 3\delta }.
$$
Moreover, since  $\Pi_{\mathrm{int}} \Pi_{\mathrm{deg}=4} R^{(p-1)} = 0$, we have $ \Pi_{\mathrm{int}} G = 0$ and so  
$$
\Pi_{\mathrm{int}} G^{(2)} = 0.
$$
Furthermore, since, by construction  (see \eqref{eq:def_B}), we have
$$
 \min_{ G_{\boldsymbol{k},\boldsymbol{m}}^{\boldsymbol{\ell},\boldsymbol{\sigma}} \neq 0 } 2 \| \boldsymbol{m} \|_{\ell^1} + \# \{ j \ | \ \boldsymbol{\ell}_j \in \mathcal{S}_{p}^c\} \geq 4,
$$
we get that $G^{(2)}$ only contains terms of degree higher than or equal to $4$, i.e.
\begin{equation}
\label{eq:deg_G2}
\Pi_{\mathrm{deg}\geq 4}G^{(2)} = G^{(2)}.
\end{equation}
Thanks to this property, since $\rho/\varrho = \delta^{-2\delta/7}$, we have that
$$
\| G^{(2)}\|_{\mathscr{H}_{\eta_{p}^{(2)},r_p,p,\mathcal{O}^{(2)} }^{\mathrm{tot}}} \leq \Big( \frac{r_p}{ \varrho} \Big)^4 \| G^{(2)}\|_{\mathscr{H}_{\eta_{p}^{(2)},\varrho,p,\mathcal{O}^{(2)} }^{\mathrm{tot}}} \leq r_p^{4- \delta/6 - 2\delta/7 } \leq r_p^{4-\delta/2}
$$
and by monotony
$$
\| L^{(2)}\|_{\mathscr{H}_{\eta_{p}^{(2)},r_p,p,\mathcal{O}^{(2)} }^{\mathrm{tot}}} \leq \| L^{(2)}\|_{\mathscr{H}_{\eta_{p}^{(2)},\varrho,p,\mathcal{O}^{(2)} }^{\mathrm{tot}}}  \leq \rho^6 r_p^{- 3\delta } \leq r_p^{6-7\delta}. 
$$
Finally, we perform a last decomposition by setting
$$
 Z^{(p,2)} := \Pi_{\mathrm{deg} \leq 4} \Pi_{\mathrm{int}}  L^{(2)},  \ J^{(2)}:= \Pi_{\mathrm{deg} \leq 3} (\mathrm{Id}-\Pi_{\mathrm{int}})  L^{(2)}, \ N^{(2)} := G^{(2)} + L^{(2)} -  Z^{(p,2)} - J^{(2)}.
$$
By construction and \eqref{eq:deg_G2}, we have that 
$$
\Pi_{\mathrm{deg}\geq 4}N^{(2)} = N^{(2)}, \quad \Pi_{\mathrm{deg}\leq  3} J^{(2)} = J^{(2)}, \quad \Pi_{\mathrm{int}} J^{(2)} = \Pi_{\mathrm{int}} \Pi_{\mathrm{deg}\leq 4} N^{(2)} = 0
$$
the bounds
$$
\| Z^{(p,2)} \|_{\mathscr{H}_{\eta_{p},r_p,p,\mathcal{O}^{(2)} }^{\mathrm{tot}}} \leq r_p^{6-8\delta} \leq r_p^{4+\frac12}, \quad \| J^{(2)}\|_{\mathscr{H}_{\eta_{p}^{(2)},r_p,p,\mathcal{O}^{(2)} }^{\mathrm{tot}}} \leq r_p^{6-8\delta}, \quad  \| N^{(2)}\|_{\mathscr{H}_{\eta_{p}^{(2)},r_p,p,\mathcal{O}^{(2)} }^{\mathrm{tot}}} \leq r_p^{4-\delta}
$$
and that for all $\xi  \in \mathcal{O}^{(2)}$ and all $u\in \mathcal{A}_\xi(3\varrho)$
\begin{equation}
\label{eq:def_H2}
H^{(1)}(\xi^{(p-1)};u) = Z_{\leq p-1}(\xi;u) +  Z^{(p,2)}(\xi;u) + J^{(2)}(\xi;u) + N^{(2)}(\xi;u) =: H^{(2)}(\xi;u).
\end{equation}
Changing the name of the Hamiltonian is only to highlight that we have performed the translation $T_{p-1}$ and then regrouped the terms in $G^{(2)} $ and $L^{(2)}$. In conclusion with the first three steps  we have passed from $G,L\in \mathscr{H}_{\eta_{p-1},r_{p-1},p-1,\mathcal{O}_{p-1}}$  to $Z^{(p,2)},J^{(2)},N^{(2)} \in \mathscr{H}_{\eta_{p}^{(2)},r_p,p,\mathcal{O}^{(2)}} $.
%\mic{I understand the spirit of this formula but I think some comment is needed} \jo{please}
\subsection{Step $4$: KAM and one last step of Birkhoff normal form} \label{sub:KAM_and_Birk} It remains only to remove the non integrable terms of degree smaller than or equal to $3$ in $H^{(2)}$ that is $J^{(2)}$. For usual degree reasons, terms or degree smaller than or equal to $2$ need a KAM iteration scheme to be removed whereas terms of degree $3$ can be removed by one single step of Birkhoff normal form. The KAM scheme employed here is standard, with its sole novelty lying in the formalism specific to this paper. To streamline the proof and avoid unnecessary complexity, we provide a detailed elimination of the cubic terms. The precise statement of the KAM result we use is deferred to the appendix (Proposition \ref{prop:KAM}). All the intricacies associated with our formalism are already captured in the process of removing the cubic terms. Eliminating the lower-degree terms then follows by iterating this construction infinitely, as is standard in KAM theory. 
\medskip

First, we remove $\Pi_{\deg \leq 2} J^{(2)}$ by applying the KAM theorem given by Proposition \ref{prop:KAM} of the Appendix. More precisely, this proposition provides a Borel set $\mathcal{O}_{p}^{(3)} \subset  \mathcal{O}_{p}^{(2)}$ satisfying
\begin{equation}
\label{eq:meas3}
\mathrm{meas}\Big(   \mathcal{O}_{p}^{(2)} \setminus \mathcal{O}_{p}^{(3)} \Big) \leq r_p^{\delta/8} \mathrm{meas}(\mathcal{O}_{p}^{\mathrm{full}} ),
\end{equation}
 an integrable Hamiltonian $Z^{(p,3)} \in  \Pi_{\mathrm{deg}\leq 4}  \Pi_{\mathrm{int}} \mathscr{H}_{\eta_p,r_{p},p,\mathcal{O}_p^{(3)}}$ of degree smaller than or equal to $4$ satisfying the bound
\begin{equation}
\label{eq:bound_Z3}
\| Z^{(p,3)} \|_{ \mathscr{H}_{\eta_p,r_{p},p,\mathcal{O}_p^{(3)}}^{\mathrm{tot}}} \leq 2 r_{p}^{4+1/2},
\end{equation}
 a Hamiltonian $ J^{(3)} \in \Pi_{\mathrm{deg} = 3} \mathscr{H}_{\eta_{p}^{(3)},r_{p},p,\mathcal{O}_{p}^{(3)}}$,  containing only terms of degree $3$  and no integrable term (i.e. $\Pi_{\mathrm{int}}  J^{(3)} =0 $) satisfying the bound 
\begin{equation}
\label{eq:bound_J3}
\| J^{(3)} \|_{ \mathscr{H}_{\eta_{p}^{(3)},r_{p},p,\mathcal{O}_{p}^{(3)}}^{\mathrm{tot}}}  \leq 2 r_p^{6-8\delta },
\end{equation}
 a Hamiltonian $N^{(3)} \in \Pi_{\mathrm{deg} \geq 4} \mathscr{H}_{\eta_{p}^{(3)},r_{p},p,\mathcal{O}_{p}^{(3)}}$,  containing only terms of degree larger than or equal to $4$ and no quartic integrable term (i.e. $\Pi_{\mathrm{deg} =4} \Pi_{\mathrm{int}} N^{(3)} =0 $), satisfying the bound 
\begin{equation}
\label{eq:bound_R3}
\| N^{(3)} \|_{ \mathscr{H}_{\eta_{p}^{(3)},r_{p},p,\mathcal{O}_{p}^{(3)}}^{\mathrm{tot}}}  \leq 2 r_p^{4-\delta},
\end{equation}
such that, for all $\xi \in \mathcal{O}_{p}^{(3)}$ and all $u\in \mathcal{A}_\xi(\frac{5}2 r_p)$  
\begin{equation}
\label{eq:def_H3}
H^{(2)}(\xi;\Psi_\xi^{(p)}(u)) =  Z_{\leq p-1}(\xi;u) +  Z^{(p,3)}(\xi;u) + J^{(3)}(\xi;u) +  N^{(3)}(\xi;u) =:H^{(3)}(\xi;u)
\end{equation}
where $\Psi_\xi^{(p)}\in C^2( \mathcal{A}_\xi(\frac{5}2 r_p); \mathcal{A}_\xi(3 r_{p}) )$ is a symplectic map which is close to the identity, i.e. for all $ j\in \{0,1,2\}$, $v^{(1)},\cdots,v^{(j)}\in H^s$,
\begin{equation}
\label{eq:Psi_close_id}
 \|\mathrm{d}^j \big(\Psi_\xi^{(p)} - \mathrm{Id}  \big)(u)  (v^{(1)},\cdots,v^{(j)}) \|_{H^s} \leq r_p^{2-j-\frac14} \prod_{1\leq i\leq j} \|v^{(i)}\|_{H^s}.
 \end{equation}

\medskip

Now, it only remains to remove $J^{(3)}$, i.e. the term of degree $3$. Since it is of degree $3$, its largest fourth largest index $\mu_{J^{(3)}}$ vanishes, i.e. $\mu_{J^{(3)}}=0$. Thus it makes sense to apply Proposition \ref{prop:res_coho} to remove it. We get a Borel set $\mathcal{O}_p \subset  \mathcal{O}^{(3)}$, satisfying the measure estimate
\begin{equation}
\label{eq:meas4}
\mathrm{meas}(\mathcal{O}^{(3)} \setminus \mathcal{O}_p ) \leq r_p^{\delta/8} \, \mathrm{meas} (\mathcal{O}_{p}^{\mathrm{full}}),
\end{equation}
and two Hamiltonians $\Gamma \in \Pi_{\mathrm{deg}=3}\mathscr{C}_{\eta_{p}^{(7/2)},r_p,p,\mathcal{O}_p }$ and $Q \in   \Pi_{\mathrm{deg \geq 5}}\mathscr{H}_{\eta_{p}^{(7/2)},r_p,p,\mathcal{O}_p }$ solving the cohomological equation
\begin{equation}
\label{eq:coho2}
\forall \xi \in \mathcal{O}_p , \forall  u\in \mathcal{A}_\xi(3\rho), \quad \{ Z_{\leq p-1}[\xi]  +  Z^{(p,3)}[\xi], \Gamma[\xi] \}(u) = (Q-J^{(3)})(\xi;u) 
\end{equation}
and satisfying the bounds\footnote{we write the estimate in this way to highlight the fact that $Q$ is smaller than $J^{(3)}$. This kind of argument would be convenient to prove Proposition \ref{prop:KAM} (else the exponent associated with the estimate of $J^{(3)}$ could be larger than the one associated with $J^{(2)}$).} 
\begin{equation} \label{eq:bound_QGamma}
\begin{split}
\| \Gamma \|_{\mathscr{C}_{\eta_{p}^{(13/4)},r_p,p,\mathcal{O}_p }^{\mathrm{tot}}} \! \! \! \! \lesssim& (r_p^{\delta/8})^{-2}   (\eta_{p}^{(3)} - \eta_p^{(13/4)})^{- 2 \# \mathcal{S}_p-25} r_{p}^{-2+7\delta/4} \| J^{(3)}\|_{ \mathscr{H}_{\eta_{p}^{(3)},r_p,p,\mathcal{O}^{(3)}}^{\mathrm{tot}}} \! \! \! \! \! \! \\ \leq & \| J^{(3)}\|_{ \mathscr{H}_{\eta_{p}^{(3)},r_p,p,\mathcal{O}^{(3)}}^{\mathrm{tot}}} r_p^{ -2+\frac{5 \delta}4} \leq r_p^{4-7\delta}\\
\| Q \|_{\mathscr{H}_{\eta_{p}^{(7/2)},r_p,p,\mathcal{O}_p }^{\mathrm{tot}}} \! \! \! \! \lesssim& r_p^{2-\delta} (\eta_{p}^{(13/4)} - \eta_p^{(7/2)})^{- 9 }\| \Gamma \|_{\mathscr{C}_{\eta_{p}^{(13/4)},r_p,p,\mathcal{O}_p }^{\mathrm{tot}}}  \! \! \leq  \| J^{(3)}\|_{ \mathscr{H}_{\eta_{p}^{(3)},r_p,p,\mathcal{O}^{(3)}}^{\mathrm{tot}}}  r_p^{ \frac{\delta}5}  \leq    r_p^{6 - 8\delta }.\\
\end{split}
\end{equation} 
Moreover, since $\Pi_{\mathrm{deg}=3} \Gamma = \Gamma$, we have 
$$
\Pi_{\|\boldsymbol{m} \| \geq 2} J^{(3)} =0.
$$

\medskip

Thus, since $r_p^{4-7\delta} \ll r_p^{2+\delta}$, we can apply Lemma \ref{lem:flows} to define the Hamiltonian flow $\Phi_\Gamma$. Moreover, using that $\Gamma$ solves the cohomological equation \eqref{eq:coho2}, we get that for all $\xi \in \mathcal{O}_p$ and all $u\in \mathcal{A}_\xi(2r_p)$
\begin{equation}
\label{eq:mod_H3}
H^{(3)}(\xi; \Phi_{\Gamma[\xi]}^1(u)) =\underbrace{Z_{\leq p-1}(\xi;u) + Z^{(p,3)}(\xi;u)}_{=: Z_{\leq p}^{(3)}(\xi;u)}  + N^{(3)}(\xi;u) + R^\sharp(\xi;u) 
\end{equation}
where
\begin{equation*}
\begin{split}
R^\sharp(\xi;u) =& Z_{\leq p}^{(3)}[\xi]\circ \Phi_{\Gamma[\xi]}^1(u) - Z_{\leq p}^{(3)}(\xi;u)  - \{ Z_{\leq p}^{(3)}[\xi],\Gamma[\xi]\}(u) \\
&+ J^{(3)}(\xi;\Phi_{\Gamma[\xi]}^1(u)) - J^{(3)}(\xi;u) +  N^{(3)}(\xi;\Phi_{\Gamma[\xi]}^1(u)) - N^{(3)}(\xi;u) +  Q(\xi;u)  \\ 
=&  \int_{0}^1 \Big\{  (tJ^{(3)} + (1-t)Q + N^{(3)})[\xi] , \Gamma[\xi]\Big\}\circ \Phi_{\Gamma[\xi]}^t(u) \, \mathrm{d}t +   Q(\xi;u).
\end{split}
\end{equation*}
As a consequence of Propositions \ref{prop:poisson_bracket_estimate} and \ref{prop:comp_ham_flow}, it follows that $R^\sharp\in \mathscr{H}_{\eta_{p},r_p,p,\mathcal{O}_p }$ and, thanks to the bounds \eqref{eq:bound_J3},\eqref{eq:bound_R3} and \eqref{eq:bound_QGamma} on $J^{(3)},N^{(3)},Q,\Gamma$, it satisfies the bound (the estimate is similar to \eqref{eq:bound_QGamma})
\begin{equation}
\label{eq:bound_Rdiese}
\|R^\sharp \|_{ \mathscr{H}_{\eta_{p},r_p,p,\mathcal{O}_p } ^{\mathrm{tot}}} \leq r_p^{6 - 8\delta }.
\end{equation}
Moreover, since $\Pi_{\mathrm{deg}=3} \Gamma = \Gamma$, we have that  
$$
\Pi_{\mathrm{deg}\geq 4} R^\sharp  =R^\sharp .
$$

\subsection{Step $5$: Conclusion of the proof} Now, we set
\begin{equation}
\label{eq:last_dec}
Z^{(p)} = Z^{(p,3)}(\xi;u) + \Pi_{\mathrm{deg}= 4}  \Pi_{\mathrm{int}}R^\sharp \quad \mathrm{and} \quad R^{(p)} = N^{(3)} + (\mathrm{Id} - \Pi_{\mathrm{deg}= 4} \Pi_{\mathrm{int}} ) R^\sharp.
\end{equation}
We deduce of that $R^{(p)}$ only contains terms of degree larger than or equal to $4$ and that by \eqref{eq:bound_Z3},\eqref{eq:bound_R3},  \eqref{eq:bound_Rdiese}, as expected, they satisfy the estimates
$$
\| Z^{(p)} \|_{ \mathscr{H}_{\eta_p,r_{p},p,\mathcal{O}_p}^{\mathrm{tot}}} \leq r_{p}^{4+1/3} \quad \mathrm{and} \quad \| R^{(p)} \|_{ \mathscr{H}_{\eta_{p},r_{p},p,\mathcal{O}_{p}}^{\mathrm{tot}}}  \leq r_p^{4-2\delta}
$$
The measure estimate \eqref{eq:measure_est_loop} of the statement is a direct consequence of \eqref{eq:meas1}, \eqref{eq:meas2}, \eqref{eq:meas3}, \eqref{eq:meas4}. Finally, the change of variable $\tau$ is naturally defined for $\xi \in \mathcal{O}_p$ by
$$
\tau_\xi^{(p)} := \Phi_{\chi[\xi^{(p-1)}]}^1 \circ \Psi_\xi \circ \Phi_{\Gamma[\xi]}^1.
$$
It is symplectic as a composition of symplectic maps. The fact that for all $u\in \mathcal{A}_\xi(r_p)$
$$
H(\xi;\tau_\xi^{(p)}(u)) =  Z_2(u)+Z_4(u) + \sum_{1\leq p'\leq p} Z^{(p')}(\xi^{(p')};u) +  R^{(p)}(\xi;u)
$$
is a direct consequence of the construction (identities \eqref{eq:def_H1}, \eqref{eq:def_H2}, \eqref{eq:def_H3}, \eqref{eq:mod_H3} and \eqref{eq:last_dec} to be precise). Finally, the fact $\tau_\xi^{(p)}$ is close to the identity (i.e. that it satisfies \eqref{eq:tau_close_id}) directly follows of the similar estimate for $\Psi_\xi$ \eqref{eq:Psi_close_id} and the similar estimates for the flows proven in Proposition \ref{lem:flows} applied for $\chi$ and $\Gamma$ using the estimates \eqref{eq:bound_Achi} and \eqref{eq:bound_QGamma} on these Hamiltonians.

%\subsection{Sketch of proof old}
%This is essentially homogeneity arguments. The key point is to perform the operation in the good order.
%
%\begin{itemize}
%\item Step 1: We restrict our Hamiltonian to a much smaller annulus of radius $r_p^{1-2\delta}$. $R^{(p-1)}$ becomes of size $r_p^{4 (1-3\delta) }$.
%\item Step 2: We decompose $R^{(p-1)} = B + G+L$ where $B$ contains all the terms of degree $\leq 5$ which are going to generate terms to be removed after the translation and $L$ all the terms of degree $\geq 6$.
%\item  Step 3: We remove $B$ by one step of BNF : no pb because small divisor are of size $r_{p-1}^2$. The size of $\mathrm{d}^2 \nabla \chi$ is $r_p^{-0^+}$ so it is ok (note that all the terms in $\chi$ are at most linear in $y$)
%\item Step 4: We perform the translation and we restrict our Hamiltonian to a annulus of radius $3r_p$. The terms to be removed are of size $r_p^{6^-}$.
%\item Step 5: We remove these terms by KAM (and one step of BNF for the terms of degree $3$). It is ok because : (small divisor loss $r_p^2$ and cost of $\mathrm{d}^2 \nabla $ is $r_p^4$ (terms to be removed are at most linear in $y$).
%\end{itemize}

\section{Convergence of the changes of variables}
\label{sec:conv}

This section is dedicated to proving the following proposition and its corollary. They will naturally be applied for $\varsigma^{(p)} = \tau_\xi^{(p)} $ where $\tau_\xi^{(p)}$ is given by Proposition \ref{prop:the_loop}.
\begin{proposition} \label{prop:cvg} Let $\xi \in \mathcal{O}_\infty^{\mathrm{full}}$ and for all $p\geq 1$, 
$$
\Psi^{(p)} = \varsigma^{(1)} \circ \cdots \circ \varsigma^{(p)}
$$
where $\varsigma^{(p)}\in \mathcal{C}^2(\mathcal{A}_\xi(\frac12 r_p); \mathcal{A}_\xi(\frac12 r_{p-1}))$  is a symplectic map which is $C^1$ close to the identity and whose second derivative is uniformly not too large, i.e. for all $u\in \mathcal{A}_\xi(\frac12 r_p)$, $ j\in \{0,1,2\}$, $v^{(1)},\cdots,v^{(j)}\in H^s$,
\begin{equation}
\label{eq:bounds_varsi}
 \|\mathrm{d}^j \big(\varsigma^{(p)}  - \mathrm{Id}  \big)(u)  (v^{(1)},\cdots,v^{(j)}) \|_{H^s} \leq r_p^{2-j-\frac13} \prod_{1\leq i\leq j} \|v^{(i)}\|_{H^s}.
\end{equation}
Then there exists $\Psi^{(\infty)} \in C^0(\mathbf{T}_\xi ; H^s)$ and $L^{(\infty)}\in C^0(\mathbf{T}_\xi ; \mathscr{L}(H^s)) $ such that
\begin{itemize}
\item $(\Psi^{(p)})_p$ converges uniformly to $\Psi^{(\infty)}$,
\item $(\mathrm{d}\Psi^{(p)})_p$ converges uniformly to $L^{(\infty)}$,
\item $L^{(\infty)}$ is the differential of $\Psi^{(\infty)} $ in the sense that
\begin{equation}
\label{eq:almost_C1}
\forall u,v \in \mathbf{T}_\xi, \quad \| \Psi^{(\infty)} (u) - \Psi^{(\infty)} (v)  - L^{(\infty)} (v)(u-v)\|_{H^s} \leq \| u-v\|_{H^s}^{3/2}
\end{equation}
\item for all $u\in \mathbf{T}_\xi$, $L^{(\infty)} (u)$ is symplectic and satisfies
\begin{equation}
\label{eq:est_L_m_Id}
 \| L^{(\infty)}(u) - \mathrm{Id} \|_{\mathscr{L}(H^s) } \leq  r_1^{1/2} .
\end{equation}
\end{itemize}

\end{proposition}
\begin{corollary} \label{cor:biLip} In the setting of Proposition \ref{prop:cvg}, the map $\Psi^{(\infty)} $ is bi-Lipschitz and, a fortiori, a homeomorphism on its image.
\end{corollary}
\begin{proof}[Proof of Corollary \ref{cor:biLip}] Let $u,v \in \mathbf{T}_\xi$ and $\theta \in (-\pi,\pi]^{\mathcal{S}_\infty}$ be such that
$$
\forall j\in \mathcal{S}_\infty, \ u_j = e^{\ic  \theta_j} v_j.
$$
We define $\gamma \in C^1([0,1]; \mathbf{T}_\xi)$ by
$$
\forall t\in [0,1],\forall x\in \mathbb{T}, \quad \gamma(t,x) = \sum_{j\in \mathcal{S}_\infty} \sqrt{\xi_j} e^{\ic t \theta_j} v_j e^{\ic j x}.
$$
Then, we note that for all $t\in [0,1]$, 
\begin{equation}
\label{eq:pas_grand_chose}
\| u-v\|_{H^s}^2 = \sum_{j\in \mathcal{S}_\infty} \xi_j \langle j\rangle^{2s} \big| e^{\ic \theta_j} -1 \big|^2 \sim \sum_{j\in \mathcal{S}_\infty} \xi_j \langle j\rangle^{2s}  \theta_j^2 = \| \gamma'(t) \|_{H^s}^2.
\end{equation}
Using \eqref{eq:almost_C1}, we get that $\Psi^{(\infty)} \circ \gamma \in C^1([0,1]; H^s)$ and
$$
\partial_t \Psi^{(\infty)} \circ \gamma = (L^{(\infty)} \circ \gamma) \gamma'.
$$
Thus, we have
\begin{equation*}
\begin{split}
\| \Psi^{(\infty)} (u) - \Psi^{(\infty)} (v) - (u-v) \|_{H^s} &= \Big\| \int_{0}^1  \partial_t (\Psi^{(\infty)}-\mathrm{Id}) \circ \gamma(t) \, \mathrm{d}t \Big\|_{H^s} \\
&\leq  \int_{0}^1  \| L^{(\infty)}(u)\gamma'(t)- \gamma'(t)\|_{H^s} \, \mathrm{d}t
\end{split}
\end{equation*}
Finally using \eqref{eq:est_L_m_Id} and \eqref{eq:pas_grand_chose}, we deduce that
$$
\| \Psi^{(\infty)} (u) - \Psi^{(\infty)} (v) - (u-v) \|_{H^s} \lesssim r_1^{1/2} \|u-v\|_{H^s}.
$$
\end{proof}
 The rest of this section is devoted to the proof of Proposition \ref{prop:cvg}. From now, we consider the maps $ \varsigma^{(p)}$ as given. For convenience, given $1\leq q \leq p $, we set
 $$
 \Psi^{(q,p)} =\varsigma^{(q)} \circ \cdots \circ \varsigma^{(p)}.
 $$
 
 \subsection{Convergence of \texorpdfstring{$\Psi^{(p)}$}{a}} In this subsection, we aim at proving that $\Psi^{(p)}$ is a Cauchy sequence and thus that it converges.
 
\medskip 
 
 We begin with two basic technical lemmas.
 \begin{lemma} \label{lem:lip}
 For all $p\geq 1$, all $u,v \in \mathcal{A}_\xi(\frac12 r_p)$, we have
 $$
  \| \varsigma^{(p)}(u) -\varsigma^{(p)}(v)\|_{H^s}  \leq (1+r_p^{1/4}) \|u-v\|_{H^s} \quad \mathrm{and} \quad \| \Psi^{(p)}(u) - \Psi^{(p)}(v)\|_{H^s}  \leq 2 \|u-v\|_{H^s}.
 $$
 \end{lemma}
\begin{proof}
 Proceeding as in the proof of Corollary \ref{cor:biLip} and using the bound \eqref{eq:bounds_varsi} on $\mathrm{d}\varsigma^{(p)}$, we directly get
 $$
   \| \varsigma^{(p)}(u) -\varsigma^{(p)}(v)\|_{H^s}  \leq (1+(1+\pi)r_p^{2/3}) \|u-v\|_{H^s} \leq  (1+r_p^{1/4}) \|u-v\|_{H^s} .
 $$
 Then, using this estimate, we deduce that
 $$
 \| \Psi^{(p)}(u) - \Psi^{(p)}(v)\|_{H^s}  \leq  \|u-v\|_{H^s} \prod_{1\leq j\leq p} ( 1+r_p^{1/4}) \leq  \|u-v\|_{H^s} \prod_{1\leq j\leq \infty}  (1+r_p^{1/4}) \leq 2 \|u-v\|_{H^s} .
 $$
\end{proof} 
 
 \begin{lemma} \label{lem:pouette}  For all $1\leq q \leq p$ and all $u \in \mathcal{A}_\xi(\frac12 r_p)$, we have
 $$
  \|\Psi^{(q,p)}(u) - u\|_{H^s} \leq r_q^{3/2}.
 $$
 \end{lemma}
 \begin{proof}
 We proceed by induction on $q$. By \eqref{eq:bounds_varsi}, we know that this estimate is true if $q=p$. Then, given $q<p$, we have
 $$
 \|\Psi^{(q,p)}(u) - u\|_{H^s} =  \|\varsigma^{(q)} \circ \Psi^{(q+1,p)}(u) - u\|_{H^s}  \leq \|\varsigma^{(q)} \circ \Psi^{(q+1,p)}(u)  - \varsigma^{(q)}(u)  \|_{H^s} + \|  \varsigma^{(q)}(u)  - u \|_{H^s}.
 $$
 Thus using the induction hypothesis, Lemma \ref{lem:lip} and  \eqref{eq:bounds_varsi}, we deduce that
 $$
 \|\Psi^{(q,p)}(u) - u\|_{H^s} \leq 2 r_{q+1}^{3/2} + r_{q}^{5/3} \leq r_{q}^{3/2}.
 $$
 \end{proof}

Then, we deduce the convergence of $\Psi^{(p)}$.
\begin{lemma} \label{lem:step1} There exists  $\Psi^{(\infty)} \in C^0(\mathbf{T}_\xi ; H^s)$ such that $(\Psi^{(p)})_p$ converges uniformly to $\Psi^{(\infty)}$.
\end{lemma}
\begin{proof} It suffices to prove that $(\Psi^{(p)})_p$ is a Cauchy sequence in $C^0(\mathbf{T}_\xi ; H^s)$. Let $u\in \mathbf{T}_\xi $ and $1\leq q \leq p$. We have
$$
 \Psi^{(p)}(u) - \Psi^{(q)}(u) =  \Psi^{(q)}(\Psi^{(q+1,p)}(u)) - \Psi^{(q)}(u) .
$$
Thus applying Lemmas \ref{lem:lip} and \ref{lem:pouette}, we get
$$
\|  \Psi^{(p)}(u) - \Psi^{(q)}(u) \|_{H^s} \leq 2 r_{q+1}^{3/2} \mathop{\longrightarrow}_{q\to \infty} 0.
$$
\end{proof}

 \subsection{Convergence of \texorpdfstring{$\mathrm{d}\Psi^{(p)}$}{a}} In this subsection, we aim at proving that $\mathrm{d}\Psi^{(p)}$ is a Cauchy sequence and thus that it converges.
 
 \medskip
 
 We begin with H\"older estimates.
 \begin{lemma} \label{lem:hold}
  For all $p\geq 1$, all $u,v \in \mathcal{A}_\xi(\frac12 r_p)$, we have
 $$
  \| \mathrm{d}\varsigma^{(p)}(u) -\mathrm{d}\varsigma^{(p)}(v)\|_{\mathscr{L}(H^s)}  \leq r_p^{1/9} \|u-v\|_{H^s}^{1/2} \quad \mathrm{and}\! \quad \| \mathrm{d}\Psi^{(p)}(u) - \mathrm{d}\Psi^{(p)}(v)\|_{\mathscr{L}(H^s)}  \lesssim r_1^{1/9} \|u-v\|_{H^s}^{1/2}.
 $$
 \end{lemma}
 \begin{proof} Proceeding as in the proof of Corollary \ref{cor:biLip} and using the bound \eqref{eq:bounds_varsi} on $\mathrm{d}^2\varsigma^{(p)}$, we directly get the Lipschitz estimate
 $$
   \| \mathrm{d}\varsigma^{(p)}(u) -\mathrm{d}\varsigma^{(p)}(v)\|_{H^s}  \leq (1+\pi)r_p^{-1/3} \|u-v\|_{H^s}.
 $$
 Then using \eqref{eq:bounds_varsi}, we deduce that
 \begin{equation*}
 \begin{split}
   \| \mathrm{d}\varsigma^{(p)}(u) -\mathrm{d}\varsigma^{(p)}(v)\|_{\mathscr{L}(H^s)} & \leq \sqrt{2} \| \mathrm{d}\varsigma^{(p)}(u) -\mathrm{d}\varsigma^{(p)}(v)\|_{\mathscr{L}( H^s)}^{1/2}\sup_{w\in  \mathcal{A}_\xi(\frac32 r_p)}\| \mathrm{d}\varsigma^{(p)} (w) - \mathrm{Id}\|_{\mathscr{L}(H^s)}^{1/2} \\
   &\leq \sqrt{2(1+\pi)} r_p^{-1/6}  \|u-v\|_{H^s}^{1/2} r_p^{1/3} \leq   r_p^{1/9} \|u-v\|_{H^s}^{1/2}.
 \end{split}
 \end{equation*}
 For the second estimate, first, we note that
 $$
 \mathrm{d}\Psi^{(p)}(u) = \prod_{1\leq j \leq p}^{\rightarrow} \mathrm{d}\varsigma^{(j)} (\Psi^{(j+1,p)}(u))
 $$
 and so
 \begin{equation*}
 \begin{split}
  &\mathrm{d}\Psi^{(p)}(u) - \mathrm{d}\Psi^{(p)}(v) \\
  =& \sum_{q=1}^p \Big(\prod_{j<q}^{\rightarrow} \mathrm{d}\varsigma^{(j)} (\Psi^{(j+1,p)}(u))   \Big)  \Big( \mathrm{d}\varsigma^{(q)} (\Psi^{(q+1,p)}(u)) -\mathrm{d}\varsigma^{(q)} (\Psi^{(q+1,p)}(v))   \Big)  \Big(\prod_{j>q}^{\rightarrow} \mathrm{d}\varsigma^{(j)} (\Psi^{(j+1,p)}(v))   \Big).
 \end{split}
 \end{equation*}
 Thus, using the H\"older estimate on $\mathrm{d}\varsigma^{(q)}$ and the estimate \eqref{eq:bounds_varsi} on $\mathrm{d}\varsigma^{(j)}$, we deduce that
 $$
 \| \mathrm{d}\Psi^{(p)}(u) - \mathrm{d}\Psi^{(p)}(v) \|_{\mathscr{L}(H^s)} \leq \Big( \prod_{j\geq 1} 1+ r_j^{2/3} \Big) \sum_{q=1}^p r_q^{1/9} \|u-v\|_{H^s}^{1/2}  \leq 2 r_1^{1/9} \|u-v\|_{H^s}^{1/2}.
 $$
 \end{proof}

Then, we deduce the convergence of $\mathrm{d}\Psi^{(p)}$.
\begin{lemma} \label{lem:step2} There exists  $L^{(\infty)} \in C^0(\mathbf{T}_\xi ; \mathscr{H^s})$ such that $(\mathrm{d}\Psi^{(p)})_p$ converges uniformly to $L^{(\infty)}$.
\end{lemma}

\begin{proof} It suffices to prove that $(\Psi^{(p)})_p$ is a Cauchy sequence in $C^0(\mathbf{T}_\xi ; \mathscr{L}(H^s))$.  Let $u\in \mathbf{T}_\xi $ and $1\leq q \leq p$. We have
$$
\mathrm{d}\Psi^{(p)}(u) = \mathrm{d}\Psi^{(q)}(\Psi^{(q+1,p)}(u)) \mathrm{d}\Psi^{(q+1,p)}(u).
$$
Thus, we have
$$
\| \mathrm{d}\Psi^{(p)}(u)  - \mathrm{d}\Psi^{(q)}(u) \|_{\mathscr{L}(H^s)} \lesssim  \| \mathrm{d}\Psi^{(q)}(u) - \mathrm{d}\Psi^{(q)}(\Psi^{(q+1,p)}(u)) \|_{\mathscr{L}(H^s)} +\| \mathrm{Id} - \mathrm{d}\Psi^{(q+1,p)}(u) \|_{\mathscr{L}(H^s)}.
$$
We estimate the first term by using Lemmas \ref{lem:hold} and \ref{lem:pouette} and the second term directly by using \eqref{eq:bounds_varsi}. We get
$$
\| \mathrm{d}\Psi^{(p)}(u)  - \mathrm{d}\Psi^{(q)}(u) \|_{\mathscr{L}(H^s)} \lesssim  r_1^{1/9} r_{q+1}^{3/6} + r_{q+1}^{2/3}  \mathop{\longrightarrow}_{q\to \infty} 0.
 $$
 
 \end{proof}
 
 \subsection{Properties of \texorpdfstring{$L^{(\infty)}$}{a}} From now, we consider the maps $\Psi^{(\infty)}$ and $L^{(\infty)}$ given by Lemmas \ref{lem:step1} and \ref{lem:step2}. The fact that for all $u\in \mathbf{T}_\xi$, $L^{(\infty)}(u)$ is symplectic is obvious because $\mathrm{d}\Psi^{(p)}(u)$ is a sequence of symplectic linear maps with converges to $L^{(\infty)}(u)$. So it only remains to prove
\begin{lemma} The estimates \eqref{eq:almost_C1} and \eqref{eq:est_L_m_Id} hold.
\end{lemma}
\begin{proof} The proof of \eqref{eq:est_L_m_Id} is direct. Indeed, as a consequence of \eqref{eq:bounds_varsi}, we have for all $p\geq 1$ and all $u\in \mathbf{T}_\xi$, $ \|  \mathrm{d}\Psi^{(p)}(u) - \mathrm{Id} \|_{\mathscr{L}(H^s) } \leq 2 r_1^{2/3}$. Thus, it suffices to take the limit as $p$ goes to $+\infty$.

\medskip

Now, we focus on \eqref{eq:almost_C1}. 
First, we note that passing to the limit in Lemma \ref{lem:hold}, we have
$$
\forall u,v\in \mathbf{T}_\xi, \ \| L^{(\infty)}(u) - L^{(\infty)}(v)\|_{\mathscr{L}(H^s)} \lesssim r_1^{1/9} \|u-v\|_{H^s}^{1/2}.
$$
Then we consider the same angles $\theta$ and path $\gamma$ as in the proof of Corollary \ref{cor:biLip}.
Thus, using the bound \eqref{eq:pas_grand_chose} on $\gamma'$, we get
\begin{equation*}
\begin{split}
\| \Psi^{(\infty)} (u) - \Psi^{(\infty)} (v)  - L^{(\infty)} (v)(u-v)\|_{H^s} &= \big\| \int_0^1 ( L^{(\infty)}(\gamma(t)) - L^{(\infty)} (v)) \gamma'(t) \, \mathrm{d}t \big\|_{H^s}\\
&\lesssim \sup_{0\leq t \leq 1} \| L^{(\infty)}(\gamma(t)) - L^{(\infty)}(v) \|_{H^s}  \|\gamma'(t)\|_{H^s} \\
&\lesssim r_1^{1/9} \sup_{0\leq t \leq 1} \| \gamma(t) - v \|_{\mathscr{L}(H^s)}^{1/2}  \|\gamma'(t)\|_{H^s} \\ &\lesssim r_1^{1/9} \|u-v\|_{H^s}^{3/2}.
\end{split}
\end{equation*}
\end{proof}

\section{Proof of the main theorem} \label{sec:proof} Now, we are in position to prove the main result of this paper, i.e. Theorem \ref{thm:main}. The proof is mainly a synthesis of what has been done in the rest of the paper. We do not specify the sites family $(\mathcal{S}_p)_{p\geq 1}$: only the choice of $\mathcal{S}_\infty$ matters. Note that the flexibility we have on the choice of this family would be important for some possible refinements of Theorem \ref{thm:main} we discuss in Section \ref{sec:main}.  We set $\varepsilon = C r_1^{1-\delta}$, where $C = (\sum_{j\in \mathcal{S}_1} \langle j \rangle^{2s})^{1/2}$  so that for all $\xi \in \mathcal{O}_\infty^{\mathrm{full}}$, we have $\varepsilon^2/4 \leq \| \xi \|_{\ell^1}^{2s} \leq 2 \varepsilon^2$. Moreover, we recall that by definition, for all $\xi \in \mathcal{O}_\infty^{\mathrm{full}}$, we have $\mathrm{supp} \, \xi = \mathcal{S}_\infty $.

Now, to prove the theorem, the point is to iterate the loop (that is Proposition \ref{prop:the_loop}) and to pass to the limit by Proposition \ref{prop:cvg}. First, for the initialization, we note that $H^{\eqref{eq:NLS}}$ satisfies the assumption of Proposition \ref{prop:the_loop} by Proposition \ref{prop:preparation_NLS}. Indeed, since $\Pi_{\mathrm{deg}\leq 3}R^{(0)}=0$, provided that $r_0$ is small enough, we have
$$
\| R^{(0)} \|_{2 \eta_{\min}+1,r_0,0} \leq (r_0/\rho_0)^4 \| R^{(0)} \|_{4 \eta_{\min},\rho_0,0} \leq r_0^{4-2\delta}.
$$ 
After applying Proposition \ref{prop:the_loop}, we apply it to $H \circ \tau^{(1)}$ and so on, inductively for all $p \geq 1$. Thus, we get Borel set $\mathcal{O}_p$, change of variables $\tau^{(p)}$, Hamiltonians $Z^{(p)},R^{(p)}$ as in Proposition \ref{prop:the_loop}. In order to pass to the limit, we set for $p\geq 0$ (recalling that $\mathcal{O}_0 := \emptyset$)
$$
 \widetilde{\mathcal{O}}_{p} := \mathcal{O}_p \times \prod_{p'>p} \big(\frac{r_{p'}^{2(1-\delta)}}2,  r_{p'}^{2(1-\delta)}\big)^{\mathcal{S}_{p'} \setminus \mathcal{S}_{p'-1}}  \subset \mathcal{O}_\infty^{\mathrm{full}} \quad \mathrm{and} \quad
\mathcal{O}_\infty := \bigcap_{p\geq 0}  \widetilde{\mathcal{O}}_{p}.
$$
They are Borel sets for the product topology. We denote by $\mathbb{P}$ the natural product probability measure on $\mathcal{O}_\infty^{\mathrm{full}}$. Since by Proposition \ref{prop:the_loop}, the sequences $(\widetilde{\mathcal{O}}_{p} )_{p\geq 1}$ is decaying, we have
\begin{equation*}
\begin{split}
\mathbb{P}(\mathcal{O}_\infty^{\mathrm{full}} \setminus \mathcal{O}_\infty) = \sum_{p\geq 1} \mathbb{P}( \widetilde{\mathcal{O}}_{p} \setminus  \widetilde{\mathcal{O}}_{p-1})  &= \sum_{p\geq 1} \frac{\mathrm{meas}\Big( \mathcal{O}_{p-1} \times \big(\frac{r_{p}^{2(1-\delta)}}2,  r_{p}^{2(1-\delta)}\big)^{\mathcal{S}_{p} \setminus \mathcal{S}_{p-1}} \setminus \mathcal{O}_{p} \Big)}{ \mathrm{meas}(\mathcal{O}_{p}^{\mathrm{full}} )} \\
&\leq  \sum_{p\geq 1} r_p^{\delta/9} \leq 2 r_1^{\delta/9}.
\end{split}
\end{equation*}
A fortiori, it implies that $\mathcal{O}_\infty\neq \emptyset$. Now, let $\xi \in \mathcal{O}_\infty$ and set $\xi^{(p)} := (\xi_j)_{j\in \mathcal{S}_p}$. Then, we set 
$$
\Psi^{(p)} = \tau_{\xi^{(1)}}^{(1)} \circ \cdots \circ \tau_{\xi^{(p)}}^{(p)}.
$$
Then applying Proposition \ref{prop:cvg}, we know that it converges uniformly on $\mathbf{T}_\xi $ to a map $\Psi^{(\infty)}=:\Psi$.  Denoting by $\mathcal{T}\subset H^s$ its image, we know by Corollary \ref{cor:biLip} that $\Psi^{(\infty)} : \mathbf{T}_\xi \to \mathcal{T}$ is a homeomorphism. The fact that it is a symplectic deformation directly results of Proposition \ref{prop:cvg} (it suffices to invoke a Neumann series argument to get the invertibility of $L^{(\infty)}$). The uniform bound on $\Psi-\mathrm{Id}$ of Theorem \ref{thm:main} is given by Lemma \ref{lem:pouette}.

Then, since $\Psi^{(p)}$ is symplectic and since $R^{(p)}$ does not contain terms of degree smaller than $2$, we have that for all $p\geq 1$
\begin{center}
$\forall v\in \mathbf{T}_{\xi},\quad  U_v^{(p)}(t) := \displaystyle t\mapsto \Psi^{(p)}( \underbrace{\sum_{j\in \mathcal{S}_p} v_j e^{-\ic \omega_j^{(p)} t} e^{ijx}}_{=:v^{(p)}(t)} )$ is a global solution of \eqref{eq:NLS}
\end{center} 
where, $\omega^{(p)}$ are the frequencies $\omega^{\mathrm{full}}$ of Lemma \ref{lem:sd_estimates}, that is, using the frequency decomposition of Lemma \ref{lem:trans_Z4} and Lemma \ref{lem:the_key}
$$
\forall j\in \mathbb{Z}, \quad \omega^{(p)}_j = j^2 +\omega_j^{\mathrm{lead}}(\xi^{(p)})  +\sum_{1\leq p' \leq p} \omega_j^{(p',p)}(\xi^{(p)}).
$$
The fact that $|\omega^{(p)}_j - j^2|\leq \varepsilon^{1/4}$ is, on the one hand, a direct consequence of the definition of $\omega^{\mathrm{lead}}$ (see Lemma \ref{lem:trans_Z4}) and, on the other hand, a direct consequence of the estimates on $\omega_j^{(p',p)}$ and $Z^{(p')}$  given by Lemma \ref{lem:the_key} and Proposition \ref{prop:the_loop} respectively.

If follows that, as a consequence of Tychonoff's theorem, the sequence $\omega^{(p)}$ converge pointwize, up to extraction of a subsequence\footnote{this compactness argument is just a shortcut, we could prove as in \cite{BGR24} that it converges.}, to a sequence $\omega\equiv \omega^{(\infty)} \in \mathbb{R}^{\mathbb{Z}}$ satisfying the same estimate.

\medskip

To conclude, it only remains to prove that $U_v^{(\infty)}$ is a global solution to \eqref{eq:NLS}. Since, for all $p\geq 1$,  $U_v^{(p)}$ is a global solution to \eqref{eq:NLS}, it suffices to prove that $U_v^{(p)}$ converges locally uniformly to $U_v^{(\infty)}$. 
Indeed, for any $t\in \mathbb{R}$, using the Lipschitz bound on $\Psi^{(p)}$ given by Lemma \ref{lem:lip}, we have 
\begin{equation*}
\begin{split}
\|U_v^{(p)}(t) - U_v^{(\infty)}(t) \|_{H^s} &= \| \Psi^{(p)}(v^{(p)}(t)) -\Psi^{(\infty)}(v^{(\infty)}(t))  \|_{H^s} \\
&\leq \| \Psi^{(p)}(v^{(p)}(t)) - \Psi^{(p)}(v^{(\infty)}(t)) \|_{H^s} + \| \Psi^{(p)}(v^{(\infty)}(t)) - \Psi^{(\infty)}(v^{(\infty)}(t))  \|_{H^s}\\
&\leq 2 \|v^{(p)}(t) - v^{(\infty)}(t)\|_{H^s} + \sup_{z\in \mathbf{T}_\xi}  \| \Psi^{(p)}(z) - \Psi^{(\infty)}(z)\|_{H^s} .
\end{split}
\end{equation*}
 The right hand side last term converges to $0$ by Proposition \ref{prop:cvg}. On the other hand, since $|\omega_j^{(p)} - \omega_j^{(\infty)}|\leq 2 \,\varepsilon^{1/4}$, applying the dominated convergence theorem, we have
$$
\|v^{(p)}(t) - v^{(\infty)}(t)\|_{H^s}  \leq   \big( \sum_{j\in \mathcal{S}_\infty \setminus \mathcal{S}_p} \xi_j \langle j \rangle^{2s} \big)^{1/2} +  \big( \sum_{j\in \mathcal{S}_p} \xi_j  |t||\omega_j^{(p)} - \omega_j^{(\infty)}| \langle j \rangle^{2s} \big)^{1/2} \mathop{\longrightarrow}_{p\to \infty} 0.
$$

\appendix

\addcontentsline{toc}{section}{Appendix}

\addtocontents{toc}{\protect\setcounter{tocdepth}{-1}} 

\section{Poisson brackets}
The aim of this Appendix is to prove Proposition \ref{prop:poisson_bracket_estimate}.
First, we prove Poisson algebra properties with respect to Poisson brackets for $\mathscr{F}_{\eta,r,p}$ and prove Proposition \ref{prop:poisson_bracket_estimate}. 
\subsection{Basic properties of Poisson brackets}
Let us start by introducing some notation, which we shall use only here to shorten formulas.

%\mic{here of course I could write everything with the notation $(\bl,\bs)$, probabily should...i look a bit at the proof, also I think it is inevitable to use $\mathcal S_p$}
\begin{definition}[Sets of indices] Given $q,p\ge 0$, 
		we denote $\mathbb I = \cup_{q\ge 2} \mathbb I_q$ where $\mathbb I_q:=(\mathcal S_p^c)^q  \times \{-1,1\}^{q}$, similarly we denote $\mathbb J= \Z^{\mathcal S_p}\times \N^{\mathcal S_p}$. Given $\bh\in \bI$, we denote by $q(\boldsymbol{h})$ the cardinality of $\boldsymbol{h}$. 
		Given $\boldsymbol{h}=  (\boldsymbol{\ell}, \boldsymbol{\sigma})\in \bI_q$, with $\bl\in (\mathcal S_p^c)^q $ and $\bs \in  \times \{-1,1\}^{q}$ we denote 
		\begin{itemize}
%			\item  %$\bl:= (\bl_i)_{i=1}^q$, $\bs= (\bs_i)_{i=1}^q$ and 
%			 $\bar{\boldsymbol{h}} :=  (\boldsymbol{\ell}, -\boldsymbol{\sigma})$, \jo{Do you use it ?} \mic{probably not!}
			\item  $ \|\boldsymbol{h}\|_{\ell^\infty}:= \langle \|\boldsymbol{\ell}\|_{\ell^\infty} \rangle$, 
			\item   $\boldsymbol h \prec i\in \mathbb Z $ if $ \boldsymbol{\ell}_1^*\prec i$.
		\end{itemize}

 For compactness of notation, we denote $\mathbb M= \cup_{q\ge 0} \mathbb M_q$ where
$$
\begin{aligned}
	\mathbb M_{q}:= \Big\{ (\boldsymbol{k} ,\boldsymbol{m} ,(\boldsymbol{\ell},\boldsymbol\sigma) )&\in \mathbb J \times  \mathbb I_q \ | \	
	&\sum_{j\in \mathcal{S}_p} \boldsymbol{k}_j + \sum_{1\leq i \leq q} \boldsymbol{\sigma}_i  =0,
	\quad 
	\displaystyle \sum_{j\in \mathcal{S}_p} j\boldsymbol{k}_j + \sum_{1\leq i \leq q} \boldsymbol{\sigma}_i \boldsymbol{\ell}_i=0\Big\}.
\end{aligned}
$$ 
\end{definition}

\begin{definition}[Cuts]\label{bohh}
	Given $(\bk,\bm,\bh)\in \mathbb M$ we say that  two families of indices 
	$
	(\bk^{(1)},\bm^{(1)},\bh^{(1)})$, $(\bk^{(2)},\bm^{(2)},\bh^{(2)})$ belonging to $ \mathbb J \times  \mathbb I
	$
	form a cut of $(\bk,\bm,\bh)$ if  %(\mic{up to a permutation?}) 
	$$
	(\bk,\bm,\bh)= (\bk^{(1)}+\bk^{(2)},\bm^{(1)}+\bm^{(2)},(\bh^{(1)};\bh^{(2)}))
	$$ 
	and moreover one of the two following  conditions holds
	\begin{itemize}
		\item  The elements $(\bk^{(1)},\bm^{(1)},\bh^{(1)}),(\bk^{(2)},\bm^{(2)},\bh^{(2)})$ belong to $\mathbb M$. We call this an $A$-cut and denote the set of all $A$-cuts as 
		$C_A(\bk,\bm,\bh)$.
		\item setting 
	\begin{equation}
			\label{taglio}
		\begin{aligned}
			&	\widehat \sigma := -(\sum_{j\in \mathcal{S}_p} \boldsymbol{k}_j^{(1)} + \sum_{1\leq i \leq q^{(1)}} \boldsymbol{\sigma}^{(1)}_i  ) = \sum_{j\in \mathcal{S}_p} \boldsymbol{k}_j^{(2)} + \sum_{1\leq i \leq q^{(2)}} \boldsymbol{\sigma}^{(2)}_i  \,,
			\\ &
			\widehat\jmath:= -\widehat\sigma (\displaystyle \sum_{j\in \mathcal{S}_p} j\boldsymbol{k}^{(1)}_j + \sum_{1\leq i \leq q^{(1)}} \boldsymbol{\sigma}^{(1)}_i \boldsymbol{\ell}^{(1)}_i) =
			\widehat\sigma (\displaystyle \sum_{j\in \mathcal{S}_p} j\boldsymbol{k}^{(2)}_j + \sum_{1\leq i \leq q^{(2)}} \boldsymbol{\sigma}^{(2)}_i \boldsymbol{\ell}^{(2)}_i)
		\end{aligned}
	\end{equation}
		we have that $(\widehat\jmath,\widehat\sigma)\in \mathcal S_p^c \times \{\pm 1\}$.  This implies that, setting $\widehat\bh^{(i)}= (\bh^{(i)};(\widehat\jmath,(-1)^{i+1}\widehat\sigma))$, the elements $(\bk^{(1)},\bm^{(1)},\widehat\bh^{(1)}),(\bk^{(2)},\bm^{(2)},\widehat\bh^{(2)})$ belong to $\mathbb M$. 
		We call this a $v$-cut and denote the set of all $v$-cuts as $C_v(\bk,\bm,\bh)$.
	\end{itemize}
\end{definition}

Next we define formal Poisson brackets on formal power series, these coincide with the usual ones when the series is uniformly convergent.

\begin{lemma} \label{lem:poisson_F}Given $p\ge 0$, $r > 0$, $\eta \geq 0$, $P,Q\in {\mathscr{F}_{\eta,r,p}^{\cste}} $, set
	\[
	\{P,Q\} :=  \sum_{\substack{ \boldsymbol{k},\boldsymbol{k}' \in \mathbb{Z}^{\mathcal{S}_p}\\ \boldsymbol{m},\boldsymbol{m}' \in \mathbb{N}^{\mathcal{S}_p}\\ q,q'\in \mathbb{N}} }    \sum_{\substack {\boldsymbol{\ell} \in (\mathcal{S}_p^c)^{q}\\  \boldsymbol{\sigma} \in  \{-1,1\}^{q} } }\sum_{\substack {\boldsymbol{\ell}' \in (\mathcal{S}_p^c)^{q'}\\  \boldsymbol{\sigma}' \in  \{-1,1\}^{q'} } } P_{\boldsymbol{k},\boldsymbol{m}}^{\boldsymbol{\ell},\boldsymbol{\sigma}} Q_{\boldsymbol{k}',\boldsymbol{m}'}^{\boldsymbol{\ell}',\boldsymbol{\sigma}'}  \,   \{v_{\boldsymbol{\ell}}^{\boldsymbol{\sigma}} y^{\boldsymbol{m}} \, z^{ \boldsymbol{k} },v_{\boldsymbol{\ell}'}^{\boldsymbol{\sigma}'}  y^{\boldsymbol{m}'} \, z^{ \boldsymbol{k}' } \}
	\]
	then  for all $0<\eta'<\eta$ one has 
	\[
	\|\{P,Q\}\|_{\mathscr{F}_{\eta',r,p}^{\cste}}  \lesssim \frac{e^{2\eta'}}{r^2(\eta-\eta')}\|P\|_{\mathscr{F}_{\eta,r,p}^{\cste}} \|Q\|_{\mathscr{F}_{\eta,r,p}^{\cste}} \,.
	\]
\end{lemma}
\begin{proof}
	%	For the norm $| \cdot|_{r,\eta}$ the most elegant proof is to show that setting $H = dP[X_Q]$ one has that $\mathtt H \prec d\mathtt P[\mathcal X_\mathtt Q]$, (which implies that the same holds for the vector fields) and then use standard Cauchy estimates on finite dimensional analytic functions. 
	%	
	%	I think a similar proof would work also for $\| \cdot\|_{r,\eta}$ but maybe a direct proof through \eqref{norm1}, is simpler.
	%	
By definition
	\begin{equation}
		\label{poisson1}
		\begin{aligned}
			\{P,Q\}
			& = \! \! \! \! \! \! \! \sum_{\substack {(\boldsymbol{k}^{(1)} ,\boldsymbol{m}^{(1)} ,\boldsymbol{h}^{(1)}),(\boldsymbol{k}^{(2)} ,\boldsymbol{m}^{(2)} ,\boldsymbol{h}^{(2)})  \in \mathbb M \\  \bh = (\bh^{(1)};\bh^{(2)})} }   %\sum_{{(\boldsymbol{k}^{(2)}, \boldsymbol{m}^{(2)}, , \boldsymbol{h}^{(2)},) \in \mathbb M} }    
			P_{\boldsymbol{k}^{(1)},\boldsymbol{m}^{(1)}}^{\boldsymbol{h}^{(1)}} Q_{ \boldsymbol{k}^{(2)},\boldsymbol{m}^{(2)}}^{\boldsymbol{h}^{(2)}}  \,  \{y^{\boldsymbol{m}^{(1)}} z^{\bk^{(1)}} ,y^{\boldsymbol{m}^{(2)}} z^{\bk^{(2)}}\} v_{\boldsymbol{h}} \ =: K \\
			&	+  \sum_{\substack {(\boldsymbol{k}^{(1)} ,\boldsymbol{m}^{(1)} ,\boldsymbol{h}^{(1)}),(\boldsymbol{k}^{(2)} ,\boldsymbol{m}^{(2)} ,\boldsymbol{h}^{(2)})  \in \mathbb M \\  \bm =\boldsymbol{m}^{(1)}+\bm^{(2)}\\ \bk = \bk^{(1)}+\bk^{(2)}} }  P_{ \boldsymbol{k}^{(1)},\boldsymbol{m}^{(1)}}^{\boldsymbol{h}^{(1)}} Q_{ \boldsymbol{k}^{(2)},\boldsymbol{m}^{(2)}}^{\boldsymbol{h}^{(2)}}  \,  \{v_{\boldsymbol{h}^{(1)}} ,v_{\boldsymbol{h}^{(2)}}\} y^{\boldsymbol{m}} z^{\boldsymbol{k}} \qquad \quad =:L
		\end{aligned}
	\end{equation}
	Setting as usual $q$ to be the cardinality of $\bh$, we have
	\begin{equation}\label{KK}
	\sharp(\mathfrak{S}_{q})K_{\bk,\bm}^{\bh}=
		2\ic\sum_{\varphi\in\mathfrak{S}_{q}}\tilde K_{\bk,\bm}^{\varphi \bh}
		\end{equation}
	\begin{equation}\label{estim:poi1}
	\begin{split}
		\tilde  K_{\bk,\bm}^{ \bh} \! \!  &= \! \!  \! \!  \! \! \! \!  \sum_{\substack{ C_A(\bk,\bm,\bh) \\
				j\in \mathcal S_p} } \! \! \! \!   \boldsymbol{k}^{(1)}_j(\boldsymbol{m}^{(2)}_j+1)P_{\bk^{(1)},\bm^{(1)}}^{ \bh^{(1)}}Q_{\bk^{(2)},\bm^{(2)}+e_j}^{ \bh^{(2)}}\! \!  \! \! - \boldsymbol{k}^{(2)}_j(\boldsymbol{m}^{(1)}_j+1)P_{\bk^{(1)},\bm^{(1)}+e_j}^{ \bh^{(1)}}Q_{\bk^{(2)},\bm^{(2)}}^{ \bh^{(2)}} \\
				& =: \sum_{\substack{ C_A(\bk,\bm,\bh) } }  \bili(P^{ \bh^{(1)}},Q^{ \bh^{(2)}}).
				\end{split}
	\end{equation}
	where    $(e_j)_{j\in \mathcal S_p}$ denotes the canonical basis of $\mathbb R^{\mathcal S_p}$. We introduced the notation $\bili$ for future use in the proof of Lemma \ref{poisson_base}. Note that, of course, $\bili$ depends on $\bk^{(1)},\bm^{(1)},\bk^{(2)},\bm^{(2)}$ but this dependency is always the same, so to highlight more important features we omit it.  Note that $|K_{\bk,\bm}^{ \bh}| \le \max_{\varphi\in \mathfrak{S}_q}|\tilde K_{\bk,\bm}^{\varphi \bh}|$ .
		Moreover, if $P,Q$ are homogeneous in $v$ of degree resp. $q^{(1)},q^{(2)}$ then $K$ is homogeneous of degree $q^{(1)}+ q^{(2)}$  and, given $\bh$ of cardinality $q^{(1)}+ q^{(2)}$, there is at most one non-trivial $A$-cut of $\bh$. 
	
	Similarly 
	\begin{equation}\label{LL}
	\sharp(\mathfrak{S}_{q})L_{\bk,\bm}^{\bh}=2\ic 		\sum_{\varphi\in\mathfrak{S}_{q}}\tilde L_{\bk,\bm}^{\varphi \bh}\end{equation}
	\begin{align}\label{pan}
		\tilde  L_{\bk,\bm}^{ \bh} =\sum_{C_v(\bk,\bm,\bh)} \widehat\sigma  (q^{(1)}+1)( q^{(2)}+1) P_{\bk^{(1)},\bm^{(1)}}^{ \widehat\bh^{(1)}}Q_{\bk^{(2)},\bm^{(2)}}^{ \widehat\bh^{(2)}}=: \sum_{C_v(\bk,\bm,\bh)} \mathcal L(P^{ \widehat\bh^{(1)}},Q^{ \widehat\bh^{(2)}}).
	\end{align}
	where $\widehat\bh^{(1)}= (\bh^{(1)};\widehat\jmath,\widehat\sigma)$ and $\widehat\bh^{(2)}={ (\bh^{(2)};\widehat\jmath,-\widehat\sigma)}$ as seen in Definition \ref{bohh} and $q^{(i)}$ is the cardinality of $\bh^{(i)}$. As before  $|L_{\bk,\bm}^{\bh}| \le \max_{\varphi\in \mathfrak{S}_q}|\tilde L_{\bk,\bm}^{\varphi \bh}|$. Similarly to the previous case, if $P,Q$ are homogeneous in $v$ of degree resp. $q^{(1)}+1,q^{(2)}+1$ then  $L$ is homogeneous of degree $q^{(1)}+ q^{(2)}$  and, given $\bh$ of cardinality $q^{(1)}+ q^{(2)}$ there is at most one non-trivial $v$-cut of $\bh$. 	By the argument above, setting
	%where $c^{(1)}$ is the cardinality of $\widehat\jmath,\widehat\sigma$ in $\bh^{(1)}$ and  $c^{(2)}$ is the cardinality of $\widehat\jmath,-\widehat\sigma$ in $\bh^{(2)}$
	\[
		\mathcal P_{\bk,\bm,q}:= \sup_{\bh\in \mathbb I_q}  \!|P^{\bh}_{{\boldsymbol{k}},{\boldsymbol{m}}}|\,,\quad 	\mathcal Q_{\bk,\bm,q}:= \sup_{\bh\in \mathbb I_q}  \!|Q^{\bh}_{{\boldsymbol{k}},{\boldsymbol{m}}}|\,,\quad 	\mathcal K_{\bk,\bm,q}:= \sup_{\bh\in \mathbb I_q}  \!|K^{\bh}_{{\boldsymbol{k}},{\boldsymbol{m}}}|\,,\quad \mathcal L_{\bk,\bm,q}:= \sup_{\bh\in \mathbb I_q}  \!|L^{\bh}_{{\boldsymbol{k}},{\boldsymbol{m}}}|
	\]
	 we have that
	\begin{equation}
		\label{stimabase1}
		\begin{aligned}
		\mathcal K_{\bk,\bm,q}
			\le 2 \!\!\! \sum_{\substack{q^{(1)}+q^{(2)}= q \\ \bk^{(1)}+\bk^{(2)}=\bk \\ \bm^{(1)}+\bm^{(2)}=\bm }}\!\!\!\sum_{j\in \mathcal S_p} & |\boldsymbol{k}^{(1)}_j|(\|\boldsymbol{m}^{(2)}\|_{\ell_\infty}+1)\mathcal P_{\bk^{(1)},\bm^{(1)},q^{(1)}}\mathcal Q_{\bk^{(2)},\bm^{(2)}+e_j,q^{(2)}}+ \\  &|\boldsymbol{k}^{(2)}_j|(\|\boldsymbol{m}^{(1)}\|_{\ell_\infty}+1)\mathcal P_{\bk^{(1)},\bm^{(1)}+e_j,q^{(1)}}\mathcal Q_{\bk^{(2)},\bm^{(2)},q^{(2)}}
		\end{aligned}
	\end{equation}
	\begin{equation}
		\label{stimabase1bis}
		\begin{aligned}
			\mathcal L_{\bk,\bm,q} \le 2 \!\!\! \sum_{\substack{q^{(1)}+q^{(2)}= q+2 \\ \bk^{(1)}+\bk^{(2)}=\bk \\ \bm^{(1)}+\bm^{(2)}=\bm }} & q^{(1)}q^{(2)}\mathcal P_{\bk^{(1)},\bm^{(1)},q^{(1)}}\mathcal Q_{\bk^{(2)},\bm^{(2)},q^{(2)}}\,.
		\end{aligned}
	\end{equation}	
	As a consequence, noticing that
	\begin{equation*}
	\begin{split}
	\|\{P,Q\}\|_{\mathscr{F}_{\eta',r,p}^{\cste}} &\leq \sum_{\substack{(\bk,\bm)\in\mathbb J\\ q\in \N}} \!\! \max(1,q ,\|\bk\|_{\ell^1} , \|\bm\|_{\ell^{\infty}})(\mathcal K_{\bk,\bm,q} +\mathcal L_{\bk,\bm,q}  ) r^{q+2\|\bm\|_{\ell^1}} e^{\eta'(\|\bk\|_{\ell^1} +q+2\|\bm\|_{\ell^1})}\\
	&\lesssim (\eta-\eta')^{-1} \sum_{\substack{(\bk,\bm)\in\mathbb J\\ q\in \N}} \!\! (\mathcal K_{\bk,\bm,q} +\mathcal L_{\bk,\bm,q}  ) r^{q+2\|\bm\|_{\ell^1}} e^{\eta(\|\bk\|_{\ell^1} +q+2\|\bm\|_{\ell^1})},
	\end{split}
	\end{equation*}
	we deduce of this convolution structure that $\|\{P,Q\}\|_{\mathscr{F}_{\eta',r,p}^{\cste}} \lesssim \frac{e^{2\eta'}}{r^2(\eta-\eta')}\|P\|_{\mathscr{F}_{\eta,r,p}^{\cste}} \|Q\|_{\mathscr{F}_{\eta,r,p}^{\cste}}$.
	
\end{proof}

\subsection{Proof of Proposition \ref{prop:poisson_bracket_estimate}.} The key step in proving Proposition \ref{prop:poisson_bracket_estimate} is the following result.

\begin{lemma}\label{poisson_base}
	For any  $\alpha,r>0$, $\eta,p\ge 0$  and for any  $P\in \mathscr{H}_{\eta+\alpha,r,p}$ and  $Q\in \mathscr{C}_{\eta+\alpha,r,p}$  one has
	\begin{equation}\label{estim:poisson2hom} 
		\begin{aligned}
			&	\| \{P,Q\}\|_{\mathscr{H}_{\eta,r,p}}  \lesssim_{p,\eta}\ \frac{1}{r^2 \alpha}\| P\|_{\mathscr{H}_{\eta+\alpha,r,p}}\| Q\|_{\mathscr{C}_{\eta+\alpha,r,p}}
		\end{aligned}
	\end{equation}
\end{lemma}

\begin{proof}
		Recalling formula \eqref{poisson1} consider separately the terms $K,L$. 
		
		\noindent \underline{\emph{Step $1$ : Estimate on $K$.}} First, we are going to establish a decomposition of $K$ of the form $K= K_0+K_1+K_2$ as in \eqref{eq:cond2}. Then, we will get explicit formulas for $K_0,K_1$. Finally, we will use these formulas to estimate the $\mathscr{H}$ norm of $K$.
		
		\noindent \underline{\emph{Step $1.1.$ : Decomposition of $K$.}} A direct analysis shows that if the decomposition exist we would have
		$$
		K_0= K- \widehat K_1 \quad \text{and} \quad K_1 =\widehat K_1- K_2
		$$
		where $\widehat K_1 \in \mathscr{T}_{1,\eta,r,p}$ and $K_2\in \mathscr{T}_{2,\eta,r,p}$ are defined through their coefficients by
		\begin{equation}
		\label{eq:def_hat_K1_and_K2}
		\widehat K_{1,\boldsymbol{k},\boldsymbol{m}}^{\boldsymbol{h}; \boldsymbol{j}_1,\pm 1} = \lim_{|\bi_1|\to  \infty}   K_{\boldsymbol{k},\boldsymbol{m}}^{\boldsymbol{h};\bi_1,\pm 1}
		\quad \mathrm{and} \quad
		K_{2,\boldsymbol{k},\boldsymbol{m}}^{\boldsymbol{h}; \boldsymbol{j}_2,\pm 1;\boldsymbol{j}_1,\pm 1} =  \lim_{|\bi_1|\to  \infty} \lim_{|\bi_2|\to  \infty}   K_{\boldsymbol{k},\boldsymbol{m}}^{\boldsymbol{h}; \bi_2,\pm 1;\bi_1,\pm 1}
		\end{equation}
		for $\bh \prec \boldsymbol{j}_2 \prec \boldsymbol{j}_1$. Conversely, to prove the existence of the decomposition, it suffices to prove that the limits exist and to define the Hamiltonians in this way. Note that the fact that $K_0 \in \mathscr{S}_{1,\eta,r,p}$ and $K_1 \in \mathscr{S}_{2,\eta,r,p}$ will follow directly of their definition.

By consistency, it shall be convenient to define  $\widehat P_1:=  P_1 +P_2\in \mathscr{T}_{1,\eta,r,p}^{\cste} $ so that for any $\bk,\bm$ and $\bh\prec \tilde\imath$,  one has
	\[
	\widehat P_{1,\bk,\bm}^{(\bh;\tilde\imath,\pm 1)} = \lim_{|i|\to \infty}   P_{\bk,\bm}^{(\bh;i,\pm 1)} \,, \quad Q_{1,\bk,\bm}^{(\bh;\tilde\imath,\pm 1)} = \lim_{|i|\to \infty}   Q_{\bk,\bm}^{(\bh;i,\pm 1)} \,.
	\]
%	\jo{I tried to introduce the following notation to shorten a lot the formulas. What do you think ? It is just to factorize things that are just copied and past each line. Maybe we could optimize more because we always use $ \bili_{\bk^{(1)},\bm^{(1)}}^{\bk^{(2)},\bm^{(2)}}$ : the indices do not vary. Maybe we could remove them for readability but insist on the fact that the dependency is implicit... Or be fully rigorous by fixing $\bk^{(1)},\bm^{(1)},\bk^{(2)},\bm^{(2)}$ but it would be an big change.}
%	In view of formula \eqref{estim:poi1}\footnote{indeed, with this notation \eqref{estim:poi1} writes $	\tilde  K_{\bk,\bm}^{ \bh} =\sum_{ C_A(\bk,\bm,\bh)  } \bili_{\bk^{(1)},\bm^{(1)}}^{\bk^{(2)},\bm^{(2)}}(P^{\bh^{(1)}},Q^{\bh^{(2)}}) $ }, to shorten the notation and highlight the important objects, given two formal Hamiltonian $A,B$, we set
%	$$
%	\bili_{\bk^{(1)},\bm^{(1)}}^{\bk^{(2)},\bm^{(2)}}(A^{\bh^{(1)}},B^{\bh^{(2)}}) := \hspace{-0.3cm}\sum_{\substack{j\in \mathcal{S}_p\\ \check{\boldsymbol{m}}^{(1)} = \boldsymbol{m}^{(1)} + e_j   \\ \check{\boldsymbol{m}}^{(2)} = \boldsymbol{m}^{(2)} + e_j  }}   \boldsymbol{k}^{(1)}_j\check{\boldsymbol{m}}^{(2)}_j A_{\bk^{(1)},\bm^{(1)}}^{ \bh^{(1)}}B_{\bk^{(2)},\check{\boldsymbol{m}}^{(2)}}^{ \bh^{(2)}}\! \!   - \boldsymbol{k}^{(2)}_j\check{\boldsymbol{m}}^{(1)}_j A_{\bk^{(1)},\check{\boldsymbol{m}}^{(1)}}^{ \bh^{(1)}}B_{\bk^{(2)},\bm^{(2)}}^{ \bh^{(2)}}
%	$$
	
\medskip		

Let us  denote\footnote{This is just the number of ways one  can insert  the elements $q+1,\dots, q+n$ in the ordered list $(1,\dots,q)$, %\mic{ hence this set is obtained as composition of the permutations of the elements $q,\dots, n+q$ with  $(n,q)$-shuffles.}\jo{too early to make sense, no ?} 
			Note that $\mathfrak B_{0,n}\:= \mathfrak S_n $, $\mathfrak B_{q,0}\:= {\rm Id}$ and that $\sharp \mathfrak B_{q,n} = \frac{(q+n)!}{q!}$. }
		\[
		\mathfrak B_{q,n}\:= \{\varphi\in \mathfrak S_{q+n}\,:\quad \varphi(i) <\varphi(i+1)\,,\;\forall i=1,\dots,q-1\}.
		\]
		
In order to prove that $\widehat{K}_1$ is well defined, it is convenient to note that, by \eqref{KK}, for $\bh\prec \bi_1$ 
\begin{equation}
\label{eq:average2}
	 K ^{(\bh;\bi_1,\pm 1)}_{\bk,\bm}= \frac{2\ic }{\sharp(\mathfrak S_{q+2})}\sum_{\vartheta\in \mathfrak S_q} \sum_{\varphi \in \mathfrak B_{q,2}} \tilde K^{\varphi(\vartheta \bh;\bi_1,\pm 1)}_{\bk,\bm}
\end{equation}
and that, by \eqref{estim:poi1}, for $\varphi \in \mathfrak B_{q,2}$, 
%	\begin{equation}
%		\label{limite1}
%		 \tilde K^{\varphi(\bh;\bi_1,\pm 1)}_{\bk,\bm}= 	 \sum_{\substack{ C_A(\bk,\bm,\varphi(\bh;\bi_1,\pm 1)) \\
%				j\in \mathcal S_p} }  \boldsymbol{k}^{(1)}_j(\boldsymbol{m}^{(2)}_j+1)P_{\bk^{(1)},\bm^{(1)}}^{ \check{\bh}^{(1)}}Q_{\bk^{(2)},\bm^{(2)}+e_j}^{ \check{\bh}^{(2)}}-	
%			 \boldsymbol{k}^{(2)}_j(\boldsymbol{m}^{(1)}_j+1)P_{\bk^{(1)},\bm^{(1)}+e_j}^{ \check{\bh}^{(1)}}Q_{\bk^{(2)},\bm^{(2)}}^{ \check{\bh}^{(2)}}
%	\end{equation}
	\begin{equation}
		\label{limite1}
		 \tilde K^{\varphi(\bh;\bi_1,\pm 1)}_{\bk,\bm}= 	 \sum_{  C_A(\bk,\bm,\varphi(\bh;\bi_1,\pm 1))  }  \bili(P^{ \check{\bh}^{(1)}},Q^{ \check{\bh}^{(2)}})
	\end{equation}
	where $ \varphi(\bh;\bi_1,\pm 1)=(\check \bh^{(1)};\check\bh^{(2)})$ implies that there exists  $\bh^{(1)},\bh^{(2)}\prec \bi_1 $ such that $\bh=(\bh^{(1)};\bh^{(2)})$ and (apart from a permutation) one of the following holds
	\begin{itemize}
		\item[\emph{Type 1}:]$\check \bh^{(1)} =(\bh^{(1)};\bi_1,\pm 1)$ and  $\check \bh^{(2)} = \bh^{(2)}$ 
		\item[\emph{Type 2}:] $\check \bh^{(1)} =\bh^{(1)}$ and  $\check \bh^{(2)} = (\bh^{(2)};\bi_1,\pm 1)$   
		\item [\emph{Type 3}:]  $\check \bh^{(1)} =(\bh^{(1)};(\bi_1,\sigma))$ and  $\check \bh^{(2)} = (\bh^{(2)};(\bi_1,-\sigma))$  for some sign $\sigma$. 
	\end{itemize}
	% Next we note that as soon as $|i|> 2\sum_{i=1}^q |\bh_i| +\sum_{s\in \mathcal S}|\bk_s|$ then  item $(3)$ is incompatible with the condition $(\bk^{(i)},\bm^{(i)}\check{\bh}^{(i)})\in \mathbb M$. 
	We note that the terms of type 3 do not contribute to the limit $|\bi_1|\to \infty$, indeed by construction in this case $\check \bh^{(2)}\notin \mathbf{C}_{q,1}$ and hence  we have 
	$|Q_{\bk^{(2)},\hat \bm^{(2)}}^{ \check{\bh}^{(2)}}| = |Q_{0,\bk^{(2)},\hat\bm^{(2)}}^{ \check{\bh}^{(2)}}|\le |Q|^{\Upsilon, q^{(2)}+1}_{\bk^{(2)},\hat\bm^{(2)}}|\bi_1|^{-1}$ for any choice of $\hat\bm^{(2)}$ ($q^{(2)}$ being the cardinality of $\bh^{(2)}$). The existence of the limit for the other terms then follows directly of the existence of the limit for $P$ and $Q$. Indeed, we have
	\begin{equation}	\label{limite2}
	\begin{split}
		\sum_{\varphi \in \mathfrak B_{q,2}}\lim_{|\bi_1|\to\infty} \tilde K^{\varphi(\bh;\bi_1,\pm 1)}_{\bk,\bm}
	=	\sum_{ C_A(\bk,\bm,\bh)  }   
	 \sum_{\varphi_1\in \mathfrak B_{q^{(1)},2}}&  \bili(\widehat P_1^{ \varphi_1(\bh^{(1)};\tilde\imath,\pm 1)},Q^{ \bh^{(2)}}) \\
	&+ \sum_{\varphi_2\in \mathfrak B_{q^{(2)},2}}  \bili(P^{ \bh^{(1)}},Q_1^{ \varphi_2(\bh^{(2)};\tilde\imath,\pm 1)}) 
	\end{split}
	\end{equation}
	where $q^{(s)}$ is the cardinality of $\bh^{(s)}$  and $\tilde\imath\in \mathcal{S}_p^c$ is any index such that $\bh\prec \tilde\imath$. Recalling \eqref{eq:average2}, this ensures that $\widehat K_1\in \mathscr{T}_{1,\eta,r,p}$ is well defined by \eqref{eq:def_hat_K1_and_K2}.

	%we obtain
	% \[
	% \begin{aligned}
		% 	\widehat K_1 %^{(\bh;i)}_{1,\bk,\bm}&:= \frac{1}{|\mathfrak S_q|}\sum_{\vartheta\in \mathfrak S_q} \sum_{\varphi \in \mathfrak B_{q+2}}\lim_{j\to\infty} \tilde K^{\varphi(\vartheta \bh;j)}_{\bk,\bm}=
		% = \{\widehat P_1, Q\}_{y,z} +\{P,\widehat Q_1\}_{y,z}
		% \end{aligned}
	% \]
	Following the same reasoning, we get that  $K_2$ is well defined and is given for  $\bh\prec \bi_2\prec \bi_1$ by
	$$ 
	K_{2,\bk,\bm}^{(\bh;\bi_2,\pm 1;\bi_1,\pm 1)} :=  \frac{2\ic }{\sharp(\mathfrak S_{q+4})}\sum_{\vartheta\in \mathfrak S_{q}} \sum_{\varphi\in \mathfrak B_{q,4}} \tilde K_{2,\bk,\bm}^{\varphi(\vartheta \bh;\bi_2,\pm 1;\bi_1,\pm 1)}  
	$$
	where
	\begin{equation}
		\label{kappa2}
			\sum_{\varphi\in \mathfrak B_{q,4}}  \tilde K_{2,\bk,\bm}^{\varphi(\bh;\bi_2\pm 1;\bi_1,\pm 1)} :=      	 \sum_{\substack{ C_A(\bk,\bm,\bh) } } \eqref{eq:dev_K_2}
	\end{equation}
		\begin{equation}
	\label{eq:dev_K_2}
	\begin{split}
		 2& \!\!\! \sum_{\substack{\varphi_1\in \mathfrak B_{q^{(1)},2} }}  \sum_{\varphi_2\in \mathfrak B_{q^{(2)},2}}\!\!\!   \bili( \widehat P_{1  }^{ \varphi_1(\bh^{(1)};\tilde\imath ,\pm 1)} ,Q_{1}^{ \varphi_2(\bh^{(2)};\tilde\jmath,\pm 1)}  )
			+\!\!\! \!\!\! \sum_{\varphi_1\in \mathfrak B_{q^{(1)},4}} \!\!\! \! \! \bili(P_{2}^{ \varphi_1(\bh^{(1)};\tilde\jmath,\pm 1;\tilde\imath,\pm 1)},Q^{ \bh^{(2)}} )
		\end{split}
	\end{equation}

	As before $\tilde\jmath\prec \tilde\imath$ are arbitrary couples of indexes in $\mathcal S_p^c$ with $\bh\prec\tilde\jmath$, so  $\tilde K_{2,\bk,\bm}^{(\bh;\bi_2,\pm 1;\bi_1,\pm 1)} $ does not  depend on $\bi_2\prec\bi_1$.
	
	We have proven that $\widehat K_1 ,K_2$  are well defined and so that
	\[
\Pi_{\mathrm{Top}=2}  K = K_2\,,\quad  \Pi_{\mathrm{Top}=1}  K = \widehat K_1 - K_2=: K_1 \,,\quad  \Pi_{\mathrm{Top}=0}  K = K- \widehat K_1 =: K_0\,.
\]

\noindent \underline{\emph{Step $1.2.$ : Explicit formulas for $K_0$ and $K_1$.}} To estimate the $\mathscr{H}$ norm of $K$, we have to prove special bounds for coefficients of the form $K ^{(\bh;\bi_1,\pm 1)}_{0,\bk,\bm}$ and $K ^{(\bh;\bi_2,\pm 1;\bi_1,\pm 1)}_{1,\bk,\bm}$. Thus, we decompose them as a sum of terms enjoying some decay properties.

	Recalling that $K_0=K-\widehat K_1$ and formula \eqref{limite1}, we have that for $\bh \prec \bi_1$
\begin{equation}
		\label{kappa0}
	K ^{(\bh;\bi_1,\pm 1)}_{0,\bk,\bm}:= \frac{2\ic }{\sharp \mathfrak S_{q+2}}\sum_{\vartheta\in \mathfrak S_q} \sum_{\varphi \in \mathfrak B_{q,2}}\tilde K^{\varphi(\vartheta \bh;\bi_1,\pm 1)}_{0,\bk,\bm} \quad \text{where} \quad \sum_{\varphi \in \mathfrak B_{q,2}}\tilde K^{\varphi( \bh;\bi_1,\pm 1)}_{0,\bk,\bm} = (1)+ (2)+ (3)
	\end{equation}
	\begin{equation*}
		\begin{aligned}
	&	(1)= \! \! \!	\sum_{\substack{ C_A(\bk,\bm,\bh) } }    \sum_{\substack{ \varphi_1\in \mathfrak B_{q^{(1)},2} }} \ \bili(P_{0}^{ \varphi_1(\bh^{(1)};\bi_1,\pm 1)},Q^{ \bh^{(2)}})  
	 \\
			& (2)=  \sum_{\substack{ C_A(\bk,\bm,\bh)   } }  \sum_{\substack{\varphi_2\in \mathfrak B_{q^{(2)},2} }} \bili(  P^{ \bh^{(1)}}, Q_{0}^{ \varphi_2(\bh^{(2)};\bi_1,\pm 1)} )  \\
			&(3)= 		  \sum_{\substack{ C_A(\bk,\bm,(\bh;\bi_1,\pm 1))\\
					\scriptscriptstyle{type}\; 3  } } 
				\sum_{\substack{\varphi_1\in \mathfrak B_{q^{(1)},1}\\ \varphi_2\in \mathfrak B_{q^{(2)},1} }} \bili(
				 P^{ \varphi_1(\bh^{(1)};(\bi_1,\sigma) )},Q_{0}^{ \varphi_2(\bh^{(2)};(\bi_1,-\sigma))} ) .
		\end{aligned}
	\end{equation*}
	In the last line we used the fact that, since  $(\bh^{(2)};\bi_1,-\sigma)\notin \mathbf{C}_{q,1}$,  $Q^{ \varphi_2(\bh^{(2)};\bi_1,-\sigma)}= Q_{0}^{ \varphi_2(\bh^{(2)};\bi_1,-\sigma)}$.

For $K_1$, we proceed in the same way.		Recalling that $K_1=\widehat K_1-K_2$ and formula \eqref{kappa2}, we have that for $\bh \prec \bi_2 \prec  \bi_1$
	\begin{equation}
		\label{kappa1}
		K_{1,\bk,\bm}^{(\bh;\bi_2,\pm 1;\bi_1,\pm 1)} :=  \frac{2\ic }{\sharp(\mathfrak S_{q+4})}\sum_{\vartheta\in \mathfrak S_{q}} \underbrace{\sum_{\varphi\in \mathfrak B_{q,4}} \tilde K_{1,\bk,\bm}^{\varphi(\vartheta \bh;\bi_2,\pm 1;\bi_1,\pm 1)} }_{=(1)+(2)+(3)+(4)+(5)+(6) } 
	\end{equation}
	where, given any index $\tilde\imath \in \mathcal S_p^c$ such that $\bi_2\prec \tilde\imath$, we have 	
	\begin{align*}
		&(1) :=    	 \sum_{\substack{ C_A(\bk,\bm,\bh) } } 
			\sum_{\substack{\varphi_1\in \mathfrak B_{q^{(1)},2} }}  \sum_{\varphi_2\in \mathfrak B_{q^{(2)},2}}
		\bili(P_{0}^{ \varphi_1(\bh^{(1)};\bi_2,\pm 1)} ,Q_{1}^{ \varphi_2(\bh^{(2)};\tilde\imath,\pm 1)})  
		\\ 
		&(2):=  	\sum_{\substack{ C_A(\bk,\bm,\bh) } } 
			\sum_{\substack{\varphi_1\in \mathfrak B_{q^{(1)},2} }}  \sum_{\varphi_2\in \mathfrak B_{q^{(2)},2}}
		\bili(\widehat P_1^{ \varphi_1(\bh^{(1)};\bi_2,\pm 1)} ,Q_{0}^{ \varphi_2(\bh^{(2)};\tilde\imath,\pm 1)})  
	\end{align*}
	% 2\{\widehat P_1, \widehat Q_1\}_{y,z} 	+\{ P_2, Q\}_{y,z} +\{P, Q_2\}_{y,z}+
	\begin{align*}
		& (3) :=   	  \sum_{\substack{ C_A(\bk,\bm,(\bh;\bi_2,\pm 1))\\
				\scriptscriptstyle{type}\; 3} }  \sum_{\substack{\varphi_1\in \mathfrak B_{q^{(1)},3}}} \sum_{ \varphi_2\in \mathfrak B_{q^{(2)},1}} \bili(\widehat P_{1}^{ \varphi_1(\bh^{(1)};\bi_2,\sigma;\tilde\imath,\pm 1)} ,  Q_{0}^{ \varphi_2(\bh^{(2)};\bi_2,-\sigma)}),
		\\
		& (4) :=   	 \sum_{\substack{ C_A(\bk,\bm,(\bh;\bi_2,\pm 1))\\
				\scriptscriptstyle{type}\; 3} }  \sum_{\substack{\varphi_1\in \mathfrak B_{q^{(1)},1}}} \sum_{\varphi_2\in \mathfrak B_{q^{(2)},3}} \bili( P^{ \varphi_1(\bh^{(1)};\bi_2,\sigma)} ,Q_{1}^{ \varphi_2(\bh^{(2)};\bi_2,-\sigma;\tilde\imath, \pm 1)})
	\end{align*}
	\begin{align*}
		&(5) :=   	 \sum_{\substack{ C_A(\bk,\bm,\bh) } } \sum_{\varphi_1\in \mathfrak B_{q^{(1)},4}} \bili( P_{1}^{ \varphi_1(\bh^{(1)};\bi_2,\pm 1;\tilde\imath, \pm 1)} ,Q^{ \bh^{(2)}}  ) \\
		&(6) :=    	 \sum_{\substack{ C_A(\bk,\bm,\bh)} }  \sum_{\varphi_2\in \mathfrak B_{q^{(2)},4}} \bili( P^{ \bh^{(1)}}  , Q_{1}^{ \varphi_2(\bh^{(2)};\bi_2,\pm 1;\tilde\imath, \pm 1)}  ).
	\end{align*}
	 As expected the expression above does not depend on $\bi_1$.
	
	\noindent \underline{\emph{Step $1.3.$ : Bounds.}} 
	We now prove that the $K_i$ satisfy the bounds in \eqref{poisson_base}. Recall that 
	\begin{equation}
		\label{stimine}
		\begin{aligned}
			&|P_{n,\bk,\bm}^{\bh;\bi_n,\pm 1;\dots;\bi_1,\pm 1}| \le \Theta_{n,\bh}  |P|^{\Theta,q+2n }_{\bk,\bm}\,,\quad |\widehat P_{1,\bk,\bm}^{\bh}| \le 2 | P|^{\Theta,q}_{\bk,\bm}\,,\quad |P_{\bk,\bm}^{\bh}| \le 3| P|^{\Theta,q}_{\bk,\bm}\\
			 &  |Q_{n,\bk,\bm}^{\bh;\bi_n,\pm 1;\dots;\bi_1,\pm 1}| \le (1+\|\bl\|_{\ell^\infty})^{-1} |Q|^{\Upsilon,q+2n}_{\bk,\bm}\,,\quad 
			\quad |Q_{\bk,\bm}^{\bh}| \le 2 | Q|^{\Upsilon,q}_{\bk,\bm}
		\end{aligned}
	\end{equation}
	From \eqref{kappa2} and \eqref{stimine}, recalling also that $\Theta_{2,\bh}=1$, 
	we immediately obtain for $q\ge 4$ %\mic{Here, when i estimate $K_{n}$  I changed $\sharp \bh =q-2n$ so that the total degree is $q$ and this contributes  to $\mathfrak K^{\Theta}_{\bk,\bm,q}$, I dont know wheter its the best choice }
	\begin{equation*}
			\sup_{\substack{\bh\in \mathbb I_{q-4}\\ \bh\prec\bi_2\prec\bi_1\\ (\bk,\bm,\bh)\in \mathbb M}}|K^{\bh;\bi_1,\pm 1;\bi_2,\pm 1}_{2,\bk,\bm}| \Theta_{2,\bh}^{-1}  \lesssim B 
	\end{equation*}
	with 
	\begin{equation}
		\label{stima2}
		\begin{aligned}
		 B:&=  \sum_{\substack{q^{(1)}+q^{(2)}= q\\ \bk^{(1)}+\bk^{(2)}=\bk \\ \bm^{(1)}+\bm^{(2)}=\bm }}\!\sum_{j\in \mathcal S_p}  |\boldsymbol{k}^{(1)}_j|(\|\boldsymbol{m}^{(2)}\|_{\ell^\infty}+1) |P|^{\Theta,q^{(1)}}_{\bk^{(1)},\bm^{(1)}}|Q|^{\Upsilon,q^{(2)}}_{\bk^{(2)},\bm^{(2)}+e_j} \\ &\qquad\qquad\qquad\qquad\qquad+|\boldsymbol{k}^{(2)}_j|(\|\boldsymbol{m}^{(1)}\|_{\ell^\infty}+1)|P|^{\Theta,q^{(1)}}_{\bk^{(1)},\bm^{(1)}+e_j} |Q|^{\Upsilon,q^{(2)}}_{\bk^{(2)},\bm^{(2)}}\,.
		\end{aligned}
	\end{equation}
	Regarding $K_{0,\bk,\bm}^{\bh}$, we note that if $\bh\notin \mathbf{C}_{q,1}$ then, by \eqref{limite1} (recall that $\Theta_{0,\bh}=1 $ in this case), we obtain for $q\ge 0$
	\begin{equation}
		\label{stima00}
		\begin{aligned}
			\sup_{\substack{\bh\in \mathbb I_{q}\\ \bh\notin \mathbf{C}_{q,1}\\ (\bk,\bm,\bh)\in \mathbb M}}|K^{(\bh)}_{0,\bk,\bm}| \Theta_{0,\bh}^{-1}  \lesssim B
%			\!\!\! \sum_{\substack{q^{(1)}+q^{(2)}= q \\ \bk^{(1)}+\bk^{(2)}=\bk \\ \bm^{(1)}+\bm^{(2)}=\bm }}\!\!\!\sum_{j\in \mathcal S_p} & |\boldsymbol{k}^{(1)}_j|(\|\boldsymbol{m}^{(2)}\|_{\ell^\infty}+1)|P|^{\Theta,q^{(1)}}_{\bk^{(1)},\bm^{(1)}}|Q|^{\Upsilon, q^{(2)}}_{\bk^{(2)},\bm^{(2)}+e_j,}+ \\  &|\boldsymbol{k}^{(2)}_j|(\|\boldsymbol{m}^{(1)}\|_{\ell^\infty}+1)|P|^{\Theta,q^{(1)}}_{\bk^{(1)},\bm^{(1)}+e_j}|Q|^{\Upsilon, q^{(2)}}_{\bk^{(2)},\bm^{(2)}})\,.
		\end{aligned}
	\end{equation}
	Otherwise, for  $(\bh;\bi_1,\pm 1)$ with $\bh\prec \bi_1$ we note that in \eqref{kappa0}
	
	\begin{itemize}
		\item  In  item (1) one has  $\Theta_{0,(\bh;\bi_1,\pm 1)} ={1+|\bi_1|}= \Theta_{0,(\bh^{(1)};\bi_1,\pm 1)}$, so that  e. g.  one has 
		$$|P_{0,\bk^{(1)},\bm^{(1)}}^{ \varphi_1(\bh^{(1)};\bi_1,\pm 1)}Q_{\bk^{(2)},\bm^{(2)}+e_j}^{ \bh^{(2)}}  | \Theta_{0,(\bh;\bi_1)}^{-1} \le 2 |P|^{\Theta,q^{(1)}+2}_{\bk^{(1)},\bm^{(1)}}|Q|^{\Upsilon,q^{(2)}}_{\bk^{(2)},\bm^{(2)}+e_j} . $$
		\item In  items (2) and (3) one has  $\Theta_{0,(\bh;\bi_1)} ={1+|\bi_1|}= \Upsilon_{0,(\bh^{(2)};\bi_1,\pm 1)}$, so that e.g. we have
		\[
		|P_{\bk^{(1)},\bm^{(1)}}^{ \bh^{(1)}} Q_{0,\bk^{(2)},\bm^{(2)}+e_j}^{ \varphi_2(\bh^{(2)};\bi_1,\pm 1)}| \Theta_{0,(\bh;\bi_1)}^{-1} \le 3  |P|^{\Theta,q^{(1)}}_{\bk^{(1)},\bm^{(1)}}|Q|^{\Upsilon,q^{(2)}+2}_{\bk^{(2)},\bm^{(2)}+e_j} 
		\] 
	\end{itemize}
%	in the first line of \eqref{kappa0} $\Theta_{0,(\bh;\bi_1)} ={1+|\bi_1|}= \Theta_{0,(\bh^{(1)};\bi_1)}$, while in the second and  the third line 
%	$\Theta_{0,(\bh;\bi_1)} ={1+|\bi_1|}= \Upsilon_{0,(\bh^{(2)};\bi_1)}$.  \mic{finqui} So, by \eqref{stimine} 
%	
	we obtain for $q\ge 2$:
	\begin{equation}
		\label{stima0}
		\begin{aligned}
			\sup_{\substack{\bh\in \mathbb I_{q-2}\\ \bh\prec\bi_1\\ (\bk,\bm,\bh)\in \mathbb M}}|K^{(\bh;\bi_1,\pm 1)}_{0,\bk,\bm}| \Theta_{0,(\bh;\bi_1,\pm 1)}^{-1}  \lesssim B
%			\!\!\! \sum_{\substack{q^{(1)}+q^{(2)}= q \\ \bk^{(1)}+\bk^{(2)}=\bk \\ \bm^{(1)}+\bm^{(2)}=\bm }}\!\!\!\sum_{j\in \mathcal S_p} & |\boldsymbol{k}^{(1)}_j|(|\boldsymbol{m}^{(2)}|_{\ell_\infty}+1)|P|^{\Theta,q^{(1)}}_{\bk^{(1)},\bm^{(1)}}|Q|^{\Upsilon,q^{(2)}}_{\bk^{(2)},\bm^{(2)}+e_j} + \\  &|\boldsymbol{k}^{(2)}_j|(|\boldsymbol{m}^{(1)}|_{\ell_\infty}+1)|P|^{\Theta,q^{(1)}}_{\bk^{(1)},\bm^{(1)}+ e_j}|Q|^{\Upsilon,q^{(2)}}_{\bk^{(2)},\bm^{(2)}} 
\,.
		\end{aligned}
	\end{equation}
	
	Regarding  $K_{1,\bk,\bm}^{(\bh;\bi_1,\pm 1)}$ with $\bh\prec\bi_1$,  we note that if $\bh\notin\mathbf{C}_{q-2,1}$ then $K_{1,\bk,\bm}^{(\bh;\bi_1,\pm 1)}= \widehat K_{1,\bk,\bm}^{(\bh;\bi_1,\pm 1)}$ and we get (recall that $\Theta_{1,\bh}=1 $ in this case)
	\begin{equation}
		\label{stima10}
		\begin{aligned}
			\sup_{\substack{\bh\in \mathbb I_{q-2}\\ \bh\notin \mathbf{C}_{q-2,1}\\ (\bk,\bm,\bh)\in \mathbb M}}|K^{(\bh;\bi_1,\pm 1)}_{1,\bk,\bm}| \Theta_{1,\bh}^{-1}  \lesssim B
%			\!\!\! \sum_{\substack{q^{(1)}+q^{(2)}= q \\ \bk^{(1)}+\bk^{(2)}=\bk \\ \bm^{(1)}+\bm^{(2)}=\bm }}\!\!\!\sum_{j\in \mathcal S_p} & |\boldsymbol{k}^{(1)}_j|(|\boldsymbol{m}^{(2)}|_{\ell_\infty}+1)\mathfrak  P^\Theta_{\bk^{(1)},\bm^{(1)},q^{(1)}}\mathfrak  Q^\Upsilon_{\bk^{(2)},\bm^{(2)}+e_j,q^{(2)}}+ \\  &|\boldsymbol{k}^{(2)}_j|(|\boldsymbol{m}^{(1)}|_{\ell_\infty}+1)\mathfrak  P^\Theta_{\bk^{(1)},\bm^{(1)}+e_j,q^{(1)}}\mathfrak  Q^\Upsilon_{\bk^{(2)},\bm^{(2)},q^{(2)}})\,.
		\end{aligned}
	\end{equation}
	Otherwise for  $K_{1,\bk,\bm}^{(\bh;\bi_2,\pm 1;\bi_1,\pm 1)}$ with $\bh\prec \bi_2$  defined in  %and $\bi_2\prec \bi_2$ 
	\eqref{kappa1}, we  note that:
	\begin{itemize}
		\item  in (1) and (5) we have  $\Theta_{1,(\bh;\bi_2,\pm 1)}^{-1} ={1+|\bi_2|}= \Theta_{n,(\bh^{(1)};\bi_2,\pm 1)}^{-1}$  with $n=0,1$, so that  for example 
		$$
		|P_{0,\bk^{(1)},\bm^{(1)}}^{ \varphi_1(\bh^{(1)};\bi_2,\pm 1)}Q_{1,\bk^{(2)},\bm^{(2)}+e_j}^{ \varphi_2(\bh^{(2)};\tilde\imath,\pm 1)}|\Theta_{1,(\bh;\bi_2,\pm 1)}^{-1} \le  |P|^{\Theta,q^{(1)}+2}_{\bk^{(1)},\bm^{(1)}}|Q|^{\Upsilon,q^{(2)}+2}_{\bk^{(2)},\bm^{(2)}+e_j} 
		$$ 
		\item in (2) and (6) we have $\Theta_{1,(\bh;\bi_2,\pm 1)}^{-1} ={1+|\bi_2|}= \Upsilon_{n,(\bh^{(2)};\bi_1,\pm 1)}^{-1}$   with $n=0,1$,
%		 then we estimate  for example $|	\widehat P_{1,\bk^{(1)},\bm^{(1)}}^{ \varphi_1(\bh^{(1)};\bi_1)} Q_{0,\bk^{(2)},\bm^{(2)}+e_j}^{ \varphi_2(\bh^{(2)};\bi_2)}|\le  2 \Upsilon_{0,(\bh^{(2)};\bi_2)} \mathfrak P^{\Theta}_{\bk^{(1)},\bm^{(1)},q^{(1)}+2} \mathfrak Q^\Upsilon_{\bk^{(2)},\bm^{(2)}+e_j,q^{(2)}+2} $ 
%		
		\item in (3) and (4) we have $\Theta_{1,(\bh;\bi_2,\pm 1)}^{-1} ={1+|\bi_2|}= \Upsilon_{0, (\bh^{(2)};\bi_2,-\sigma)}^{-1}$.
%		moreover, since $(\bh^{(2)};\bi_2,-\sigma)\notin \mathbf P$ then
%		$Q_{\bk^{(2)},\bm^{(2)}+e_j}^{ \varphi_2(\bh^{(2)};\bi_2,-\sigma}= Q_{0,\bk^{(2)},\bm^{(2)}+e_j}^{ \varphi_2(\bh^{(2)};\bi_2,-\sigma}$. Thus  we estimate  for example
%		$$
%		\widehat P_{1,\bk^{(1)},\bm^{(1)}}^{ \varphi_1(\bh^{(1)};\bi_2,\sigma);\bi_1)}   Q_{\bk^{(2)},\bm^{(2)}+e_j}^{ \varphi_2(\bh^{(2)};\bi_2,-\sigma)}|
%		\le \Upsilon_{0, (\bh^{(2)};\bi_2,-\sigma)} \mathfrak P^{\Theta}_{\bk^{(1)},\bm^{(1)},q^{(1)}+3} \mathfrak Q^\Upsilon_{\bk^{(2)},\bm^{(2)}+e_j,q^{(2)}+1} 
%		$$
%%		
%		\item  in (4), since  $(\bh^{(2)};\bi_2,-\sigma)\notin \mathbf P$, then, by \eqref{stimine}, we have  $|\widehat Q_{1,\bk^{(2)},\bm^{(2)}+e_j}^{ \varphi_2(\bh^{(2)};\bi_2,-\sigma;\bi_1)}|= | Q_{1,\bk^{(2)},\bm^{(2)}+e_j}^{ \varphi_2(\bh^{(2)};\bi_2,-\sigma;\bi_1)}|\le \Upsilon_{1,(\bh^{(2)};\bi_2,-\sigma)} \mathfrak Q_{\bk,\bm,q^{(2)}+3}=  \Theta_{1,(\bh;\bi_2)} \mathfrak Q_{\bk,\bm,q^{(2)}+3}$.
	\end{itemize}
	In conclusion  we obtain:
	\begin{equation}
		\label{stima1}
		\begin{aligned}
			\sup_{\substack{\bh\in \mathbb I_{q-4}\\ \bh\prec\bi_1\\ (\bk,\bm,\bh)\in \mathbb M}}|K^{(\bh;\bi_2,\pm 1;\bi_1,\pm 1)}_{1,\bk,\bm}| \Theta_{1,(\bh;\bi_2,\pm 1)}^{-1}  \lesssim B
		\end{aligned}
	\end{equation}
	We have thus proved that $|K|_{\bk,\bm}^{\Theta,q} \lesssim B$.
Then, we conclude as in the proof of Lemma \ref{lem:poisson_F}.
	
			\noindent \underline{\emph{Step $2$ : Estimate on $L$.}} 
We follow the same scheme as for $K$, by first proposing a decomposition $L= L_0+L_1+L_2$.

\noindent \underline{\emph{Step $2.1.$ : Decomposition of $L$.}} We write 
$L_0= L- \widehat L_1 $ and $L_1 =\widehat L_1- L_2$
where
\begin{equation}
	\label{eq:def_hat_L1_and_L2}
	\widehat L_{1,\boldsymbol{k},\boldsymbol{m}}^{(\boldsymbol{h}; \boldsymbol{j}_1,\pm 1)} = \lim_{|\bi_1|\to  \infty}   L_{\boldsymbol{k},\boldsymbol{m}}^{(\boldsymbol{h};\bi_1,\pm 1)}
	\quad \mathrm{and} \quad
	L_{2,\boldsymbol{k},\boldsymbol{m}}^{(\boldsymbol{h}; \boldsymbol{j}_2,\pm 1;\boldsymbol{j}_1,\pm 1)} =  \lim_{|\bi_1|\to  \infty} \lim_{|\bi_2|\to  \infty}   L_{\boldsymbol{k},\boldsymbol{m}}^{(\boldsymbol{h}; \bi_2,\pm 1;\bi_1,\pm 1)}
\end{equation}
for $\bh \prec \boldsymbol{j}_2 \prec \boldsymbol{j}_1$. Now, we aim at proving that these limits exist.
%Then in order to prove that this is a decomposition as in \eqref{eq:cond2}, we just need to verify that the corresponding norm \eqref{eq:normcst2} is finite. 

As for the term $K$, it is convenient to note that for $\bh\prec \bi_1$ 
\begin{equation}
	\label{eq:averageL}
	L ^{(\bh;\bi_1,\pm 1)}_{\bk,\bm}= \frac{2\ic }{\sharp(\mathfrak S_{q+2})}\sum_{\vartheta\in \mathfrak S_q} \sum_{\varphi \in \mathfrak B_{q,2}} \tilde L^{\varphi(\vartheta \bh;\bi_1,\pm 1)}_{\bk,\bm}
\end{equation}
and that, by \eqref{pan}, for $\varphi \in \mathfrak B_{q,2}$, 
%	\begin{equation}
	%		\label{limite1}
	%		 \tilde K^{\varphi(\bh;\bi_1,\pm 1)}_{\bk,\bm}= 	 \sum_{\substack{ C_A(\bk,\bm,\varphi(\bh;\bi_1,\pm 1)) \\
			%				j\in \mathcal S_p} }  \boldsymbol{k}^{(1)}_j(\boldsymbol{m}^{(2)}_j+1)P_{\bk^{(1)},\bm^{(1)}}^{ \check{\bh}^{(1)}}Q_{\bk^{(2)},\bm^{(2)}+e_j}^{ \check{\bh}^{(2)}}-	
	%			 \boldsymbol{k}^{(2)}_j(\boldsymbol{m}^{(1)}_j+1)P_{\bk^{(1)},\bm^{(1)}+e_j}^{ \check{\bh}^{(1)}}Q_{\bk^{(2)},\bm^{(2)}}^{ \check{\bh}^{(2)}}
	%	\end{equation}
\begin{equation}
	\label{limiteL}
	\tilde L^{\varphi(\bh;\bi_1,\pm 1)}_{\bk,\bm}= 	 \sum_{  C_v(\bk,\bm,\varphi(\bh;\bi_1,\pm 1))  }  \mathcal{L}(P^{ \widehat{\bh}^{(1)}},Q^{ \widehat{\bh}^{(2)}})
\end{equation}
where $ \varphi(\bh;\bi_1,\pm 1)=(\tilde\bh^{(1)};\tilde\bh^{(2)})$,  $\widehat\bh^{(1)}= (\tilde\bh^{(1)};\widehat\jmath,\widehat\sigma)$ and $\widehat\bh^{(2)}={ (\tilde\bh^{(2)};\widehat\jmath,-\widehat\sigma)}$, see formula \eqref{taglio}, and $\tilde q^{(i)}$ is the cardinality of $\tilde\bh^{(i)}$. 

This implies that there exists  $\bh^{(1)},\bh^{(2)}\prec \bi_1 $ such that $\bh=(\bh^{(1)};\bh^{(2)})$ and (apart from a permutation) one of the following holds 
\begin{itemize}
	\item[ \emph{Type 1}:] $\tilde \bh^{(1)} =(\bh^{(1)};\bi_1,\pm 1)$ and  $\tilde \bh^{(2)} = \bh^{(2)}$ so that $\widehat \bh^{(1)}=(\bh^{(1)};
	\bi_1,\pm 1;\hat\jmath,\hat \sigma)$  and $\widehat \bh^{(2)}=(\bh^{(2)};\hat\jmath,-\hat \sigma)$ with 
	$$\
	\hat\sigma = -\sum_{j\in \mathcal{S}_p} \boldsymbol{k}_j^{(1)} + \sum_{1\leq i \leq q^{(1)}} \boldsymbol{\sigma}^{(1)}_i  \quad \mbox{and} \quad  \hat\jmath =  -\hat\sigma (\displaystyle \sum_{j\in \mathcal{S}_p} j\boldsymbol{k}^{(1)}_j + \sum_{1\leq i \leq q^{(1)}} \boldsymbol{\sigma}^{(1)}_i \boldsymbol{\ell}^{(1)}_i).
	$$
	\item[ \emph{Type 2}:] $\tilde \bh^{(1)} =\bh^{(1)}$ and  $\tilde \bh^{(2)} = (\bh^{(2)};\bi_1,\pm 1)$   so that $\widehat \bh^{(1)}=(\bh^{(1)};\hat\jmath,\hat \sigma)$  and $\widehat \bh^{(2)}=(\bh^{(2)};\bi_1,\pm 1;\hat\jmath,-\hat \sigma)$ with $\hat\sigma,\hat\jmath$ defined as in \emph{Type 1.}
	\item[ \emph{Type 3}:] $\tilde \bh^{(1)} =(\bh^{(1)};(\bi_1,\sigma))$ and  $\tilde\bh^{(2)} = (\bh^{(2)};\bi_1,-\sigma)$  for some  $\sigma\in\{ -1,1\}$ so that  $\widehat\bh^{(1)} =(\bh^{(1)};(\bi_1,\sigma);(\hat\jmath,\hat \sigma))$ and  $\widehat\bh^{(2)} = (\bh^{(2)};\bi_1,-\sigma;\hat\jmath,-\hat \sigma)$. In this last case
	\begin{equation}
		\label{caso3}
		\hat\sigma = -\sum_{j\in \mathcal{S}_p} \boldsymbol{k}_j^{(1)} + \sum_{1\leq i \leq q^{(1)}} \boldsymbol{\sigma}^{(1)}_i - \sigma \,,\quad  \hat\jmath =  -\hat\sigma (\displaystyle \sum_{j\in \mathcal{S}_p} j\boldsymbol{k}^{(1)}_j + \sum_{1\leq i \leq q^{(1)}} \boldsymbol{\sigma}^{(1)}_i \boldsymbol{\ell}^{(1)}_i) -\hat\sigma \sigma \bi_1 \,.
	\end{equation}
\end{itemize}

It is important to notice that in \eqref{limiteL}, the terms of \emph{Type 3} such that $\hat\jmath =\bi_1$ and $\hat\sigma =-\sigma$ cancel out. Indeed they would contribute
\[
\sum_{\substack{C_v(\bk,\bm,\varphi(\bh;\bi_1,\pm 1))\\ \mbox{\tiny{\emph{Type 3}}} \\ \hat\jmath =\bi_1\,,\; \hat\sigma =-1}} \mathcal{L}(P^{ \widehat{\bh}^{(1)}},Q^{ \widehat{\bh}^{(2)}})+ \sum_{\substack{C_v(\bk,\bm,\varphi(\bh;i))\\ \mbox{\tiny{\emph{Type 3}}} \\ \hat\jmath =\bi_1\,,\; \hat\sigma =1}} \mathcal{L}(P^{ \widehat{\bh}^{(1)}},Q^{ \widehat{\bh}^{(2)}})
\]
and of course (recall that $\tilde q^{(i)}= q^{(i)}+1 $)
\begin{equation}
	\label{stoc}
	\sum_{\sigma= \pm 1} %\sum_{\substack{C_v(\bk,\bm,\varphi(\bh;i))\\ \mbox{\tiny{\emph{Type 3}}} \\ \hat\jmath =i\,,\; \hat\sigma =-\sigma}} 
	(-\sigma)  {(q^{(1)} +2)(q^{(2)}+2}) P_{\bk^{(1)},\bm^{(1)}}^{ (\bh^{(1)};\bi_1,\sigma; \bi_1,-\sigma)} Q_{\bk^{(2)},\bm^{(2)}}^{ (\bh^{(2)};\bi_1,-\sigma;\bi_1,\sigma)}=0\,.
\end{equation}
Moreover if $(\bi_1,\sigma) \ne (\hat\jmath,-\hat\sigma)$, then $\widehat \bh^{(2)}\notin \mathbf C_{q^{(2)},1}$  (recall that either $\bi_1$ or $\hat\jmath$ is the highest mode) so, by \eqref{stimine}, we have 
$$
|Q_{\bk^{(2)},\bm^{(2)}}^{ \widehat\bh^{(2)}}| =  |Q_{0,\bk^{(2)},\bm^{(2)}}^{ \widehat\bh^{(2)}}| \le  |Q|^{\Upsilon,q^{(2)}}_{\bk^{(2)},\bm^{(2)}}  \|\hat\bl^{(2)}\|_{\ell^\infty}^{-1} \le   |Q|^{\Upsilon,q^{(2)}}_{\bk^{(2)},\bm^{(2)}}  |\bi_1|^{-1}\,.
$$  
In conclusion the terms of \emph{Type 3}  give no contribution in the limit $|\bi_1|\to \infty$ of \eqref{limiteL}. Thus we obtain 
\begin{align}\label{ellehat}
	\lim_{|\bi_1|\to\infty} \sum_{\varphi \in \mathfrak B_{q,2}}	\tilde  L_{\bk,\bm}^{ \varphi(\bh;\bi_1,\pm 1)} =   \sum_{C_v(\bk,\bm,\bh)}  & \sum_{\varphi_1 \in \mathfrak B_{q^{(1)},2}} \mathcal L(\widehat P_{1,\bk^{(1)},\bm^{(1)}}^{ (\varphi_1(\bh^{(1)};\tilde\imath,\pm 1);\hat\jmath,\hat\sigma)},Q_{\bk^{(2)},\bm^{(2)}}^{ (\bh^{(2)};\hat\jmath,-\hat\sigma)}) \\ &  +  \sum_{\varphi_2 \in \mathfrak B_{q^{(2)},2}} \mathcal L( P_{\bk^{(1)},\bm^{(1)}}^{ (\bh^{(1)};\hat\jmath,\hat\sigma)} , Q_{1,\bk^{(2)},\bm^{(2)}}^{ (\varphi_2(\bh^{(2)};\tilde\imath,\pm 1);\hat\jmath,-\hat\sigma)})\nonumber%- \mic{c^{(1)}_+ c^{(2)}_-}  P_{\bk^{(1)},\bm^{(1)}}^{ \bh^{(1)};\hat\jmath,-\hat\sigma}\chi_{\bk^{(2)},\bm^{(2)}}^{ \bh^{(2)};\hat\jmath,\hat\sigma}
\end{align}
where $\tilde\imath\in \mathcal S^c_p$ is any mode which satisfies $(\bh;\hat\jmath,\hat\sigma)\prec\tilde\imath$; this defines $\widehat L$ according to \eqref{eq:averageL}.

	Following the same reasoning, we get that  $L_2$ is well defined and is given for  $\bh\prec \bi_2\prec \bi_1$ by
$$ 
L_{2,\bk,\bm}^{(\bh;\bi_2,
	\pm 1;\bi_1,\pm 1)} =  \frac{2 \ic}{|\mathfrak S_{q+4}|}\sum_{\vartheta\in \mathfrak S_{q}} \sum_{\varphi\in \mathfrak B_{q,4}} \tilde L_{2,\bk,\bm}^{\varphi(\vartheta \bh;\bi_2,
	\pm 1;\bi_1,\pm 1)}  
$$
where (see \eqref{pan} for the definition of $\mathcal L$):
\begin{align}\label{elle2}
	\sum_{\varphi\in \mathfrak B_{q,4}} \tilde L_{2,\bk,\bm}^{(\bh;\bi_2,
		\pm 1;\bi_1,\pm 1)}  =  \sum_{C_v(\bk,\bm,\bh)}
	\Big( 2&\sum_{\varphi_s \in \mathfrak B_{q^{(s)},2}}\mathcal L(\widehat P_{1,\bk^{(1)},\bm^{(1)}}^{ (\varphi_1(\bh^{(1)};\tilde\imath,\pm 1);\hat\jmath,\hat\sigma)}, Q_{1,\bk^{(2)},\bm^{(2)}}^{ (\varphi_2(\bh^{(2)};\tilde\jmath,\pm 1);\hat\jmath,-\hat\sigma)}) \\	
	 &+\sum_{\varphi_1 \in \mathfrak B_{q^{(1)},4}}\mathcal L(P_{2,\bk^{(1)},\bm^{(1)}}^{ (\varphi_1(\bh^{(1)};\tilde\jmath,\pm 1;\tilde\imath,\pm 1);\hat\jmath,\hat\sigma)},Q_{\bk^{(2)},\bm^{(2)}}^{ (\bh^{(2)};\hat\jmath,-\hat\sigma)})\Big) \,.\nonumber
\end{align}
% $(\bh;\hat\jmath)\prec\tilde\jmath\prec\tilde\imath$.
		As before $\tilde\jmath\prec \tilde\imath$ are arbitrary couples of indexes in $\mathcal S_p^c$ with $(\bh;\hat\jmath,\hat\sigma)\prec\tilde\jmath$.
		We have proven that $\widehat L_1 ,L_2$  are well defined and so that
	\[
	\Pi_{\mathrm{Top}=2}  L = L_2\,,\quad  \Pi_{\mathrm{Top}=1}  L = \widehat L_1 - L_2=: L_1 \,,\quad  \Pi_{\mathrm{Top}=0}  L = L- \widehat L_1 =: L_0\,.
	\]
	
	\noindent \underline{\emph{Step $2.2.$ : Explicit formulas for $L_0$ and $L_1$.}} 
	Recalling that $L_0=L-\widehat L_1$ and formula \eqref{limiteL}, we have that  $L_{0,\bk,\bm}^{(\bh)}= L_{\bk,\bm}^{(\bh)}$ defined in \eqref{pan},  for any $\bh\not\in \mathbf C_{q,1}$. Otherwise for  $\bh \prec \bi_1$ we have
	\begin{equation}
		\label{kappa01}
		\sum_{\varphi \in \mathfrak B_{q,2}}	\tilde  L_{0,\bk,\bm}^{ \varphi(\bh;\bi_1,\pm 1)} = (1) +(2)+(3)
	\end{equation}
	with 
	\[
	\begin{aligned}
		&(1):=\sum_{\substack{C_v(\bk,\bm,\bh)\\ \hat\jmath \prec \bi_1}} \sum_{\varphi_1 \in \mathfrak B_{q^{(1)},2}} \mathcal L(P_{0,\bk^{(1)},\bm^{(1)}}^{ (\varphi_1(\bh^{(1)};\bi_1,\pm 1);\hat\jmath,\hat\sigma)},Q_{\bk^{(2)},\bm^{(2)}}^{ (\bh^{(2)};\hat\jmath,-\hat\sigma)})\\ & \qquad\qquad\qquad\qquad\qquad\qquad\qquad\qquad+   \sum_{\varphi_2 \in \mathfrak B_{q^{(2)},2}} \mathcal L(P_{\bk^{(1)},\bm^{(1)}}^{ (\bh^{(1)};\hat\jmath,\hat\sigma}, Q_{0,\bk^{(2)},\bm^{(2)}}^{ (\varphi_2(\bh^{(2)};\bi_1,\pm 1);\hat\jmath,-\hat\sigma)})%- \mic{c^{(1)}_+ c^{(2)}_-}  P_{\bk^{(1)},\bm^{(1)}}^{ \bh^{(1)};(\hat\jmath,-\hat\sigma)}\chi_{\bk^{(2)},\bm^{(2)}}^{ \bh^{(2)};\hat\jmath,\hat\sigma}
		\\
		& (2):= \sum_{\substack{C_v(\bk,\bm,\bh)\\ \hat\jmath \not\prec \bi_1}} \sum_{\varphi_1 \in \mathfrak B_{q^{(1)},2}}  \mathcal L( P_{\bk^{(1)},\bm^{(1)}}^{ (\varphi_1(\bh^{(1)};\bi_1,\pm 1));\hat\jmath,\hat\sigma)}- \widehat P_{1,\bk^{(1)},\bm^{(1)}}^{ (\varphi_1(\bh^{(1)};\tilde\imath,\pm 1);\hat\jmath,\hat\sigma)}, Q_{\bk^{(2)},\bm^{(2)}}^{ (\bh^{(2)};\hat\jmath,-\hat\sigma)})\\ &  \qquad\qquad\qquad  +   \sum_{\varphi_2 \in \mathfrak B_{q^{(2)},2}} \mathcal L( P_{\bk^{(1)},\bm^{(1)}}^{ (\bh^{(1)};\hat\jmath,\hat\sigma)} , Q_{\bk^{(2)},\bm^{(2)}}^{ (\varphi_2(\bh^{(2)};\bi_1,\pm 1);\hat\jmath,-\hat\sigma)}- Q_{1,\bk^{(2)},\bm^{(2)}}^{ (\varphi_2(\bh^{(2)};\tilde\imath,\pm 1);\hat\jmath,-\hat\sigma)})
		\\
		&(3):=    \sum_{\substack{ C_v(\bk,\bm,(\bh;\bi_1))\\
				\scriptscriptstyle{type}\; 3} }\sum_{\sigma=\pm 1} \sum_{\varphi_s\in \mathfrak B_{q^{(s)},1}}  \mathcal L(P_{\bk^{(1)},\bm^{(1)}}^{ (\varphi_1(\bh^{(1)};(\bi_1,\sigma));\hat\jmath,\hat\sigma )},Q_{\bk^{(2)},\bm^{(2)}}^{ (\varphi_2(\bh^{(2)};(\bi_1,-\sigma));\hat\jmath,-\hat\sigma)} )
	\end{aligned}
	\]
	where $(\bh;\hat\jmath,\hat\sigma)\prec\tilde\imath$.
	It remains to deal with $L_1$. As before, if $\bh\not\in\mathbf{C}_{q,1}$, then  $L_{1,\bk,\bm}^{(\bh;\bi_1)} =  \widehat L_{1,\bk,\bm}^{(\bh;\bi_1)}$ defined in \eqref{ellehat}. 
	Otherwise
	\begin{equation}
		\label{L1}
		\sum_{\varphi\in \mathfrak B_{q,4}}  \tilde L_{1,\bk,\bm}^{(\bh;\bi_2;\bi_1)} =  (1)+(2)+(3)+(4)+(5)
	\end{equation}
	
	\[\begin{aligned}
		(1):=	 \sum_{\substack{C_v(\bk,\bm,\bh)\\ \hat\jmath \prec \bi_2}} \sum_{\varphi_1 \in \mathfrak B_{q^{(1)},2}} \mathcal L(P_{0,\bk^{(1)},\bm^{(1)}}^{ (\varphi_1(\bh^{(1)};\bi_2,\pm 1);\hat\jmath,\hat\sigma)}&, Q_{1,\bk^{(2)},\bm^{(2)}}^{ (\varphi_2(\bh^{(2)};\tilde\imath,\pm 1);\hat\jmath,-\hat\sigma)})\\ & +   \sum_{\varphi_2 \in \mathfrak B_{q^{(2)},2}}  \mathcal L(\widehat P_{1,\bk^{(1)},\bm^{(1)}}^{ (\varphi_1(\bh^{(1)};\tilde\imath,\pm 1);\hat\jmath,\hat\sigma)}, Q_{0,\bk^{(2)},\bm^{(2)}}^{ (\varphi_2(\bh^{(2)};\bi_2,\pm 1);\hat\jmath,-\hat\sigma)})%- \mic{c^{(1)}_+ c^{(2)}_-}  P_{\bk^{(1)},\bm^{(1)}}^{ \bh^{(1)};(\hat\jmath,-\hat\sigma)}\chi_{\bk^{(2)},\bm^{(2)}}^{ \bh^{(2)};\hat\jmath,\hat\sigma}
	\end{aligned}
	\]
	\[\begin{aligned}
		(2):=	  \sum_{\substack{C_v(\bk,\bm,\bh)\\ \hat\jmath \not\prec \bi_2}}\sum_{\varphi_1 \in \mathfrak B_{q^{(1)},2}}&\mathcal L(P_{\bk^{(1)},\bm^{(1)}}^{ (\varphi_1(\bh^{(1)};\bi_2,\pm 1);\hat\jmath,\hat\sigma)} - \widehat P_{1,\bk^{(1)},\bm^{(1)}}^{ (\varphi_1(\bh^{(1)};\tilde\jmath,\pm 1);\hat\jmath,\hat\sigma)}, Q_{1,\bk^{(2)},\bm^{(2)}}^{ \varphi_2(\bh^{(2)};\tilde\imath,\pm 1);\hat\jmath,-\hat\sigma)})\\ & +   \sum_{\varphi_2 \in \mathfrak B_{q^{(2)},2}} \mathcal L(\widehat P_{1,\bk^{(1)},\bm^{(1)}}^{ (\varphi_1(\bh^{(1)};\tilde\imath,\pm 1);\hat\jmath,\hat\sigma)},Q_{\bk^{(2)},\bm^{(2)}}^{ (\varphi_2(\bh^{(2)};\bi_2,\pm 1);\hat\jmath,-\hat\sigma)}-Q_{1,\bk^{(2)},\bm^{(2)}}^{ (\varphi_2(\bh^{(2)};\tilde\jmath,\pm 1);\hat\jmath,-\hat\sigma)})%- \mic{c^{(1)}_+ c^{(2)}_-}  P_{\bk^{(1)},\bm^{(1)}}^{ \bh^{(1)};(\hat\jmath,-\hat\sigma)}\chi_{\bk^{(2)},\bm^{(2)}}^{ \bh^{(2)};\hat\jmath,\hat\sigma}
	\end{aligned}
	\]
	\[
	\begin{aligned}
		(3)	:=  \sum_{\substack{C_v(\bk,\bm,\bh)\\ \scriptscriptstyle{\emph{type} \,3}}} \sum_{\substack{\varphi_1 \in \mathfrak B_{q^{(1)},3}\\\varphi_2 \in \mathfrak B_{q^{(2)},1}}} \sum_{\sigma=\pm 1}&\mathcal L(\widehat P_{1,\bk^{(1)},\bm^{(1)}}^{ (\varphi_1(\bh^{(1)};\tilde\imath,\pm 1; \bi_2,\sigma);\hat\jmath,\hat\sigma)},Q_{\bk^{(2)},\bm^{(2)}}^{ (\varphi_2(\bh^{(2)};\bi_2,-\sigma);\hat\jmath,-\hat\sigma)})+
		\\ &\quad + \sum_{\substack{\varphi_1 \in \mathfrak B_{q^{(1)},3}\\\varphi_2 \in \mathfrak B_{q^{(2)},1}}} \sum_{\sigma=\pm 1} \mathcal L(P_{\bk^{(1)},\bm^{(1)}}^{ (\varphi_1(\bh^{(1)}; \bi_2,\sigma);\hat\jmath,\hat\sigma)},Q_{1,\bk^{(2)},\bm^{(2)}}^{ (\varphi_2(\bh^{(2)};\tilde\imath,\pm 1;\bi_2,-\sigma);\hat\jmath,-\hat\sigma)})
	\end{aligned}
	\]
	\[
	\begin{aligned}
		(4):=  \sum_{\substack{C_v(\bk,\bm,\bh)\\ \hat\jmath \prec \bi_2,\pm 1}} \sum_{\varphi_1 \in \mathfrak B_{q^{(1)},4}}&\mathcal L(P_{1,\bk^{(1)},\bm^{(1)}}^{ (\varphi_1(\bh^{(1)};\bi_2,\pm 1;\tilde\imath,\pm 1);\hat\jmath,\hat\sigma)},Q_{\bk^{(2)},\bm^{(2)}}^{ (\bh^{(2)};\hat\jmath,-\hat\sigma)}) \nonumber\\ &+   \sum_{\varphi_2 \in \mathfrak B_{q^{(2)},4}} \mathcal L(P_{\bk^{(1)},\bm^{(1)}}^{ (\bh^{(1)};\hat\jmath,\hat\sigma)}, Q_{1,\bk^{(2)},\bm^{(2)}}^{ (\varphi_2(\bh^{(2)};\bi_2,\pm 1;\tilde\imath,\pm 1);\hat\jmath,-\hat\sigma)})\,.
	\end{aligned}
	\]
	\[
	\begin{aligned}
		(5):=  \sum_{\substack{C_v(\bk,\bm,\bh)\\ \hat\jmath \not\prec \bi_2}}   & \sum_{\varphi_1 \in \mathfrak B_{q^{(1)},4}}\mathcal L(\widehat P_{1,\bk^{(1)},\bm^{(1)}}^{ (\varphi_1(\bh^{(1)};\bi_2,\pm 1;\tilde\imath,\pm 1);\hat\jmath,\hat\sigma)}-P_{2,\bk^{(1)},\bm^{(1)}}^{ (\varphi_1(\bh^{(1)};\tilde\jmath,\pm 1;\tilde\imath,\pm 1);\hat\jmath,\hat\sigma)} , Q_{\bk^{(2)},\bm^{(2)}}^{ (\bh^{(2)};\hat\jmath,-\hat\sigma)}) \nonumber\\ &+   \sum_{\varphi_2 \in \mathfrak B_{q^{(2)},4}} \mathcal L(P_{\bk^{(1)},\bm^{(1)}}^{ (\bh^{(1)};\hat\jmath,\hat\sigma)}  ,  Q_{1,\bk^{(2)},\bm^{(2)}}^{ (\varphi_2(\bh^{(2)};\bi_2,\pm 1;\tilde\imath,\pm 1);\hat\jmath,-\hat\sigma)}-  Q_{2,\bk^{(2)},\bm^{(2)}}^{ (\varphi_2(\bh^{(2)};\tilde\jmath,\pm 1;\tilde\imath,\pm 1);\hat\jmath,-\hat\sigma)} )\,.
	\end{aligned}
	\]
	with $(\bh;\hat\jmath,\hat\sigma)\prec\tilde\jmath\prec\tilde\imath$.
	
		\noindent \underline{\emph{Step $2.3.$ : Bounds.}} 
		We shall systematically use \eqref{stimine} and define
		\begin{equation}
			\label{comelochiamo}
			C:= \!\!\! \sum_{\substack{q^{(1)}+q^{(2)}= q+2 \\ \bk^{(1)}+\bk^{(2)}=\bk \\ \bm^{(1)}+\bm^{(2)}=\bm }}q^{(1)}q^{(2)} |P|^{\Theta,q^{(1)}}_{\bk^{(1)},\bm^{(1)}}|Q|^{\Upsilon,q^{(2)}}_{\bk^{(2)},\bm^{(2)}}\,.
		\end{equation}
		From \eqref{elle2} and \eqref{stimine}, recalling also \eqref{stimabase1bis} and that $\Theta_{2,\bh}=1$, 
		we immediately obtain for $q\ge 4$ %\mic{Here, when i estimate $K_{n}$  I changed $\sharp \bh =q-2n$ so that the total degree is $q$ and this contributes  to $\mathfrak K^{\Theta}_{\bk,\bm,q}$, I dont know wheter its the best choice }
		\begin{equation*}
			\sup_{\substack{\bh\in \mathbb I_{q-4}\\ \bh\prec\bi_2\prec\bi_1\\ (\bk,\bm,\bh)\in \mathbb M}}|L^{\bh;\bi_1,\pm 1;\bi_2,\pm 1}_{2,\bk,\bm}| \Theta_{2,\bh}^{-1}  \lesssim C\,.
		\end{equation*}
		Regarding $L_{0,\bk,\bm}^{\bh}$ we obtain % for $q\ge 0$ 
		$|L^{(\bh)}_{0,\bk,\bm}| \Theta_{0,\bh}^{-1}  \lesssim C$ for all $\bh\notin \mathbf{C}_{q,1}$.
%		\begin{equation*}
%				\sup_{\substack{\bh\in \mathbb I_{q}\\ \bh\notin \mathbf{C}_{q,1}\\ (\bk,\bm,\bh)\in \mathbb M}}|L^{(\bh)}_{0,\bk,\bm}| \Theta_{0,\bh}^{-1}  \lesssim C
%				\end{equation*}
		Otherwise, for  $(\bh;\bi_1,\pm 1)$ with $\bh\prec \bi_1$ we have the following bounds for the terms in \eqref{kappa01}
		
		\noindent $\bullet$ In $(1)$ we have $ \Theta_{0,(\bh;\bi_1,\pm 1)}=\Theta_{0,(\bh^{(1)};\bi_1,\pm 1;\hat\jmath,\hat\sigma)}=\Theta_{0,(\bh^{(2)};\bi_1,\pm 1; \hat\jmath,-\hat\sigma)}$, so we bound e.g.
		\[
		|P_{0,\bk^{(1)},\bm^{(1)}}^{ (\varphi_1(\bh^{(1)};\bi_1,\pm 1);\hat\jmath,\hat\sigma)}Q_{\bk^{(2)},\bm^{(2)}}^{ (\bh^{(2)};\hat\jmath,-\hat\sigma)}|\Theta_{0,(\bh;\bi_1,\pm 1)}^{-1} \le 2  |P|^{\Theta,q^{(1)}+3}_{\bk^{(1)},\bm^{(1)}}|Q|^{\Upsilon,q^{(2)}+1}_{\bk^{(2)},\bm^{(2)}} 
		\]

		\noindent $\bullet$  In  $(2)$, we have that $\tilde\imath$ is an arbitrary  mode in $\mathcal S^c$ such that $\hat\jmath \prec \tilde\imath$. Note that, since $\hat\jmath\not\prec\bi_1$ and $\bh\prec\bi_1$, we have automatically that $\bh^{(s)}\prec\hat\jmath\prec \tilde\imath$ for $s=1,2$. %Thus    ${ (\bh^{(2)};\hat\jmath,-\hat\sigma)}\notin \mathbf C_{q^{(2)},1}$. 
		As a consequence 
		$\Upsilon_{0, (\bh^{(2)};\hat\jmath,-\hat\sigma)} \le \langle \bi_1 \rangle^{-1}$ and  $(\bh^{(2)};\hat\jmath,-\hat\sigma)\not\in \mathbf C_{q^{(2)}+1,1}$ so that 
		$$ 
		|Q_{\bk^{(2)},\bm^{(2)}}^{ (\bh^{(2)};\hat\jmath,-\hat\sigma)}|= |Q_{0,\bk^{(2)},\bm^{(2)}}^{ (\bh^{(2)};\hat\jmath,-\hat\sigma)}| \le   \frac{1}{\langle \bi_1 \rangle } |Q|^{\Upsilon,q^{(2)}+1}_{\bk^{(2)},\bm^{(2)}}.
		$$
		  Similarly 
		$$|Q_{1,\bk^{(2)},\bm^{(2)}}^{ (\bh^{(2)};\hat\jmath,-\hat\sigma);\tilde\imath,\pm 1)}|, |Q_{\bk^{(2)},\bm^{(2)}}^{ (\varphi_2(\bh^{(2)};\bi_1,\pm 1);\hat\jmath,-\hat\sigma)}|  \le \frac{1}{\langle \bi_1 \rangle } |Q|^{\Upsilon,q^{(2)}+3}_{\bk^{(2)},\bm^{(2)}}.$$
%		Similarly $(\bh^{(2)};\bi_1);\hat\jmath,-\hat\sigma)\notin\pmo$ (if $\bi_1\prec \hat\jmath$ then it is not in $\mathbf P$ otherwise if $\bi_1=\hat\jmath$ it is in $\mathbf O$) so that  again
%		$$
%		|Q_{\bk^{(2)},\bm^{(2)}}^{ (\varphi_2(\bh^{(2)};\bi_1);\hat\jmath,-\hat\sigma)}|= |Q_{0.\bk^{(2)},\bm^{(2)}}^{ (\varphi_2(\bh^{(2)};\bi_1);\hat\jmath,-\hat\sigma)}| \le 
%		\Upsilon_{0, (\bh^{(2)};\bi_1;\hat\jmath,-\hat\sigma)}\mathfrak Q^{\Upsilon}_{\bk^{(2)},\bm^{(2)},q^{(2)}+3}\le \frac{1}{|\bi_1| } \mathfrak Q^{\Upsilon}_{\bk^{(2)},\bm^{(2)},q^{(2)}+3}.
%		$$
		The factors coming from $P$ in line $(2)$  are just estimated by  $|P|^{\Theta}_{\bk^{(1)},\bm^{(1)}, q^{(1)}+ a}$ with $a=1,3$.
		
		\noindent $\bullet$  In  $(3)$, as explained in \eqref{stoc}, we may assume that $(\hat\jmath,\hat\sigma)\neq (\bi_1,-\sigma)$. 
		As a consequence  $(\bh^{(2)};\bi_1,- \sigma;\hat\jmath,-\hat\sigma) \notin \mathbf C_{q^{(2)}+2,1}$  and 
		\[
		\begin{aligned}
			|Q_{\bk^{(2)},\bm^{(2)}}^{ (\bh^{(2)};(\bi_1,-\sigma);\hat\jmath,-\hat\sigma} | &= |Q_{0,\bk^{(2)},\bm^{(2)}}^{ (\bh^{(2)};(\bi_1,-\sigma);\hat\jmath,-\hat\sigma} |\le  
			\frac{1}{|\bi_1| }| Q|^{\Upsilon,q^{(2)}+2}_{\bk^{(2)},\bm^{(2)}}.
		\end{aligned}
		\]
		
		In conclusion
	\[
	\begin{aligned}
		\sup_{\substack{\bh\in \mathbb I_{q-2}, \bi_1\in \mathcal S_p^c\\ \bh\prec \bi_1 \\ (\bk,\bm,\bh)\in \mathbb M}}|L^{(\bh;\bi_1,\pm 1)}_{0,\bk,\bm}| \Theta_{0,(\bh;\bi)}^{-1}  \lesssim C.
	\end{aligned}
	\]
	It remains to deal with $L_1$. As for $L_0$, we have $|L^{(\bh;\bi_1,\pm 1 )}_{1,\bk,\bm}| \Theta_{1,\bh}^{-1}\lesssim C$ if $h\not\in \mathbf C_{q-2,1}$ with $q\ge 2$. Now item (1) and (4) are  bounded as (1) in \eqref{kappa01}, items (2) and (5) are bounded as (2) of \eqref{kappa01}    and item (3) as (3) in \eqref{kappa01}.
	In conclusion we obtain
	\[
	\begin{aligned}
		\sup_{\substack{\bh\in \mathbb I_{q-4}, \bi_2,\bi_1 \in \mathcal S^c_p\\ \bh\prec \bi_2\prec\bi_1 \\ (\bk,\bm,\bh)\in \mathbb M}}|L^{(\bh;\bi_2,\pm 1;\bi_1,\pm 1)}_{1,\bk,\bm}| \Theta_{1,(\bh;\bi_2)}^{-1}  \lesssim C\,.
	\end{aligned}
	\]
		We have thus proved that $|L|_{\bk,\bm}^{\Theta,q} \lesssim C$.
	Then, we conclude as in the proof of Lemma \ref{lem:poisson_F}.
\end{proof}

\section{Translations} \label{sub:trans}
The aim of this subsection is to prove Proposition \ref{prop:open}.
The starting point is to pass form formal power series $P$ in the variables $(y,z,v)\in \R^{\mathcal S_p}\times \R^{\mathcal S_p} \times \C^{\mathcal S^c_p}$ to a formal power series $R$ in $(y_{\rm new},z_{\rm new},,v_{\rm new},)\in \R^{\mathcal S_{p+1}}\times \R^{\mathcal S_{p+1}} \times \C^{\mathcal S^c_{p+1}}$.  In view of condition 
\eqref{identita} we have to apply a  formal change of variables  defined as
%\[
%\begin{cases}
%	(y_j, z_j )&= (y^{\rm new}_j , z^{\rm new}_j )\;\forall j\in \mathcal S_p\,,\quad v_{j}= v^{\rm new}_j \,\forall j\in \mathcal S_{p+1}^c 
%\\
%v^{(\pm 1)}_j &= \sqrt{\xi_j^{\rm new} + y^{\rm new}_j} (z^{\rm new}_j )^{\pm 1} =\sqrt{\xi_j^{\rm new} }\sum_{m}\binom{1/2}{m} (\frac{ y^{\rm new}_j}{\xi^{\rm new}_j})^m
%\end{cases}
%\]
\begin{equation}
	\label{radiceq}
v^{(\pm 1)}_j = \sqrt{\xi_j^{\rm new} + y^{\rm new}_j} (z^{\rm new}_j )^{\pm 1} =\sum_{m}\binom{1/2}{m} (y^{\rm new}_j)^m (\xi^{\rm new}_j)^{1/2-m} \quad \forall j\in \mathcal S_{p+1} \setminus \mathcal S_{p}\,.
\end{equation}
and as the identity on all other variables. 
Our purpose is to show that not only this operator, which we denote by $T_p$, preserves the convergence of power series, but it also maintains the action-T\"oplitz structure. The first property is quite standard, we have the following
\begin{lemma}[Translation 1] \label{lem:trans_F1} Let $p\geq 0$, $\eta \in [\eta_{\min},9 \eta_{\min}]$, $\mathcal{O} \subset \mathcal{O}_p^{\mathrm{full}}$ be a Borel set and $\rho,r>0$ such that
	$$
	2r_p \geq \rho \gtrsim_{p} r_{p+1}^{1-\delta} \gtrsim_{p} r.
	$$
	Then, there exists a unique expansion operator
	$$
	T_p : \mathscr{F}_{\eta,\rho,p,\mathcal{O}} \to \mathscr{F}_{\eta,r,p+1,\mathcal{O}^{\sharp}}
	\quad \mathrm{where} \quad
	\mathcal{O}^\sharp := \mathcal{O}\times  \big(\frac{r_{p+1}^{2(1-\delta)}}2,  r_{p+1}^{2(1-\delta)}\big)^{\mathcal{S}_{p+1}\setminus \mathcal{S}_{p}},
	$$ 
	such that for all $P\in \mathscr{F}_{\eta,\rho,p,\mathcal{O}}$, $\xi=:(\xi^{\flat},\xi^{\mathrm{new}})\in \mathcal{O}^\sharp$ and $u\in \mathcal{A}_{\xi}(3r_{p+1})$  we have
	$$
	T_p P(\xi;u) = P(\xi^{\flat};u)
	$$
	and which satisfies the bounds
	$$
	\forall j\geq 1, \forall\alpha>0, \quad \| \mathrm{D}_j^\alpha T_p P  \|_{\mathscr{F}_{\eta,r,p+1,\mathcal{O}^\sharp}^{\mathrm{tot}}} \lesssim_{j,\alpha,p} \| \mathrm{D}_j^\alpha  P  \|_{\mathscr{F}_{\eta,\rho,p,\mathcal{O}}^{\mathrm{tot}}}
	$$
	where $D_j^\alpha$ denotes the (unbounded) multiplier defined by $(D_j^\alpha P)_{\bk,\bm}^{\bl,\bs} = \langle \bl_j^\ast \rangle^{\alpha} P_{\bk,\bm}^{\bl,\bs}$.
\end{lemma}
\begin{proof}
	We refer to \cite[Proposition 7.1.]{BGR24}, where a similar result is proved in detail. We emphasize that the proof is very explicit but tedious. It mainly involves performing the Taylor expansion of a square root as in \eqref{radiceq}, reordering the terms, and estimating them by controlling binomial coefficients. In particular, the coefficients of $T_p P$ are known in terms of those of $P$ (see formula (116) in \cite{BGR24}).
\end{proof}
To control the action-Toplitz norm of translations, we introduce some definitions.
 First, we define  the following useful operator
%\mic{there is something that slight;y bothers me: of course our definition of $\mathbf C_{q,n}$ depends on $p$ (because implicitly $\bl_i\in (S_p)^c$) \jo{I am not sure about this, it is not what I had in mind. Do we really use it implicitly ? But I agree with you that in the rest of the paper we only use it when $\bl_i\in S_p^c$}, however in this definition below it is not clear , I hesitate to propose this but maybe we could add a p in definition 2.18 or write below that we assume $\bl_i\in (S_p)^c$ and define for $p'\geq p$
%	\[
%	\mathbf{C}_{q,n}^{(p,p')} := \{ (\bl,\bs) \in \mathbf{C}^{(p)}_{q,n} \ | \ \bl_{2n}^* \succ \max_{j\in \mathcal{S}_{p'}}^{\prec} j \} \quad \mathrm{and} \quad \mathbf{C}_{n}^{(p)} := \bigcup_{q\geq 2n}   \mathbf{C}_{q,n}^{(p)} .   
%	\]}
%	\jo{We could. It involves more notations but if you think it avoid possible misunderstanding it could be worth.}
%	\mic{maybe a remark is better}
\begin{definition}[Projections $\Pi_{\mathrm{C}^{(p)}_n}$] Let $p\geq 0$. Given $q\geq 2n\geq 2$, we denote
$$
 \mathbf{C}_{q,n}^{(p)} := \{ (\bl,\bs) \in \mathbf{C}_{q,n} \ | \ \bl_{2n}^* \succ \max_{j\in \mathcal{S}_{p}}^{\prec} j \} \quad \mathrm{and} \quad \mathbf{C}_{n}^{(p)} := \bigcup_{q\geq 2n}   \mathbf{C}_{q,n}^{(p)} .   
 $$ 
We denote by $\Pi_{\mathrm{C}^{(p)}_n}$ the associated projection defined by restriction of the coefficients. 
\end{definition}
Note that, by construction
$$
\lim_{|\boldsymbol{i}_n|\to \infty} \cdots \lim_{|\boldsymbol{i}_1|\to \infty} P^{\boldsymbol{\ell},\boldsymbol{\sigma};\boldsymbol{i}_n,\pm 1; \dots ; \boldsymbol{i}_1,\pm 1 }_{\boldsymbol{k},\boldsymbol{m}} =  \lim_{|\boldsymbol{i}_n|\to \infty} \cdots \lim_{|\boldsymbol{i}_1|\to \infty} (\Pi_{\mathrm{C}^{(p)}_n} P)^{\boldsymbol{\ell},\boldsymbol{\sigma};\boldsymbol{i}_n,\pm 1; \dots ; \boldsymbol{i}_1,\pm 1 }_{\boldsymbol{k},\boldsymbol{m}} \,.
$$
%moreover if $(\bl,\bs) \in \mathbf{C}_{n} $  then $\Pi_{\mathrm{C}^{(p)}_n} T_p v^{\bs}_{\bl }= T_p \Pi_{\mathrm{C}^{(p)}_n} v^{\bs}_{\bl }$.
As a direct consequence of Lemma \ref{lem:trans_F1} and formula (116) in \cite{BGR24} we have 
\begin{corollary}\label{lem:trans_F}
The operator 	$T_p$ satisfies  the following  properties
	\begin{itemize}
		\item[(i)] for all $n\geq 1$, $\Pi_{\mathbf{C}^{(p+1)}_n} T_p = T_p \Pi_{\mathbf{C}^{(p+1)}_n} $, 
		\item[(ii)] for all $n\geq 1$,  $T_p \Pi_{\mathbf{C}_n^{(p+1)}} \mathscr{T}_{n,\eta,\rho,p,\mathcal{O}} \subset \Pi_{\mathbf{C}_n^{(p+1)}} \mathscr{T}_{n,\eta,r,p+1,\mathcal{O}^\sharp}^{\cste}$, 
		\item[(iii)] 
		$
		\Pi_{\mathrm{deg}\leq d-1} T_p P  =0 \quad \mathrm{where} \quad d= \min_{ P_{\boldsymbol{k},\boldsymbol{m}}^{\boldsymbol{\ell},\boldsymbol{\sigma}} \neq 0 } 2 \| \boldsymbol{m} \|_{\ell^1} + \# \{ j \ | \ \boldsymbol{\ell}_j \in \mathcal{S}_{p+1}^c\}
		$
		\item[(iv)] if $\Pi_{\mathrm{int}} P=0$ then $\Pi_{\mathrm{int}} R = 0$.
	\end{itemize}
\end{corollary}
\begin{definition}[Extension operator $E^{(p)}_n$] Let $p,n\geq 1$. We denote by $E^{(p)}_n$ the linear operator such that for all $r>0,\eta\geq 0$, $E^{(p)}_n$ maps  $\Pi_{\mathbf{C}_n^{(p)}} \mathscr{T}_{n,\eta,r,p}^{\cste}$ to $\mathscr{T}_{n,\eta,r,p}^{\cste}$ and $\Pi_{\mathbf{C}_n^{(p)}}  E^{(p)}_n = \mathrm{Id}$.
\end{definition}

We are in position to prove Proposition \ref{prop:open}.
\begin{proof}[Proof of Proposition \ref{prop:open}] We set $R:=T_p P$ and $P_n :=\Pi_{\mathrm{Top}=n} P$ for $n=0,1,2$. Thanks to Lemma \ref{lem:trans_F1}, we have that $R\in \mathscr{F}_{\eta,r,p+1,\mathcal{O}^{\sharp}}$ and it actually only remains to prove the $\mathscr{H}$ bound \eqref{eq:bound_open} on $R$.

First, we set
$$
R_0 := R -  E_{1}^{(p+1)}T_p \Pi_{\mathbf{C}^{(p+1)}_1} (P-P_0).
$$
%\mic{ $T_p P_0 +T_p \widehat P_1 - E_{1}^{(p+1)}T_p \Pi_{\mathbf{C}^{(p+1)}_1} \widehat P_1 $} 
Since, by construction $P-P_0\in \mathscr{T}_{1,\eta,\rho,p,\mathcal{O}}$, applying point \emph{(ii)} of Corollary  \ref{lem:trans_F}, it makes sense to apply the extension operator and so  $R_0$ is well defined. Then we note that
\begin{multline*}
\| D_1 \Pi_{\mathbf{C}_{1}^{(p+1)}} R_0\|_{\mathscr{F}_{\eta,r,p+1,\mathcal{O}^\sharp}^{\mathrm{tot}}} = \| D_1 \Pi_{\mathbf{C}_{1}^{(p+1)}} T_p (P -\Pi_{\mathbf{C}^{(p+1)}_1} (P-P_0)) \|_{\mathscr{F}_{\eta,r,p+1,\mathcal{O}^\sharp}^{\mathrm{tot}}} \\ = \| D_1  T_p \Pi_{\mathbf{C}_{1}^{(p+1)}} P_0 \|_{\mathscr{F}_{\eta,r,p+1,\mathcal{O}^\sharp}^{\mathrm{tot}}} 
\lesssim_p  \| D_1 \Pi_{\mathbf{C}_{1}^{(p+1)}} P_0 \|_{\mathscr{F}_{\eta,\rho,p,\mathcal{O}}^{\mathrm{tot}}} \lesssim_p \|P\|_{\mathscr{H}_{\eta,\rho,p,\mathcal{O}}^{\mathrm{tot}}}.
\end{multline*}
 Since, by the previous Lemma, $R_0  \in \mathscr{F}_{1,\eta,r,p+1,\mathcal{O}^\sharp}$, it implies that $R_0  \in \mathscr{S}_{1,\eta,r,p+1,\mathcal{O}^\sharp}$ and 
 $$
 \| R_0\|_{\mathscr{H}_{\eta,r,p+1,\mathcal{O}^\sharp}^{\mathrm{tot}}} \sim_p   \| R_0\|_{\mathscr{F}_{\eta,r,p+1,\mathcal{O}^\sharp}^{\mathrm{tot}}} + \| D_1 \Pi_{\mathbf{C}_{1}^{(p+1)}} R_0\|_{\mathscr{F}_{\eta,r,p+1,\mathcal{O}^\sharp}^{\mathrm{tot}}} \lesssim_p \|P\|_{\mathscr{H}_{\eta,\rho,p,\mathcal{O}}^{\mathrm{tot}}}
 $$
   by definition of the $\mathscr{H}$ norm.

Then, we set 
$$
R_2 := E_2^{(p+1)}T_p \Pi_{\mathbf{C}_2^{(p+1)}} P_2.
$$
Thanks to property \emph{(ii)} of Lemma \ref{lem:trans_F}, it is $2$-action T\"oplitz. Moreover, its $\mathscr{H}$ norm and its $\mathscr{F}$ norm coincide, so we have $\| R_2\|_{\mathscr{H}_{\eta,r,p+1,\mathcal{O}^\sharp}^{\mathrm{tot}}} \lesssim_p \|P\|_{\mathscr{H}_{\eta,\rho,p,\mathcal{O}}^{\mathrm{tot}}}$.

Finally, we set 
$$
R_1 := R-R_0-R_2 = E_{1}^{(p+1)}T_p \Pi_{\mathbf{C}^{(p+1)}_1} (P_1+P_2) -E_2^{(p+1)}T_p \Pi_{\mathbf{C}_2^{(p+1)}} P_2 .
$$
By definition of the extension operators it is $1$-action T\"oplitz. Thus, as for $R_0$, it suffices to prove that $\| D_3 \Pi_{\mathbf{C}_{2}^{(p+1)}} R_1\|_{\mathscr{F}_{\eta,r,p+1,\mathcal{O}^\sharp}^{\mathrm{tot}}} \lesssim_p \|P\|_{\mathscr{H}_{\eta,\rho,p,\mathcal{O}}^{\mathrm{tot}}}$.
Indeed, proceeding as for $R_0$, we have
\begin{equation*}
\begin{split}
\| D_3 \Pi_{\mathbf{C}_{2}^{(p+1)}} R_1\|_{\mathscr{F}_{\eta,r,p+1,\mathcal{O}^\sharp}^{\mathrm{tot}}}  
= \| D_3  \Pi_{\mathbf{C}_{2}^{(p+1)}} T_p ( \Pi_{\mathbf{C}_1^{(p+1)}} (P_1+P_2) - \Pi_{\mathbf{C}_2^{(p+1)}} P_2) \|_{\mathscr{F}_{\eta,r,p+1,\mathcal{O}^\sharp}^{\mathrm{tot}}} \\
= \| D_3  T_p   \Pi_{\mathbf{C}_2^{(p+1)}} P_1\|_{\mathscr{F}_{\eta,r,p+1,\mathcal{O}^\sharp}^{\mathrm{tot}}} \lesssim_p \ \| D_3  \Pi_{\mathbf{C}_2^{(p+1)}} P_1 \|_{\mathscr{F}_{\eta,\rho,p,\mathcal{O}}^{\mathrm{tot}}} 
\lesssim_p \|P\|_{\mathscr{H}_{\eta,\rho,p,\mathcal{O}}^{\mathrm{tot}}}.
\end{split}
\end{equation*}
\end{proof}

\section{Boundedness of a non trivial projection}

\begin{lemma} \label{lem:non_trivial_proj}Let $p\geq 1$, $\eta \geq 0$, $r>0$, $R\in \Pi_{4\leq \mathrm{deg}  \leq 5}\mathscr{H}_{\eta,r,p}^{\cste}$ and set $B\in \Pi_{4\leq \mathrm{deg}  \leq 5}\mathscr{F}_{\eta,r,p}^{\cste}$ by
	\begin{equation}
		\label{eq:def_B_app}
		B_{\boldsymbol{k} ,\boldsymbol{m} }^{\boldsymbol{\ell} ,\boldsymbol{\sigma} }  =   R_{\boldsymbol{k} ,\boldsymbol{m} }^{\boldsymbol{\ell} ,\boldsymbol{\sigma} } \mathds{1}_{\# \{j \ | \ \boldsymbol{\ell}_j \in \mathcal{S}_{p}\}\geq q+2\|\boldsymbol{m} \|_{\ell^1}-3}
	\end{equation}
	Then, $B\in \Pi_{4\leq \mathrm{deg}  \leq 5}\mathscr{H}_{\eta,r,p}^{\cste}$, satisfies the bound
	\begin{equation}
		\label{eq:petite_borne}
		\| B\|_{\mathscr{H}_{\eta,r,p}^{\cste}} \lesssim_p  \| R\|_{\mathscr{H}_{\eta,r,p}^{\cste}}.
	\end{equation}
\end{lemma}
\begin{proof} We divide the proof in two steps.

	\noindent \underline{\emph{Step 1: algebraic decomposition.}} First, we aim at proving that 
	$$
	B=B_0 +B_1  
	$$
	for some explicit Hamiltonians $B_0 \in \mathscr{S}_{1,\eta,r,p}^{\cste}$ and $B_1 \in  \mathscr{T}_{1,\eta,r,p}^{\cste}  \cap \mathscr{S}_{2,\eta,r,p}^{\cste}$. Note that it will directly imply that  $\Pi_{\mathrm{Top}=2} B =0$.
	
	First, we define the Hamiltonian $B_1 \in \mathscr{T}_{1,\eta,r,p}^{\cste} $ by the relation
	\begin{equation}
		\label{eq:def_B1}
		B_{1,\boldsymbol{k} ,\boldsymbol{m} }^{\boldsymbol{\ell} ,\boldsymbol{\sigma};i,\pm 1 } =\big( \lim_{|\boldsymbol{i}_1| \to +\infty} R_{\boldsymbol{k} ,\boldsymbol{m} }^{\boldsymbol{\ell} ,\boldsymbol{\sigma};\boldsymbol{i}_1,\pm 1 } \big) \mathds{1}_{\# \{j \ | \ \boldsymbol{\ell}_j \in \mathcal{S}_{p}\}\geq q+2\|\boldsymbol{m} \|_{\ell^1}-1}, \quad \forall i\succ \max_{1\leq j\leq q}^{\prec} \boldsymbol{\ell}_j
	\end{equation}
	where $\boldsymbol{k} \in \mathbb{Z}^{\mathcal{S}_{p-1}},\boldsymbol{m} \in \mathbb{N}^{\mathcal{S}_{p-1}},q\in \mathbb{N}, \boldsymbol{\ell} \in (\mathcal{S}_{p-1}^c)^q ,\boldsymbol{\sigma} \in  \{-1,1\}^q $.
	We notice that this limit exist because $R\in \mathscr{H}_{\eta,r,p}^{\cste}$. We notice that, by definition of $B$, we have
	\begin{equation}
		\label{eq:bien_vu}
		\lim_{|\boldsymbol{i}_1| \to +\infty} B_{\boldsymbol{k} ,\boldsymbol{m} }^{\boldsymbol{\ell} ,\boldsymbol{\sigma};\boldsymbol{i}_1,\pm 1 }  = B_{1,\boldsymbol{k} ,\boldsymbol{m} }^{\boldsymbol{\ell} ,\boldsymbol{\sigma};i,\pm 1 }, \quad  \forall i\succ \max_{1\leq j\leq q}^{\prec} \boldsymbol{\ell}_j.
	\end{equation}
	Then, we set 
	$$
	B_0 := B-B_1 \in  \mathscr{F}_{\eta,r,p}^{\cste}.
	$$
	As a consequence of \eqref{eq:bien_vu}, we directly note that $B_0 \in \mathscr{S}_{1,\eta,r,p}^{\cste}$. To prove that $B_1 \in \mathscr{S}_{2,\eta,r,p}^{\cste}$, we have to use the degree property. Indeed, assume by contradiction that there exists $\boldsymbol{i}_1,\boldsymbol{i}_2 \in \mathbb{Z}$, $\boldsymbol{k} \in \mathbb{Z}^{\mathcal{S}_{p-1}},\boldsymbol{m} \in \mathbb{N}^{\mathcal{S}_{p-1}},q\in \mathbb{N}, \boldsymbol{\ell} \in (\mathcal{S}_{p-1}^c)^q ,\boldsymbol{\sigma} \in  \{-1,1\}^q $ such that
	$$
	B_{1,\boldsymbol{k} ,\boldsymbol{m} }^{\boldsymbol{\ell} ,\boldsymbol{\sigma}; \boldsymbol{i}_2 ,\pm 1 ; \boldsymbol{i}_1 ,\pm 1 } \neq 0 \quad \mathrm{where} \quad \boldsymbol{i}_1 \succ \boldsymbol{i}_2 \succ \max_{i\in \mathcal{S}_{p-1}}^{\prec} i.  
	$$
	Using \eqref{eq:def_B1}, it would imply that 
	$$
	\big( \lim_{|i|\to +\infty} R_{\boldsymbol{k} ,\boldsymbol{m} }^{\boldsymbol{\ell} ,\boldsymbol{\sigma}; \boldsymbol{i}_2 ,\pm 1;i,\pm 1 } \big) \neq 0 \quad \mathrm{and} \quad  \# \{j \ | \ \boldsymbol{\ell}_j \in \mathcal{S}_{p}\}\geq q+2\|\boldsymbol{m} \|_{\ell^1}+1.
	$$
	Since $R\in \Pi_{4\leq \mathrm{deg}  \leq 5}\mathscr{H}_{\eta,r,p}^{\cste}$, the first condition implies that $\boldsymbol{m}=0$ and $0\leq q\leq 1$. Thus, it suffices to notice that $\# \{j \ | \ \boldsymbol{\ell}_j \in \mathcal{S}_{p}\}\leq q$ to get a contradiction. Thus we have proven that  
	\begin{equation}
		\label{eq:petite_observation}
		B_{1,\boldsymbol{k} ,\boldsymbol{m} }^{\boldsymbol{\ell} ,\boldsymbol{\sigma}; \boldsymbol{i}_2 ,\pm 1 ; \boldsymbol{i}_1 ,\pm 1 } = 0 \quad \mathrm{when} \quad \boldsymbol{i}_1 \succ \boldsymbol{i}_2 \succ \max_{i\in \mathcal{S}_{p-1}}^{\prec} i.  
	\end{equation}
	It directly follows that $B_1 \in \mathscr{S}_{2,\eta,r,p}^{\cste}$.

	\noindent \underline{\emph{Step 2: bound.}} To deduce that $B\in \mathscr{H}_{\eta,r,p}^{\cste}$ and to get the bound \eqref{eq:petite_borne}, it suffices to prove that
	$$
	\forall n\in \{0,1\}, \forall \boldsymbol{k} \in \mathbb{Z}^{\mathcal{S}_{p-1}},\forall \boldsymbol{m} \in \mathbb{N}^{\mathcal{S}_{p-1}},\forall q\in \mathbb{N}, \quad  |B_n|^{\Theta,q}_{\boldsymbol{k},\boldsymbol{m}}    \lesssim_p  |R|^{\Theta,q}_{\boldsymbol{k},\boldsymbol{m}} 
	$$
	where the notation  $|\cdot|^{\Theta,q}_{\boldsymbol{k},\boldsymbol{m}} $  is defined in Definition \ref{def:class_take3} of the $\mathscr{H}$ norm.
	
	\medskip
	
	The bound on $B_1$ is a direct consequence of \eqref{eq:petite_observation}. Indeed, it implies that 
	$$
	(1+ \max_{i\in \mathcal{S}_p} |i|)^{-1} |B_1|^{\Theta,q}_{\boldsymbol{k},\boldsymbol{m}}  \leq \sup_{\substack{  {\boldsymbol{\ell}}\in (\mathcal S_p^c)^{q}\\  \boldsymbol{\sigma}\in\{-1,1\}^{q}}}  \! \! \! \! \! \! |B^{\boldsymbol{\ell},\boldsymbol{\sigma}}_{1,{\boldsymbol{k}},{\boldsymbol{m}}}| \leq \sup_{\substack{  {\boldsymbol{\ell}}\in (\mathcal S_p^c)^{q}\\  \boldsymbol{\sigma}\in\{-1,1\}^{q}}}  \! \! \! \! \! \! |R^{\boldsymbol{\ell},\boldsymbol{\sigma}}_{{\boldsymbol{k}},{\boldsymbol{m}}}| \leq 3 |R|^{\Theta,q}_{\boldsymbol{k},\boldsymbol{m}} .
	$$
	The bound on $B_0$ requires more care. Applying the triangular inequality, it suffices to deal with the tree special cases where $R = \Pi_{\mathrm{Top}=n} R =: R_n$, $n=0,1,2$.

	\noindent  \underline{$\bullet$ \emph{Case 1: $R =R_0$}}. In that case, we have $B_1 = 0$ and so $B_0=B$. Since the coefficients of $B$ are smaller than or equal to those of $R$, it follows that directly that $|B_0|^{\Theta,q}_{\boldsymbol{k},\boldsymbol{m}} \leq |R|^{\Theta,q}_{\boldsymbol{k},\boldsymbol{m}} $.
	
	\noindent \underline{$\bullet$  \emph{Case 2: $R =R_1$ with $n\in \{1,2\}$}}. In that case if $B^{\boldsymbol{\ell},\boldsymbol{\sigma}}_{0,{\boldsymbol{k}},{\boldsymbol{m}}}\neq 0$ then $(\boldsymbol{\ell},\boldsymbol{\sigma})\in \mathbf{C}_{q,1} $ (because this property holds also for $B_1$ and $B$). Thus, it suffices to control special terms of the form
	$$
	B_{0,{\boldsymbol{k}},{\boldsymbol{m}}}^{\boldsymbol{\ell},\boldsymbol{\sigma};i,\pm 1} \quad \mathrm{where} \quad i \succ \max_{1 \leq j\leq q-2}^{\prec} \boldsymbol{\ell}_j
	$$ 
	and  $\boldsymbol{\ell} \in (\mathcal{S}_{p-1}^c)^{q-2} ,\boldsymbol{\sigma} \in  \{-1,1\}^{q-2}$. Recalling the definition \eqref{eq:def_B1} of $B_1$, in this special case, we have
	\begin{equation*}
		\begin{split}
			B_{0,{\boldsymbol{k}},{\boldsymbol{m}}}^{\boldsymbol{\ell},\boldsymbol{\sigma};i,\pm 1} &=(\mathds{1}_{ 2 \mathds{1}_{i\in \mathcal{S}_p} + \# \{j \ | \ \boldsymbol{\ell}_j \in \mathcal{S}_{p}\}\geq q+2\|\boldsymbol{m} \|_{\ell^1}-3} - \mathds{1}_{\# \{j \ | \ \boldsymbol{\ell}_j \in \mathcal{S}_{p}\}\geq q+2\|\boldsymbol{m} \|_{\ell^1}-3}) R_{n,{\boldsymbol{k}},{\boldsymbol{m}}}^{\boldsymbol{\ell},\boldsymbol{\sigma};i,\pm 1}
			\\ &= \mathds{1}_{\# \{j \ | \ \boldsymbol{\ell}_j \in \mathcal{S}_{p}\}<q+2\|\boldsymbol{m} \|_{\ell^1}-3}\mathds{1}_{i \in S_p} R_{n,{\boldsymbol{k}},{\boldsymbol{m}}}^{\boldsymbol{\ell},\boldsymbol{\sigma};i,\pm 1}
		\end{split}
	\end{equation*}
	(because $R_{n,{\boldsymbol{k}},{\boldsymbol{m}}}^{\boldsymbol{\ell},\boldsymbol{\sigma};i,\pm 1}=  \lim_{|\boldsymbol{i}_1| \to +\infty} R_{\boldsymbol{k} ,\boldsymbol{m} }^{\boldsymbol{\ell} ,\boldsymbol{\sigma};\boldsymbol{i}_1,\pm 1 }$ and $R\in \Pi_{4\leq \mathrm{deg}  \leq 5} R =R$). Thus, if $B_{0,{\boldsymbol{k}},{\boldsymbol{m}}}^{\boldsymbol{\ell},\boldsymbol{\sigma};i,\pm 1} \neq 0$, the largest index is $i$ and it has to belong to $\mathcal{S}_p$. So as previously, we have
	$$
	|B_1|^{\Theta,q}_{\boldsymbol{k},\boldsymbol{m}}  \leq  (1+ \max_{i\in \mathcal{S}_p} |i|) |R|^{\Theta,q}_{\boldsymbol{k},\boldsymbol{m}} .
	$$

\end{proof}

\section{A KAM theorem} In this subsection, we state precisely the KAM theorem we use in this paper. Its sole novelty lies in the formalism specific to this paper. Thus its proof is omitted. Nevertheless all the specific features related with our context are proven and explained in details in Subsection \ref{sub:KAM_and_Birk}. Concerning notations, the radii $\eta_p$ and $\eta_p^{(\alpha)}$ used in the statement are respectively defined by \eqref{eq:def_etap} and \eqref{eq:def_etapa}.

\begin{proposition} \label{prop:KAM} Let $p\geq 1$ and $\mathcal{O}_{p}^{(2)} \subset \mathcal{O}_{p}^{\mathrm{full}}$ be a Borel set. We set, for all $\xi \in \mathcal{O}_{p}^{(2)}$ and all $u\in \mathcal{A}_\xi(3r_{p})$  
	$$
	H^{(2)}(\xi;u) := Z_2(u)+Z_4(u) + \sum_{1\leq p'\leq p-1} Z^{(p')}(\xi;u) +  Z^{(p,2)}(\xi;u) +    J^{(2)}(\xi;u) + N^{(2)}(\xi;u) 
	$$
	where $\xi^{(p')} := (\xi_j)_{j\in \mathcal{S}_{p'}}$,
	\begin{itemize}
		\item $Z_2,Z_4$ are given by Definition \ref{def:Z_24} and, as in Proposition \ref{prop:res_coho}, for all $p'\in \{1,\cdots,p-1\}$, $Z^{(p')} \in  \Pi_{\mathrm{deg}\leq 4}  \Pi_{\mathrm{int}} \mathscr{H}_{\eta_{p'},r_{p'},p',\mathcal{O}_{p'}}$ is an integrable Hamiltonian  of degree smaller than or equal to $4$, defined for parameters in a Borel set $\mathcal{O}_{p'} \subset \mathcal{O}_{p'}^{\mathrm{full}}$ satisfying $\mathcal{O}_{p}  \subset \mathcal{O}_{p'} \times \mathbb{R}^{\mathcal{S}_{p} \setminus \mathcal{S}_{p'}} $ and satisfying the bound
		$$
		\| Z^{(p')} \|_{ \mathscr{H}_{\eta_{p'},r_{p'},p',\mathcal{O}_{p'}}^{\mathrm{tot}}} \leq r_{p'}^{4+1/3},
		$$
		\item $Z^{(p,2)} \in  \Pi_{\mathrm{deg}\leq 4}  \Pi_{\mathrm{int}} \mathscr{H}_{\eta_p,r_{p},p,\mathcal{O}_{p}^{(1)}}$ is an integrable Hamiltonian  of degree smaller than or equal to $4$ satisfying the bound
		$$
		\| Z^{(p,2)} \|_{ \mathscr{H}_{\eta_p,r_{p},p,\mathcal{O}_{p}^{(1)}}^{\mathrm{tot}}} \leq r_{p}^{4+1/2},
		$$
		\item  $ N^{(2)} \in \Pi_{\mathrm{deg} \geq 4 } \mathscr{H}_{\eta_{p}^{(2)},r_{p},p,\mathcal{O}_{p}^{(2)}}$ is a Hamiltonian, containing only terms of degree larger than or equal to $4$, no quartic integrable term (i.e. $\Pi_{\mathrm{deg} =4} \Pi_{\mathrm{int}}  N^{(2)} =0 $) and satisfying the bound
		$$
		\|  N^{(2)} \|_{ \mathscr{H}_{\eta_{p}^{(2)},r_{p},p,\mathcal{O}_{p}^{(2)}}^{\mathrm{tot}}}  \leq r_p^{4-\delta},
		$$
		\item  $ J^{(2)} \in \Pi_{\mathrm{deg} \leq 3} \mathscr{H}_{\eta_{p}^{(2)},r_{p},p,\mathcal{O}_{p}^{(2)}}$ is a Hamiltonian, containing only terms of degree smaller than $4$ and no integrable term (i.e. $\Pi_{\mathrm{int}}  J^{(2)} =0 $), satisfying the bound
		$$
		\|  J^{(2)} \|_{ \mathscr{H}_{\eta_{p}^{(2)},r_{p},p,\mathcal{O}_{p}^{(2)}}^{\mathrm{tot}}}  \leq r_p^{6-8\delta}.
		$$
	\end{itemize}
	Then, there exist
	\begin{itemize}
		\item a Borel set $\mathcal{O}_{p}^{(3)} \subset  \mathcal{O}_{p}^{(2)}$ satisfying
		$$
		\mathrm{meas}\Big(   \mathcal{O}_{p}^{(2)} \setminus \mathcal{O}_{p}^{(3)} \Big) \leq r_p^{\delta/8} \mathrm{meas}(\mathcal{O}_{p}^{\mathrm{full}} ).
		$$
		\item an integrable Hamiltonian $Z^{(p,3)} \in  \Pi_{\mathrm{deg}\leq 4}  \Pi_{\mathrm{int}} \mathscr{H}_{\eta_p,r_{p},p,\mathcal{O}_p^{(3)}}$ of degree smaller than or equal to $4$ satisfying the bound
		$$
		\| Z^{(p,3)} \|_{ \mathscr{H}_{\eta_p,r_{p},p,\mathcal{O}_p^{(3)}}^{\mathrm{tot}}} \leq 2 r_{p}^{4+1/2},
		$$
		\item a Hamiltonian $ J^{(3)} \in \Pi_{\mathrm{deg} = 3} \mathscr{H}_{\eta_{p}^{(3)},r_{p},p,\mathcal{O}_{p}^{(3)}}$,  containing only terms of degree $3$  and no integrable term (i.e. $\Pi_{\mathrm{int}}  J^{(3)} =0 $) satisfying the bound 
		$$
		\| J^{(3)} \|_{ \mathscr{H}_{\eta_{p}^{(3)},r_{p},p,\mathcal{O}_{p}^{(3)}}^{\mathrm{tot}}}  \leq 2 r_p^{6-8\delta },
		$$
		\item a Hamiltonian $N^{(3)} \in \Pi_{\mathrm{deg} \geq 4} \mathscr{H}_{\eta_{p}^{(3)},r_{p},p,\mathcal{O}_{p}^{(3)}}$,  containing only terms of degree larger than or equal to $4$ and no quartic integrable term (i.e. $\Pi_{\mathrm{deg} =4} \Pi_{\mathrm{int}} N^{(3)} =0 $), satisfying the bound 
		$$
		\| N^{(3)} \|_{ \mathscr{H}_{\eta_{p}^{(3)},r_{p},p,\mathcal{O}_{p}^{(3)}}^{\mathrm{tot}}}  \leq 2 r_p^{4-\delta},
		$$
	\end{itemize}
	such that, for all $\xi \in \mathcal{O}_{p}^{(3)}$ and all $u\in \mathcal{A}_\xi(2 r_p)$  
	$$
	H^{(2)}(\xi;\Psi_\xi^{(p)}(u)) =  Z_2(u)+Z_4(u) + \sum_{1\leq p'\leq p-1} Z^{(p')}(\xi;u) +  Z^{(p,3)}(\xi;u) + J^{(3)}(\xi;u) +  N^{(3)}(\xi;u) 
	$$
	where $\Psi_\xi^{(p)} \in C^2( \mathcal{A}_\xi(2r_p); \mathcal{A}_\xi(3 r_{p}) )$ is a symplectic map which is close to the identity, i.e. for all $ j\in \{0,1,2\}$, $v^{(1)},\cdots,v^{(j)}\in H^s$,
	$$
	\|\mathrm{d}^j \big(\Psi_\xi^{(p)} - \mathrm{Id}  \big)(u)  (v^{(1)},\cdots,v^{(j)}) \|_{H^s} \leq r_p^{2-j-\frac14} \prod_{1\leq i\leq j} \|v^{(i)}\|_{H^s}.
	$$
	
\end{proposition}
{\begin{proof}[Sketch of the proof] We apply a rapidly converging sequence of changes of variables $\Psi^{(p,i)}_\xi= \Phi^{(p,i)}_\xi \circ \Psi^{(p,i-1)}_\xi$ so that
		$$
		H^{(2)}(\xi;\Psi_\xi^{(p,i)}(u)) =  Z_2(u)+Z_4(u) + \sum_{1\leq p'\leq p-1} Z^{(p')}(\xi;u) +  Z^{(p,2,i)}(\xi;u) + J^{(2,i)}(\xi;u) +  N^{(2,i)}(\xi;u) 
		$$
		and $\Phi^{(p+1,i)}$ is the time one Hamiltonian flow generated by  solving three times  the cohomological  equation \eqref{eq:coho} with $P_d= \Pi_{\rm{deg}=d}J^{(2,i)} + \tilde P_d $ where $d=0,1,2$ is the homogeneity degree and $\tilde P_d$ is a triangular correction term. As is standard in KAM theory at each step of our iteration we have to impose non-resonance condition as in \eqref{eq:the_small_divisor_estimate} with a  different frequency $\omega^{\rm{full},i}$ and $\gamma_i = c_i r_p^{\delta/8}$ with $c_i\le 1$ and an  appropriate sequence which decreases to zero.
\end{proof}}

\end{document}